\documentclass[british,english,11pt]{article}
\usepackage[T1]{fontenc}
\usepackage[utf8x]{inputenc}
\usepackage{geometry}
\usepackage{babel}
\usepackage{float}
\usepackage{booktabs}
\usepackage{mathtools}
\usepackage{bm}
\usepackage{multirow}
\usepackage{amsmath}
\usepackage{amsthm}
\usepackage{amssymb}
\usepackage{graphicx}
\usepackage[authoryear]{natbib}
\usepackage[unicode=true,
 bookmarks=true,bookmarksnumbered=false,bookmarksopen=false,
 breaklinks=false,pdfborder={0 0 1},backref=false,colorlinks=false]
 {hyperref}
\hypersetup{
 colorlinks,linkcolor=red,anchorcolor=blue,citecolor=blue}

\makeatletter

\providecommand{\tabularnewline}{\\}
\floatstyle{ruled}
\newfloat{algorithm}{tbp}{loa}
\providecommand{\algorithmname}{Algorithm}
\floatname{algorithm}{\protect\algorithmname}

\theoremstyle{plain}
\newtheorem{thm}{\protect\theoremname}
\theoremstyle{remark}
\newtheorem{claim}[thm]{\protect\claimname}

\usepackage{babel}

\usepackage{cite}\usepackage{dsfont}\usepackage{array}\usepackage{mathrsfs}\usepackage{comment}\onecolumn
\usepackage{color}

\allowdisplaybreaks

\usepackage{enumitem}
\setlist[itemize]{leftmargin=1.5em}
\setlist[enumerate]{leftmargin=1.5em}

\usepackage{algorithm}% http://ctan.org/pkg/algorithms
\usepackage{algorithmic}% http://ctan.org/pkg/algorithmicx
\usepackage{arydshln}

\DeclareMathOperator{\ind}{\mathds{1}}  % Indicator

\numberwithin{equation}{section}

\definecolor{yxc}{RGB}{255,0,0}
\definecolor{yjc}{RGB}{125,0,0}
\definecolor{cm}{RGB}{0,0,200}
\definecolor{yly}{RGB}{0,150,0}

\addto\captionsbritish{\renewcommand{\algorithmname}{Algorithm}}

\makeatother

\addto\captionsbritish{\renewcommand{\algorithmname}{Algorithm}}
\addto\captionsbritish{\renewcommand{\claimname}{Claim}}
\addto\captionsbritish{\renewcommand{\theoremname}{Theorem}}
\addto\captionsenglish{\renewcommand{\claimname}{Claim}}
\addto\captionsenglish{\renewcommand{\theoremname}{Theorem}}
\providecommand{\claimname}{Claim}
\providecommand{\theoremname}{Theorem}

\begin{document}
\theoremstyle{plain} \newtheorem{lemma}{\textbf{Lemma}} \newtheorem{proposition}{\textbf{Proposition}}\newtheorem{theorem}{\textbf{Theorem}}\setcounter{theorem}{0}
\newtheorem{corollary}{\textbf{Corollary}} \newtheorem{assumption}{\textbf{Assumption}}
\newtheorem{example}{\textbf{Example}} \newtheorem{definition}{\textbf{Definition}}
\newtheorem{fact}{\textbf{Fact}}\newtheorem{property}{Property}
\theoremstyle{definition}

\theoremstyle{remark}\newtheorem{remark}{\textbf{Remark}}\newtheorem{condition}{Condition}\newtheorem{conjecture}{Conjecture} 
\title{\bf Robust Matrix Completion with Heavy-tailed Noise}
\author{Bingyan Wang \thanks{Department of Operations Research and Financial Engineering, Princeton University, Princeton, NJ 08544, USA; Email: \texttt{\{bingyanw,jqfan\}@princeton.edu}.} \and Jianqing Fan\footnotemark[1]}
\date{}

\maketitle

\begin{abstract}
This paper studies low-rank matrix completion in the presence of heavy-tailed and possibly asymmetric noise, where we aim to estimate an underlying low-rank matrix given
a set of highly incomplete noisy entries. Though the matrix completion
problem has attracted much attention in the past decade, there is still lack of theoretical understanding when the observations are contaminated by heavy-tailed noises. Prior theory
falls short of explaining the empirical results and is unable to capture
the optimal dependence of the estimation error on  the noise level. In this paper,
we adopt an adaptive Huber loss to accommodate heavy-tailed noise, which is robust
against large and possibly asymmetric errors when the parameter in the loss function is carefully designed to balance the Huberization biases and robustness to outliers. Then, we propose an efficient nonconvex algorithm via a balanced low-rank Burer-Monteiro matrix factorization and gradient decent with robust spectral initialization.  We prove that under merely bounded second moment condition on the error distributions, rather than the sub-Gaussian assumption, the Euclidean error of the iterates generated by the proposed algorithm decrease geometrically fast until achieving a minimax-optimal statistical estimation error, which has the same order as that in the sub-Gaussian case. The key technique behind this significant advancement is a powerful leave-one-out analysis framework. The theoretical results are corroborated by our simulation studies.
\end{abstract}


\tableofcontents{}

\section{Introduction\label{sec:Introduction}}

In a diverse array of real-world applications such as collaborative
filtering \citep{rao2015collaborative}, quantum-state tomography
\citep{gross2011recovering}, spectrum sensing \citep{corroy2011distributed}
and recommender system \citep{ramlatchan2018survey}, we are interested
in recovering a large-scale low-rank data matrix from noisy and highly
incomplete observations. This problem, usually termed as \emph{matrix
completion}, has attracted a lot of attention over a decade \citep{candes2010matrix,keshavan2010matrix,candes2011tight,ma2017implicit,chi2019nonconvex,chen2020nonconvex, chen2020noisy}.

Suppose the matrix of interest is $\bm{M}^{\star}=\{M_{i,j}^{\star}\}\in\mathbb{R}^{n\times n}$
of rank $r$ \footnote{Here we assume $\bm{M}^{\star}$ is a square matrix for simplicity
of presentation. It is straightforward to extend our results to the
case of a rectangular matrix $\bm{M}^{\star}\in\mathbb{R}^{n_{1}\times n_{2}}$. }, and we can observe a subset of noisy entries 
\[
M_{i,j}=M_{i,j}^{\star}+\varepsilon_{i,j},\qquad\left(i,j\right)\in\Omega,
\]
where $\varepsilon_{i,j}$ denotes the additive noise at index $(i,j)$,
and $\Omega\subseteq\{1,\cdots,n\}\times\{1,\cdots,n\}$ represents
sampling set. A variety of algorithms has been proposed for estimating
$\bm{M}^{\star}$, among which two paradigms have received much attention:
convex relaxation and nonconvex optimization. Both of them minimize
a form of regularized loss function \citep{candes2010matrix,chen2015fast,chen2020nonconvex}.
To enforce low-rank structure, the convex relaxation approach usually
adds a penalty term to the loss function and then implements well-developed
convex programming algorithms to obtain the estimates. There is vast
literature regarding this and Section \ref{sec:Prior-arts} 
provides a more detailed coverage on this part. However, a major drawback
of convex approach is that it suffers from high computational costs.
To remedy this issue, one turns to the nonconvex approach with a good initialization, which
often enjoys better computational performance and thus can be applied
to data of larger scale.  This will also be the focus of this paper.

For both convex and nonconvex methods, the vast majority of prior
literature relies heavily on the sub-Gaussian assumption of the noise
\citep{candes2010matrix,negahban2012restricted,
klopp2014noisy,chen2015fast,ma2017implicit,chen2020noisy}.
Under this assumption, regularized least-squares methods have been
proposed for nonconvex regularization and widely studied over the past decade, which usually takes
the form of

\begin{equation}
\underset{\bm{X},\bm{Y}\in\mathbb{R}^{n\times r}}{\mathrm{minimize}}\qquad\frac{1}{2p}\sum_{\left(i,j\right)\in\Omega}\left(\left(\bm{X}\bm{Y}^{\top}\right)_{i,j}-M_{i,j}\right)^{2}+\text{regularization terms},\label{eq:squareloss-obj}
\end{equation}
where the factorization $\bm{X} \bm{Y}^\top$ is used to model rank-$r$ matrix $\bm{M}^*$.
Theoretical studies are often conducted under the sub-Gaussian assumption, which can easily fail in many
modern applications.  In fact, heavy-tailed data is ubiquitous and can be encountered
in various domains such as functional magnetic resonance imaging \citep{eklund2016cluster},
financial markets \citep{cont2001empirical}, gene microarray analysis
\citep{wang2015high}, to name just a few. See also \citet{fan2021shrinkage} for additional examples where they also argue that by chance alone, some noises will have heavy tails in high dimensions.  Therefore, it is compelling to address the robustness issue in the matrix completion.

In this paper, we focus
on recovering a low-rank matrix from its highly incomplete subset
of entries contaminated by heavy-tailed noise. To accommodate the new
challenges here, we cannot stick to the least-squares formulation
\eqref{eq:squareloss-obj}, since it is well-known that square loss
can be vulnerable when dealing with heavy-tailed noise \citep{huber1973robust,catoni2012challenging}.
To address this issue, a natural solution is to resort to more robust
loss functions such as $\ell_{1}$-loss \citep{bassett1978asymptotic,huber2004robust},
Huber loss \citep{huber1973robust} and quantile loss \citep{koenker2001quantile}.
As in \citet{fan2017estimation} and \citet{sun2020adaptive}, we allow the distributions of $\varepsilon_{i,j}$ to be asymmetric so that $\tau$ should diverge appropriately in order to control the bias due to Huberization.  For this reason, the loss function is also referred to as the adaptive Huber loss, which will be our main focus.  We shall adopt the following nonconvex minimization problem
\begin{equation}
\underset{\bm{X},\bm{Y}\in\mathbb{R}^{n\times r}}{\mathrm{minimize}}\qquad\frac{1}{2p}\sum_{\left(i,j\right)\in\Omega}\rho_{\tau}\left(\left(\bm{X}\bm{Y}^{\top}\right)_{i,j}-M_{i,j}\right)+\frac{1}{8}\left\Vert \bm{X}^{\top}\bm{X}-\bm{Y}^{\top}\bm{Y}\right\Vert _{\mathrm{F}}^{2},\label{eq:huberloss-obj}
\end{equation}
where $\rho_{\tau}(\cdot)$ is the Huber loss function which will
be defined formally later and $\tau$ diverges at an appropriate rate. Here, the penalty term is intended to control
the balance between two low-rank factors $\bm{X}$ and $\bm{Y}$
which is crucial to the establishments of our theoretical guarantees
as we shall present later.  
Due to the fact that \eqref{eq:huberloss-obj}
is a highly nonconvex function, the objective function has numerous
local minima which prevents us from solving it easily by applying
some standard algorithm. This calls for a carefully designed algorithm
with provable performance guarantees. The constant $1/8$ is sufficient to guarantee the locally strong convexity of the objective function \eqref{eq:huberloss-obj} around the ground truth \citep{zheng2016convergence,tu2016low}.

\subsection{Comparison with prior theory}

\paragraph{Inadequacy of prior works.}

Among prior literature, matrix completion with heavy-tailed noise
has been studied by a series of papers.  \citet{elsener2018robust}
assumes a constant lower bound of the density function and adopts
Huber loss with nuclear norm penalty. \citet{minsker2018sub} deals with the case that
the noise has finite second moment and preprocesses the data before
feeding into the nuclear norm penalized square loss function. 
\citet{fan2021shrinkage} truncates the data before passing into the estimation 
method and it requires $2+\varepsilon$ moment of the noise to be finite. 
Despite the various assumptions and somewhat different formulations, their
statistical estimation errors behave similarly (cf.~Table \ref{table:comparison}).
They have a trailing term and prevent the upper bound from being
proportional to the noise level, even for sub-Gaussian noise. Consequently,
when the noise level is small, there would be a considerable gap between
their results and the optimal results available on sub-Gaussian noise \citep{ma2017implicit,chen2020noisy}.

Furthermore, a recent work \citep{shen2022computationally} proposes
a nonconvex Riemannian sub-gradient algorithm under the condition
that the noise is symmetric (which transfers means to medians) and some regularity conditions hold. Their
algorithm has improved over the convex approach in terms of computational
costs. However, their estimation error still has a trailing term as
listed in Table \ref{table:comparison}, which summarized the state-of-art theoretical results discussed above.

In view of these prior theories,  the following questions arise naturally. 
\begin{enumerate}
\item \emph{Is it possible to complete the matrix under only bounded second moment condition with the same rate of convergence as the sub-Gaussian case?}

\item \emph{Is it possible to close the theoretical gap by incorporating
the techniques of robust statistics into nonconvex optimization?} 

\item \emph{Is it possible to design an efficient nonconvex algorithm to
achieve the desired statistical accuracy? } 
\end{enumerate}
These questions are important but poorly understood.  They form the subject of this paper.

\paragraph{Our contribution. }

The current paper is devoted to providing a satisfactory answer to
the aforementioned questions. In a nutshell, we propose a two-stage nonconvex gradient
descent algorithm with a robust spectral initialization and establish theories to guarantee the optimality of its iterates after running for a sufficient number (logarithmically dependent on the model parameters) of iterations. Our result is also listed in Table \ref{table:comparison}, which, to the best of our knowledge, is the first one that achieves the optimal error rate under only bounded second moment condition (without assuming symmetric distribution). 

\begin{table}[t]
\caption{Comparison of our theoretical guarantees to prior theory, where we
hide all logarithmic factors. Here, the Euclidean estimation error
refers to $\|\bm{X}\bm{Y}^{\top}-\bm{M}^{\star}\|_{\mathrm{F}}$.}
\label{table:comparison} \vspace{0.8em}
 \centering %
\begin{tabular}{ccc}
\toprule 
$\vphantom{2_{2_{2_{2}}}^{2^{2^{2}}}}$  & \multirow{1}{*}{Algorithm} & Euclidean estimation error\tabularnewline
\midrule 
\citet{minsker2018sub}  & convex relaxation  & $\sqrt{\frac{\left(\sigma^{2}+\left\Vert \bm{M}^{\star}\right\Vert _{\infty}^{2}\right)rn}{p}}$\vphantom{$\frac{1^{7^{7^{7}}}}{1^{7^{7^{7}}}}$}\hspace{-0.4em}\tabularnewline
\midrule 
\citet{elsener2018robust}  & convex relaxation  & $\sqrt{\frac{\left(\tau^{2}+\left\Vert \bm{M}^{\star}\right\Vert _{\infty}^{2}\right)rn}{p}}$\vphantom{$\frac{1^{7^{7^{7}}}}{1^{7^{7^{7}}}}$}\hspace{-0.4em}\tabularnewline
\midrule 
\citet{fan2021shrinkage}  & convex relaxation  & $\sqrt{\frac{\left(\mu^{2}r+\sigma^{2}\right)rn}{p}}$\vphantom{$\frac{1^{7^{7^{7}}}}{1^{7^{7^{7}}}}$}\hspace{-0.4em}\tabularnewline
\midrule 
\citet{shen2022computationally}  & Riemannian sub-gradient  & $\max\left\{ \tau+\mathbb{E}\left|\varepsilon\right|,1+\mathbb{E}\left|\varepsilon\right|\tau^{-1}\right\} \sqrt{\frac{rn}{p}}$\vphantom{$\frac{1^{7^{7^{7}}}}{1^{7^{7^{7}}}}$}\hspace{-0.4em}\tabularnewline
\midrule 
\textbf{This paper}  & nonconvex GD  & $\sigma\sqrt{\frac{rn}{p}}$\vphantom{$\frac{1^{7^{7^{7}}}}{1^{7^{7^{7}}}}$}\hspace{-0.4em}\tabularnewline
\bottomrule
\end{tabular}
\end{table}

\subsection{Paper organization and notation}

The outline of the paper is as follows. Section \ref{sec:Main-results}
provides a formal statement of the model assumptions and presents
our main results. Section \ref{sec:Prior-arts} gives a review on
prior literature of matrix completion. Section \ref{sec:Numerical-experiments}
conducts numerical experiments that verify our theoretical results. Section \ref{sec:Proof-sketch}
gives a sketch of the proof techniques.  We conclude the paper in Section \ref{sec:Discussion}
by discussing several future directions.  All the proof details are
deferred to the Appendix.

Throughout the paper, for two functions $f(\cdot)$ and $g(\cdot)$,
we use the notations $f(n)\lesssim g(n)$ and $f(n)=O(g(n))$ to indicate
that there exists some constant $C_{1}>0$ such that $f(n)\leq C_{1}g(n)$
holds when $n$ is sufficiently large. Analogously, we adopt the notation
$f(n)\gtrsim g(n)$ to indicate that $f(n)\geq C_{2}g(n)$ for some
constant $C_{2}>0$ for all $n$ that are large enough. Moreover,
$f(n)\asymp g(n)$ means that $f(n)\lesssim g(n)$ and $f(n)\gtrsim g(n)$
hold simultaneously. In our proof, $C$ and $\widetilde{C}$ serve
as constants whose value might change from line to line.

Additionally, the matrix notations $\bm{X}$, $\bm{Y}\in\mathbb{R}^{n\times r}$
and $\bm{M}\in\mathbb{R}^{n\times n}$ shall be frequently used. 
%For any vector $\bm{v}$ and any matrix $\bm{M}$, we denote by $\bm{v}^{\top}$
%and $\bm{M}^{\top}$ their transposes, respectively. 
The notation $\|\bm{v}\|_{2}$ represents the $\ell_{2}$ norm of an vector $\bm{v}$,
and we let $\left\Vert \bm{M}\right\Vert $ and $\Vert\bm{M}\Vert_{\mathrm{F}}$
represent the spectral norm and the Frobenius norm of $\bm{M}$, respectively.
Moreover, we define $\Vert\bm{M}\Vert_{\infty}\coloneqq\max_{i,j}\vert M_{i,j}\vert$
and $\Vert\bm{X}\Vert_{2,\infty}\coloneqq\max_{i}\Vert\bm{X}_{i,\cdot}\Vert_{2}$.
We use $\mathcal{P}_{\Omega}(\cdot):\mathbb{R}^{n\times n}\mapsto\mathbb{R}^{n\times n}$
to stand for the projection onto the subspace of matrices whose support
is $\Omega$, i.e. 
\[
\left[\mathcal{P}_{\Omega}\left(\bm{M}\right)\right]_{i,j}=\begin{cases}
M_{i,j}, & \text{if\ }\left(i,j\right)\in\Omega\\
0, & \text{otherwise}
\end{cases}
\]
for any matrix $\bm{M}\in\mathbb{R}^{n\times n}$. Furthermore, for
any matrix $\bm{Z}$, we denote by $\bm{Z}_{\cdot,l}$ (resp.~$\bm{Z}_{l,\cdot}$)
the $l$th column (resp.~column) of $\bm{Z}$. For a function $f(\bm{X},\bm{Y})$,
we use $\nabla f_{\bm{X}}(\bm{X},\bm{Y})$ (resp.~$\nabla f_{\bm{Y}}(\bm{X},\bm{Y})$)
to denote the gradient of $f(\cdot)$ with respect to $\bm{X}$ (resp.~$\bm{Y}$).
For a non-singular matrix $\bm{R}\in\mathbb{R}^{r\times r}$ with
SVD $\bm{R}=\bm{U}_{\bm{R}}\bm{\Sigma}_{\bm{R}}\bm{V}_{\bm{R}}^{\top}$,
we define the orthogonal matrix $\mathsf{sgn}(\bm{R})$ by 
\begin{equation}
\mathsf{sgn}\left(\bm{R}\right)\eqqcolon\bm{U}_{\bm{R}}\bm{V}_{\bm{R}}^{\top}.\label{defn-sgn}
\end{equation}

\section{Robust matrix completion and main results\label{sec:Main-results}}

\subsection{Model and algorithm}

\paragraph{Model.}

As elucidated in Section \ref{sec:Introduction}, we are interested
in recovering a rank-$r$ matrix $\bm{M}^{\star}$. Let $\bm{M}^{\star}=\bm{U}^{\star}\bm{\Sigma}^{\star}\bm{V}^{\star\top}$
be the SVD of $\bm{M}^{\star}$ where $\bm{U}^{\star}$, $\bm{V}^{\star}\in\mathbb{R}^{n\times r}$
consists of orthogonal columns and $\bm{\Sigma}^{\star}\in\mathbb{R}^{r\times r}$
is a diagonal matrix with decreasing singular values $\sigma_{1}^{\star}\geq\sigma_{2}^{\star}\geq\cdots\geq\sigma_{r}^{\star}>0$.
Denote by $\kappa\coloneqq\sigma_{\max}^{\star}/\sigma_{\min}^{\star}$
the condition number of $\bm{M}^{\star}$, where $\sigma_{\max}^{\star}\coloneqq\sigma_{1}^{\star}$
and $\sigma_{\min}^{\star}\coloneqq\sigma_{r}^{\star}$. In addition,
let $\bm{X}^{\star}=\bm{U}^{\star}(\bm{\Sigma}^{\star})^{1/2}$ and
$\bm{Y}^{\star}=\bm{V}^{\star}(\bm{\Sigma}^{\star})^{1/2}$ be the
balanced low-rank factors of $\bm{M}^{\star}$, namely $\bm{X}^{\star} \bm{Y}^{\star}{}^\top = \bm{M}^{\star}$ and $\bm{X}^{\star}{}^\top \bm{X}^{\star}= \bm{Y}^{\star}{}^\top \bm{Y}^{\star}$. We consider the following assumptions
regarding the highly incomplete and noisy observations of $\bm{M}^{\star}$.

\begin{assumption}\label{assum}We assume 
\begin{enumerate}
\item \textbf{\emph{(Random sampling)}} The entry at each index $(i,j)$
can be observed independently with probability $p$, namely the entry missing at random with probability $1-p$.

\item \textbf{\emph{(Heavy-tailed noise)}} The noise matrix $\bm{E}=\{\varepsilon_{i,j}\}_{1\leq i,j\leq n}$
is composed of independent heteroskedastic noise with zero mean and bounded variance: 
\[
\mathbb{E}\left[\varepsilon_{i,j}\right]=0,\qquad\mathbb{E}\left[\varepsilon_{i,j}^{2}\right]=\sigma_{i,j}^{2}\leq\sigma^{2}.
\]
\end{enumerate}
\end{assumption}

Note that the heavy-tailed noise should be contrasted to the sub-Gaussian assumption in high-dimenison.  In particular, the bounded second moment can include distribution such as
the mixture normal $(1-\delta)* N(0, 1) + \delta N(0, \delta^{-1})$ with $\delta \to 0$.  It can contain data points with outliers of order $n^{1/2}$ among $n$ data points by taking $\delta \asymp n^{-1}$ or $O(n^{1/4})$ numbers of outliers of order $n^{3/8}$ by taking $\delta \asymp n^{-3/4}$.
As introduced before, the robust nonconvex problem
studied here is 
\begin{equation}
\min_{\bm{X},\bm{Y}\in\mathbb{R}^{n\times r}}\qquad f\left(\bm{X},\bm{Y}\right)=\frac{1}{2p}\sum_{\left(i,j\right)\in\Omega}\rho_{\tau}\left(\left(\bm{X}\bm{Y}^{\top}\right)_{i,j}-M_{i,j}\right)+\frac{1}{8}\left\Vert \bm{X}^{\top}\bm{X}-\bm{Y}^{\top}\bm{Y}\right\Vert _{\mathrm{F}}^{2},\label{eq:obj}
\end{equation}
where the Huber loss function with parameter $\tau$ is defined as
\begin{equation}
\label{defn:huber}
\rho_{\tau}\left(x\right)\coloneqq\begin{cases}
x^{2}/2, & \mathrm{if\ }\left|x\right|\leq\tau,\\
\tau\left|x\right|-\tau^{2}/2, & \mathrm{if\ }\left|x\right|>\tau.
\end{cases} 
\end{equation}
The regularization term in \eqref{eq:obj} is widely employed in the
literature \citep{zheng2016convergence,tu2016low,chen2020nonconvex}
to control the discrepancy or balance between $\bm{X}$ and $\bm{Y}$. It accommodates
an unavoidable scaling issue underlying this model, since there is
no hope to distinguish between $(\bm{X}\bm{R},\bm{Y}\bm{R}^{-\top})$
with $\bm{R}\in\mathbb{R}^{r\times r}$ being any invertible matrix
and $(\bm{X},\bm{Y})$ given only observations based on $\bm{X}\bm{Y}^{\top}$.

\paragraph{Algorithm. }

This paper considers an algorithm consists of two stages: (i) robust spectral
initialization which would generate a consistent yet not optimal initial
estimate, (ii) a gradient descent (GD) algorithm which update the
estimate iteratively. It can be seen momentarily that the initial
estimate given by (i) would fall into a \emph{local} region in the
neighborhood of the global minimum where restricted strong convexity
holds true, and then the GD algorithm can iteratively refine the estimates
within the local region. The complete algorithm is summarized in Algorithm
\ref{alg:gd-rmc}. 
\begin{itemize}
\item \emph{Spectral initialization.} Due to the nonconvex landscape, nonconvex
algorithms typically require initialization point with \emph{good}
properties to avoid getting stuck into some highly sub-optimal local
minima. To achieve this goal, in the first stage of Algorithm \ref{alg:gd-rmc},
we initialize the algorithm by the top-$r$ SVD of \eqref{eq:spectral-method-matrix}
where 
\begin{equation}
\psi_{\tau}\left(t\right)=\frac{\partial}{\partial t}\rho_{\tau}\left(t\right),\label{defn-psi}
\end{equation}
is the truncation (Winsorization) operator.
Under this definition, $\bm{M}^{0}$ given by \eqref{eq:spectral-method-matrix}
is a nearly unbiased estimator of $\bm{M}^{\star}$ when $\tau$ is sufficiently large, suggesting that
the top-$r$ SVD of $\bm{M}^{0}$ shall be a proper estimate for the
low rank factors of $\bm{M}^{\star}$. 
\item \emph{Gradient descent.} In the second stage, we proceed by performing
gradient descent iteratively to refine our estimates. As we shall
see momentarily, the number of iterations $t_{0}$ is logarithmically
dependent on the model parameters. This implies superior computational
performance of Algorithm \ref{alg:gd-rmc}, since in each iteration,
we only need to compute the gradient $\nabla f(\cdot)$ and update
the estimates $\bm{X}^{t}$ and $\bm{Y}^{t}$. In addition, the step size
$\eta$ is fixed throughout iterations. We will state how to choose
its value shortly. 
\end{itemize}
\begin{algorithm}[h]
\caption{Gradient descent for robust matrix completion (with spectral initialization)}

\label{alg:gd-rmc}\begin{algorithmic}

\STATE \textbf{{Input}}: data matrix $\bm{M}$, sampling set $\Omega$,
rank $r$, observation probability $p$, and maximum number of iterations
$t_{0}$.

\STATE \textbf{{Spectral initialization}}: let $\bm{U}^{0}\bm{\Sigma}^{0}\bm{V}^{0\top}$
be the top-$r$ SVD of 
\begin{equation}
\bm{M}^{0}\coloneqq\frac{1}{p}\mathcal{P}_{\Omega}\left(\psi_{\tau}\left(\bm{M}\right)\right),\label{eq:spectral-method-matrix}
\end{equation}
(see \eqref{defn-psi} for definition of $\psi_{\tau}(\cdot)$) and
set $\bm{X}^{0}=\bm{U}^{0}\left(\bm{\Sigma}^{0}\right)^{1/2}$, $\bm{Y}^{0}=\bm{V}^{0}\left(\bm{\Sigma}^{0}\right)^{1/2}$.

\STATE \textbf{{Gradient updates}}: \textbf{for }$t=0,1,\ldots,t_{0}-1$
\textbf{do}

\STATE \vspace{-1em}
 
\begin{subequations}
\label{subeq:gradient_update_ncvx} 
\begin{align}
\bm{X}^{t+1} & =\bm{X}^{t}-\eta\nabla f_{\bm{X}}\left(\bm{X}^{t},\bm{Y}^{t}\right)\label{eq:gd-x}\\
\bm{Y}^{t+1} & =\bm{Y}^{t}-\eta\nabla f_{\bm{Y}}\left(\bm{X}^{t},\bm{Y}^{t}\right)\label{eq:gd-y}
\end{align}
\end{subequations}
where $\nabla f_{\bm{X}}(\cdot)$ and $\nabla f_{\bm{Y}}(\cdot)$
represent the gradient of $f(\cdot)$ given by \eqref{eq:obj} w.r.t.~$\bm{X}$ and $\bm{Y}$,
respectively.

\end{algorithmic} 
\end{algorithm}

\subsection{Theoretical guarantees}

In this section, we present our theory for Algorithm \ref{alg:gd-rmc}
and elaborate on the implications of our results.

Before proceeding to the main results, we introduce a crucial condition
on $\bm{M}^{\star}$, which allows for reliable estimation schemes.
It is standard and widely adopted in the literature of matrix completion
\citep{candes2009exact,candes2010matrix,chen2015incoherence,chen2015fast,sun2016guaranteed}.

\begin{definition}\text{(Incoherence).}\label{defn:incoherence}A
rank-$r$ matrix $\bm{M}^{\star}$ with SVD $\bm{M}^{\star}=\bm{U}^{\star}\bm{\Sigma}^{\star}\bm{V}^{\star\top}$
is said to satisfy the incoherence condition with parameter $\mu$
if 
\[
\left\Vert \bm{U}^{\star}\right\Vert _{2,\infty}\leq\sqrt{\frac{\mu}{n}}\left\Vert \bm{U}^{\star}\right\Vert _{\mathrm{F}}=\sqrt{\frac{\mu r}{n}},\qquad\text{and}\qquad\left\Vert \bm{V}^{\star}\right\Vert _{2,\infty}\leq\sqrt{\frac{\mu}{n}}\left\Vert \bm{V}^{\star}\right\Vert _{\mathrm{F}}=\sqrt{\frac{\mu r}{n}}.
\]
\end{definition}

We note an identifiability issue underlying this
problem, namely, given any orthonormal matrix $\bm{R}\in\mathbb{R}^{r\times r}$,
there always holds $\bm{X}^{\star}\bm{Y}^{\star\top}=\bm{X}^{\star}\bm{R}(\bm{Y}^{\star}\bm{R})^{\top}$.
In view of this, when measuring the discrepancy between 
\[
\bm{F}^{t}\coloneqq\left[\begin{array}{c}
\bm{X}^{t}\\
\bm{Y}^{t}
\end{array}\right]\qquad\text{and}\qquad\bm{F}^{\star}\coloneqq\left[\begin{array}{c}
\bm{X}^{\star}\\
\bm{Y}^{\star}
\end{array}\right],
\]
we shall consider the distance metric modulo the global rotation matrix
$\bm{H}^{t}$ which best align $\bm{F}^{t}$ and $\bm{F}^{\star}$
in the sense that 
\begin{align}
\bm{H}^{t} & \eqqcolon\arg\min_{\bm{R}\in\mathcal{O}^{r\times r}}\left\Vert \bm{F}^{t}\bm{R}-\bm{F}^{\star}\right\Vert _{\mathrm{F}}
=\mathsf{sgn}(\bm{F}^t{}^\top \bm{F}^*),\label{eq:def-ht}
\end{align}
where $\mathsf{sgn}(\cdot)$ is defined in \eqref{defn-sgn}.
With these notions in hand, we are ready to present the main theorems.
To begin with, the first theorem below presents the theoretical results
when the condition number $\kappa$, the incoherence parameter $\mu$,
and the rank $r$ of $\bm{M}^{\star}$ are all constantly bounded.
It makes the requirement of sample size and noise level clearer to
recognize.   It also presents the property of the robust special spectral method as the initialization.  

\begin{theorem}\label{thm:main}Let $\bm{M}^{\star}$ be rank-$r$
and $\mu$-incoherent with condition number $\kappa$. Suppose $\kappa$,
$\mu$, $r\sim O(1)$ and Assumption \ref{assum} holds. Take $\tau=C_{\tau}(\Vert\bm{M}^{\star}\Vert_{\infty}+\sigma\sqrt{np})$
for some large enough constant $C_{\tau}>0$. Assume the sample size
and the noise level satisfy 
\begin{equation}
n^{2}p\geq Cn\log^{2}n\qquad\text{and}\qquad\sigma\sqrt{\frac{n}{p}}\leq c\frac{\sigma_{\min}}{{\log n}},\label{eq:condition}
\end{equation}
where $C>0$ is some large enough constant and $c>0$ is some sufficiently
small constant. Then with probability exceeding $1-O(n^{-3})$, the
iterates of Algorithm \ref{alg:gd-rmc} obey 
\begin{align}
\left\Vert \bm{F}^{0}\bm{H}^{0}-\bm{F}^{\star}\right\Vert _{\mathrm{F}} & \leq C_{0}\left(\frac{\sigma}{\sigma_{\min}}+\frac{\left\Vert \bm{M}^{\star}\right\Vert _{\infty}}{\sigma_{\min}}\right)
\sqrt{\frac{n}{p}} \left\Vert \bm{F}^{\star}\right\Vert _{\mathrm{F}},\label{eq:thm-initial}\\
\left\Vert \bm{F}^{t}\bm{H}^{t}-\bm{F}^{\star}\right\Vert _{\mathrm{F}} & \leq\rho^{t}\left\Vert \bm{F}^{0}\bm{H}^{0}-\bm{F}^{\star}\right\Vert _{\mathrm{F}}+C_{1}\frac{\sigma}{\sigma_{\min}}\sqrt{\frac{n}{p}}\left\Vert \bm{F}^{\star}\right\Vert _{\mathrm{F}}. \label{eq:thm-decay}
\end{align}
for all $0\leq t\leq t_{0}=O(n^{5})$, where $C_{0}$ and $C_{1}$
are some absolute constants and $\rho=1-\frac{\sigma_{\min}}{20}\eta$,
as long as $0\leq\eta\leq c'/(\sigma_{\max}\log n)$ for some small
constant $c'>0$.
\end{theorem}

\begin{remark} The analysis behind Theorem \ref{thm:main} remains
valid as long as the total number of iterations $t_{0}$ is polynomially
dependent on the problem dimension, i.e. $t_{0}=O(n^{c})$ for some
constant $c>0$.  To make the contraction term in \eqref{eq:thm-decay} neglible, the number of iterations should be at least of order $\log (\left\Vert \bm{M}^{\star}\right\Vert/\sigma)$.
\end{remark}

\begin{remark} 
Theorem~\ref{thm:main} contains both the results for the robust spectral inialization and its subsequence iterates.  On the way of our proof, we also establish
\begin{align}
 \left\Vert \bm{F}^{t}\bm{H}^{t}-\bm{F}^{\star}\right\Vert _{2,\infty} & \lesssim\kappa^{1.5}\sqrt{r}\left(\frac{\sigma}{\sigma_{\min}} +\frac{\left\Vert \bm{M}^{\star}\right\Vert _{\infty}}{\sigma_{\min}}\right)  \sqrt{\frac{n}{p}} \log n\left\Vert \bm{F}^{\star}\right\Vert _{2,\infty}.\label{eq:hyp-2infty}
\end{align}
We expect this result can be improved further.
\end{remark}

Despite its simplicity, Theorem \ref{thm:main} reveals deep insights
into the core idea of our newly-developed theoretical understanding
towards robust matrix completion. As can been seen from \eqref{eq:thm-decay},
Theorem \ref{thm:main} guarantees that the robust spectral initialization point $\bm{F}^{0}$
falls close enough to the ground truth $\bm{F}^{\star}$ and the estimation
error of the iterates $\{\bm{X}^{t},\bm{Y}^{t}\}_{t>0}$ generated
by the gradient descent step decay geometrically fast until reaching
some error floor. This behavior will be further illustrated numerically
in Section \ref{sec:Numerical-experiments}. A few remarks are in
order. 
\begin{itemize}
\item \textbf{Minimax optimality.} An immediate consequence of the theorem is that 
\begin{align*}
\left\Vert \bm{X}^{t}\bm{Y}^{t\top}-\bm{M}^{\star}\right\Vert _{\mathrm{F}} & \leq\widetilde{C}\rho^{t} \left(\sigma+\left\Vert \bm{M}^{\star}\right\Vert _{\infty}\right)\sqrt{\frac{n}{p}} +C_{1}\sigma\sqrt{\frac{n}{p}}.
\end{align*}
Consequently, as $t$ increases, $\Vert\bm{X}^{t}\bm{Y}^{t\top}-\bm{M}^{\star}\Vert_{\mathrm{F}}$
converges to $\sigma\sqrt{\frac{n}{p}}$, matching the lower bound
developed in \citet{koltchinskii2011nuclear,negahban2012restricted}
in the presence of sub-Gaussian noise. This confirms the minimaxity
of nonconvex optimization for matrix completion with heavy-tailed
noise. Furthermore, this implies that when addressed properly, heavy-tailed
noise can even behave analogously to sub-Gaussian noise in matrix
completion \citep{chen2020noisy}. Compared with the other methods
listed in Table \ref{table:comparison}, this error level gets rid
of the trailing term and is proportional to the noise level $\sigma$
even as $\sigma$ becomes vanishingly small, which also coincides
with our intuition. 
\item \textbf{Fast convergence.} In stark contrast to the convex approaches,
which usually suffer from high computational costs, the gradient descent
algorithm here is easy to implement and demonstrates linear convergence
with contraction rate $\rho$, resulting in an iteration complexity
scaling logarithmically with the model parameters. Hence it is straightforward
to see that to reach the error floor, the computational complexity
is $\widetilde{O}(n)$ (up to some log factors) under our sample size
condition, which is almost the best we can expect since the time spent
loading the data is $\widetilde{O}(n^2)$ in our case. Note that previous
work \citep{shen2022computationally} also achieves geometric convergence,
while its estimation error is not optimal. 
\item \textbf{Minimal sample size and noise conditions.} As stated in Theorem
\ref{thm:main}, when $\kappa$, $\mu$ and $r$ are all $O(1)$,
the sample size requirement scales as 
\begin{equation}
n^{2}p\gtrsim n\mathrm{poly}\log\left(n\right),\label{eq:samplesize}
\end{equation}
which matches the information-theoretic lower limit even in the absence
of noise \citep{candes2009exact,candes2010matrix}. 

Under the same
conditions, the noise level requirement \eqref{eq:condition} in our
main theorem is $\sigma\sqrt{n\log^2 n/p}\lesssim\sigma_{\min}$. Therefore,
if we adopt the following definition of signal-to-noise ratio (SNR)
\[
\mathsf{SNR}\coloneqq\frac{\mathbb{E}\left[\left\Vert \mathcal{P}_{\Omega}\left(\bm{M}^{\star}\right)\right\Vert _{\mathrm{F}}^{2}/\left|\Omega\right|\right]}{\sigma^{2}},
\]
the noise level requirement \eqref{eq:condition} implies that our theory will work as long 
as 
\[
\mathsf{SNR} = \frac{\left\Vert \bm{M}^{\star}\right\Vert _{\mathrm{F}}^{2}}{n^{2}\sigma^{2}} \gtrsim\frac{\log n}{np}.
\]
The lower bound in the above equation can be vanishingly small in
view of the sample size condition \eqref{eq:samplesize}. This shows
that our theory works even in the low-SNR regime. 

Furthermore, it is worth noting that Theorem \ref{thm:main} removes the symmetric noise assumption which is required in previous works \citep{elsener2018robust, shen2022computationally}.  This will also be verified shortly by the numerical experiments reported in Section \ref{sec:Numerical-experiments}.  Note that the symmetric noise assumption makes Huberization bias zero and problem becomes easier.

\item \textbf{Implicit regularization.} On closer inspection of the works
listed in Table \ref{table:comparison}, the convex problems need
either the constraint $\Vert\bm{M}\Vert_{\infty}\leq\alpha$ \citep{elsener2018robust,fan2021shrinkage}
or the nuclear norm penalty $\Vert\bm{M}\Vert_{*}$ in loss function
\citep{elsener2018robust,minsker2018sub, fan2021shrinkage}, while
these can be removed in our study by a powerful entrywise control,
as we shall elaborate on shortly in Section \ref{sec:Proof-sketch}.
This indicates that gradient descent can implicitly bound the spikiness
of the estimates. 
\end{itemize}

Next, we present the more geneal case where $\kappa$, $\mu$, $r$ are allowed
to grow with $n$.  It allows us to examine the explicit dependence on $\kappa$, $\mu$, $r$, which is not available in 
Theorem \ref{thm:main}.

\begin{theorem}\label{thm:main-1}Let $\bm{M}^{\star}$ be rank-$r$
and $\mu$-incoherent with condition number $\kappa$. Suppose Assumption
\ref{assum} holds and take $\tau=C_{\tau}(\Vert\bm{M}^{\star}\Vert_{\infty}+\sigma\sqrt{np})$
for some large enough constant $C_{\tau}>0$. Assume the sample size
and the noise level satisfy 
\begin{equation}
n^{2}p\geq C\kappa^{6}\mu^{2}r^{4}n\log^{2}n\qquad\text{and}\qquad\sigma\sqrt{\frac{n}{p}}\leq c\frac{\sigma_{\min}}{\sqrt{\kappa^{4}\mu r^{2}\log^{2}n}},\label{eq:condition-1}
\end{equation}
where $C>0$ is some large enough constant and $c>0$ is some sufficiently
small constant. Then with probability exceeding $1-O(n^{-3})$, the
iterates of Algorithm \ref{alg:gd-rmc} obey 
\begin{align}
\left\Vert \bm{F}^{t}\bm{H}^{t}-\bm{F}^{\star}\right\Vert _{\mathrm{F}} & \leq\rho^{t}\left\Vert \bm{F}^{0}\bm{H}^{0}-\bm{F}^{\star}\right\Vert _{\mathrm{F}}+C_{1}\frac{\sigma}{\sigma_{\min}}\sqrt{\frac{n}{p}}\left\Vert \bm{F}^{\star}\right\Vert _{\mathrm{F}},\label{eq:thm-decay-1}\\
\left\Vert \bm{F}^{0}\bm{H}^{0}-\bm{F}^{\star}\right\Vert _{\mathrm{F}} & \leq C_{0}\sqrt{\kappa} \left(\frac{\sigma}{\sigma_{\min}}+\frac{\left\Vert \bm{M}^{\star}\right\Vert _{\infty}}{\sigma_{\min}} \right)\sqrt{\frac{n}{p}}\left\Vert \bm{F}^{\star}\right\Vert _{\mathrm{F}}.\label{eq:thm-initial-1}
\end{align}
for all $0\le t\leq t_{0}=O(n^{-5})$, where $C_{0}$ and $C_{1}$
are some absolute constants and $\rho=1-\frac{\sigma_{\min}}{20}\eta$,
as long as $0\leq\eta\leq c'/(\mu\kappa^{3}r^{2}\sigma_{\max}\log n)$
for some small constant $c'>0$. \end{theorem}

In an analogy to Theorem \ref{thm:main}, Theorem \ref{thm:main-1}
also exhibits that Algorithm \ref{alg:gd-rmc} starts from a proper
initialization and then the iterates $\{\bm{F}^{t}\}_{t=0}^{t_{0}}$
demonstrates geometric convergence of to some error floor. Furthermore,
there are several aspects of Theorem \ref{thm:main-1} calling for
future improvement. For example, the sample size condition \eqref{eq:condition-1}
requires that the sample complexity to scale as $O(\kappa^{4}nr^{3})$.
In contrast, in the noiseless setting, \citet{gross2011recovering} and \citet{chen2015incoherence}
have shown that the sample complexity needed to recover the low-rank
matrix scales as $O(nr)$. Moreover, in the presence of sub-Gaussian
noise, \citet{chen2020noisy} has established minimax optimal estimation
error with sample complexity $O(\kappa^{4}nr^{2})$. Hence, there
might still exist much room for improvement of the dependency on $r$
and $\kappa$. To put it into perspective, this sub-optimal scaling
in $r$ and $\kappa$ appears frequently in theory of nonconvex low-rank
matrix recovery \citep{chen2015fast,sun2016guaranteed,ding2020leave,shen2022computationally}.
New analysis techniques shall be explored to sharpen the results.

\section{Prior arts\label{sec:Prior-arts}}

%Due to its superior computational advantage over the convex approach,
%the nonconvex approach has been employed to study a diverse array
%of high dimensional statistical estimation problems with low-rank
%structure, including matrix sensing \citep{zheng2015convergent,tu2016low,jain2013low},
%matrix completion \citep{chen2015fast,jain2013low,hardt2014understanding,wei2016guarantees,zheng2016convergence,chen2020noisy,chen2020nonconvex,ma2017implicit,sun2016guaranteed,keshavan2010matrix,keshavan2009matrix,hardt2014fast},
%phase retrieval \citep{candes2015phase,cai2016optimal,zhang2016provable,chen2017solving,wang2017solving,zhang2017nonconvex,chen2019gradient,yang2019misspecified,ma2017implicit,zhang2020phase,netrapalli2013phase,waldspurger2018phase,ma2019optimization},
%blind deconvolution \citep{ling2017blind,dong2018nonconvex,huang2018blind,li2019rapid,chen2021convex},
%Robust PCA \citep{netrapalli2014non,yi2016fast,gu2016low,cherapanamjeri2017nearly,chen2021bridging}, to name just a few. 

Due to its superior computational advantage over the convex approach, the nonconvex approach has been employed to study a diverse array of high dimensional statistical estimation problems with low-rank structure, including matrix sensing, matrix completion, phase retrieval, blind deconvolution, robust PCA, to name just a few. The readers are referred to \citet{chi2019nonconvex} %and \citet{chen2021spectral}
for an overview of this topic and references therein. Among these problems, matrix completion
is the focus of this paper and the recent decade has witnessed a flurry
of research activities under this topic since the seminal work by \citet{candes2009exact}.  
% See, for example, \citep{chen2015fast,jain2013low,hardt2014understanding,wei2016guarantees,zheng2016convergence,chen2020noisy,chen2020nonconvex,ma2017implicit,sun2016guaranteed,keshavan2010matrix,keshavan2009matrix,hardt2014fast} and its closed related work on robust PCA  \citep{netrapalli2014non,yi2016fast,gu2016low,cherapanamjeri2017nearly,chen2021bridging}.
A variety of nonconvex algorithms have been analyzed, such as projected
gradient descent \citep{chen2015fast,zheng2016convergence,sun2016guaranteed},
alternating minimization \citep{jain2013low,hardt2014understanding,hardt2014fast},
Riemannian gradient descent \citep{wei2016guarantees} and gradient
descent with Burer-Monteiro factorization \citep{zheng2016convergence,burer2003nonlinear}.
Besides, nuclear norm minimization has also attracted much attention
\citep{candes2010matrix,gross2011recovering,candes2011tight,koltchinskii2011nuclear,negahban2012restricted,chen2015incoherence}.
All these papers consider either noiseless setting or sub-Gaussian
noise, while heavy-tailed noise is not allowed for.

Heavy-tailed noise is an ubiquitous and widely studied issue arising
in a variety of modern statistical problems. A number of papers have
been dedicated to resolving this problem. For instance, \citet{catoni2012challenging}
proposes using a robust loss function to estimate the mean and variance
of data with bounded variance; \citet{brownlees2015empirical} investigates
empirical risk minimization based on the robust estimator proposed
by \citet{catoni2012challenging}. In high-dimensional linear regression problem,
 \citet{fan2017estimation,loh2017statistical} and \citet{sun2020adaptive}
study the usage of robust loss functions 
and analyze the theoretical properties of the proposed robust estimators.
\citet{charisopoulos2019low,tong2021low,li2020nonconvex} study the
nonsmooth and nonconvex formulation of low-rank matrix recovery with
the help of $\ell_{1}$ loss, which subsumes many well-known problems
including phase retrieval, matrix completion, blind deconvolution,
etc. \citet{alquier2019estimation} studies the applications of general
Lipschitz loss functions in a series of statistical problems including
matrix completion, logistic LASSO and kernel methods. Another line
of research follows the ``median of means'' approach \citep{nemirovskij1983problem,minsker2015geometric,hsu2016loss}
to attenuate the effects of heavy-tailed data.

Taking a closer look at robust matrix completion, a variety of papers
have been devoted to studying the scenarios when observations are
contaminated by outliers or heavy-tailed noises. In this regime, the
proposed methods can also be roughly categorized as convex and nonconvex
ones. Regarding the convex approach, to mitigate the effects of heavy-tailed
noises, \citet{fan2021shrinkage} proposes to first shrink the data
to construct robust covariance estimators and then minimizes $\ell_{2}$
risk with nuclear norm penalty under the assumption of finite $2+\varepsilon$
moment of noise. \citet{elsener2018robust} assumes a constant lower
bound of the density function and a regularity condition on the distribution
function of the errors. Under such conditions, it then studies the
performance of $\ell_{1}$ and Huber loss with nuclear norm penalty,
and obtains estimation error rates for approximately low-rank matrices.
\citet{minsker2018sub} introduces a robust estimator inspired by
\citet{catoni2012challenging} and proves a similar estimation error
bound to \citet{fan2021shrinkage} under the finite second moment
condition. Turning to robust nonconvex optimization, \citet{shen2022computationally}
introduces a nonconvex Riemannian sub-gradient algorithm to study
matrix completion with heavy-tailed noise and $\ell_{1}$ loss, Huber
loss and quantile loss respectively, under some regularity conditions
on the density function and distribution function of the errors similar
to that of \citet{elsener2018robust}. Another collection of works
focuses on studying robust matrix completion in the presence of outliers.
For example, \citet{cambier2016robust} considers the case where the
observed entries are corrupted by random outliers and studies $\ell_{1}$
loss function with the help of Riemannian optimization. \citet{klopp2017robust} and \citet{chen2021bridging}
extend the model setting to incorporate both outliers and sub-Gaussian
noise, and the latter achieves optimal estimation error.

\section{Numerical experiments\label{sec:Numerical-experiments}}

In this section, we conduct a variety of numerical experiments to
corroborate the validity of our theory established in Section \ref{sec:Main-results}.
Throughout the experiments, we fix the dimension to be $n=1000$ and
the rank $r=5$. The observation probability is $p=0.3$. The ground
truth matrix $\bm{U}^{\star}$, $\bm{V}^{\star}\in\mathbb{R}^{n\times r}$
are generated by sampling from standard Gaussian distribution and
then orthogonalizing their columns. The diagonal of $\bm{\Sigma}^{\star}$
is set to be equidistant from $\sigma_{1}^{\star}=r$ to $\sigma_{r}^{\star}=1$.

\selectlanguage{british}%
\begin{figure}
\hfill{}\includegraphics[clip,width=0.25\paperwidth]{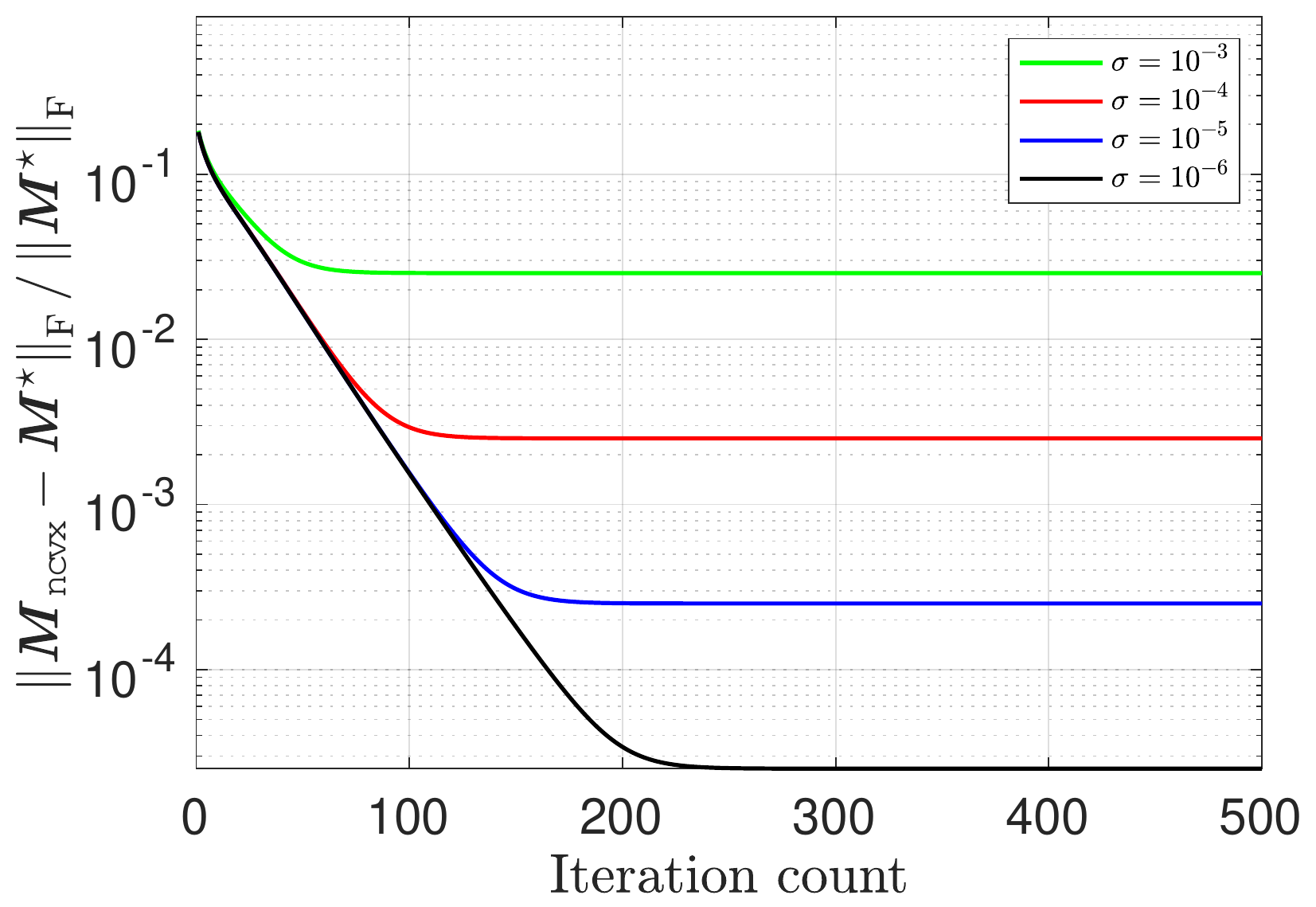}
\hfill{}\includegraphics[clip,width=0.25\paperwidth]{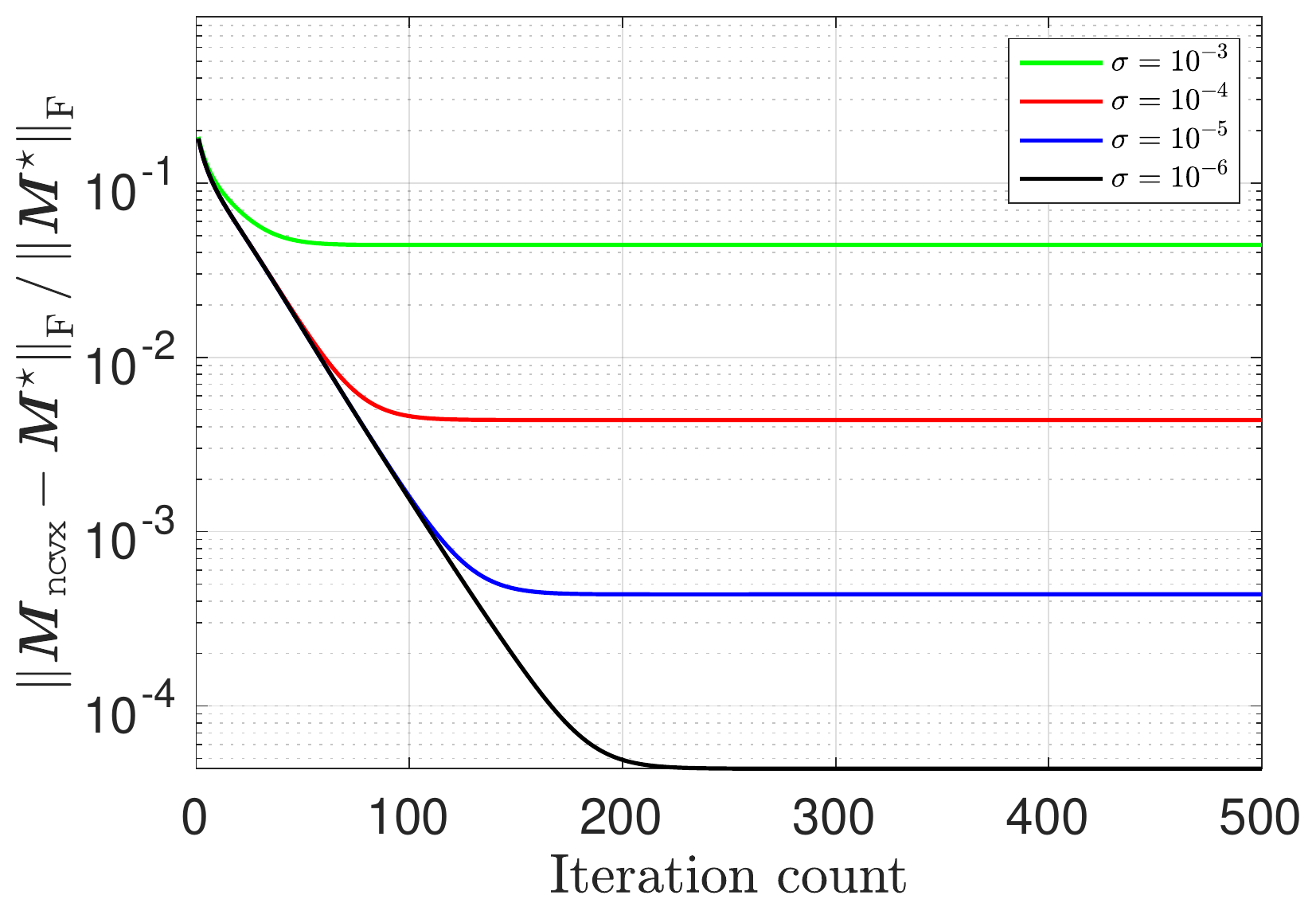}
\hfill{}\includegraphics[clip,width=0.25\paperwidth]{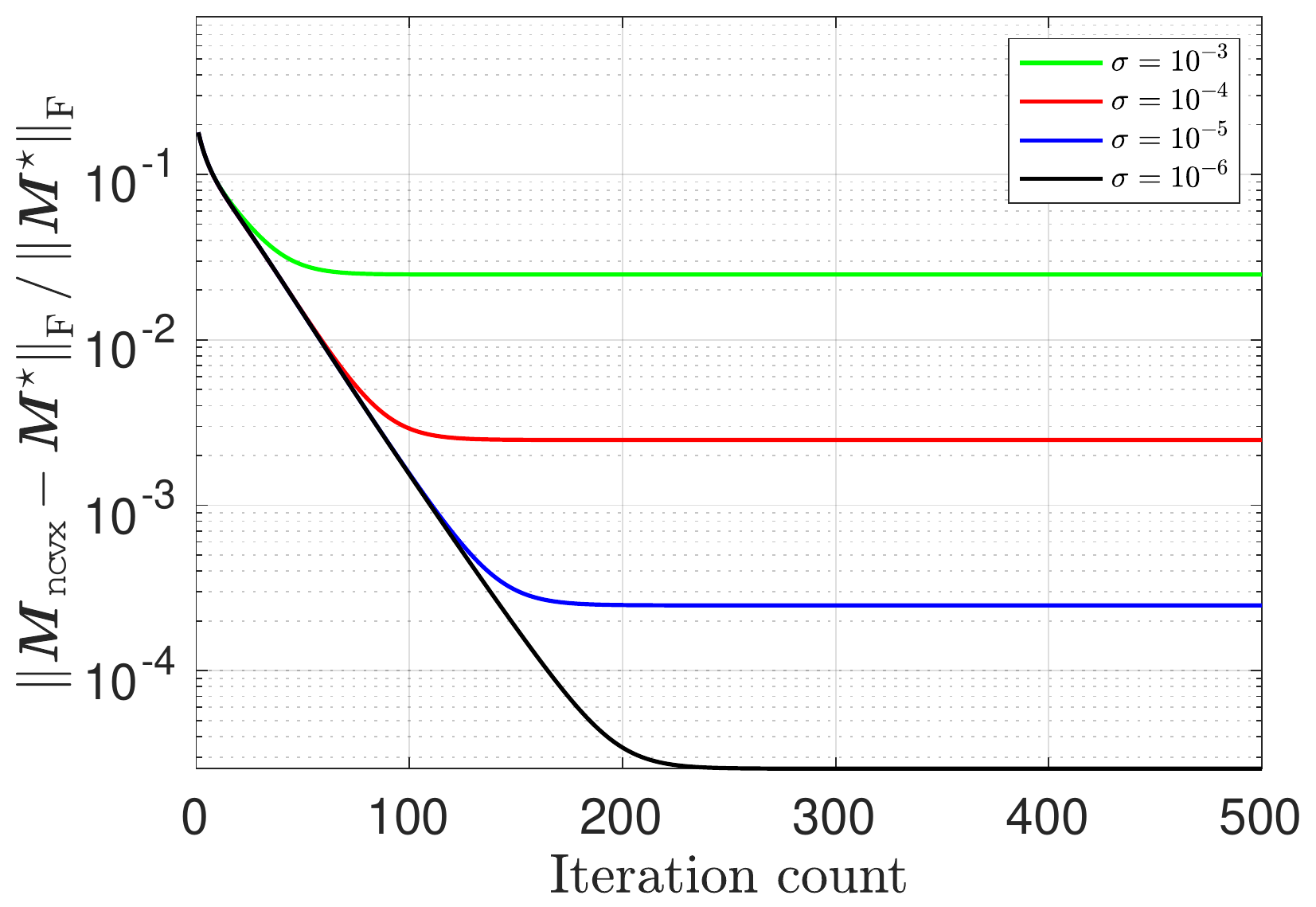}
\begin{minipage}[t]{0.36\textwidth}%
\selectlanguage{english}%
\centering(a) Gaussian distribution. \selectlanguage{english}%
\end{minipage}%
\begin{minipage}[t]{0.32\textwidth}%
\selectlanguage{english}%
\centering(b) Student's distribution $t(3)$. \selectlanguage{english}%
\end{minipage}\hfill{}%
\begin{minipage}[t]{0.3\textwidth}%
\selectlanguage{english}%

\centering(c) Distribution  in \eqref{eq:sparse-dist}. \selectlanguage{english}%
\end{minipage}

\selectlanguage{english}%
\caption{Relative Euclidean estimation errors of \eqref{eq:obj} with different error distributions vs. iteration count. \label{fig:dist_ncvx}}
\selectlanguage{english}%
\end{figure}

\selectlanguage{english}%
In the first series of experiments, we report the numerical convergence
of gradient descent (cf.~Algorithm \ref{alg:gd-rmc}) as the noise
level $\sigma$ varies from $10^{-6}$ to $10^{-3}$. The step size $\eta$ is set to be $0.05$ and the threshold parameter in Huber loss function \eqref{defn:huber} is taken to be $\tau=3\left(\Vert\bm{M}^{\star}\Vert_{\infty}+\sigma\sqrt{np}\right)$. Let $\bm{M}_{\mathsf{ncvx}}=\bm{X}_{\mathsf{ncvx}}\bm{Y}_{\mathsf{ncvx}}^{\top}$
be the nonconvex solution from Algorithm \ref{alg:gd-rmc} and $\bm{M}^{\star}$ be the ground truth.
Figure \ref{fig:dist_ncvx} displays the relative Euclidean estimation
errors ($\left\Vert \bm{M}_{\mathsf{ncvx}}-\bm{M}^{\star}\right\Vert _{\text{F}}/\left\Vert \bm{M}^{\star}\right\Vert _{\text{F}}$)
vs.~the iteration count respectively. In (a) and (b), the noises
are generated from Gaussian distribution and Student's $t$-distribution
with 3 degrees of freedom respectively. In (c), we adopt the noise
distribution defined by the following probability mass function for a trinomial distribution: 
\begin{equation}
f\left(x\right)=\begin{cases}
\frac{1}{2}\delta, & x=\frac{\sigma}{\sqrt{\delta}}\\
\frac{1}{2}\delta, & x=-\frac{\sigma}{\sqrt{\delta}}\\
1-\delta, &  x = 0, 
\end{cases}\label{eq:sparse-dist}
\end{equation}
and take $\delta =0.01$. In this case, only a fraction of observed entries
are corrupted by noise and the magnitude of noise can be much larger (10 times) than $\sigma$. Here, (b) and (c) focus on heavy-tailed noise distribution
with finite second moments, while (a) considers Gaussian distribution
which does not have heavy tail and serves as a benchmark to be compared
with. As can be seen from the plots, the nonconvex gradient descent
algorithm studied here converges linearly (in fact, within around
200 iterations) before it hits an error floor. In addition, the relative
error of matrix completion increases as the noise level $\sigma$
increases, which is consistent with Theorem \ref{thm:main}. 

In the second series of experiments, we report the statistical estimation
errors as the noise level $\sigma$ varies. The parameter $\tau$ in Huber loss function \eqref{defn:huber} is also chosen to be $\tau=3\left(\Vert\bm{M}^{\star}\Vert_{\infty}+\sigma\sqrt{np}\right)$. For each value of
$\sigma$, we conduct 50 random trials and use their average as the
reported estimation error. In each trial, we run the nonconvex algorithm
(cf.~Algorithm \ref{alg:gd-rmc}) until convergence or the maximum
number of iterations is reached. Figure \ref{fig:dist_norm} depicts
the relative Euclidean error $\left\Vert \bm{M}_{\mathsf{ncvx}}-\bm{M}^{\star}\right\Vert _{\text{F}}/\left\Vert \bm{M}^{\star}\right\Vert _{\text{F}}$
vs. the noise level $\sigma$. It captures the behavior of \eqref{eq:obj} as $\sigma$ varies from $10^{-6}$ to $10^{-3}$. The results suggest that the relative Euclidean estimation
error scales linearly with $\sigma$, providing empirical evidence
for the theories developed in Theorems \ref{thm:main} and \ref{thm:main-1}.
 
\begin{figure}
\selectlanguage{british}%
\hfill{}\includegraphics[clip,width=0.33\paperwidth]{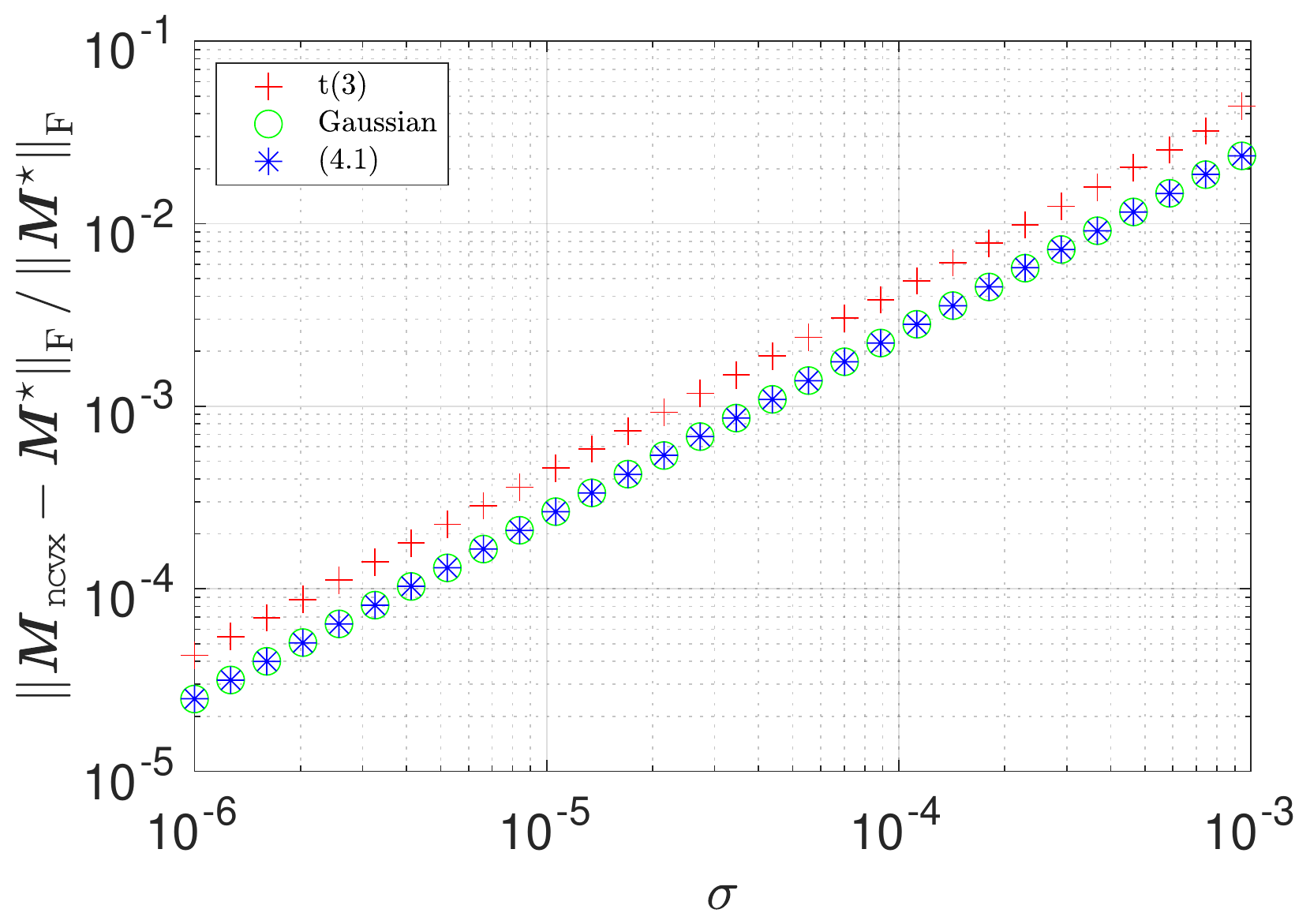}\hfill{}
% \hfill{}\includegraphics[clip,width=0.33\paperwidth]{figure/fig_nh}\hfill{}
%\begin{minipage}[t]{0.55\textwidth}%
%	\selectlanguage{english}%
%	\centering(a) \selectlanguage{english}%
%	%
%\end{minipage}%
%\begin{minipage}[t]{0.45\textwidth}%
%	\selectlanguage{english}%
%	\centering(b) \selectlanguage{english}%
%	%
%\end{minipage}\hfill{}%

\selectlanguage{english}%
\caption{Relative Euclidean estimation errors of \eqref{eq:obj} with different error distributions vs. the noise level $\sigma$. The results are averaged over 50 independent random trials. }
\label{fig:dist_norm}
\end{figure}

Next, we study how the estimation error depends on the choice of parameter $\tau$ in Huber loss function \eqref{defn:huber}. The noise level $\sigma$ is fixed to be $10^{-3}$ and the parameter $\tau$ is varied from $10^{-5}$ to $10^2$. To demonstrate the capacity of our theory to incorporate asymmetric distribution, we adopt a highly asymmetric noise distribution which is defined as follows: 
\begin{equation}
	f\left(x\right)=\begin{cases}
		\delta, & x=\sigma\sqrt{\frac{1-\delta}{\delta}}\\
		1-\delta, & x=-\sigma\sqrt{\frac{\delta}{1-\delta}}, \\
	\end{cases}\label{eq:asym-dist}
\end{equation}
and choose $\delta =0.0001$. It is straightforward to check that the distribution defined above is zero-mean with variance $\sigma^2$. The other two distributions used in Figure \ref{fig:dist_tau} are (b) Student's $t$-distribution with shape parameter $\nu=2.1$ and scale parameter $\sigma$; (c) Gaussian distribution with variance $\sigma^2$. Figure \ref{fig:dist_tau} displays the relative Euclidean error $\left\Vert \bm{M}_{\tau}-\bm{M}^{\star}\right\Vert _{\text{F}}/\left\Vert \bm{M}^{\star}\right\Vert _{\text{F}}$ vs. $\tau$. Here, we denote the estimator associated with parameter $\tau$ by $\bm{M}_{\tau}$. As can be seen in (a) and (b), adopting Huber loss with proper choice of $\tau$ can indeed significantly improve the estimation error, with the minimum is achieved approximately by $\tau=0.01$. This corresponds to the constant $C_{\tau}$ defined in Theorems \ref{thm:main} and \ref{thm:main-1} being roughly $0.15$, much smaller than our choice of $C_{\tau}=3$ in Figures \ref{fig:dist_ncvx} and \ref{fig:dist_norm}.  For all distributions, there are trunction biases in the spectral initialization:  the smaller $\tau$, the bigger the bias.  This gives various qualities of spectral initializaitons, which clearly have adverse impact on the convergence of the gradient decent for the non-convex loss.  That explains the poor performance of the estimator when $\tau$ is small even for symmetric error distributions.  There is additional Huberization bias for the error distribution \eqref{eq:asym-dist} that makes the performance for small $\tau$ much worse than the optimally chosen one in Figure  \ref{fig:dist_tau}.  As $\tau$ increases, the biases get smaller, but the impact of heavy tails gradually becomes dominant (except for the Gaussian noise)  and this is why we observe the shapes in Figure \ref{fig:dist_tau} (the fluctations are probably due to large value of variance for $t(2.1)$ distirbution and relatively small number of simulations).   In contrast, Figure \ref{fig:dist_tau}(c) plots the results of Gaussian distribution, which remains roughly the same after $\tau$ becomes large enough so that the bias in the initialization is small. These intriguing observations lend further support to our theoretical results and highlight the benefits of adopting Huber loss function when encountering heavy-tailed noise and also provide some guidance on how to choose $\tau$ in real applications. 
\begin{figure}
	\hfill{}\includegraphics[clip,width=0.25\paperwidth]{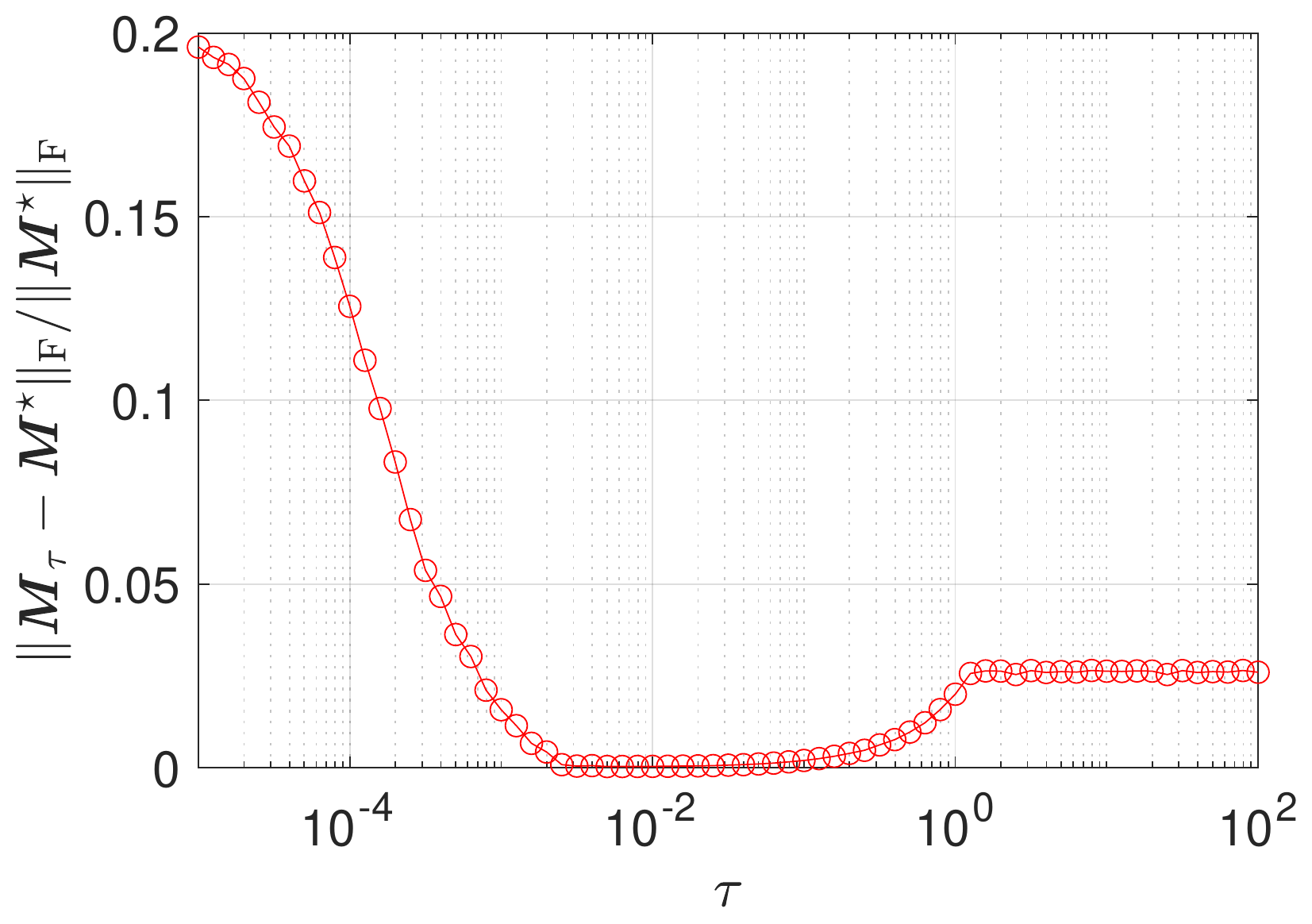}
	\hfill{}\includegraphics[clip,width=0.25\paperwidth]{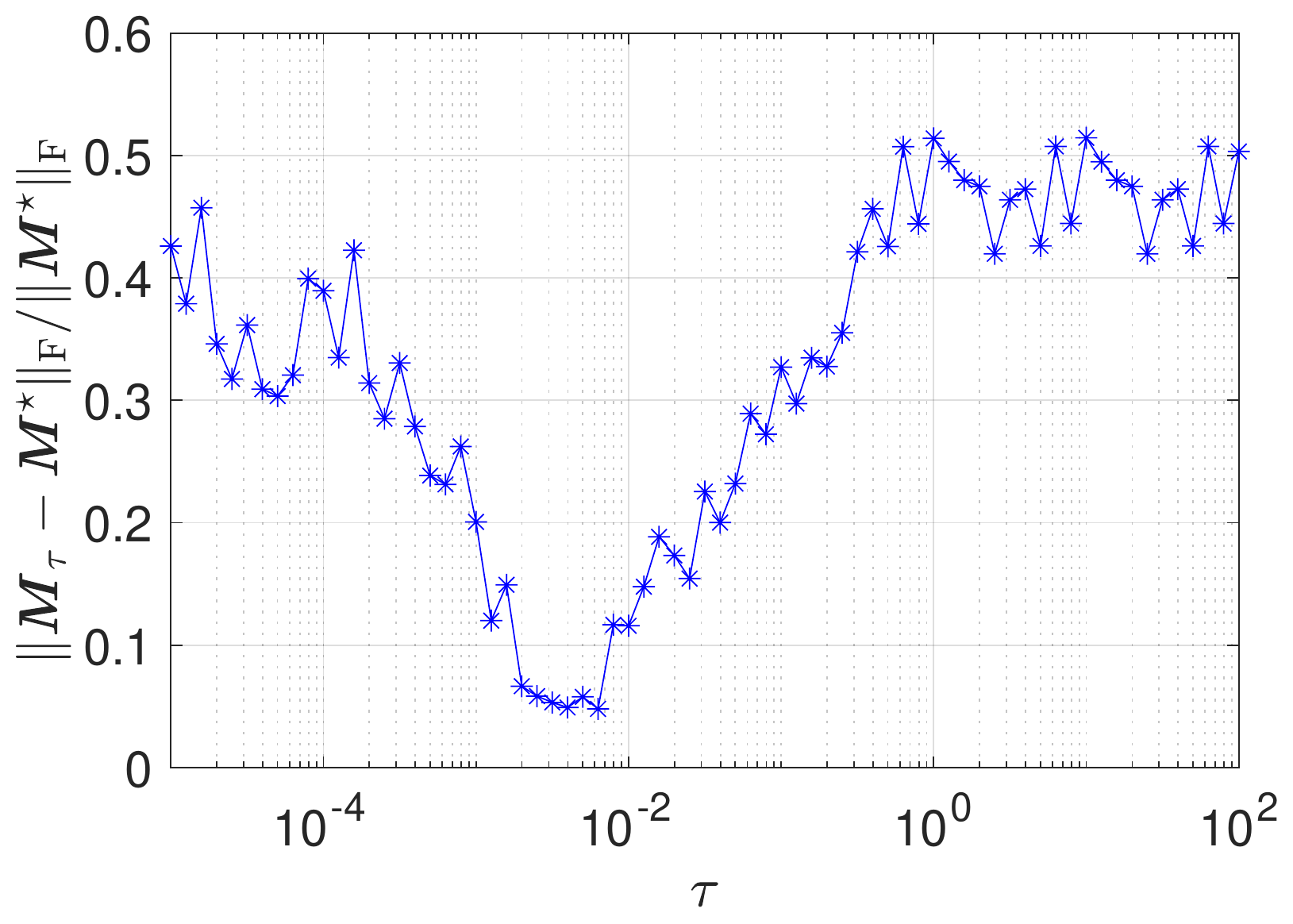}
	\hfill{}\includegraphics[clip,width=0.25\paperwidth]{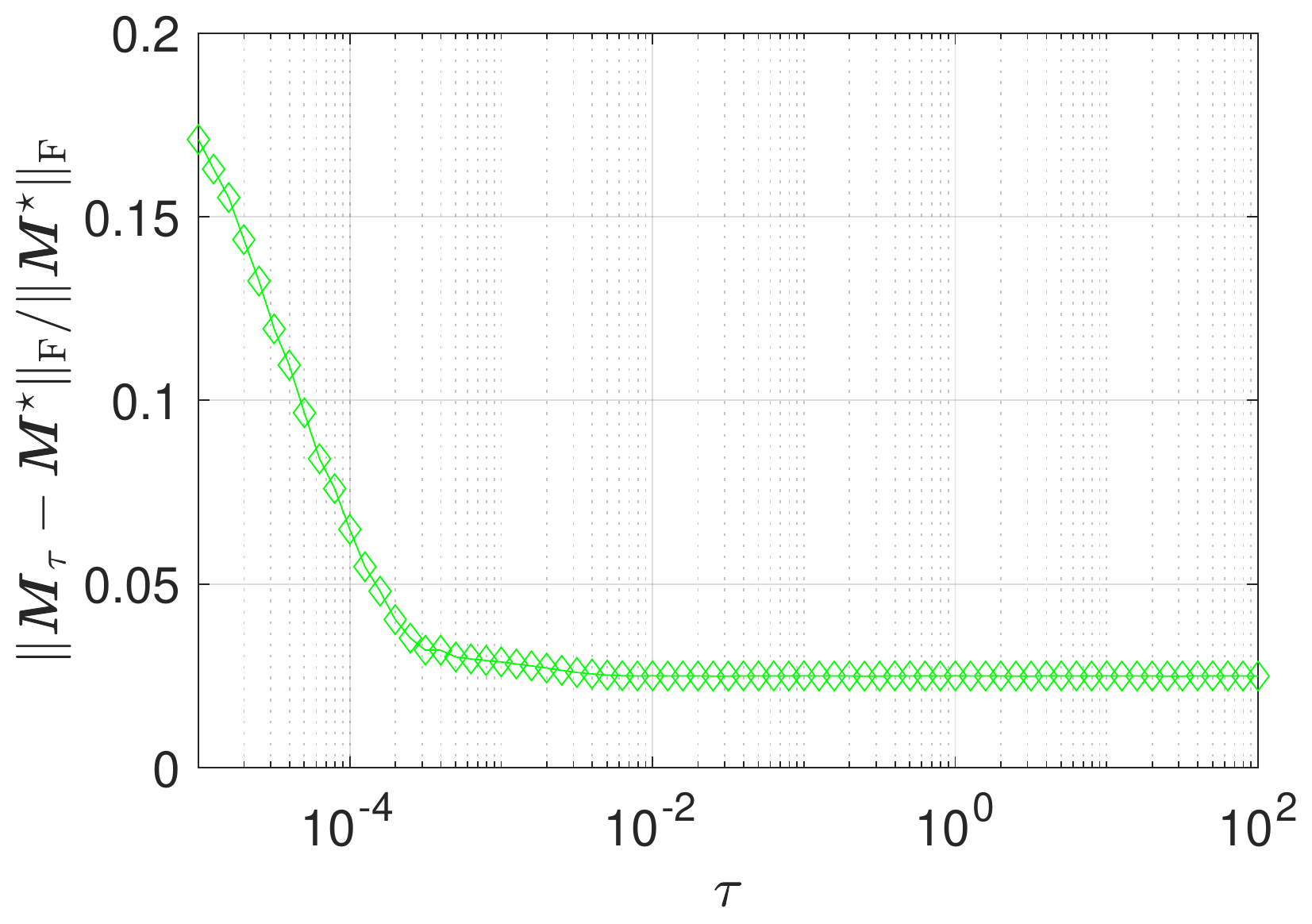}
	
	\begin{minipage}[t]{0.36\textwidth}%
		\selectlanguage{english}%
		\centering(a) Noise distribution defined in \eqref{eq:asym-dist}. \selectlanguage{english}%
	\end{minipage}%
	\begin{minipage}[t]{0.32\textwidth}%
		\selectlanguage{english}%
		\centering(b) Student's $t$-distribution with shape parameter $\nu=2.1$. \selectlanguage{english}%
	\end{minipage}\hfill{}%
	\begin{minipage}[t]{0.3\textwidth}%
		\selectlanguage{english}%
		\centering(c) Gaussian distribution. \selectlanguage{english}%
	\end{minipage}
	
	\selectlanguage{english}%
	\caption{Relative Euclidean estimation errors of \eqref{eq:obj} with different error distributions vs. $\tau$. The results are averaged over 50 independent random trials. \label{fig:dist_tau}}
	\selectlanguage{english}%
\end{figure}

Finally, we investigate the improvement of \eqref{eq:obj} over the regularized least-squares estimator defined by 
\begin{equation}
	\underset{\bm{X},\bm{Y}\in\mathbb{R}^{n\times r}}{\mathrm{minimize}}\qquad\frac{1}{2p}\sum_{\left(i,j\right)\in\Omega}\left(\left(\bm{X}\bm{Y}^{\top}\right)_{i,j}-M_{i,j}\right)^2+\frac{1}{8}\left\Vert \bm{X}^{\top}\bm{X}-\bm{Y}^{\top}\bm{Y}\right\Vert _{\mathrm{F}}^{2},\label{eq:ls-obj}
\end{equation}
as the noise level $\sigma$ varies from $10^{-6}$ to $10^{-3}$. Note that \eqref{eq:ls-obj} is equivalent to \eqref{eq:obj} with $\tau=\infty$. For each value of $\sigma$, we experiment with a series of $\tau$ ranging from $10^{-4}$ to $10^{-1}$. Specifically, for each pair of $\sigma$ and $\tau$, we conduct 50 random trials and calculate their average estimation error. Then for each value of $\sigma$, we record the minimum estimation error across different values of $\tau$. In addition, we also calculate the estimation error of the regularized least-squares estimator \eqref{eq:ls-obj}. We denote the minimizer of \eqref{eq:ls-obj} by $(\bm{X}_{LS},\bm{Y}_{LS})$ and define $\bm{M}_{LS}=\bm{X}_{LS}\bm{Y}_{LS}^{\top}$. Figure \ref{fig:dist_ratio} depicts the ratio between the estimation error of \eqref{eq:obj} with best choice of $\tau$ and the estimation error of $\bm{M}_{LS}$ ($\min_{\tau}\Vert\bm{M}_{\tau}-\bm{M}^{\star}\Vert_F/\Vert\bm{M}_{LS}-\bm{M}^{\star}\Vert_F$) vs. the noise level $\sigma$. For Gaussian distribution, adopting the Huber loss with the best choice of $\tau$ has barely improved over the least-squares estimator, as expected. In contrast, for asymmetric distribution \eqref{eq:asym-dist} and Student's $t$-distribution, we can observe considerable improvement over least-squares estimator when $\sigma$ is not exceedingly small. When $\sigma=10^{-3}$, the estimation error of adopting square loss in the objective function can be almost 10 times larger than using the best Huber loss. This impressive result emphasizes the superior advantage of the Huber loss over square loss when dealing with heavy-tailed distribution. 
\begin{figure}
	\selectlanguage{british}%
	\hfill{}\includegraphics[clip,width=0.33\paperwidth]{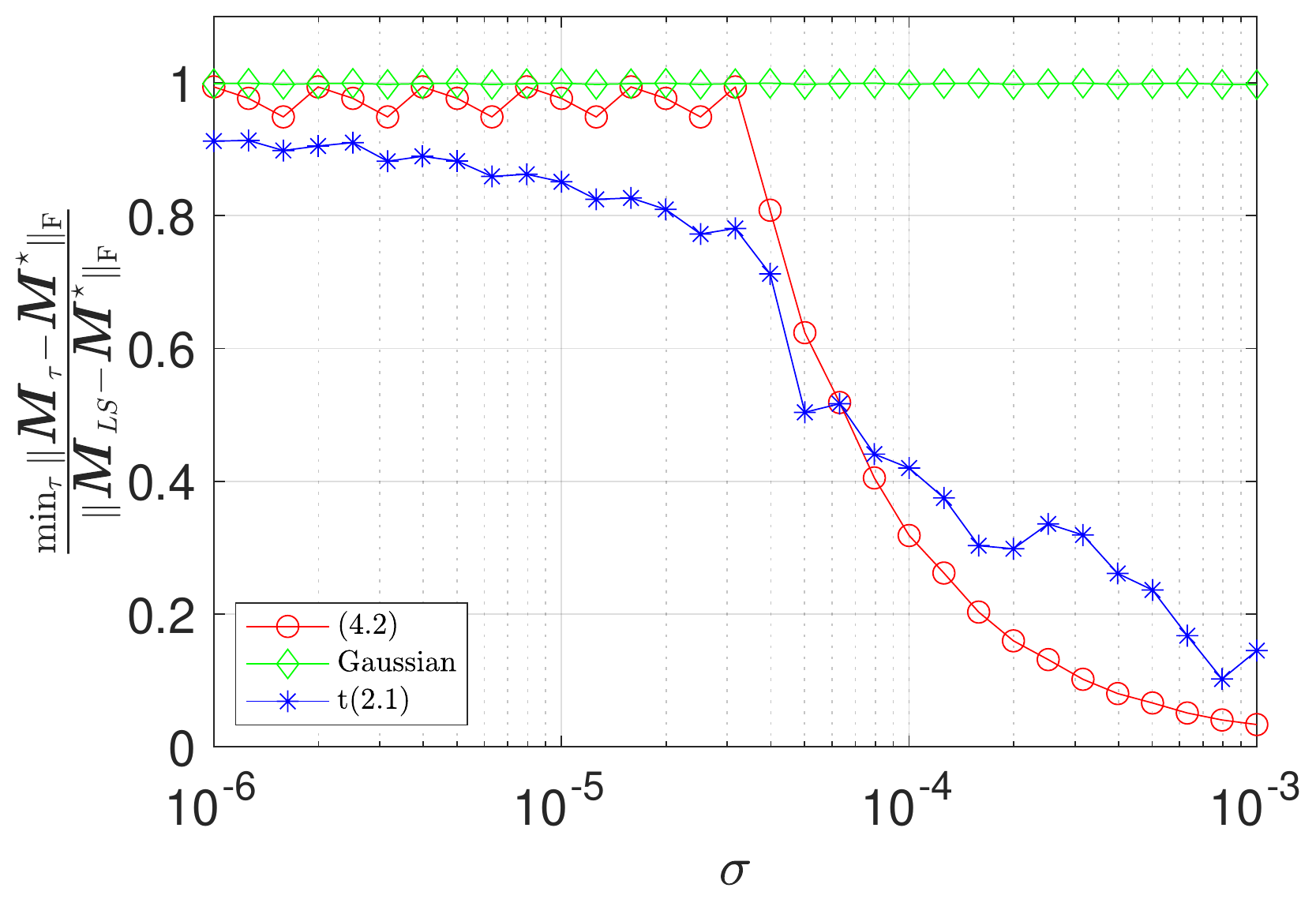}\hfill{}
	\selectlanguage{english}%
	\caption{The ratio between the minimum estimation error of \eqref{eq:obj} with respect to $\tau$ (i.e.$~\min_{\tau}\Vert\bm{M}_{\tau}-\bm{M}^{\star}\Vert_F$) and the estimation error of \eqref{eq:ls-obj} (i.e.$~\Vert\bm{M}_{LS}-\bm{M}^{\star}\Vert_F$) vs. the noise level $\sigma$.}
	\label{fig:dist_ratio}
\end{figure}
\section{Proof sketch\label{sec:Proof-sketch}}

In this section, we sketch the proof of Theorem \ref{thm:main}. The
proof details are all deferred to the Appendix.
We would establish
the following set of induction hypotheses for all $n^5 \gtrsim  t\geq0$:
\begin{subequations}
	\label{eq:hyp} 
	\begin{align}
		\left\Vert \bm{F}^{t}\bm{H}^{t}-\bm{F}^{\star}\right\Vert _{\mathrm{F}} & \lesssim\left(1-\frac{\sigma_{\min}}{20}\eta\right)^{t}\left\Vert \bm{F}^{0}\bm{H}^{0}-\bm{F}^{\star}\right\Vert _{\mathrm{F}}+C_{1}\frac{\sigma}{\sigma_{\min}}\sqrt{\frac{n}{p}}\left\Vert \bm{F}^{\star}\right\Vert _{\mathrm{F}},\label{eq:hyp-frob}\\
		\left\Vert \bm{F}^{0}\bm{H}^{0}-\bm{F}^{\star}\right\Vert _{\mathrm{F}} & \leq C_{0}\sqrt{\kappa}\left(\frac{\sigma}{\sigma_{\min}}\sqrt{\frac{n}{p}}+\frac{\left\Vert \bm{M}^{\star}\right\Vert _{\infty}}{\sigma_{\min}}\sqrt{\frac{n}{p}}\right)\left\Vert \bm{F}^{\star}\right\Vert _{\mathrm{F}},\label{eq:hyp-initial}\\
		\left\Vert \bm{F}^{t}\bm{H}^{t}-\bm{F}^{\star}\right\Vert _{2,\infty} & \lesssim\kappa^{1.5}\sqrt{r}\left(\frac{\sigma}{\sigma_{\min}}\sqrt{\frac{n}{p}}+\frac{\left\Vert \bm{M}^{\star}\right\Vert _{\infty}}{\sigma_{\min}}\sqrt{\frac{n}{p}}\right)\log n\left\Vert \bm{F}^{\star}\right\Vert _{2,\infty},\label{eq:hyp-2infty}
	\end{align}
\end{subequations}
for all $t\geq0$. With these in hand, Theorem \ref{thm:main} follows
immediately. In what follows, Section \ref{subsec:Local-geometry}
proves the hypothesis \eqref{eq:hyp-frob} by a careful investigation
of the landscape. Section \ref{subsec:Leave-one-out-sequences} is
devoted to justifying \eqref{eq:hyp-2infty} for all $t>0$. Finally,
Section \ref{subsec:Spectral-Initialization} verifies \eqref{eq:hyp-initial}
and \eqref{eq:hyp} for the base case, i.e. $t=0$. \eqref{eq:hyp-initial}, the base case with $t=0$.  

\subsection{Local geometry\label{subsec:Local-geometry}}

In this section, we start from the following lemma which characterizes
the region where the empirical loss function $f(\cdot)$ enjoys both
restricted strong convexity and smoothness, and then establish the
contraction of the error $\Vert\bm{F}^{t}\bm{H}^{t}-\bm{F}^{\star}\Vert_{\mathrm{F}}$
by use of Lemma \ref{lem:strongcvx}.

\begin{lemma}\label{lem:strongcvx} \textbf{\emph{(Restricted strong convexity and smoothness). }}Set $\tau=C_{\tau}\left(\left\Vert \bm{M}^{\star}\right\Vert _{\infty}+\sigma\sqrt{np}\right)$
for some constant $C_{\tau}>0$. Suppose the sample size obeys $n^{2}p\geq C\mu^{2}r^{2}\kappa^{2}n\log n$
for some sufficiently large constant $C>0$ and the noise satisfies
$\frac{\sigma}{\sigma_{\min}}\sqrt{\frac{n\log n}{p}}\leq c$ for
some sufficiently small constant $c>0$. Then with probability exceeding
$1-O(n^{-10})$, one has 
\begin{align*}
\mathsf{vec}\left(\bm{\Delta}\right)^{\top}\nabla^{2}f\left(\bm{X},\bm{Y}\right)\mathsf{vec}\left(\bm{\Delta}\right) & \geq\frac{\sigma_{\min}}{20}\left\Vert \bm{\Delta}\right\Vert _{\mathrm{F}}^{2},\\
\left\Vert \nabla^{2}f\left(\bm{X},\bm{Y}\right)\right\Vert  & \leq10\sigma_{\max},
\end{align*}
hold uniformly over all $\bm{X},\bm{Y}\in\mathbb{R}^{n\times r}$
obeying 
\begin{equation}
\left\Vert \left[\begin{array}{c}
\bm{X}-\bm{X}^{\star}\\
\bm{Y}-\bm{Y}^{\star}
\end{array}\right]\right\Vert _{2,\infty}\leq\frac{c}{\kappa\sqrt{n}}\left\Vert \bm{F}^{\star}\right\Vert ,\label{eq:strongcvx-2infty}
\end{equation}
and all $\bm{\Delta}=\left[\begin{array}{c}
\bm{\Delta}_{\bm{X}}\\
\bm{\Delta}_{\bm{Y}}
\end{array}\right]\in\mathbb{R}^{2n\times r}$ lying in the set 
\begin{equation}
\left\{ \left[\begin{array}{c}
\bm{X}_{1}\\
\bm{Y}_{1}
\end{array}\right]\widehat{\bm{H}}-\left[\begin{array}{c}
\bm{X}_{2}\\
\bm{Y}_{2}
\end{array}\right]:  \quad \left\Vert \left[\begin{array}{c}
\bm{X}_{2}-\bm{X}^{\star}\\
\bm{Y}_{2}-\bm{Y}^{\star}
\end{array}\right]\right\Vert \leq c'\left\Vert \bm{X}^{\star}\right\Vert ,\widehat{\bm{H}}\coloneqq\arg\min_{\bm{R}\in\mathcal{O}^{r\times r}}\left\Vert \left[\begin{array}{c}
\bm{X}_{1}\\
\bm{Y}_{1}
\end{array}\right]\bm{R}-\left[\begin{array}{c}
\bm{X}_{2}\\
\bm{Y}_{2}
\end{array}\right]\right\Vert _{\mathrm{F}} \right\} ,\label{eq:strongcvx-dir}
\end{equation}
where $c$ and $c'$ are some sufficiently small constants.

\end{lemma}

In words, Lemma \ref{lem:strongcvx} shows that when restricted to
points close to the ground truth $\bm{F}^{\star}$ in the sense of
$\ell_{2}/\ell_{\infty}$ norm, the Hessian $\nabla^{2}f(\cdot)$
is well-conditioned along directions defined in \eqref{eq:strongcvx-dir}.
Armed with this lemma, we are ready to establish the first induction
hypothesis \eqref{eq:hyp-frob} as follows.

\begin{lemma}\label{lem:contraction}\textbf{\emph{(Frobenius and spectral norm errors).}} Set $\tau=C_{\tau}\left(\left\Vert \bm{M}^{\star}\right\Vert _{\infty}+\sigma\sqrt{np}\right)$
for some constant $C_{\tau}>0$. Suppose the sample size obeys $n^{2}p\geq\kappa^{6}\mu^{2}r^{4}n\log^{2}n$
for some sufficiently large constant $C>0$ and the noise satisfies
$\frac{\sigma}{\sigma_{\min}}\sqrt{\frac{n}{p}}\leq\frac{c}{\sqrt{\kappa^{4}\mu r^{2}\log^{2}n}}$
for some sufficiently small constant $c>0$. If the iterates satisfy
at the $t$th iteration, then with probability over $1-O(n^{-100})$,
one has 
\begin{align*}
\left\Vert \bm{F}^{t+1}\bm{H}^{t+1}-\bm{F}^{\star}\right\Vert _{\mathrm{F}} & \leq\left(1-\frac{\sigma_{\min}}{20}\eta\right)\left\Vert \bm{F}^{t}\bm{H}^{t}-\bm{F}^{\star}\right\Vert _{\mathrm{F}}+C\eta\sigma\sqrt{\frac{n}{p}}\left\Vert \bm{F}^{\star}\right\Vert _{\mathrm{F}}\\
 & \leq\left(1-\frac{\sigma_{\min}}{20}\eta\right)^{t+1}\left\Vert \bm{F}^{0}\bm{H}^{0}-\bm{F}^{\star}\right\Vert _{\mathrm{F}}+C_{1}\frac{\sigma}{\sigma_{\min}}\sqrt{\frac{n}{p}}\left\Vert \bm{F}^{\star}\right\Vert _{\mathrm{F}},
\end{align*}
for some given step size $\eta$ such that $0\leq\eta\leq c'/(\mu\kappa^{3}r^{2}\sigma_{\max}\log n)$
with some small constant $c'>0$.

\end{lemma}

As an immediate consequence of Lemma \ref{lem:contraction}, one has 
\begin{align}
\left\Vert \bm{F}^{t+1}\bm{H}^{t+1}-\bm{F}^{\star}\right\Vert  & \leq\left\Vert \bm{F}^{t+1}\bm{H}^{t+1}-\bm{F}^{\star}\right\Vert _{\mathrm{F}}\leq\left\Vert \bm{F}^{0}\bm{H}^{0}-\bm{F}^{\star}\right\Vert _{\mathrm{F}}+C_{1}\frac{\sigma}{\sigma_{\min}}\sqrt{\frac{n}{p}}\left\Vert \bm{F}^{\star}\right\Vert _{\mathrm{F}}\nonumber \\
 & \lesssim\sqrt{\kappa}\left(\frac{\sigma}{\sigma_{\min}}\sqrt{\frac{n}{p}}+\frac{\left\Vert \bm{M}^{\star}\right\Vert _{\infty}}{\sigma_{\min}}\sqrt{\frac{n}{p}}\right)\left\Vert \bm{F}^{\star}\right\Vert _{\mathrm{F}},\label{eq:operator-norm}
\end{align}
where we make use of \eqref{eq:thm-initial} in the last inequality.  This crude inequality will be used in our proof.

\subsection{Leave-one-out sequences\label{subsec:Leave-one-out-sequences}}

In this section, we introduce a powerful leave-one-out analysis framework,
which assists us to decouple the dependence between noise and iterates.
This has already been employed to study various statistical problems
\citep{el2018impact,zhong2018near,li2019nonconvex,chen2019gradient,ding2020leave,chen2021convex}.

In what follows, we shall introduce a collection of auxiliary leave-one-out
sequences $\{\bm{F}^{t,(l)}\}_{t\geq0}$ for each $1\leq l\leq2n$
to decouple the complicated dependency structure and thus establish
\eqref{eq:hyp-2infty}. Specifically, for each $1\leq l\leq n$, $\{\bm{F}^{t,(l)}\}_{t\geq0}$
are constructed to be the gradient descent iterates generated by Algorithm
\ref{alg:gd-rmc-loo} with the following auxiliary loss function

\begin{align}
f^{\left(l\right)}\left(\bm{X},\bm{Y}\right) & =\frac{1}{2p}\sum_{(i,j)\in\Omega,i\neq l}\rho_{\tau}\left(\left(\bm{X}\bm{Y}^{\top}\right)_{i,j}-M_{i,j}\right)+\frac{1}{2}\sum_{j=1}^{n}\rho_{\tau}\left(\left(\bm{X}\bm{Y}^{\top}\right)_{l,j}-M_{l,j}^{\star}\right)\nonumber \\
 & \qquad+\frac{1}{8}\left\Vert \bm{X}^{\top}\bm{X}-\bm{Y}^{\top}\bm{Y}\right\Vert _{\mathrm{F}}^{2}.\label{defn-floo-1}
\end{align}
For $n+1\leq l\leq2n$, $\{\bm{F}^{t,(l)}\}_{t\geq0}$ are generated
similarly by running Algorithm \ref{alg:gd-rmc-loo} with the loss
function 
\begin{align}
f^{\left(l\right)}\left(\bm{X},\bm{Y}\right) & =\frac{1}{2p}\sum_{(i,j)\in\Omega,j\neq l-n}\rho_{\tau}\left(\left(\bm{X}\bm{Y}^{\top}\right)_{i,j}-M_{i,j}\right)+\frac{1}{2}\sum_{i=1}^{n}\rho_{\tau}\left(\left(\bm{X}\bm{Y}^{\top}\right)_{i,l-n}-M_{i,l-n}^{\star}\right)\nonumber \\
 & \qquad+\frac{1}{8}\left\Vert \bm{X}^{\top}\bm{X}-\bm{Y}^{\top}\bm{Y}\right\Vert _{\mathrm{F}}^{2}.\label{defn-floo-2}
\end{align}
When constructing the auxiliary loss functions \eqref{defn-floo-1}
and \eqref{defn-floo-2}, we drop the error term in each single row
(or column) respectively, and thus the resulting loss function is
independent of the randomness in that row (or column). In this
way, at the cost of a small perturbation, we are able to eliminate
the dependence between $\{\bm{F}^{t,(l)}\}_{t\geq0}$ and $\bm{F}_{l,\cdot}^{t}$,
which plays a key role in our analysis of $\ell_{2}/\ell_{\infty}$
norm.

\begin{algorithm}[h]
\caption{Construction of the $l$th leave-one-out sequence}

\label{alg:gd-rmc-loo}\begin{algorithmic}

\STATE \textbf{{Input}}: $\bm{M}$, $r$, $p$.

\STATE \textbf{{Spectral initialization}}: let $\bm{U}^{0,\left(l\right)}\bm{\Sigma}^{0,\left(l\right)}\bm{V}^{0,\left(l\right)\top}$
be the top-$r$ SVD of 
\begin{equation}
\bm{M}^{0,\left(l\right)}\coloneqq\frac{1}{2p}\mathcal{P}_{\Omega_{-l}}\left(\psi_{\tau}\left(\bm{M}\right)\right)+\mathcal{P}_{l}\left(\bm{M}^{\star}\right),\label{eq:spectral-method-matrix-1}
\end{equation}
and set $\bm{X}^{0,\left(l\right)}=\bm{U}^{0,\left(l\right)}\left(\bm{\Sigma}^{0,\left(l\right)}\right)^{1/2}$,
$\bm{Y}^{0,\left(l\right)}=\bm{V}^{0,\left(l\right)}\left(\bm{\Sigma}^{0,\left(l\right)}\right)^{1/2}$.

\STATE \textbf{{Gradient updates}}: \textbf{for }$t=0,1,\ldots,t_{0}-1$
\textbf{do} 

\STATE \vspace{-1em}
 
\begin{subequations}
\label{subeq:gradient_update_ncvx-1} 
\begin{align}
\bm{X}^{t+1,\left(l\right)} & =\bm{X}^{t,\left(l\right)}-\eta\nabla f_{\bm{X}}^{\left(l\right)} \left(\bm{X}^{t,\left(l\right)}\right)\label{eq:gd-x-1}\\
\bm{Y}^{t+1,\left(l\right)} & =\bm{Y}^{t,\left(l\right)}-\eta\nabla f_{\bm{Y}}^{\left(l\right)}\left(\bm{Y}^{t,\left(l\right)}\right)\label{eq:gd-y-1}
\end{align}
\end{subequations}
where $\nabla f_{\bm{X}}(\cdot)$ and $\nabla f_{\bm{Y}}(\cdot)$
represent the gradient of $f(\cdot)$ w.r.t.~$\bm{X}$ and $\bm{Y}$,
respectively.

\end{algorithmic} 
\end{algorithm}

To facilitate our analysis, we define the rotation matrices 
\begin{align}
\bm{H}^{t,\left(l\right)} & \triangleq\arg\min_{\bm{R}\in\mathcal{O}^{r\times r}}\left\Vert \bm{F}^{t,\left(l\right)}\bm{R}-\bm{F}^{\star}\right\Vert _{\mathrm{F}},\label{defn-Htl}\\
\bm{R}^{t,\left(l\right)} & \triangleq\arg\min_{\bm{R}\in\mathcal{O}^{r\times r}}\left\Vert \bm{F}^{t,\left(l\right)}\bm{R}-\bm{F}^{t}\bm{H}^{t}\right\Vert _{\mathrm{F}}.\label{defn-Rtl}
\end{align}
In the sequel, in order to justify \eqref{eq:hyp-2infty}, we shall
establish the following set of hypotheses:
\begin{subequations}
\label{subeq:additional}

\begin{align}
\max_{1\leq l\leq2n}\left\Vert \left(\bm{F}^{t,\left(l\right)}\bm{H}^{t,\left(l\right)}-\bm{F}^{\star}\right)_{l,\cdot}\right\Vert _{2} & \lesssim\kappa\sqrt{r}\left\Vert \bm{F}^{\star}\right\Vert _{2,\infty}\left(\frac{\sigma}{\sigma_{\min}}\sqrt{\frac{n\log n}{p}}+\frac{\left\Vert \bm{M}^{\star}\right\Vert _{\infty}}{\sigma_{\min}}\sqrt{\frac{n}{p}}\right),\label{subeq:add-1}\\
\max_{1\leq l\leq2n}\left\Vert \bm{F}^{t}\bm{H}^{t}-\bm{F}^{t,\left(l\right)}\bm{R}^{t,\left(l\right)}\right\Vert _{\mathrm{F}} & \lesssim\sqrt{\kappa}\left(\frac{\sigma}{\sigma_{\min}}\sqrt{\frac{n}{p}}+\frac{\left\Vert \bm{M}^{\star}\right\Vert _{\infty}}{\sigma_{\min}}\sqrt{\frac{n}{p}}\right)\left\Vert \bm{F}^{\star}\right\Vert _{2,\infty}\log n,\label{subeq:add-2}\\
\left\Vert \bm{F}^{t}\bm{H}^{t}-\bm{F}^{\star}\right\Vert _{2,\infty} & \lesssim\kappa^{3/2}\sqrt{r}\left(\frac{\sigma}{\sigma_{\min}}\sqrt{\frac{n}{p}}+\frac{\left\Vert \bm{M}^{\star}\right\Vert _{\infty}}{\sigma_{\min}}\sqrt{\frac{n}{p}}\right)\log n\left\Vert \bm{F}^{\star}\right\Vert _{2,\infty}.\label{subeq:add-3}
\end{align}
\end{subequations}

The results are summarized in the following three lemmas.

\begin{lemma}\label{lem:loo}\textbf{\emph{($\ell_{2}/\ell_{\infty}$
norm error of leave-one-out sequences).}} Set $\tau=C_{\tau}\left(\left\Vert \bm{M}^{\star}\right\Vert _{\infty}+\sigma\sqrt{np}\right)$
for some constant $C_{\tau}>0$. Suppose the sample size obeys $n^{2}p\geq C\kappa^{4}\mu^{2}r^{3}n\log n$
for some sufficiently large constant $C>0$ and the noise satisfies
$\sigma\sqrt{\frac{n}{p}}\leq\frac{c\sigma_{\min}}{\sqrt{\kappa^{2}\log^{2}n}}$
for some sufficiently small constant $c>0$. If the iterates satisfy
at the $t$th iteration, then with probability over $1-O(n^{-100})$,
one has 
\[
\max_{1\leq l\leq2n}\left\Vert \left(\bm{F}^{t+1,\left(l\right)}\bm{H}^{t+1,\left(l\right)}-\bm{F}^{\star}\right)_{l,\cdot}\right\Vert _{2}\lesssim\kappa\sqrt{r}\left\Vert \bm{F}^{\star}\right\Vert _{2,\infty}\left(\frac{\sigma}{\sigma_{\min}}\sqrt{\frac{n\log n}{p}}+\frac{\left\Vert \bm{M}^{\star}\right\Vert _{\infty}}{\sigma_{\min}}\sqrt{\frac{n}{p}}\right).
\]
\end{lemma}

This lemma justifies \eqref{subeq:add-1} and establishes the incoherence
of $\{\bm{F}^{t+1,(l)}\}_{l=1}^{2n}$. Next, we turn to show that
up to some orthogonal transformation, $\bm{F}^{t+1}$ can indeed be
well approximated by $\{\bm{F}^{t+1,(l)}\}_{l=1}^{2n}$.

\begin{lemma}\label{lem:loo-1}\textbf{\emph{(Leave-one-out perturbation).}}
Set $\tau=C_{\tau}\left(\left\Vert \bm{M}^{\star}\right\Vert _{\infty}+\sigma\sqrt{np}\right)$
for some constant $C_{\tau}>0$. Suppose the sample size obeys $n^{2}p\geq\kappa^{6}\mu^{2}r^{4}n\log^{2}n$
for some sufficiently large constant $C>0$ and the noise satisfies
$\frac{\sigma}{\sigma_{\min}}\sqrt{\frac{n}{p}}\leq\frac{c}{\sqrt{\kappa^{4}\mu r^{2}\log^{2}n}}$
for some sufficiently small constant $c>0$. If the iterates satisfy
at the $t$th iteration, then with probability over $1-O(n^{-100})$,
one has 
\begin{align*}
 & \max_{1\leq l\leq2n}\left\Vert \bm{F}^{t+1}\bm{H}^{t+1}-\bm{F}^{t+1,\left(l\right)}\bm{R}^{t+1,\left(l\right)}\right\Vert _{\mathrm{F}}\\
 & \leq\left(1-\frac{\sigma_{\min}}{20}\eta\right)^{t+1}\left\Vert \bm{F}^{0}\bm{H}^{0}-\bm{F}^{0,\left(l\right)}\bm{R}^{0,\left(l\right)}\right\Vert _{\mathrm{F}}+C\left(\frac{\sigma}{\sigma_{\min}}\sqrt{\frac{n}{p}}+\frac{\left\Vert \bm{M}^{\star}\right\Vert _{\infty}}{\sigma_{\min}}\sqrt{\frac{n}{p}}\right)\left\Vert \bm{F}^{\star}\right\Vert _{2,\infty}\log n\\
 & \lesssim\sqrt{\kappa}\left(\frac{\sigma}{\sigma_{\min}}\sqrt{\frac{n}{p}}+\frac{\left\Vert \bm{M}^{\star}\right\Vert _{\infty}}{\sigma_{\min}}\sqrt{\frac{n}{p}}\right)\left\Vert \bm{F}^{\star}\right\Vert _{2,\infty}\log n,
\end{align*}
provided that $0\leq\eta\leq c'/(\mu\kappa^{3}r^{2}\sigma_{\max}\log n)$
with some small constant $c'>0$. \end{lemma}

Now we are positioned to prove the induction hypothesis \eqref{subeq:add-3}
by combining the previous two lemmas.

\begin{lemma}\label{lem:loo-2}\textbf{\emph{($\ell_{2}/\ell_{\infty}$
norm error).}} Set $\tau=C_{\tau}\left(\left\Vert \bm{M}^{\star}\right\Vert _{\infty}+\sigma\sqrt{np}\right)$
for some constant $C_{\tau}>0$. Suppose the sample size obeys $n^{2}p\geq C\mu^{2}r^{2}\kappa n\log n$
for some sufficiently large constant $C>0$ and the noise satisfies
$\sigma\sqrt{\frac{n}{p}}\leq\frac{c\sigma_{\min}}{\sqrt{\kappa^{3}\mu r\log^{3}n}}$
for some sufficiently small constant $c>0$. If the iterates satisfy
at the $t$th iteration, then with probability over $1-O(n^{-100})$,
one has 
\[
\left\Vert \bm{F}^{t+1}\bm{H}^{t+1}-\bm{F}^{\star}\right\Vert _{2,\infty}\lesssim\kappa^{1.5}\sqrt{r}\left(\frac{\sigma}{\sigma_{\min}}\sqrt{\frac{n}{p}}+\frac{\left\Vert \bm{M}^{\star}\right\Vert _{\infty}}{\sigma_{\min}}\sqrt{\frac{n}{p}}\right)\log n\left\Vert \bm{F}^{\star}\right\Vert _{2,\infty}.
\]
\end{lemma}

The incoherence of $\bm{F}^{t+1}$ has thus been established. Finally,
we introduce another induction hypothesis which demonstrates the approximate
balancedness between $\bm{X}^{t}$ and $\bm{Y}^{t}$ through iterations.

\begin{lemma}\label{lem:bal}\textbf{\emph{(Approximate balancedness).}} 
Set $\tau=C_{\tau}\left(\left\Vert \bm{M}^{\star}\right\Vert +\sigma\sqrt{np}\right)$
for some constant $C_{\tau}>0$. Suppose the sample size obeys $n^{2}p\geq C\mu\kappa^{2}n\log n$
for some sufficiently large constant $C>0$ and the noise satisfies
$\sigma\sqrt{\frac{n}{p}}\leq\frac{c\sigma_{\min}}{\sqrt{\log n}}$
for some sufficiently small constant $c>0$. If the iterates satisfy
at the $t$th iteration, then with probability over $1-O(n^{-100})$,
one has 
\begin{align*}
\left\Vert \bm{X}^{t+1\top}\bm{X}^{t+1}-\bm{Y}^{t+1\top}\bm{Y}^{t+1}\right\Vert _{\mathrm{F}} & \lesssim\eta\mu\kappa^{4}r^{3.5}\sigma_{\max}^{2}\left(\frac{\sigma}{\sigma_{\min}}\sqrt{\frac{n}{p}}+\frac{\left\Vert \bm{M}^{\star}\right\Vert _{\infty}}{\sigma_{\min}}\sqrt{\frac{n}{p}}\right)^{2}\log^{2}n,
\end{align*}
given that $\eta\ll1/(\mu\kappa^{3}r^{2}\sigma_{\max}\log n)$.

\end{lemma}

Until now, we have finished verifying the induction hypotheses for
$t>0$. It remains to justify the base case $t=0$ in the next section.

\subsection{Spectral Initialization\label{subsec:Spectral-Initialization}}

According to Algorithm \ref{alg:gd-rmc}, the robust spectral method
initializes the algorithm by top-$r$ SVD of the matrix 
\begin{equation}
\bm{M}^{0}=\frac{1}{p}\mathcal{P}_{\Omega}\left(\psi_{\tau}\left(\bm{M}\right)\right),\label{eq:psi-tau}
\end{equation}
with $\psi_{\tau}(\cdot)$ is defined in \eqref{defn-psi}. Now we
are ready to present the following several lemmas justifying \eqref{eq:thm-initial}
and \eqref{subeq:additional} with $t=0$.

\begin{lemma}\label{lem:2}Suppose the sample size obeys $n^{2}p\geq C\kappa^{3}\mu^{2}r^{3}n\log n$
for some sufficiently large constant $C>0$, the noise satisfies $\frac{\sigma}{\sigma_{\min}}\sqrt{\frac{\kappa\mu rn\log n}{p}}\leq c$
for some sufficiently small constant $c>0$. Then with probability
over $1-O(n^{-10})$, one has 
\begin{align}
\left\Vert \bm{F}^{0}\bm{H}^{0}-\bm{F}^{\star}\right\Vert  & \lesssim\left(\frac{\sigma}{\sigma_{\min}}\sqrt{\frac{n}{p}}+\frac{\left\Vert \bm{M}^{\star}\right\Vert _{\infty}}{\sigma_{\min}}\sqrt{\frac{n}{p}}\right)\left\Vert \bm{F}^{\star}\right\Vert ,\label{eq:initial-operator}
\end{align}
and 
\begin{align}
\left\Vert \bm{F}^{0}\bm{H}^{0}-\bm{F}^{\star}\right\Vert _{\mathrm{F}} & \lesssim\sqrt{r}\left(\frac{\sigma}{\sigma_{\min}}\sqrt{\frac{n}{p}}+\frac{\left\Vert \bm{M}^{\star}\right\Vert _{\infty}}{\sigma_{\min}}\sqrt{\frac{n}{p}}\right)\left\Vert \bm{F}^{\star}\right\Vert \nonumber \\
 & \lesssim\sqrt{\kappa}\left(\frac{\sigma}{\sigma_{\min}}\sqrt{\frac{n}{p}}+\frac{\left\Vert \bm{M}^{\star}\right\Vert _{\infty}}{\sigma_{\min}}\sqrt{\frac{n}{p}}\right)\left\Vert \bm{F}^{\star}\right\Vert _{\mathrm{F}}.\label{eq:initial-fro}
\end{align}

\end{lemma}

This lemma verifies induction hypothesis \eqref{eq:thm-initial}.
Then we focus on establishing the incoherence condition of spectral
initialization.

\begin{lemma}\label{lem:3}Suppose the sample size obeys $n^{2}p\geq C\mu^{2}r^{2}\kappa n\log n$
for some sufficiently large constant $C>0$, the noise satisfies $\sigma\sqrt{\frac{n}{p}}\leq\frac{c\sigma_{\min}}{\sqrt{\kappa^{3}\mu r\log^{3}n}},$for
some sufficiently small constant $c>0$. Then with probability over
$1-O(n^{-10})$, one has

\begin{align}
\left\Vert \bm{F}^{0}\bm{H}^{0}-\bm{F}^{\star}\right\Vert _{2,\infty} & \lesssim\sqrt{\kappa}\left(\frac{\sigma}{\sigma_{\min}}\sqrt{\frac{n}{p}}+\frac{\left\Vert \bm{M}^{\star}\right\Vert _{\infty}}{\sigma_{\min}}\sqrt{\frac{n}{p}}\right)\log n\left\Vert \bm{F}^{\star}\right\Vert _{2,\infty}.\label{eq:initial-incoherence}
\end{align}

\end{lemma}

Next, we move on to proving the incoherence property of spectral initialization
in the leave-one-out algorithm (cf.~Algorithm \ref{alg:gd-rmc-loo}).

\begin{lemma}\label{lem:4}Suppose the sample size obeys $n^{2}p\geq C\mu^{2}r^{2}\kappa^{2}n\log n$
for some sufficiently large constant $C>0$, the noise satisfies $\frac{\sigma}{\sigma_{\min}}\sqrt{\frac{n}{p}}\leq c,$for
some sufficiently small constant $c>0$. Then with probability over
$1-O(n^{-10})$, one has 
\begin{equation}
\left\Vert \left(\bm{F}^{0,\left(l\right)}\bm{H}^{0,\left(l\right)}-\bm{F}^{\star}\right)_{l,\cdot}\right\Vert _{2}\lesssim\sqrt{\kappa}\left(\frac{\sigma}{\sigma_{\min}}\sqrt{\frac{n}{p}}+\frac{\left\Vert \bm{M}^{\star}\right\Vert _{\infty}}{\sigma_{\min}}\sqrt{\frac{n}{p}}\right)\left\Vert \bm{F}^{\star}\right\Vert _{2,\infty}.\label{eq:initial-loo-1}
\end{equation}

\end{lemma}

Finally, we turn to justify \eqref{subeq:add-2} for $t=0$, demonstrating
the proximity between $\bm{F}^{0}$ and $\{\bm{F}^{0,(l)}\}_{l=1}^{2n}$
up to some orthonormal transformation.

\begin{lemma}\label{lem:5}Suppose the sample size obeys $n^{2}p\geq C\kappa^{2}\mu^{2}r^{3}n\log n$
for some sufficiently large constant $C>0$, the noise satisfies $\frac{\sigma}{\sigma_{\min}}\sqrt{\frac{\kappa rn\log^{2}n}{p}}\leq c$
for some sufficiently small constant $c>0$. Then with probability
over $1-O(n^{-10})$, one has 
\begin{equation}
\left\Vert \bm{F}^{0}\bm{H}^{0}-\bm{F}^{0,\left(l\right)}\bm{R}^{0,\left(l\right)}\right\Vert _{\mathrm{F}}\lesssim\sqrt{\kappa}\left(\frac{\sigma}{\sigma_{\min}}\sqrt{\frac{n}{p}}+\frac{\left\Vert \bm{M}^{\star}\right\Vert _{\infty}}{\sigma_{\min}}\sqrt{\frac{n}{p}}\right)\log n\left\Vert \bm{F}^{\star}\right\Vert _{2,\infty}.\label{eq:initial-loo-2}
\end{equation}

\end{lemma} 

\section{Discussion\label{sec:Discussion}}

This paper establishes the convergence guarantees for gradient descent
of robust matrix completion with second moment condition. Going beyond
this, there are a few interesting directions for future study, as
elaborated below. 
\begin{itemize}
\item \emph{Tightening the dependency on $r$ and $\kappa$. }As discussed
below Theorem \ref{thm:main-1}, the dependence of the sample size
requirement on $r$ and $\kappa$ is sub-optimal, calling for the
application of novel analysis techniques for improvement. 
\item \emph{Incorporating outliers. }As stated in Section \ref{sec:Prior-arts},
outlier is another important source of contamination, while as far
as we are concerned, there is no existing theory of model incorporating
heavy-tailed noise and outliers simultaneously. The most similar setting
might be robust PCA which includes outliers and sub-Gaussian noise
\citet{agarwal2012noisy,klopp2017robust,chen2021bridging}. The techniques
and insights of this paper may enlighten studies on the more general
setting. 
\item \emph{Approximate low-rank structure. }This present paper requires
the matrix of interest $\bm{M}^{\star}$ to be exactly low-rank, while
in many real applications, $\bm{M}^{\star}$ may be only approximately
low-rank. Specifically, papers \citet{fan2021shrinkage,elsener2018robust}
allowing for approximate low-rank structure typically suppose the
$\ell_{q}$ norm of the singular values of $\bm{M}^{\star}$ is bounded,
which reduces to Assumption \ref{assum} by setting $q=0$. It remains
largely unclear whether the nonconvex approach still works in the
approximate low-rank scenario. 
\item \emph{Convex estimator for robust matrix completion. }As elucidated
in Table \ref{table:comparison}, existing theoretical guarantees
of convex approach have a trailing term which is not proportional
to noise, creating a considerable gap from optimal results when noise
level vanishes. To handle this problem, the idea of connecting convex
relaxation and nonconvex optimization in \citet{chen2020noisy} might
be inspiring and worth future exploration. 
\item \emph{Valid inference procedures. }This present paper focuses on estimation
of robust matrix completion and establishes a minimax optimal statistical
error. To move forward, we note it is vastly under-explored how to
assess the uncertainty of the estimates obtained from Algorithm \ref{alg:gd-rmc}.\emph{
}Methods and techniques in \citet{chen2019inference,xia2021statistical,yan2021inference}
may shed light on the procedure to perform valid inference on these
matrices. However, the focal point of \citet{chen2019inference,xia2021statistical}
is matrix completion with sub-Gaussian noise, while inference in the
presence of heavy-tailed noise has been a long-standing open question. 
\end{itemize}

\section*{Acknowledgements}

J.~Fan is supported by the NSF grants  DMS-1712591, DMS-2052926,  DMS-2053832, and DMS-2210833. B.~Wang	
is supported in part by Gordon Y.~S.~Wu Fellowships in Engineering from Princeton University.

\bibliographystyle{apalike}
\bibliography{bibfile}

\appendix

\section{Proofs for gradient descent}

It is straightforward to obtain the gradients of $f(\cdot)$: 
\begin{align}
\nabla_{\bm{X}}f\left(\bm{X},\bm{Y}\right) & =\frac{1}{2p}\mathcal{P}_{\Omega}\left(\left\{ \psi_{\tau}\left(\left(\bm{X}\bm{Y}^{\top}\right)_{i,j}-M_{i,j}\right)\right\} _{i,j}\right)\bm{Y}+\frac{1}{2}\bm{X}\left(\bm{X}^{\top}\bm{X}-\bm{Y}^{\top}\bm{Y}\right),\label{eq:gradient-x}\\
\nabla_{\bm{Y}}f\left(\bm{X},\bm{Y}\right) & =\frac{1}{2p}\mathcal{P}_{\Omega}\left(\left\{ \psi_{\tau}\left(\left(\bm{X}\bm{Y}^{\top}\right)_{i,j}-M_{i,j}\right)\right\} _{i,j}\right)^{\top}\bm{X}+\frac{1}{2}\bm{Y}\left(\bm{Y}^{\top}\bm{Y}-\bm{X}^{\top}\bm{X}\right).\label{eq:gradient-y}
\end{align}

We start with an auxiliary lemma which is immediate consequence of
Lemma \ref{lem:contraction}-\ref{lem:loo-2}.

\begin{lemma}\label{lem:aux}Instate the notation and assumptions
in Theorem \ref{thm:main}. For an integer $t>0$, suppose that the
hypotheses \eqref{eq:hyp} and \eqref{subeq:additional} hold in the
$t$th iteration. Then under the assumptions that $\frac{\sigma}{\sigma_{\min}}\sqrt{\frac{n}{p}}\ll1/\sqrt{\kappa r\log^{2}n}$
and $np\gg\mu^{2}\kappa^{3}r^{3}\log^{2}n$, one has 
\begin{subequations}
\label{subeq:aux} 
\begin{align}
\left\Vert \bm{F}^{t,\left(l\right)}\bm{R}^{t,\left(l\right)}-\bm{F}^{\star}\right\Vert _{2,\infty} & \lesssim\kappa^{3/2}\left(\frac{\sigma}{\sigma_{\min}}\sqrt{\frac{n}{p}}+\frac{\left\Vert \bm{M}^{\star}\right\Vert _{\infty}}{\sigma_{\min}}\sqrt{\frac{n}{p}}\right)\log n\left\Vert \bm{F}^{\star}\right\Vert _{2,\infty},\label{subeq:aux-1}\\
\left\Vert \bm{F}^{t,\left(l\right)}\bm{R}^{t,\left(l\right)}-\bm{F}^{\star}\right\Vert  & \lesssim\sqrt{r}\left(\frac{\sigma}{\sigma_{\min}}\sqrt{\frac{n}{p}}+\frac{\left\Vert \bm{M}^{\star}\right\Vert _{\infty}}{\sigma_{\min}}\sqrt{\frac{n}{p}}\right)\left\Vert \bm{F}^{\star}\right\Vert ,\label{subeq:aux-2}\\
\left\Vert \bm{F}^{t,\left(l\right)}\bm{H}^{t,\left(l\right)}-\bm{F}^{\star}\right\Vert  & \leq\left(\frac{\sigma}{\sigma_{\min}}\sqrt{\frac{n}{p}}+\frac{\left\Vert \bm{M}^{\star}\right\Vert _{\infty}}{\sigma_{\min}}\sqrt{\frac{n}{p}}\right)\sqrt{r}\left\Vert \bm{F}^{\star}\right\Vert ,\label{subeq:aux-3}\\
\left\Vert \bm{F}^{t,\left(l\right)}\bm{H}^{t,\left(l\right)}-\bm{F}^{t}\bm{H}^{t}\right\Vert _{\mathrm{F}} & \leq5\kappa\left\Vert \bm{F}^{t,\left(l\right)}\bm{R}^{t,\left(l\right)}-\bm{F}^{t}\bm{H}^{t}\right\Vert _{\mathrm{F}},\label{subeq:aux-4}\\
\left\Vert \bm{F}^{t}\right\Vert \leq2\left\Vert \bm{F}^{\star}\right\Vert ,\quad & \left\Vert \bm{F}^{t}\right\Vert _{\mathrm{F}}\leq2\left\Vert \bm{F}^{\star}\right\Vert _{\mathrm{F}},\quad\left\Vert \bm{F}^{t}\right\Vert _{2,\infty}\leq2\left\Vert \bm{F}^{\star}\right\Vert _{2,\infty},\label{subeq:aux-5}\\
\left\Vert \bm{F}^{t,\left(l\right)}\right\Vert \leq2\left\Vert \bm{F}^{\star}\right\Vert ,\quad & \left\Vert \bm{F}^{t,\left(l\right)}\right\Vert _{\mathrm{F}}\leq2\left\Vert \bm{F}^{\star}\right\Vert _{\mathrm{F}},\quad\left\Vert \bm{F}^{t,\left(l\right)}\right\Vert _{2,\infty}\leq2\left\Vert \bm{F}^{\star}\right\Vert _{2,\infty},\label{subeq:aux-6}\\
\frac{\sigma_{\min}}{2}\leq\sigma_{\min}\bigg[\left(\bm{Y}^{t,\left(l\right)}\bm{H}^{t,\left(l\right)}\right)^{\top}&\bm{Y}^{t,\left(l\right)}\bm{H}^{t,\left(l\right)}\bigg]  \leq\sigma_{\max}\left[\left(\bm{Y}^{t,\left(l\right)}\bm{H}^{t,\left(l\right)}\right)^{\top}\bm{Y}^{t,\left(l\right)}\bm{H}^{t,\left(l\right)}\right]\leq2\sigma_{\max}.\label{subeq:aux-7}
\end{align}
\end{subequations}

\end{lemma}

\subsection{Proof of Lemma \ref{lem:strongcvx}\label{subsec:Proof-of-Lemma-strongcvx}}

In view of \eqref{eq:obj}, simple calculation yields 
\begin{align*}
 & \mathsf{vec}\left(\bm{\Delta}\right)^{\top}\nabla^{2}f\left(\bm{X},\bm{Y}\right)\mathsf{vec}\left(\bm{\Delta}\right)\\
 & =\frac{2}{p}\left\langle \mathcal{P}_{\Omega}\left(\left(\bm{X}\bm{Y}^{\top}-\bm{M}\right)_{i,j}\ind_{\left|\left(\bm{X}\bm{Y}^{\top}\right)_{i,j}-M_{i,j}\right|\leq\tau}\right),\mathcal{P}_{\Omega}\left(\bm{\Delta}_{\bm{X}}\bm{\Delta}_{\bm{Y}}^{\top}\right)\right\rangle \\
 & \qquad+\frac{1}{p}\left\Vert \mathcal{P}_{\Omega}\left(\left(\bm{\Delta}_{\bm{X}}\bm{Y}^{\top}+\bm{X}\bm{\Delta}_{\bm{Y}}^{\top}\right)_{i,j}\ind_{\left|\left(\bm{X}\bm{Y}^{\top}\right)_{i,j}-M_{i,j}\right|\leq\tau}\right)\right\Vert _{\mathrm{F}}^{2}\\
 & \qquad+\frac{1}{2}\left\langle \bm{X}^{\top}\bm{X}-\bm{Y}^{\top}\bm{Y},\bm{\Delta}_{\bm{X}}^{\top}\bm{\Delta}_{\bm{X}}-\bm{\Delta}_{\bm{Y}}^{\top}\bm{\Delta}_{\bm{Y}}\right\rangle +\frac{1}{4}\left\Vert \bm{\Delta}_{\bm{X}}^{\top}\bm{X}+\bm{X}^{\top}\bm{\Delta}_{\bm{X}}-\bm{Y}^{\top}\bm{\Delta}_{\bm{Y}}-\bm{\Delta}_{\bm{Y}}^{\top}\bm{Y}\right\Vert _{\mathrm{F}}^{2}\\
 & \qquad+\frac{2}{p}\left\langle \mathcal{P}_{\Omega}\left(\left\{ \tau\mathrm{sgn}\left(\left(\bm{X}\bm{Y}^{\top}\right)_{i,j}\right)\ind_{\left|\left(\bm{X}\bm{Y}^{\top}\right)_{i,j}-M_{i,j}\right|>\tau}\right\} _{i,j}\right),\mathcal{P}_{\Omega}\left(\bm{\Delta}_{\bm{X}}\bm{\Delta}_{\bm{Y}}^{\top}\right)\right\rangle .
\end{align*}
\citet[Lemma 3.1]{chen2020nonconvex} has proved that 
\begin{align}
\mathcal{E}_{0}\coloneqq & \frac{2}{p}\left\langle \mathcal{P}_{\Omega}\left(\bm{X}\bm{Y}^{\top}-\bm{M}^{\star}\right),\mathcal{P}_{\Omega}\left(\bm{\Delta}_{\bm{X}}\bm{\Delta}_{\bm{Y}}^{\top}\right)\right\rangle +\frac{1}{p}\left\Vert \mathcal{P}_{\Omega}\left(\bm{\Delta}_{\bm{X}}\bm{Y}^{\top}+\bm{X}\bm{\Delta}_{\bm{Y}}^{\top}\right)\right\Vert _{\mathrm{F}}^{2}\nonumber \\
 & +\frac{1}{2}\left\langle \bm{X}^{\top}\bm{X}-\bm{Y}^{\top}\bm{Y},\bm{\Delta}_{\bm{X}}^{\top}\bm{\Delta}_{\bm{X}}-\bm{\Delta}_{\bm{Y}}^{\top}\bm{\Delta}_{\bm{Y}}\right\rangle +\frac{1}{4}\left\Vert \bm{\Delta}_{\bm{X}}^{\top}\bm{X}+\bm{X}^{\top}\bm{\Delta}_{\bm{X}}-\bm{Y}^{\top}\bm{\Delta}_{\bm{Y}}-\bm{\Delta}_{\bm{Y}}^{\top}\bm{Y}\right\Vert _{\mathrm{F}}^{2}\nonumber \\
\geq & \frac{\sigma_{\min}}{10}\left\Vert \bm{\Delta}\right\Vert _{\mathrm{F}}^{2}.\label{eq:strongcvx-E0}
\end{align}
We are left with considering 
\begin{align}
 & \mathsf{vec}\left(\bm{\Delta}\right)^{\top}\nabla^{2}f\left(\bm{X},\bm{Y}\right)\mathsf{vec}\left(\bm{\Delta}\right)-\mathcal{E}_{0}\nonumber \\
= & -\underbrace{\frac{2}{p}\left\langle \mathcal{P}_{\Omega}\left(\left(\bm{X}\bm{Y}^{\top}-\bm{M}^{\star}\right)_{i,j}\ind_{\left|\left(\bm{X}\bm{Y}^{\top}\right)_{i,j}-M_{i,j}\right|>\tau}\right),\mathcal{P}_{\Omega}\left(\bm{\Delta}_{\bm{X}}\bm{\Delta}_{\bm{Y}}^{\top}\right)\right\rangle }_{\eqqcolon\alpha_{1}}\nonumber \\
 & \qquad-\underbrace{\frac{2}{p}\left\langle \mathcal{P}_{\Omega}\left(\varepsilon_{i,j}\ind_{\left|\left(\bm{X}\bm{Y}^{\top}\right)_{i,j}-M_{i,j}\right|\leq\tau}\right),\mathcal{P}_{\Omega}\left(\bm{\Delta}_{\bm{X}}\bm{\Delta}_{\bm{Y}}^{\top}\right)\right\rangle }_{\eqqcolon\alpha_{2}}\nonumber \\
 & \qquad+\underbrace{\frac{1}{p}\left\Vert \mathcal{P}_{\Omega}\left(\left(\bm{\Delta}_{\bm{X}}\bm{Y}^{\top}+\bm{X}\bm{\Delta}_{\bm{Y}}^{\top}\right)_{i,j}\ind_{\left|\left(\bm{X}\bm{Y}^{\top}\right)_{i,j}-M_{i,j}\right|>\tau}\right)\right\Vert _{\mathrm{F}}^{2}}_{\eqqcolon\alpha_{3}}\nonumber \\
 & \qquad+\underbrace{\frac{2}{p}\left\langle \mathcal{P}_{\Omega}\left(\left\{ \tau\mathrm{sgn}\left(\left(\bm{X}\bm{Y}^{\top}\right)_{i,j}\right)\ind_{\left|\left(\bm{X}\bm{Y}^{\top}\right)_{i,j}-M_{i,j}\right|>\tau}\right\} _{i,j}\right),\mathcal{P}_{\Omega}\left(\bm{\Delta}_{\bm{X}}\bm{\Delta}_{\bm{Y}}^{\top}\right)\right\rangle }_{\eqqcolon\alpha_{4}}.\label{eq:strongcvx-E}
\end{align}

\begin{enumerate}
\item Regarding $\alpha_{1}$, it can be further decomposed as 
\begin{align}
\alpha_{1} & =\left|\frac{2}{p}\left\langle \mathcal{P}_{\Omega}\left(\left(\bm{X}\bm{Y}^{\top}-\bm{X}^{\star}\bm{Y}^{\star}\right)_{i,j}\ind_{\left|\left(\bm{X}\bm{Y}^{\top}\right)_{i,j}-M_{i,j}\right|>\tau}\right),\mathcal{P}_{\Omega}\left(\bm{\Delta}_{\bm{X}}\bm{\Delta}_{\bm{Y}}^{\top}\right)\right\rangle \right|\nonumber \\
 & \leq\underbrace{\left|\frac{2}{p}\left\langle \mathcal{P}_{\Omega}\left(\left(\bm{X}\left(\bm{Y}-\bm{Y}^{\star}\right)^{\top}\right)_{i,j}\ind_{\left|\left(\bm{X}\bm{Y}^{\top}\right)_{i,j}-M_{i,j}\right|>\tau}\right),\mathcal{P}_{\Omega}\left(\bm{\Delta}_{\bm{X}}\bm{\Delta}_{\bm{Y}}^{\top}\right)\right\rangle \right|}_{\eqqcolon\beta_{1}}\nonumber \\
 & \qquad+\underbrace{\left|\frac{2}{p}\left\langle \mathcal{P}_{\Omega}\left(\left(\left(\bm{X}-\bm{X}^{\star}\right)\bm{Y}^{\star}{}^{\top}\right)_{i,j}\ind_{\left|\left(\bm{X}\bm{Y}^{\top}\right)_{i,j}-M_{i,j}\right|>\tau}\right),\mathcal{P}_{\Omega}\left(\bm{\Delta}_{\bm{X}}\bm{\Delta}_{\bm{Y}}^{\top}\right)\right\rangle \right|}_{\eqqcolon\beta_{2}}.\label{eq:strongcvx-alpha1}
\end{align}
For $\beta_{1}$, we have 
\begin{align}
 & \left|\frac{1}{p}\left\langle \mathcal{P}_{\Omega}\left(\left(\bm{X}\left(\bm{Y}-\bm{Y}^{\star}\right)^{\top}\right)_{i,j}\ind_{\left|\left(\bm{X}\bm{Y}^{\top}\right)_{i,j}-M_{i,j}\right|>\tau}\right),\mathcal{P}_{\Omega}\left(\bm{\Delta}_{\bm{X}}\bm{\Delta}_{\bm{Y}}^{\top}\right)\right\rangle \right|\nonumber \\
 & =\left|\left\langle \frac{1}{p}\left\{ \delta_{i,j}\ind_{\left|\left(\bm{X}\bm{Y}^{\top}\right)_{i,j}-M_{i,j}\right|>\tau}\right\} _{i,j},\bm{X}\left(\bm{Y}-\bm{Y}^{\star}\right)^{\top}\circ\bm{\Delta}_{\bm{X}}\bm{\Delta}_{\bm{Y}}^{\top}\right\rangle \right|\nonumber \\
 & \leq\left\Vert \frac{1}{p}\left\{ \delta_{i,j}\ind_{\left|\left(\bm{X}\bm{Y}^{\top}\right)_{i,j}-M_{i,j}\right|>\tau}\right\} _{i,j}\right\Vert \left\Vert \bm{X}\left(\bm{Y}-\bm{Y}^{\star}\right)^{\top}\circ\bm{\Delta}_{\bm{X}}\bm{\Delta}_{\bm{Y}}^{\top}\right\Vert _{*}\nonumber \\
 & \overset{\text{(i)}}{\leq}\left\Vert \frac{1}{p}\left\{ \delta_{i,j}\ind_{\left|\varepsilon_{i,j}\right|>\tau/2}\right\} _{i,j}\right\Vert \left\Vert \bm{X}\right\Vert _{2,\infty}\left\Vert \bm{Y}-\bm{Y}^{\star}\right\Vert _{2,\infty}\left\Vert \bm{\Delta}\right\Vert _{\mathrm{F}}^{2}\nonumber \\
 & \overset{\text{(ii)}}{\lesssim}n\mathbb{P}\left(\left|\varepsilon_{i,j}\right|>\tau/2\right)\frac{c\mu r\sigma_{\max}}{n}\left\Vert \bm{\Delta}\right\Vert _{\mathrm{F}}^{2}\nonumber \\
 & \overset{\text{(iii)}}{\lesssim}\frac{\mu r\sigma_{\max}}{np}\left\Vert \bm{\Delta}\right\Vert _{\mathrm{F}}^{2}.\label{eq:beta1}
\end{align}
Here (i) comes from the fact that 
\[
\ind_{\left|\left(\bm{X}\bm{Y}^{\top}\right)_{i,j}-M_{i,j}\right|>\tau}=\ind_{\left|\left(\bm{X}\bm{Y}^{\top}\right)_{i,j}-M_{i,j}^{\star}-\varepsilon_{i,j}\right|>\tau}\leq\ind_{\left|\varepsilon_{i,j}\right|>\tau-\left|\left(\bm{X}\bm{Y}^{\top}\right)_{i,j}-M_{i,j}^{\star}\right|}\leq\ind_{\left|\varepsilon_{i,j}\right|>\tau/2},
\]
where the last inequality is due to 
\begin{equation}
\max_{i,j}\left|\left(\bm{X}\bm{Y}^{\top}\right)_{i,j}-M_{i,j}^{\star}\right|\leq\left\Vert \bm{F}-\bm{F}^{\star}\right\Vert _{2,\infty}\left\Vert \bm{F}^{\star}\right\Vert _{2,\infty}\ll\left\Vert \bm{F}^{\star}\right\Vert _{2,\infty}^{2}\leq\frac{\mu r\sigma_{\max}}{n}\leq\frac{\tau}{2};\label{eq:strongcvx-error}
\end{equation}
(ii) comes from Lemma \ref{lem:bernoulli}, \eqref{eq:strongcvx-2infty}
and the fact that $\left\Vert \bm{F}^{\star}\right\Vert _{2,\infty}\leq\sqrt{\mu r\sigma_{\max}/n}$;
(iii) applies Markov inequality to obtain that 
\[
\mathbb{P}\left(\left|\varepsilon_{i,j}\right|>\frac{\tau}{2}\right)\leq\frac{\sigma_{i,j}^{2}}{\left(\tau/2\right)^{2}}\leq\frac{\sigma^{2}}{\left(\tau/2\right)^{2}}.
\]
Analogously, one has 
\begin{equation}
\beta_{2}\lesssim\frac{\mu r\sigma_{\max}}{np}\left\Vert \bm{\Delta}\right\Vert _{\mathrm{F}}^{2}.\label{eq:beta2}
\end{equation}
Hence, plugging \eqref{eq:beta1} and \eqref{eq:beta2} into \eqref{eq:strongcvx-alpha1}
yields 
\begin{equation}
\alpha_{1}\lesssim\frac{\mu r\sigma_{\max}}{np}\left\Vert \bm{\Delta}\right\Vert _{\mathrm{F}}^{2}.\label{eq:strongcvx-alpha1-bound}
\end{equation}
\item Turning attention to $\alpha_{2}$, one has 
\begin{align*}
 & \left|\frac{2}{p}\left\langle \mathcal{P}_{\Omega}\left(\varepsilon_{i,j}\ind_{\left|\left(\bm{X}\bm{Y}^{\top}\right)_{i,j}-M_{i,j}\right|\leq\tau}\right),\mathcal{P}_{\Omega}\left(\bm{\Delta}_{\bm{X}}\bm{\Delta}_{\bm{Y}}^{\top}\right)\right\rangle \right|\\
 & \leq\frac{2}{p}\left\Vert \mathcal{P}_{\Omega}\left(\varepsilon_{i,j}\ind_{\left|\left(\bm{X}\bm{Y}^{\top}\right)_{i,j}-M_{i,j}\right|\leq\tau}\right)\right\Vert \left\Vert \bm{\Delta}_{\bm{X}}\bm{\Delta}_{\bm{Y}}^{\top}\right\Vert _{*}\\
 & \leq\frac{2}{p}\left\Vert \mathcal{P}_{\Omega}\left(\varepsilon_{i,j}\ind_{\left|\left(\bm{X}\bm{Y}^{\top}\right)_{i,j}-M_{i,j}\right|\leq\tau}\right)\right\Vert \left\Vert \bm{\Delta}\right\Vert _{\mathrm{F}}^{2},
\end{align*}
where the last inequality can be easily obtained from the elementary
fact of the nuclear norm that 
\begin{equation}
\left\Vert \bm{Z}\right\Vert _{*}=\inf_{\bm{U},\bm{V}\in\mathbb{R}^{n\times r},\bm{U}\bm{V}^{\top}=\bm{Z}}\left\{ \frac{1}{2}\left\Vert \bm{U}\right\Vert _{\mathrm{F}}^{2}+\frac{1}{2}\left\Vert \bm{V}\right\Vert _{\mathrm{F}}^{2}\right\} .\label{eq:nuclear}
\end{equation}
To bound $\frac{2}{p}\Vert\mathcal{P}_{\Omega}(\varepsilon_{i,j}\ind_{|(\bm{X}\bm{Y}^{\top})_{i,j}-M_{i,j}|\leq\tau})\Vert$,
one has 
\begin{align}
 & \frac{2}{p}\left\Vert \mathcal{P}_{\Omega}\left(\varepsilon_{i,j}\ind_{\left|\left(\bm{X}\bm{Y}^{\top}\right)_{i,j}-M_{i,j}\right|\leq\tau}\right)\right\Vert \nonumber \\
 & \leq\underbrace{\frac{2}{p}\left\Vert \mathcal{P}_{\Omega}\left(\varepsilon_{i,j}\ind_{\left|\varepsilon_{i,j}\right|\leq\tau}-\mathbb{E}\left[\varepsilon_{i,j}\ind_{\left|\varepsilon_{i,j}\right|\leq\tau}\right]\right)\right\Vert }_{\eqqcolon\beta_{1}}\nonumber \\
 & \qquad+\underbrace{\frac{2}{p}\left\Vert \mathcal{P}_{\Omega}\left(\left|\varepsilon_{i,j}\left(\ind_{\left|\left(\bm{X}\bm{Y}^{\top}\right)_{i,j}-M_{i,j}\right|\leq\tau}-\ind_{\left|\varepsilon_{i,j}\right|\leq\tau}\right)\right|\right)\right\Vert }_{\eqqcolon\beta_{2}}+\underbrace{\frac{2}{p}\left\Vert \mathcal{P}_{\Omega}\left(\left|\mathbb{E}\left[\varepsilon_{i,j}\ind_{\left|\varepsilon_{i,j}\right|\leq\tau}\right]\right|\right)\right\Vert }_{\eqqcolon\beta_{3}},\label{eq:strongcvx-alpha2}
\end{align}
To bound $\beta_{1}$, we intend to apply \eqref{lem:noise}, which needs the following quantities:
\begin{align*}
		\mathbb{V}\left[\frac{1}{p}\left(\varepsilon_{i,j}\ind_{\left|\varepsilon_{i,j}\right|\leq\tau}-\mathbb{E}\left[\varepsilon_{i,j}\ind_{\left|\varepsilon_{i,j}\right|\leq\tau}\right]\right)\right]&=\mathbb{V}\left[\frac{1}{p}\varepsilon_{i,j}\ind_{\left|\varepsilon_{i,j}\right|\leq\tau}\right]\leq\mathbb{E}\left[\frac{1}{p^2}\varepsilon^2_{i,j}\ind_{\left|\varepsilon_{i,j}\right|\leq\tau}\right]\\
		&\leq\mathbb{E}\left[\frac{1}{p^2}\varepsilon^2_{i,j}\right] \leq\frac{\sigma^2}{p^2}\eqqcolon\widetilde{\sigma}^2,
\end{align*}
and \begin{align*}
	B&\coloneqq \max_{i,j}\left|\frac{1}{p}\left(\varepsilon_{i,j}\ind_{\left|\varepsilon_{i,j}\right|\leq\tau}-\mathbb{E}\left[\varepsilon_{i,j}\ind_{\left|\varepsilon_{i,j}\right|\leq\tau}\right]\right)\right|\leq\frac{2\tau}{p}.
\end{align*}

Therefore, applying \eqref{lem:noise} to $\beta_{1}$ yields
\begin{equation}
\beta_{1}\lesssim\widetilde{\sigma}\sqrt{n}+B\sqrt{\log n}\lesssim\frac{\sigma\sqrt{n}+\sqrt{\log n}\left\Vert \bm{M}^{\star}\right\Vert _{\infty}}{\sqrt{p}}.\label{eq:strongcvx-beta1}
\end{equation}
For $\beta_{2}$, one has 
\begin{align*}
\beta_{2} & \overset{\text{(i)}}{\leq}\frac{2}{p}\left\Vert \mathcal{P}_{\Omega}\left(\left|\varepsilon_{i,j}\ind_{\tau-\left\Vert \bm{X}\bm{Y}^{\top}-\bm{M}^{\star}\right\Vert _{\infty}\leq\varepsilon_{i,j}\leq\tau+\left\Vert \bm{X}\bm{Y}^{\top}-\bm{M}^{\star}\right\Vert _{\infty}}\right|\right)\right\Vert \\
 & \qquad+\frac{2}{p}\left\Vert \mathcal{P}_{\Omega}\left(\left|\varepsilon_{i,j}\ind_{-\tau-\left\Vert \bm{X}\bm{Y}^{\top}-\bm{M}^{\star}\right\Vert _{\infty}\leq\varepsilon_{i,j}\leq-\tau+\left\Vert \bm{X}\bm{Y}^{\top}-\bm{M}^{\star}\right\Vert _{\infty}}\right|\right)\right\Vert \\
 & \overset{\text{(ii)}}{\leq}\frac{2}{p}\left\Vert \mathcal{P}_{\Omega}\left(\left|\varepsilon_{i,j}\ind_{\tau/2\leq\varepsilon_{i,j}\leq3\tau/2}\right|\right)\right\Vert +\frac{2}{p}\left\Vert \mathcal{P}_{\Omega}\left(\left|\varepsilon_{i,j}\ind_{-3\tau/2\leq\varepsilon_{i,j}\leq-\tau/2}\right|\right)\right\Vert \\
 & \leq\underbrace{\frac{2}{p}\left\Vert \mathcal{P}_{\Omega}\left(\left|\varepsilon_{i,j}\ind_{\tau/2\leq\varepsilon_{i,j}\leq3\tau/2}\right|-\mathbb{E}\left[\left|\varepsilon_{i,j}\ind_{\tau/2\leq\varepsilon_{i,j}\leq3\tau/2}\right|\right]\right)\right\Vert }_{\eqqcolon\theta_{1}}\\
 & \qquad+\underbrace{\frac{2}{p}\left\Vert \mathcal{P}_{\Omega}\left(\left|\varepsilon_{i,j}\ind_{-3\tau/2\leq\varepsilon_{i,j}\leq-\tau/2}\right|-\mathbb{E}\left[\left|\varepsilon_{i,j}\ind_{-3\tau/2\leq\varepsilon_{i,j}\leq-\tau/2}\right|\right]\right)\right\Vert }_{\eqqcolon\theta_{2}}\\
 & \qquad+\underbrace{\frac{2}{p}\left\Vert \mathcal{P}_{\Omega}\left(\mathbb{E}\left[\left|\varepsilon_{i,j}\ind_{\tau/2\leq\varepsilon_{i,j}\leq3\tau/2}\right|\right]\right)\right\Vert }_{\eqqcolon\theta_{3}}+\underbrace{\frac{2}{p}\left\Vert \mathcal{P}_{\Omega}\left(\mathbb{E}\left[\left|\varepsilon_{i,j}\ind_{-3\tau/2\leq\varepsilon_{i,j}\leq-\tau/2}\right|\right]\right)\right\Vert }_{\eqqcolon\theta_{4}},
\end{align*}
where (i) is due to 
\begin{align*}
 & \left|\varepsilon_{i,j}\left(\ind_{\left|\left(\bm{X}\bm{Y}^{\top}\right)_{i,j}-M_{i,j}\right|>\tau}-\ind_{\left|\varepsilon_{i,j}\right|>\tau}\right)\right|=\left|\varepsilon_{i,j}\left(\ind_{\left|\left(\bm{X}\bm{Y}^{\top}\right)_{i,j}-M_{i,j}\right|\leq\tau}-\ind_{\left|\varepsilon_{i,j}\right|\leq\tau}\right)\right|\\
 & \leq\left|\varepsilon_{i,j}\left(\ind_{\left|\left(\bm{X}\bm{Y}^{\top}\right)_{i,j}-M_{i,j}\right|\leq\tau}-\ind_{\left|\varepsilon_{i,j}\right|\leq\tau}\right)\ind_{\varepsilon_{i,j}\geq0}\right|+\left|\varepsilon_{i,j}\left(\ind_{\left|\left(\bm{X}\bm{Y}^{\top}\right)_{i,j}-M_{i,j}\right|\leq\tau}-\ind_{\left|\varepsilon_{i,j}\right|\leq\tau}\right)\ind_{\varepsilon_{i,j}<0}\right|\\
 & \leq\left|\varepsilon_{i,j}\ind_{\tau-\left\Vert \bm{X}\bm{Y}^{\top}-\bm{M}^{\star}\right\Vert _{\infty}\leq\varepsilon_{i,j}\leq\tau+\left\Vert \bm{X}\bm{Y}^{\top}-\bm{M}^{\star}\right\Vert _{\infty}}\right|+\left|\varepsilon_{i,j}\ind_{-\tau-\left\Vert \bm{X}\bm{Y}^{\top}-\bm{M}^{\star}\right\Vert _{\infty}\leq\varepsilon_{i,j}\leq-\tau+\left\Vert \bm{X}\bm{Y}^{\top}-\bm{M}^{\star}\right\Vert _{\infty}}\right|,
\end{align*}
and (ii) comes from \eqref{eq:strongcvx-error}. Another application
of \eqref{lem:noise} gives rise to 
\begin{align*}
\frac{2}{p}\left\Vert \mathcal{P}_{\Omega}\left(\left|\varepsilon_{i,j}\ind_{\tau/2\leq\varepsilon_{i,j}\leq3\tau/2}\right|-\mathbb{E}\left[\left|\varepsilon_{i,j}\ind_{\tau/2\leq\varepsilon_{i,j}\leq3\tau/2}\right|\right]\right)\right\Vert  & \lesssim\frac{\sigma\sqrt{n}+\sqrt{\log n}\left\Vert \bm{M}^{\star}\right\Vert _{\infty}}{\sqrt{p}},\\
\frac{2}{p}\left\Vert \mathcal{P}_{\Omega}\left(\left|\varepsilon_{i,j}\ind_{-3\tau/2\leq\varepsilon_{i,j}\leq-\tau/2}\right|-\mathbb{E}\left[\left|\varepsilon_{i,j}\ind_{-3\tau/2\leq\varepsilon_{i,j}\leq-\tau/2}\right|\right]\right)\right\Vert  & \lesssim\frac{\sigma\sqrt{n}+\sqrt{\log n}\left\Vert \bm{M}^{\star}\right\Vert _{\infty}}{\sqrt{p}}.
\end{align*}
Regarding $\theta_{3}$, we have 
\begin{align}
\theta_{3} & \overset{\text{(i)}}{\lesssim}\frac{2}{p}\left\Vert \mathcal{P}_{\Omega}\left(\bm{1}\bm{1}^{\top}\right)\right\Vert \cdot\max_{i,j}\mathbb{E}\left[\left|\varepsilon_{i,j}\ind_{\tau/2\leq\varepsilon_{i,j}\leq3\tau/2}\right|\right]\nonumber \\
 & \overset{\text{(ii)}}{\lesssim}\frac{2}{p}\left\Vert \mathcal{P}_{\Omega}\left(\bm{1}\bm{1}^{\top}\right)\right\Vert \sqrt{\mathbb{E}\left[\varepsilon_{i,j}^{2}\right]\mathbb{E}\left[\ind_{\tau/2\leq\varepsilon_{i,j}\leq3\tau/2}\right]}\nonumber \\
 & \overset{\text{(iii)}}{\lesssim}\sigma\sqrt{\frac{n}{p}},\label{eq:strongcvx-beta2}
\end{align}
where (i) applies the observation that for any matrix $\bm{W}$, it
holds that 
\begin{equation}
\left\Vert \bm{W}\right\Vert \leq\left\Vert \left\{ \left|W_{i,j}\right|\right\} _{i,j}\right\Vert \leq\left\Vert \left\{ \ind_{\left\{ W_{i,j}\neq0\right\} }\right\} _{i,j}\right\Vert \max_{i,j}\left|W_{i,j}\right|;\label{eq:operator-ob}
\end{equation}
(ii) arises from the Cauchy-Schwartz inequality; (iii) follows from
Lemma \ref{lem:bernoulli} and the application of Markov inequality
\[
\mathbb{E}\left[\ind_{\tau/2\leq\varepsilon_{i,j}\leq3\tau/2}\right]\leq\mathbb{P}\left[\tau/2\leq\varepsilon_{i,j}\right]\leq\frac{\sigma_{i,j}^{2}}{\left(\tau/2\right)^{2}}\leq\frac{\sigma^{2}}{\left(\tau/2\right)^{2}}.
\]
The bound of $\theta_{4}$ follows analogously from \eqref{eq:beta2}.
Finally, we have 
\begin{align}
\beta_{3} & \overset{\text{(i)}}{\leq}\frac{2}{p}\left\Vert \mathcal{P}_{\Omega}\left(\bm{1}\bm{1}^{\top}\right)\right\Vert \cdot\max_{i,j}\left|\mathbb{E}\left[\varepsilon_{i,j}\ind_{\left|\varepsilon_{i,j}\right|>\tau}\right]\right|\nonumber \\
 & \overset{\text{(ii)}}{\lesssim}n\sqrt{\mathbb{E}\left[\varepsilon_{i,j}^{2}\right]\mathbb{E}\left[\ind_{\left|\varepsilon_{i,j}\right|>\tau}\right]}\nonumber \\
 & \overset{\text{(iii)}}{\lesssim}\sigma\sqrt{\frac{n}{p}},\label{eq:strongcvx-beta4}
\end{align}
where (i) holds due to \eqref{eq:operator-ob}; (ii) comes from Lemma
\eqref{lem:bernoulli} and the Cauchy-Schwartz inequality; (iii) relies
on Markov inequality 
\[
\mathbb{E}\left[\ind_{\left|\varepsilon_{i,j}\right|>\tau}\right]\leq\mathbb{P}\left[\left|\varepsilon_{i,j}\right|>\tau\right]\leq\frac{\sigma^{2}}{\tau^{2}}.
\]
Plugging the bounds of $\{\beta_{i}\}_{i=1}^{3}$ into \eqref{eq:strongcvx-alpha2}
yields 
\begin{equation}
\left\Vert \mathcal{P}_{\Omega}\left(\varepsilon_{i,j}\ind_{\left|\left(\bm{X}\bm{Y}^{\top}\right)_{i,j}-M_{i,j}\right|\leq\tau}\right)\right\Vert \lesssim\sigma\sqrt{\frac{n}{p}},\label{eq:noise}
\end{equation}
and therefore, 
\begin{equation}
\alpha_{2}\leq\frac{2}{p}\left\Vert \mathcal{P}_{\Omega}\left(\varepsilon_{i,j}\ind_{\left|\left(\bm{X}\bm{Y}^{\top}\right)_{i,j}-M_{i,j}\right|\leq\tau}\right)\right\Vert \left\Vert \bm{\Delta}\right\Vert _{\mathrm{F}}^{2}\lesssim\left(\frac{\sigma\sqrt{n}+\sqrt{\log n}\left\Vert \bm{M}^{\star}\right\Vert _{\infty}}{\sqrt{p}}\right)\left\Vert \bm{\Delta}\right\Vert _{\mathrm{F}}^{2}.\label{eq:strongcvx-alpha2-bound}
\end{equation}
\item Next, $\alpha_{3}$ can be decomposed as 
\begin{align}
 & \frac{1}{p}\left\Vert \mathcal{P}_{\Omega}\left(\left(\bm{\Delta}_{\bm{X}}\bm{Y}^{\top}+\bm{X}\bm{\Delta}_{\bm{Y}}^{\top}\right)_{i,j}\ind_{\left|\left(\bm{X}\bm{Y}^{\top}\right)_{i,j}-M_{i,j}\right|>\tau}\right)\right\Vert _{\mathrm{F}}^{2}\nonumber \\
 & \leq\frac{2}{p}\left\Vert \mathcal{P}_{\Omega}\left(\left\{ \left(\bm{\Delta}_{\bm{X}}\bm{Y}^{\top}\right)_{i,j}\ind_{\left|\left(\bm{X}\bm{Y}^{\top}\right)_{i,j}-M_{i,j}\right|>\tau}\right\} _{i,j}\right)\right\Vert _{\mathrm{F}}^{2}\nonumber \\
 &\qquad +\frac{2}{p}\left\Vert \mathcal{P}_{\Omega}\left(\left\{ \left(\bm{X}\bm{\Delta}_{\bm{Y}}^{\top}\right)_{i,j}\ind_{\left|\left(\bm{X}\bm{Y}^{\top}\right)_{i,j}-M_{i,j}\right|>\tau}\right\} _{i,j}\right)\right\Vert _{\mathrm{F}}^{2}.\label{eq:strongcvx-alpha3}
\end{align}
The first term on the right-hand side can be bounded by 
\begin{align}
 & \frac{2}{p}\left\Vert \mathcal{P}_{\Omega}\left(\left\{ \left(\bm{\Delta}_{\bm{X}}\bm{Y}^{\top}\right)_{i,j}\ind_{\left|\left(\bm{X}\bm{Y}^{\top}\right)_{i,j}-M_{i,j}\right|>\tau}\right\} _{i,j}\right)\right\Vert _{\mathrm{F}}^{2}\nonumber \\
 & \overset{\text{(i)}}{\leq}\frac{2}{p}\left\Vert \mathcal{P}_{\Omega}\left(\left\{ \left(\bm{\Delta}_{\bm{X}}\bm{Y}^{\top}\right)_{i,j}\ind_{\left|\varepsilon_{i,j}\right|>\tau/2}\right\} _{i,j}\right)\right\Vert _{\mathrm{F}}^{2}\nonumber \\
 & \overset{\text{(ii)}}{\leq}4n\min\left\{ \left\Vert \bm{\Delta}_{\bm{X}}\right\Vert _{\mathrm{F}}^{2}\left\Vert \bm{Y}\right\Vert _{2,\infty}^{2},\left\Vert \bm{Y}\right\Vert _{\mathrm{F}}^{2}\left\Vert \bm{\Delta}_{\bm{X}}\right\Vert _{2,\infty}^{2}\right\} \mathbb{P}\left(\left|\varepsilon_{i,j}\right|>\frac{\tau}{2}\right)\nonumber \\
 & \overset{\text{(iii)}}{\lesssim}\frac{\mu r\sigma_{\max}}{np}\left\Vert \bm{\Delta}\right\Vert _{\mathrm{F}}^{2}.\label{eq:strongcvx-alpha3-1}
\end{align}
Here (i) comes from the fact that 
\begin{equation}
\ind_{\left|\left(\bm{X}\bm{Y}^{\top}\right)_{i,j}-M_{i,j}\right|>\tau}=\ind_{\left|\left(\bm{X}\bm{Y}^{\top}\right)_{i,j}-M_{i,j}^{\star}-\varepsilon_{i,j}\right|>\tau}\leq\ind_{\left|\varepsilon_{i,j}\right|>\tau-\left|\left(\bm{X}\bm{Y}^{\top}\right)_{i,j}-M_{i,j}^{\star}\right|}\leq\ind_{\left|\varepsilon_{i,j}\right|>\tau/2},\label{eq:ind}
\end{equation}
where the last inequality is due to 
\[
\max_{i,j}\left|\left(\bm{X}\bm{Y}^{\top}\right)_{i,j}-M_{i,j}^{\star}\right|\leq\left\Vert \bm{F}-\bm{F}^{\star}\right\Vert _{2,\infty}\left\Vert \bm{F}^{\star}\right\Vert _{2,\infty}\ll\left\Vert \bm{F}^{\star}\right\Vert _{2,\infty}^{2}\leq\frac{\mu r\sigma_{\max}}{n}\leq\frac{\tau}{2};
\]
(ii) invokes \citet[Lemma A.3]{chen2020nonconvex} and holds uniformly
for all matrices $\bm{\Delta}_{\bm{X}}\in\mathbb{R}^{n\times r}$
and $\bm{Y}\in\mathbb{R}^{n\times r}$ with probability over $1-O(n^{-10})$;
(iii) applies Markov inequality to obtain that 
\begin{equation}
\mathbb{P}\left(\left|\varepsilon_{i,j}\right|>\frac{\tau}{2}\right)\leq\frac{\sigma_{i,j}^{2}}{\left(\tau/2\right)^{2}}\leq\frac{\sigma^{2}}{\left(\tau/2\right)^{2}}.\label{eq:markov}
\end{equation}
Analogously, one has 
\begin{equation}
\frac{2}{p}\left\Vert \mathcal{P}_{\Omega}\left(\left\{ \left(\bm{\Delta}_{\bm{X}}\bm{Y}^{\top}\right)_{i,j}\ind_{\left|\left(\bm{X}\bm{Y}^{\top}\right)_{i,j}-M_{i,j}\right|>\tau}\right\} _{i,j}\right)\right\Vert _{\mathrm{F}}^{2}\lesssim\frac{\mu r\sigma_{\max}}{np}\left\Vert \bm{\Delta}\right\Vert _{\mathrm{F}}^{2}.\label{eq:strongcvx-alpha3-2}
\end{equation}
Taking \eqref{eq:strongcvx-alpha3-1}, \eqref{eq:strongcvx-alpha3-2}
together with \eqref{eq:strongcvx-alpha3} yields 
\begin{align*}
 & \frac{1}{p}\left\Vert \mathcal{P}_{\Omega}\left(\left(\bm{\Delta}_{\bm{X}}\bm{Y}^{\top}+\bm{X}\bm{\Delta}_{\bm{Y}}^{\top}\right)_{i,j}\ind_{\left|\left(\bm{X}\bm{Y}^{\top}\right)_{i,j}-M_{i,j}\right|>\tau}\right)\right\Vert _{\mathrm{F}}^{2}\lesssim\frac{\mu r\sigma_{\max}}{np}\left\Vert \bm{\Delta}\right\Vert _{\mathrm{F}}^{2},
\end{align*}
and hence, 
\begin{equation}
\alpha_{3}\lesssim\frac{\mu r\sigma_{\max}}{np}\left\Vert \bm{\Delta}\right\Vert _{\mathrm{F}}^{2}.\label{eq:strongcvx-alpha3-bound}
\end{equation}
\item Finally, we turn to consider $\alpha_{4}$. Simple calculation reveals
that 
\end{enumerate}
\begin{align*}
 & \left|\frac{1}{p}\left\langle \mathcal{P}_{\Omega}\left(\left\{ \tau\mathrm{sgn}\left(\left(\left(\bm{X}\bm{Y}^{\top}\right)_{i,j}-M_{i,j}\right)_{i,j}\right)\ind_{\left|\left(\bm{X}\bm{Y}^{\top}\right)_{i,j}-M_{i,j}\right|>\tau}\right\} _{i,j}\right),\mathcal{P}_{\Omega}\left(\bm{\Delta}_{\bm{X}}\bm{\Delta}_{\bm{Y}}^{\top}\right)\right\rangle \right|\\
 & =\left|\frac{1}{p}\left\langle \mathcal{P}_{\Omega}\left(\left\{ \tau\mathrm{sgn}\left(\left(\left(\bm{X}\bm{Y}^{\top}\right)_{i,j}-M_{i,j}\right)_{i,j}\right)\ind_{\left|\left(\bm{X}\bm{Y}^{\top}\right)_{i,j}-M_{i,j}\right|>\tau}\right\} _{i,j}\right),\bm{\Delta}_{\bm{X}}\bm{\Delta}_{\bm{Y}}^{\top}\right\rangle \right|\\
 & \leq\left\Vert \frac{1}{p}\mathcal{P}_{\Omega}\left(\left\{ \tau\mathrm{sgn}\left(\left(\left(\bm{X}\bm{Y}^{\top}\right)_{i,j}-M_{i,j}\right)_{i,j}\right)\ind_{\left|\left(\bm{X}\bm{Y}^{\top}\right)_{i,j}-M_{i,j}\right|>\tau}\right\} _{i,j}\right)\right\Vert \left\Vert \bm{\Delta}_{\bm{X}}\bm{\Delta}_{\bm{Y}}^{\top}\right\Vert _{*}.
\end{align*}
One has \eqref{eq:nuclear} suggests that $\Vert\bm{\Delta}_{\bm{X}}\bm{\Delta}_{\bm{Y}}^{\top}\Vert_{*}\leq\Vert\bm{\Delta}\Vert_{\mathrm{F}}^{2}$.
In addition, we have

\begin{align}
 & \left\Vert \frac{1}{p}\mathcal{P}_{\Omega}\left(\left\{ \tau\mathrm{sgn}\left(\left(\left(\bm{X}\bm{Y}^{\top}\right)_{i,j}-M_{i,j}\right)_{i,j}\right)\ind_{\left|\left(\bm{X}\bm{Y}^{\top}\right)_{i,j}-M_{i,j}\right|>\tau}\right\} _{i,j}\right)\right\Vert \nonumber \\
 & \overset{\text{(i)}}{\leq}\frac{\tau}{p}\left\Vert \mathcal{P}_{\Omega}\left(\left\{ \ind_{\left|\left(\bm{X}\bm{Y}^{\top}\right)_{i,j}-M_{i,j}\right|>\tau}\right\} _{i,j}\right)\right\Vert \\
 & \overset{\text{(ii)}}{\leq}\frac{\tau}{p}\left\Vert \left\{ \delta_{i,j}\ind_{\left|\varepsilon_{i,j}\right|>\tau/2}\right\} _{i,j}\right\Vert \nonumber \\
 & \overset{\text{(iii)}}{\lesssim}\tau n\mathbb{P}\left(\left|\varepsilon_{i,j}\right|>\tau/2\right)\overset{\text{(iv)}}{\lesssim}\tau n\frac{\sigma^{2}}{\left(\tau/2\right)^{2}}\lesssim\sigma\sqrt{\frac{n}{p}}.\label{eq:tau}
\end{align}

Here, (i) holds due to \eqref{eq:operator-ob}; (ii) follows from
\eqref{eq:ind}; (iii) applies Lemma \ref{lem:bernoulli}; (iv) arises
from \eqref{eq:markov}. Therefore, one has 
\begin{equation}
\alpha_{4}\lesssim\sigma\sqrt{\frac{n}{p}}\left\Vert \bm{\Delta}\right\Vert _{\mathrm{F}}^{2}.\label{eq:strongcvx-alpha4-bound}
\end{equation}

Plugging \eqref{eq:strongcvx-alpha1-bound}, \eqref{eq:strongcvx-alpha2-bound},
\eqref{eq:strongcvx-alpha3-bound}, and \eqref{eq:strongcvx-alpha4-bound}
into \eqref{eq:strongcvx-E} yields 
\[
\left|\mathsf{vec}\left(\bm{\Delta}\right)^{\top}\nabla^{2}f\left(\bm{X},\bm{Y}\right)\mathsf{vec}\left(\bm{\Delta}\right)-\mathcal{E}_{0}\right|\lesssim\left(\frac{\sigma\sqrt{n}+\sqrt{\log n}\left\Vert \bm{M}^{\star}\right\Vert _{\infty}}{\sqrt{p}}\right)\left\Vert \bm{\Delta}\right\Vert _{\mathrm{F}}^{2}.
\]
Combining this with \eqref{eq:strongcvx-E0} gives 
\begin{align*}
\mathsf{vec}\left(\bm{\Delta}\right)^{\top}\nabla^{2}f\left(\bm{X},\bm{Y}\right)\mathsf{vec}\left(\bm{\Delta}\right)\sigma & \geq\frac{\sigma_{\min}}{10}\left\Vert \bm{\Delta}\right\Vert _{\mathrm{F}}^{2}-\left(\frac{\sigma\sqrt{n}+\sqrt{\log n}\left\Vert \bm{M}^{\star}\right\Vert _{\infty}}{\sqrt{p}}\right)\left\Vert \bm{\Delta}\right\Vert _{\mathrm{F}}^{2}\\
 & \ge\frac{\sigma_{\min}}{20}\left\Vert \bm{\Delta}\right\Vert _{\mathrm{F}}^{2},
\end{align*}
where the last inequality holds provided that $\frac{\sigma}{\sigma_{\min}}\sqrt{\frac{n}{p}}\ll1$
and $np\gg\mu^{2}r^{2}\kappa^{2}$.

\subsection{Proof of Lemma \ref{lem:contraction}\label{subsec:Proof-of-Lemma-contraction}}

The definition of $\bm{H}^{t+1}$ (cf.~\eqref{eq:def-ht}) and the
update rule \eqref{subeq:gradient_update_ncvx} give

\begin{align}
\left\Vert \bm{F}^{t+1}\bm{H}^{t+1}-\bm{F}^{\star}\right\Vert _{\mathrm{F}} & \leq\left\Vert \bm{F}^{t+1}\bm{H}^{t}-\bm{F}^{\star}\right\Vert _{\mathrm{F}}=\left\Vert \left[\bm{F}^{t}-\eta\nabla f\left(\bm{F}^{t}\right)\right]\bm{H}^{t}-\bm{F}^{\star}\right\Vert _{\mathrm{F}}\nonumber \\
 & \overset{\text{(i)}}{=}\left\Vert \bm{F}^{t}\bm{H}^{t}-\eta\nabla f\left(\bm{F}^{t}\bm{H}^{t}\right)-\bm{F}^{\star}\right\Vert _{\mathrm{F}}\nonumber \\
 & \leq\underbrace{\left\Vert \bm{F}^{t}\bm{H}^{t}-\eta\nabla f\left(\bm{F}^{t}\bm{H}^{t}\right)-\left[\bm{F}^{\star}-\eta\nabla f\left(\bm{F}^{\star}\right)\right]\right\Vert _{\mathrm{F}}}_{\eqqcolon\alpha_{1}}+\underbrace{\eta\left\Vert \nabla f\left(\bm{F}^{\star}\right)\right\Vert _{\mathrm{F}}}_{\eqqcolon\alpha_{2}},\label{eq:fro}
\end{align}
where (i) holds due to the fact that $\nabla f(\bm{F}\bm{R})=\nabla f(\bm{F})\bm{R}$
for all $\bm{R}\in\mathcal{O}^{r\times r}$. $\alpha_{1}$ is exactly
the term $\alpha_{1}$ in \citet[Section D.3]{chen2020noisy} with
$f_{\mathsf{aug}}$ replaced by $f$. Reusing the results therein,
we obtain 
\[
\alpha_{1}\leq\left(1-\frac{\sigma_{\min}}{20}\eta\right)\left\Vert \bm{F}^{t}\bm{H}^{t}-\bm{F}^{\star}\right\Vert _{\mathrm{F}},
\]
holds as long as $\frac{\sigma}{\sigma_{\min}}\sqrt{\frac{n}{p}}\ll\frac{1}{\sqrt{\kappa^{4}\mu r^{2}\log^{2}n}}$
and $np\gg\kappa^{6}\mu^{2}r^{4}\log^{2}n$. Regarding $\alpha_{2}$,
one has 
\begin{align*}
\left\Vert \nabla f\left(\bm{F}^{\star}\right)\right\Vert _{\mathrm{F}} & =\left\Vert \left[\begin{array}{c}
\frac{1}{2p}\mathcal{P}_{\Omega}\left(\left\{ \psi_{\tau}\left(-\varepsilon_{i,j}\right)\right\} _{i,j}\right)\bm{Y}^{\star}\\
\frac{1}{2p}\mathcal{P}_{\Omega}\left(\left\{ \psi_{\tau}\left(-\varepsilon_{i,j}\right)\right\} _{i,j}\right)^{\top}\bm{X}^{\star}
\end{array}\right]\right\Vert _{\mathrm{F}}\leq\frac{1}{2p}\left\Vert \mathcal{P}_{\Omega}\left(\left\{ \psi_{\tau}\left(-\varepsilon_{i,j}\right)\right\} _{i,j}\right)\right\Vert \left\Vert \bm{F}^{\star}\right\Vert _{\mathrm{F}}.
\end{align*}
Substitution of the definition of $\psi_{\tau}(\cdot)$ (cf.~\eqref{eq:psi-tau})
into the equation above gives 
\begin{align}
&	\frac{1}{2p}\left\Vert \mathcal{P}_{\Omega}\left(\left\{ \psi_{\tau}\left(-\varepsilon_{i,j}\right)\right\} _{i,j}\right)\right\Vert  \\
&\quad =\frac{1}{p}\left\Vert \mathcal{P}_{\Omega}\left(\left\{ 2\varepsilon_{i,j}\ind_{\left|\varepsilon_{i,j}\right|\leq\tau}+\tau\mathrm{sgn}\left(\varepsilon_{i,j}\right)\ind_{\left|\varepsilon_{i,j}\right|>\tau}\right\} _{i,j}\right)\right\Vert \nonumber \\
 &\quad \leq\underbrace{\frac{1}{p}\left\Vert \mathcal{P}_{\Omega}\left(\left\{ \varepsilon_{i,j}\ind_{\left|\varepsilon_{i,j}\right|\leq\tau}\right\} _{i,j}\right)\right\Vert }_{\eqqcolon\beta_{1}}+\underbrace{\frac{1}{2p}\left\Vert \mathcal{P}_{\Omega}\left(\left\{ \tau\mathrm{sgn}\left(\varepsilon_{i,j}\right)\ind_{\left|\varepsilon_{i,j}\right|>\tau}\right\} _{i,j}\right)\right\Vert }_{\eqqcolon\beta_{2}}.\label{eq:psi-tau-epsilon}
\end{align}
$\beta_{1}$ can be controlled as 
\begin{align}
\frac{1}{p}\left\Vert \mathcal{P}_{\Omega}\left(\left\{ \varepsilon_{i,j}\ind_{\left|\varepsilon_{i,j}\right|\leq\tau}\right\} _{i,j}\right)\right\Vert  & \leq\frac{1}{p}\left\Vert \mathcal{P}_{\Omega}\left(\left\{ \varepsilon_{i,j}\ind_{\left|\varepsilon_{i,j}\right|\leq\tau}-\mathbb{E}\left[\varepsilon_{i,j}\ind_{\left|\varepsilon_{i,j}\right|\leq\tau}\right]\right\} _{i,j}\right)\right\Vert \nonumber \\
 & \qquad+\frac{1}{p}\left\Vert \mathcal{P}_{\Omega}\left(\left\{ \mathbb{E}\left[\varepsilon_{i,j}\ind_{\left|\varepsilon_{i,j}\right|\leq\tau}\right]\right\} _{i,j}\right)\right\Vert \\
 & \overset{\mathrm{(i)}}{\lesssim}\sigma\sqrt{\frac{n}{p}},\label{eq:epsilon-tau}
\end{align}
where (i) comes from the bounds of $\beta_{1}$ and $\beta_{4}$ in
\eqref{eq:strongcvx-beta1} and \eqref{eq:strongcvx-beta4}.

For the next, one has 
\begin{align}
\frac{1}{2p}\left\Vert \mathcal{P}_{\Omega}\left(\left\{ \tau\mathrm{sgn}\left(\varepsilon_{i,j}\right)\ind_{\left|\varepsilon_{i,j}\right|>\tau}\right\} _{i,j}\right)\right\Vert  & \overset{\text{(i)}}{\leq}\frac{\tau}{2p}\left\Vert \mathcal{P}_{\Omega}\left(\left\{ \ind_{\left|\varepsilon_{i,j}\right|>\tau}\right\} _{i,j}\right)\right\Vert \nonumber \\
 & \overset{\text{(ii)}}{\lesssim}\tau n\mathbb{P}\left(\left|\varepsilon_{i,j}\right|>\tau\right)\nonumber \\
 & \overset{\text{(iii)}}{\lesssim}\sigma\sqrt{\frac{n}{p}},\label{eq:epsilon-ind}
\end{align}
where (i) is due to \eqref{eq:ind}; (ii) comes from Lemma \ref{lem:bernoulli};
(iii) comes from \eqref{eq:markov}. Therefore, combining \eqref{eq:psi-tau-epsilon},
\eqref{eq:epsilon-tau} and \eqref{eq:epsilon-ind} yields 
\[
\left\Vert \nabla f\left(\bm{F}^{\star}\right)\right\Vert _{\mathrm{F}}\lesssim\frac{1}{2p}\left\Vert \mathcal{P}_{\Omega}\left(\left\{ \psi_{\tau}\left(-\varepsilon_{i,j}\right)\right\} _{i,j}\right)\right\Vert \left\Vert \bm{F}^{\star}\right\Vert _{\mathrm{F}}\lesssim\sigma\sqrt{\frac{n}{p}}\left\Vert \bm{F}^{\star}\right\Vert _{\mathrm{F}}.
\]
Plugging this result into \eqref{eq:fro} reveals that 
\begin{align*}
\left\Vert \bm{F}^{t+1}\bm{H}^{t+1}-\bm{F}^{\star}\right\Vert _{\mathrm{F}} & \leq\left(1-\frac{\sigma_{\min}}{20}\eta\right)\left\Vert \bm{F}^{t}\bm{H}^{t}-\bm{F}^{\star}\right\Vert _{\mathrm{F}}+\widetilde{C}\eta\sigma\sqrt{\frac{n}{p}}\left\Vert \bm{F}^{\star}\right\Vert _{\mathrm{F}}\\
 & \overset{\text{(i)}}{\lesssim}\left(1-\frac{\sigma_{\min}}{20}\eta\right)^{t+1}C\sqrt{\kappa}\left(\frac{\sigma}{\sigma_{\min}}\sqrt{\frac{n}{p}}+\frac{\left\Vert \bm{M}^{\star}\right\Vert _{\infty}}{\sigma_{\min}}\sqrt{\frac{n}{p}}\right)\left\Vert \bm{F}^{\star}\right\Vert _{\mathrm{F}}+C_{1}\frac{\sigma}{\sigma_{\min}}\sqrt{\frac{n}{p}}\left\Vert \bm{F}^{\star}\right\Vert _{\mathrm{F}}\\
 & \lesssim\left(\frac{\sigma}{\sigma_{\min}}\sqrt{\frac{n}{p}}+\frac{\left\Vert \bm{M}^{\star}\right\Vert _{\infty}}{\sigma_{\min}}\sqrt{\frac{n}{p}}\right)\sqrt{\kappa}\left\Vert \bm{F}^{\star}\right\Vert _{\mathrm{F}},
\end{align*}
where (i) arises from the induction hypothesis. 

\subsection{Proof of Lemma \ref{lem:loo}}

For $1\leq l\leq n$, the update rule \eqref{subeq:gradient_update_ncvx-1}
implies the decomposition

\begin{align}
 & \left(\bm{F}^{t+1,\left(l\right)}\bm{H}^{t+1,\left(l\right)}-\bm{F}^{\star}\right)_{l,\cdot}\nonumber \\
 & =\left\{ \bm{X}^{t,\left(l\right)}-\eta\nabla f_{\bm{X}}\left(\bm{X}^{t,\left(l\right)}\right)\right\} _{l,\cdot}\bm{H}^{t+1,\left(l\right)}-\bm{X}_{l,\cdot}^{\star}\nonumber \\
 & =\underbrace{\left\{ \bm{X}^{t,\left(l\right)}-\eta\nabla f_{\bm{X}}\left(\bm{X}^{t,\left(l\right)}\right)\right\} _{l,\cdot}\bm{H}^{t,\left(l\right)}-\bm{X}_{l,\cdot}^{\star}}_{\eqqcolon\bm{h}_{1}}\nonumber \\
 & \qquad+\underbrace{\left\{ \bm{X}^{t,\left(l\right)}-\eta\nabla f_{\bm{X}}\left(\bm{X}^{t,\left(l\right)}\right)\right\} _{l,\cdot}\left[\left(\bm{H}^{t,\left(l\right)}\right)^{-1}\bm{H}^{t+1,\left(l\right)}-\bm{I}_{r}\right]}_{\eqqcolon\bm{h}_{2}},\label{eq:decompose-loo}
\end{align}
where the gradient is

\begin{align}
\left\{ \nabla f_{\bm{X}}\left(\bm{X}^{t,\left(l\right)}\right)\right\} _{l,\cdot}&=\left\{ \sum_{j=1}^{n}\psi_{\tau}\left(\left(\bm{X}^{t,\left(l\right)}\bm{Y}^{t,\left(l\right)\top}\right)_{l,j}-M_{l,j}^{\star}\right)\right\} _{j}\bm{Y}^{t,\left(l\right)}\nonumber \\
&\qquad+\frac{1}{2}\bm{X}_{l,\cdot}^{t,\left(l\right)}\left(\bm{X}^{t,\left(l\right)\top}\bm{X}^{t,\left(l\right)}-\bm{Y}^{t,\left(l\right)\top}\bm{Y}^{t,\left(l\right)}\right).
\end{align}
Then we proceed by controlling $\bm{h}_{1}$ and $\bm{h}_{2}$ separately.
For notational simplicity we denote 
\[
\bm{\Delta}^{t,\left(l\right)}\coloneqq\left[\begin{array}{c}
\bm{\Delta}_{\bm{X}}^{t,\left(l\right)}\\
\bm{\Delta}_{\bm{Y}}^{t,\left(l\right)}
\end{array}\right]=\left[\begin{array}{c}
\bm{X}^{t,\left(l\right)}\bm{H}^{t,\left(l\right)}-\bm{X}^{\star}\\
\bm{Y}^{t,\left(l\right)}\bm{H}^{t,\left(l\right)}-\bm{Y}^{\star}
\end{array}\right].
\]

\begin{enumerate}
\item Regarding the first term $\bm{h}_{1}$, one has 
\begin{align}
\bm{h}_{1} & =\bm{X}_{l,\cdot}^{t,\left(l\right)}\bm{H}^{t,\left(l\right)}-\bm{X}_{l,\cdot}^{\star}-\eta\left\{ \psi_{\tau}\left(\left(\bm{X}^{t,\left(l\right)}\bm{Y}^{t,\left(l\right)\top}\right)_{l,j}-M_{l,j}^{\star}\right)\right\} _{j}\bm{Y}^{t,\left(l\right)}\bm{H}^{t,\left(l\right)}\nonumber \\
 & \qquad-\eta\frac{1}{2}\bm{X}_{l,\cdot}^{t,\left(l\right)}\left(\bm{X}^{t,\left(l\right)\top}\bm{X}^{t,\left(l\right)}-\bm{Y}^{t,\left(l\right)\top}\bm{Y}^{t,\left(l\right)}\right)\bm{H}^{t,\left(l\right)}\nonumber \\
 & \overset{\text{(i)}}{=}\bm{X}_{l,\cdot}^{t,\left(l\right)}\bm{H}^{t,\left(l\right)}-\bm{X}_{l,\cdot}^{\star}-\eta\left\{ \left(\left(\bm{X}^{t,\left(l\right)}\bm{Y}^{t,\left(l\right)\top}\right)_{l,j}-M_{l,j}^{\star}\right)\ind_{\left|\left(\bm{X}^{t,\left(l\right)}\bm{Y}^{t,\left(l\right)\top}\right)_{l,j}-M_{l,j}^{\star}\right|\leq\tau}\right\} _{j}\bm{Y}^{t,\left(l\right)}\bm{H}^{t,\left(l\right)}\nonumber \\
 & \qquad-\eta\left\{ \tau\mathrm{sgn}\left(\left(\bm{X}^{t,\left(l\right)}\bm{Y}^{t,\left(l\right)\top}\right)_{l,j}-M_{l,j}^{\star}\right)\ind_{\left|\left(\bm{X}^{t,\left(l\right)}\bm{Y}^{t,\left(l\right)\top}\right)_{l,j}-M_{l,j}^{\star}\right|>\tau}\right\} _{j}\bm{Y}^{t,\left(l\right)}\bm{H}^{t,\left(l\right)}\nonumber \\
 & \qquad-\eta\frac{1}{2}\bm{X}_{l,\cdot}^{t,\left(l\right)}\left(\bm{X}^{t,\left(l\right)\top}\bm{X}^{t,\left(l\right)}-\bm{Y}^{t,\left(l\right)\top}\bm{Y}^{t,\left(l\right)}\right)\bm{H}^{t,\left(l\right)}\nonumber \\
 & \overset{\text{(ii)}}{=}\bm{X}_{l,\cdot}^{t,\left(l\right)}\bm{H}^{t,\left(l\right)}-\bm{X}_{l,\cdot}^{\star}-\eta\left(\bm{X}^{t,\left(l\right)}\bm{Y}^{t,\left(l\right)\top}-\bm{M}^{\star}\right)_{l,\cdot}\bm{Y}^{t,\left(l\right)}\bm{H}^{t,\left(l\right)}\nonumber \\
 & \qquad-\eta\frac{1}{2}\bm{X}_{l,\cdot}^{t,\left(l\right)}\left(\bm{X}^{t,\left(l\right)\top}\bm{X}^{t,\left(l\right)}-\bm{Y}^{t,\left(l\right)\top}\bm{Y}^{t,\left(l\right)}\right)\bm{H}^{t,\left(l\right)}\nonumber \\
 & =\left(\bm{X}_{l,\cdot}^{t,\left(l\right)}\bm{H}^{t,\left(l\right)}-\bm{X}_{l,\cdot}^{\star}\right)\left(\bm{I}_{r}-\eta\left(\bm{Y}^{t,\left(l\right)}\bm{H}^{t,\left(l\right)}\right)^{\top}\bm{Y}^{t,\left(l\right)}\bm{H}^{t,\left(l\right)}\right)\nonumber \\
 & \qquad-\eta\bm{X}_{l,\cdot}^{\star}\left(\bm{Y}^{t,\left(l\right)}\bm{H}^{t,\left(l\right)}-\bm{Y}^{\star}\right)^{\top}\bm{Y}^{t,\left(l\right)}\bm{H}^{t,\left(l\right)}\nonumber \\
 &\qquad-\eta\frac{1}{2}\bm{X}_{l,\cdot}^{t,\left(l\right)}\left(\bm{X}^{t,\left(l\right)\top}\bm{X}^{t,\left(l\right)}-\bm{Y}^{t,\left(l\right)\top}\bm{Y}^{t,\left(l\right)}\right)\bm{H}^{t,\left(l\right)}.\label{eq:h1-decompose}
\end{align}
Here (i) makes use of the definition of $\psi_{\tau}(\cdot)$ (cf.~\eqref{defn-psi});
(ii) follows from the fact that 
\[
\ind_{\left|\left(\bm{X}^{t,\left(l\right)}\bm{Y}^{t,\left(l\right)\top}\right)_{l,j}-M_{l,j}^{\star}\right|>\tau}=0,
\]
which can be verified by 
\begin{align*}
 & \max_{j}\left|\left(\bm{X}^{t,\left(l\right)}\bm{Y}^{t,\left(l\right)\top}\right)_{l,j}-M_{l,j}^{\star}\right|\\
 & =\max_{i,j}\left|\left(\bm{X}^{t,\left(l\right)}\bm{H}^{t,\left(l\right)}\right)_{l,\cdot}\left(\bm{Y}^{t,\left(l\right)}\bm{H}^{t,\left(l\right)}\right)_{j,\cdot}^{\top}-\bm{X}_{l,\cdot}^{\star}\left(\bm{Y}_{j,\cdot}^{\star}\right)^{\top}\right|\\
 & \leq\max_{i,j}\left|\left(\bm{X}^{t,\left(l\right)}\bm{H}^{t,\left(l\right)}\right)_{l,\cdot}\left(\bm{Y}^{t,\left(l\right)}\bm{H}^{t,\left(l\right)}\right)_{j,\cdot}^{\top}-\bm{X}_{l,\cdot}^{\star}\left(\bm{Y}^{t,\left(l\right)}\bm{H}^{t,\left(l\right)}\right)_{j,\cdot}^{\top}\right|\\
 & \qquad+\max_{i,j}\left|\bm{X}_{l,\cdot}^{\star}\left(\bm{Y}^{t,\left(l\right)}\bm{H}^{t,\left(l\right)}\right)_{j,\cdot}^{\top}-\bm{X}_{l,\cdot}^{\star}\left(\bm{Y}_{j,\cdot}^{\star}\right)^{\top}\right|\\
 & \leq2\left\Vert \left(\bm{F}^{t,\left(l\right)}\bm{H}^{t,\left(l\right)}-\bm{F}^{\star}\right)_{l,\cdot}\right\Vert _{2}\left\Vert \bm{F}^{t,\left(l\right)}\right\Vert _{2,\infty}+2\left\Vert \bm{F}^{\star}\right\Vert _{2,\infty}\left\Vert \bm{F}^{t,\left(l\right)}\bm{H}^{t,\left(l\right)}-\bm{F}^{\star}\right\Vert _{2,\infty}\\
 & \overset{\text{(i)}}{\lesssim}\kappa\sqrt{r}\left(\frac{\sigma}{\sigma_{\min}}\sqrt{\frac{n\log n}{p}}+\frac{\left\Vert \bm{M}^{\star}\right\Vert _{\infty}}{\sigma_{\min}}\sqrt{\frac{n}{p}}\right)\left\Vert \bm{F}^{\star}\right\Vert _{2,\infty}^{2}\\
 & \ll\tau,
\end{align*}
where (i) relies on Lemma \ref{lem:loo} and the result holds provided
$np\gg\mu^{2}\kappa^{4}r^{3}\log n$. Hence, the bound of $\bm{h}_{1}$
follows from \eqref{eq:h1-decompose} that 
\begin{align}
\left\Vert \bm{h}_{1}\right\Vert _{2} & \leq\left\Vert \bm{I}_{r}-2\eta\left(\bm{Y}^{t,\left(l\right)}\bm{H}^{t,\left(l\right)}\right)^{\top}\bm{Y}^{t,\left(l\right)}\bm{H}^{t,\left(l\right)}\right\Vert \left\Vert \left(\bm{\Delta}_{\bm{X}}^{t,\left(l\right)}\right)_{l,\cdot}\right\Vert _{2}+2\eta\left\Vert \bm{X}_{l,\cdot}^{\star}\right\Vert _{2}\left\Vert \bm{\Delta}_{\bm{Y}}^{t,\left(l\right)}\right\Vert \left\Vert \bm{Y}^{t,\left(l\right)}\bm{H}^{t,\left(l\right)}\right\Vert \nonumber \\
 & \qquad+\frac{\eta}{2}\left\Vert \bm{X}_{l,\cdot}^{t,\left(l\right)}\left(\bm{X}^{t,\left(l\right)\top}\bm{X}^{t,\left(l\right)}-\bm{Y}^{t,\left(l\right)\top}\bm{Y}^{t,\left(l\right)}\right)\bm{H}^{t,\left(l\right)}\right\Vert \nonumber \\
 & \leq\left(1-\eta\sigma_{\min}\right)\left\Vert \left(\bm{\Delta}_{\bm{X}}^{t,\left(l\right)}\right)_{l,\cdot}\right\Vert _{2}+8\eta\left\Vert \bm{\Delta}^{t,\left(l\right)}\right\Vert \left\Vert \bm{F}^{\star}\right\Vert \left\Vert \bm{X}^{\star}\right\Vert _{2,\infty},\label{eq:h1}
\end{align}
where the last line utilizes \eqref{subeq:aux-7} and the fact that
\begin{align*}
 & \left\Vert \bm{X}_{l,\cdot}^{t,\left(l\right)}\left(\bm{X}^{t,\left(l\right)\top}\bm{X}^{t,\left(l\right)}-\bm{Y}^{t,\left(l\right)\top}\bm{Y}^{t,\left(l\right)}\right)\bm{H}^{t,\left(l\right)}\right\Vert _{2}\\
 & \overset{\text{(i)}}{\leq}2\left\Vert \bm{X}^{\star}\right\Vert _{2,\infty}\left\Vert \bm{X}^{t,\left(l\right)\top}\bm{X}^{t,\left(l\right)}-\bm{Y}^{t,\left(l\right)\top}\bm{Y}^{t,\left(l\right)}\right\Vert \\
 & \overset{\text{(ii)}}{\leq}\left\Vert \bm{X}^{\star}\right\Vert _{2,\infty}\left(\left\Vert \bm{X}^{t,\left(l\right)\top}\bm{X}^{t,\left(l\right)}-\bm{X}^{\star\top}\bm{X}^{\star}\right\Vert +\left\Vert \bm{Y}^{\star\top}\bm{Y}^{\star}-\bm{Y}^{t,\left(l\right)\top}\bm{Y}^{t,\left(l\right)}\right\Vert \right)\\
 & \leq\left\Vert \bm{X}^{\star}\right\Vert _{2,\infty}\left(\left\Vert \bm{X}^{t,\left(l\right)}\bm{H}^{t,\left(l\right)}-\bm{X}^{\star}\right\Vert \left\Vert \bm{X}^{t,\left(l\right)}\right\Vert +\left\Vert \bm{X}^{t,\left(l\right)}\bm{H}^{t,\left(l\right)}-\bm{X}^{\star}\right\Vert \left\Vert \bm{X}^{\star}\right\Vert \right)\\
 & \qquad+\left\Vert \bm{X}^{\star}\right\Vert _{2,\infty}\left(\left\Vert \bm{Y}^{t,\left(l\right)}\bm{H}^{t,\left(l\right)}-\bm{Y}^{\star}\right\Vert \left\Vert \bm{Y}^{t,\left(l\right)}\right\Vert +\left\Vert \bm{Y}^{t,\left(l\right)}\bm{H}^{t,\left(l\right)}-\bm{Y}^{\star}\right\Vert \left\Vert \bm{Y}^{\star}\right\Vert \right)\\
 & \overset{\text{(iii)}}{\leq}\left\Vert \bm{X}^{\star}\right\Vert _{2,\infty}\left(3\left\Vert \bm{X}^{t,\left(l\right)}\bm{H}^{t,\left(l\right)}-\bm{X}^{\star}\right\Vert \left\Vert \bm{X}^{\star}\right\Vert +3\left\Vert \bm{Y}^{t,\left(l\right)}\bm{H}^{t,\left(l\right)}-\bm{Y}^{\star}\right\Vert \left\Vert \bm{Y}^{\star}\right\Vert \right)\\
 & \leq6\left\Vert \bm{\Delta}^{t,\left(l\right)}\right\Vert \left\Vert \bm{F}^{\star}\right\Vert \left\Vert \bm{X}^{\star}\right\Vert _{2,\infty},
\end{align*}
where (i) and (iii) are due to \eqref{subeq:aux-6}; (ii) relies on
the fact that $\bm{X}^{\star\top}\bm{X}^{\star}=\bm{Y}^{\star\top}\bm{Y}^{\star}$.
Consequently, it is easy to obtain that 
\begin{align}
\left\Vert \bm{h}_{1}\right\Vert _{2} & \leq\left(1-\eta\sigma_{\min}\right)C\kappa\sqrt{r}\left\Vert \bm{F}^{\star}\right\Vert _{2,\infty}\left(\frac{\sigma}{\sigma_{\min}}\sqrt{\frac{n\log n}{p}}+\frac{\left\Vert \bm{M}^{\star}\right\Vert _{\infty}}{\sigma_{\min}}\sqrt{\frac{n}{p}}\right)\nonumber \\
 & \qquad+8\eta C\left(\frac{\sigma}{\sigma_{\min}}\sqrt{\frac{n}{p}}+\frac{\left\Vert \bm{M}^{\star}\right\Vert _{\infty}}{\sigma_{\min}}\sqrt{\frac{n}{p}}\right)\sqrt{r}\left\Vert \bm{F}^{\star}\right\Vert \left\Vert \bm{X}^{\star}\right\Vert \left\Vert \bm{X}\right\Vert _{2,\infty}\nonumber \\
 & \leq\left\Vert \bm{F}^{\star}\right\Vert _{2,\infty},\label{eq:h1-coarse}
\end{align}
where the last link utilizes the induction hypothesis \eqref{subeq:add-1}
and \eqref{subeq:aux-3}. 
\item For the next, making use of \eqref{eq:h1-coarse} reveals that 
\begin{align}
\left\Vert \bm{h}_{2}\right\Vert _{2} & \leq\left\Vert \left(\bm{H}^{t,\left(l\right)}\right)^{-1}\bm{H}^{t+1,\left(l\right)}-\bm{I}_{r}\right\Vert \left\Vert \bm{h}_{1}+\bm{X}_{l,\cdot}^{\star}\right\Vert _{2}\leq2\left\Vert \left(\bm{H}^{t,\left(l\right)}\right)^{-1}\bm{H}^{t+1,\left(l\right)}-\bm{I}_{r}\right\Vert \left\Vert \bm{F}^{\star}\right\Vert _{2,\infty}.\label{eq:h2}
\end{align}
To bound the right-hand side, we invoke the following claim whose
proof is deferred to Section \ref{subsec:Proof-of-Claim-H}. 
\begin{claim}
\label{claim:H}With probability exceeding $1-O(n^{-100})$, one has
\begin{align*}
\left\Vert \left(\bm{H}^{t,\left(l\right)}\right)^{-1}\bm{H}^{t+1,\left(l\right)}-\bm{I}_{r}\right\Vert  & \lesssim\eta\sqrt{r}\left(\frac{\sigma}{\sigma_{\min}}\sqrt{\frac{n}{p}}+\frac{\left\Vert \bm{M}^{\star}\right\Vert _{\infty}}{\sigma_{\min}}\sqrt{\frac{n}{p}}\right)^{2}\left\Vert \bm{F}^{\star}\right\Vert ^{2}\log n\\
 & \qquad+\frac{\eta}{2}\kappa^{1.5}\left(\frac{\sigma}{\sigma_{\min}}\sqrt{\frac{n}{p}}+\frac{\left\Vert \bm{M}^{\star}\right\Vert _{\infty}}{\sigma_{\min}}\sqrt{\frac{n}{p}}\right)\left\Vert \bm{F}^{\star}\right\Vert _{2,\infty}\left\Vert \bm{F}^{\star}\right\Vert \log n,
\end{align*}
as long as $0\leq\eta\leq c/(\mu\kappa^{5}r^{3}\sigma_{\max}\log n)$
for some sufficiently small constant $c>0$. 
\end{claim}

Finally plugging \eqref{eq:h1} and \eqref{eq:h2} into \eqref{eq:decompose-loo}
gives 
\begin{align*}
\left\Vert \left(\bm{\Delta}_{\bm{X}}^{t+1,\left(l\right)}\right)_{l,\cdot}\right\Vert _{2} & \leq\left(1-\eta\sigma_{\min}\right)\left\Vert \left(\bm{\Delta}_{\bm{X}}^{t,\left(l\right)}\right)_{l,\cdot}\right\Vert _{2}+8\eta\left\Vert \bm{\Delta}^{t,\left(l\right)}\right\Vert \left\Vert \bm{F}^{\star}\right\Vert \left\Vert \bm{X}^{\star}\right\Vert _{2,\infty}\\
 & \qquad+\widetilde{C}\eta\sqrt{r}\left(\frac{\sigma}{\sigma_{\min}}\sqrt{\frac{n}{p}}+\frac{\left\Vert \bm{M}^{\star}\right\Vert _{\infty}}{\sigma_{\min}}\sqrt{\frac{n}{p}}\right)^{2}\left\Vert \bm{F}^{\star}\right\Vert ^{2}\left\Vert \bm{F}^{\star}\right\Vert _{2,\infty}\log n\\
 & \qquad+\widetilde{C}\frac{\eta}{2}\kappa^{1.5}\left(\frac{\sigma}{\sigma_{\min}}\sqrt{\frac{n}{p}}+\frac{\left\Vert \bm{M}^{\star}\right\Vert _{\infty}}{\sigma_{\min}}\sqrt{\frac{n}{p}}\right)\left\Vert \bm{F}^{\star}\right\Vert _{2,\infty}^{2}\left\Vert \bm{F}^{\star}\right\Vert \log n\\
 & \lesssim\kappa\sqrt{r}\left(\frac{\sigma}{\sigma_{\min}}\sqrt{\frac{n}{p}}+\frac{\left\Vert \bm{M}^{\star}\right\Vert _{\infty}}{\sigma_{\min}}\sqrt{\frac{n}{p}}\right)\left\Vert \bm{F}^{\star}\right\Vert _{2,\infty},
\end{align*}
where the last line utilizes the hypothesis \eqref{subeq:add-1} and
\eqref{subeq:aux-3}. 
\end{enumerate}

\subsubsection{Proof of Claim \ref{claim:H}\label{subsec:Proof-of-Claim-H}}

To start with, we introduce an auxiliary sequence 
\[
\widetilde{\bm{F}}^{t+1,\left(l\right)}=\left[\begin{array}{c}
\widetilde{\bm{X}}^{t+1,\left(l\right)}\\
\widetilde{\bm{Y}}^{t+1,\left(l\right)}
\end{array}\right],
\]
with

\begin{align*}
\widetilde{\bm{X}}^{t+1,\left(l\right)} & =\bm{X}^{t,\left(l\right)}\bm{H}^{t,\left(l\right)}-\eta\frac{1}{p}\mathcal{P}_{\Omega_{-l},\cdot}\left(\psi_{\tau}\left(\left(\bm{X}^{t,\left(l\right)}\bm{Y}^{t,\left(l\right)\top}\right)_{i,j}-M_{i,j}\right)\right)\bm{Y}^{\star}\\
 & \qquad-\eta\mathcal{P}_{l,\cdot}\left(\psi_{\tau}\left(\left(\bm{X}^{t,\left(l\right)}\bm{Y}^{t,\left(l\right)\top}\right)_{i,j}-M_{i,j}^{\star}\right)\right)\bm{Y}^{\star}\\
 & \qquad-\frac{\eta}{2}\bm{X}^{\star}\bm{H}^{t,\left(l\right)\top}\left(\bm{X}^{t,\left(l\right)\top}\bm{X}^{t,\left(l\right)}-\bm{Y}^{t,\left(l\right)\top}\bm{Y}^{t,\left(l\right)}\right)\bm{H}^{t,\left(l\right)},\\
\widetilde{\bm{Y}}^{t+1,\left(l\right)} & =\bm{Y}^{t,\left(l\right)}\bm{H}^{t,\left(l\right)}-\eta\left[\frac{1}{p}\mathcal{P}_{\Omega_{-l},\cdot}\left(\psi_{\tau}\left(\left(\bm{X}^{t,\left(l\right)}\bm{Y}^{t,\left(l\right)\top}\right)_{i,j}-M_{i,j}\right)\right)\right]^{\top}\bm{X}^{\star}\\
 & \qquad-\eta\left[\mathcal{P}_{l,\cdot}\left(\psi_{\tau}\left(\left(\bm{X}^{t,\left(l\right)}\bm{Y}^{t,\left(l\right)\top}\right)_{i,j}-M_{i,j}^{\star}\right)\right)\right]^{\top}\bm{X}^{\star}\\
 & \qquad-\frac{\eta}{2}\bm{Y}^{\star}\bm{H}^{t,\left(l\right)\top}\left(\bm{Y}^{t,\left(l\right)\top}\bm{Y}^{t,\left(l\right)}-\bm{X}^{t,\left(l\right)\top}\bm{X}^{t,\left(l\right)}\right)\bm{H}^{t,\left(l\right)}.
\end{align*}

Then we turn attention to $\Vert(\bm{H}^{t,(l)})^{-1}\bm{H}^{t+1,(l)}-\bm{I}_{r}\Vert$.
We begin with a claim \citep[Claim 4]{chen2020noisy} showing that
$\bm{I}_{r}$ aligns $\widetilde{\bm{F}}^{t+1,(l)}$ with $\bm{F}^{\star}$.
\begin{claim}
One has 
\[
\bm{I}_{r}=\arg\min_{\bm{R}\in\mathcal{O}^{r\times r}}\left\Vert \widetilde{\bm{F}}^{t+1,\left(l\right)}\bm{R}-\bm{F}^{\star}\right\Vert _{\mathrm{F}}\qquad\text{and}\qquad\sigma_{\min}\left(\widetilde{\bm{F}}^{t+1,\left(l\right)\top}\bm{F}^{\star}\right)\geq\sigma_{\min}/2.
\]
\end{claim}

Invoking \citet[Lemma 36]{ma2017implicit} yields 
\begin{align}
\left\Vert \left(\bm{H}^{t,\left(l\right)}\right)^{-1}\bm{H}^{t+1,\left(l\right)}-\bm{I}_{r}\right\Vert  & =\left\Vert \mathsf{sgn}\left(\left(\bm{F}^{t+1,\left(l\right)}\bm{H}^{t,\left(l\right)}\right)^{\top}\bm{F}^{\star}\right)-\mathsf{sgn}\left(\widetilde{\bm{F}}^{t+1,\left(l\right)\top}\bm{F}^{\star}\right)\right\Vert \nonumber \\
 & \leq\frac{1}{\sigma_{\min}\left(\widetilde{\bm{F}}^{t+1,\left(l\right)\top}\bm{F}^{\star}\right)}\left\Vert \left(\bm{F}^{t+1,\left(l\right)}\bm{H}^{t,\left(l\right)}-\widetilde{\bm{F}}^{t+1,\left(l\right)}\right)^{\top}\bm{F}^{\star}\right\Vert \nonumber \\
 & \leq\frac{2}{\sigma_{\min}}\left\Vert \bm{F}^{t+1,\left(l\right)}\bm{H}^{t,\left(l\right)}-\widetilde{\bm{F}}^{t+1,\left(l\right)}\right\Vert \left\Vert \bm{F}^{\star}\right\Vert .\label{eq:Htl}
\end{align}
To control $\bm{F}^{t+1,(l)}\bm{H}^{t,(l)}-\widetilde{\bm{F}}^{t+1,(l)}$,
we can decompose it into the sum of three terms as follows, 
\begin{align}
\bm{F}^{t+1,\left(l\right)}\bm{H}^{t,\left(l\right)}-\widetilde{\bm{F}}^{t+1,\left(l\right)} & =\eta\left[\begin{array}{cc}
\bm{B} & \bm{0}\\
\bm{0} & \bm{B}^{\top}
\end{array}\right]\left[\begin{array}{c}
\bm{\Delta}_{\bm{Y}}^{t,\left(l\right)}\\
\bm{\Delta}_{\bm{X}}^{t,\left(l\right)}
\end{array}\right]+\frac{\eta}{2}\left[\begin{array}{c}
\bm{X}^{\star}\\
-\bm{Y}^{\star}
\end{array}\right]\bm{H}^{t,\left(l\right)\top}\bm{C}\bm{H}^{t,\left(l\right)}\nonumber \\
 & \qquad+\frac{\eta}{2}\left[\begin{array}{c}
\bm{\Delta}_{\bm{X}}^{t,\left(l\right)}\\
-\bm{\Delta}_{\bm{Y}}^{t,\left(l\right)}
\end{array}\right]\bm{H}^{t,\left(l\right)\top}\bm{C}\bm{H}^{t,\left(l\right)},\label{eq:Ftl}
\end{align}
where 
\begin{align*}
\bm{B} & \coloneqq-\frac{1}{p}\mathcal{P}_{\Omega_{-l},\cdot}\left(\psi_{\tau}\left(\left(\bm{X}^{t,\left(l\right)}\bm{Y}^{t,\left(l\right)\top}\right)_{i,j}-M_{i,j}\right)\right)-\mathcal{P}_{l,\cdot}\left(\psi_{\tau}\left(\left(\bm{X}^{t,\left(l\right)}\bm{Y}^{t,\left(l\right)\top}\right)_{i,j}-M_{i,j}\right)\right)\\
\bm{C} & \coloneqq\bm{X}^{t,\left(l\right)\top}\bm{X}^{t,\left(l\right)}-\bm{Y}^{t,\left(l\right)\top}\bm{Y}^{t,\left(l\right)}.
\end{align*}
Regarding $\bm{B}$, one has 
\begin{align*}
\bm{B} & =-\frac{1}{p}\mathcal{P}_{\Omega}\left(\psi_{\tau}\left(\left(\bm{X}^{t,\left(l\right)}\bm{Y}^{t,\left(l\right)\top}\right)_{i,j}-M_{i,j}\right)\right)+\frac{1}{p}\mathcal{P}_{\Omega_{l}}\left(\psi_{\tau}\left(\left(\bm{X}^{t,\left(l\right)}\bm{Y}^{t,\left(l\right)\top}\right)_{i,j}-M_{i,j}\right)\right)\\
 & \qquad-\mathcal{P}_{l,\cdot}\left(\psi_{\tau}\left(\left(\bm{X}^{t,\left(l\right)}\bm{Y}^{t,\left(l\right)\top}\right)_{i,j}-M_{i,j}^{\star}\right)\right)\\
 & =\bm{\Theta}_{1}+\bm{\Theta}_{2}+\bm{\Theta}_{3}+\bm{\Theta}_{4}+\bm{\Theta}_{5},
\end{align*}
where 
\begin{align*}
\bm{\Theta}_{1}\coloneqq & -\frac{1}{p}\mathcal{P}_{\Omega}\left(\left(\left(\bm{X}^{t,\left(l\right)}\bm{Y}^{t,\left(l\right)\top}\right)_{i,j}-M_{i,j}^{\star}\right)\ind_{\left|\left(\bm{X}^{t,\left(l\right)}\bm{Y}^{t,\left(l\right)\top}\right)_{i,j}-M_{i,j}\right|\leq\tau}\right),\\
\bm{\Theta}_{2}\coloneqq & \frac{1}{p}\mathcal{P}_{\Omega_{l}}\left(\left(\left(\bm{X}^{t,\left(l\right)}\bm{Y}^{t,\left(l\right)\top}\right)_{i,j}-M_{i,j}^{\star}\right)\ind_{\left|\left(\bm{X}^{t,\left(l\right)}\bm{Y}^{t,\left(l\right)\top}\right)_{i,j}-M_{i,j}\right|\leq\tau}\right)\\
 & \quad-\mathcal{P}_{l,\cdot}\left(\left(\left(\bm{X}^{t,\left(l\right)}\bm{Y}^{t,\left(l\right)\top}\right)_{i,j}-M_{i,j}^{\star}\right)\ind_{\left|\left(\bm{X}^{t,\left(l\right)}\bm{Y}^{t,\left(l\right)\top}\right)_{i,j}-M_{i,j}^{\star}\right|\leq\tau}\right),\\
\bm{\Theta}_{3}\coloneqq & -\frac{1}{p}\mathcal{P}_{\Omega}\left(\tau\mathrm{sgn}\left(\left(\bm{X}^{t,\left(l\right)}\bm{Y}^{t,\left(l\right)\top}\right)_{i,j}-M_{i,j}\right)\ind_{\left|\left(\bm{X}^{t,\left(l\right)}\bm{Y}^{t,\left(l\right)\top}\right)_{i,j}-M_{i,j}\right|>\tau}\right),\\
\bm{\Theta}_{4}\coloneqq & \frac{1}{p}\mathcal{P}_{\Omega_{l}}\left(\tau\mathrm{sgn}\left(\left(\bm{X}^{t,\left(l\right)}\bm{Y}^{t,\left(l\right)\top}\right)_{i,j}-M_{i,j}\right)\ind_{\left|\left(\bm{X}^{t,\left(l\right)}\bm{Y}^{t,\left(l\right)\top}\right)_{i,j}-M_{i,j}\right|>\tau}\right)\\
 & \quad-\mathcal{P}_{l,\cdot}\left(\tau\mathrm{sgn}\left(\left(\bm{X}^{t,\left(l\right)}\bm{Y}^{t,\left(l\right)\top}\right)_{i,j}-M_{i,j}^{\star}\right)\ind_{\left|\left(\bm{X}^{t,\left(l\right)}\bm{Y}^{t,\left(l\right)\top}\right)_{i,j}-M_{i,j}^{\star}\right|>\tau}\right),\\
\bm{\Theta}_{5}\coloneqq & \frac{1}{p}\mathcal{P}_{\Omega_{-l}}\left(\varepsilon_{i,j}\ind_{\left|\left(\bm{X}^{t,\left(l\right)}\bm{Y}^{t,\left(l\right)\top}\right)_{i,j}-M_{i,j}\right|\leq\tau}\right).
\end{align*}
We shall bound these terms separately as follows. 
\begin{enumerate}
\item In terms of $\bm{\Theta}_{1}$, one has 
\begin{align*}
 & \left\Vert \frac{1}{p}\mathcal{P}_{\Omega}\left(\left(\left(\bm{X}^{t,\left(l\right)}\bm{Y}^{t,\left(l\right)\top}\right)_{i,j}-M_{i,j}^{\star}\right)\ind_{\left|\left(\bm{X}^{t,\left(l\right)}\bm{Y}^{t,\left(l\right)\top}\right)_{i,j}-M_{i,j}\right|\leq\tau}\right)\right\Vert \\
 & \overset{\text{(i)}}{\leq}\left\Vert \frac{1}{p}\mathcal{P}_{\Omega}\left(\bm{1}\bm{1}^{\top}\right)\right\Vert \cdot\max_{i,j}\left|\left(\left(\bm{X}^{t,\left(l\right)}\bm{Y}^{t,\left(l\right)\top}\right)_{i,j}-M_{i,j}^{\star}\right)\ind_{\left|\left(\bm{X}^{t,\left(l\right)}\bm{Y}^{t,\left(l\right)\top}\right)_{i,j}-M_{i,j}\right|\leq\tau}\right|\\
 & \overset{\text{(ii)}}{\lesssim}n\left\Vert \bm{X}^{t,\left(l\right)}\bm{Y}^{t,\left(l\right)\top}-\bm{M}^{\star}\right\Vert _{\infty}\\
 & \lesssim n\left\Vert \bm{F}^{t,\left(l\right)}\bm{R}^{t,\left(l\right)}-\bm{F}^{\star}\right\Vert _{2,\infty}\left(\left\Vert \bm{F}^{t,\left(l\right)}\right\Vert _{2,\infty}+\left\Vert \bm{F}^{\star}\right\Vert _{2,\infty}\right)\\
 & \overset{\text{(iii)}}{\lesssim}n\left\Vert \bm{F}^{t,\left(l\right)}\bm{R}^{t,\left(l\right)}-\bm{F}^{\star}\right\Vert _{2,\infty}\left\Vert \bm{F}^{\star}\right\Vert _{2,\infty}
\end{align*}
where (i) is due to \eqref{eq:operator-ob}; (ii) comes from Lemma
\ref{lem:bernoulli}; (iii) utilizes \eqref{subeq:aux-6}. 
\item Next, we turn attention to $\bm{\Theta}_{3}$, which can be bounded
as 
\begin{align*}
 & \left\Vert \frac{1}{p}\mathcal{P}_{\Omega}\left(\tau\mathrm{sgn}\left(\left(\bm{X}^{t,\left(l\right)}\bm{Y}^{t,\left(l\right)\top}\right)_{i,j}-M_{i,j}\right)\ind_{\left|\left(\bm{X}^{t,\left(l\right)}\bm{Y}^{t,\left(l\right)\top}\right)_{i,j}-M_{i,j}\right|>\tau}\right)\right\Vert \\
 & \overset{\text{(i)}}{\leq}\tau\left\Vert \frac{1}{p}\mathcal{P}_{\Omega}\left(\ind_{\left|\left(\bm{X}^{t,\left(l\right)}\bm{Y}^{t,\left(l\right)\top}\right)_{i,j}-M_{i,j}\right|>\tau}\right)\right\Vert \\
 & \overset{\text{(ii)}}{\leq}\tau\left\Vert \frac{1}{p}\mathcal{P}_{\Omega}\left(\ind_{\left|\varepsilon_{i,j}\right|>\tau/2}\right)\right\Vert \\
 & \overset{\text{(iii)}}{\lesssim}\tau n\frac{\sigma^{2}}{\left(\tau/2\right)^{2}}\lesssim\sigma\sqrt{\frac{n}{p}}.
\end{align*}
Here (i) arises from \eqref{eq:operator-ob}; (ii) comes from the
observation that 
\begin{align}
\ind_{\left|\left(\bm{X}^{t,\left(l\right)}\bm{Y}^{t,\left(l\right)\top}\right)_{i,j}-M_{i,j}\right|>\tau} & =\ind_{\left|\left(\bm{X}^{t,\left(l\right)}\bm{Y}^{t,\left(l\right)\top}\right)_{i,j}-M_{i,j}^{\star}-\varepsilon_{i,j}\right|>\tau}\nonumber \\
 & \leq\ind_{\left|\varepsilon_{i,j}\right|>\tau-\left|\left(\bm{X}^{t,\left(l\right)}\bm{Y}^{t,\left(l\right)\top}\right)_{i,j}-M_{i,j}^{\star}\right|}\leq\ind_{\left|\varepsilon_{i,j}\right|>\tau/2},\label{eq:ind-1}
\end{align}
where the last inequality is due to \eqref{subeq:aux-1} resulting
in 
\begin{align}
\max_{i,j}\left|\left(\bm{X}^{t,\left(l\right)}\bm{Y}^{t,\left(l\right)\top}\right)_{i,j}-M_{i,j}^{\star}\right| & \leq\left\Vert \bm{F}^{t,\left(l\right)}\bm{R}^{t,\left(l\right)}-\bm{F}^{\star}\right\Vert _{2,\infty}\left\Vert \bm{F}^{\star}\right\Vert _{2,\infty}\nonumber \\
 & \ll\left\Vert \bm{F}^{\star}\right\Vert _{2,\infty}^{2}\leq\frac{\mu r\sigma_{\max}}{n}\leq\frac{\tau}{2};\label{eq:tau-tl}
\end{align}
(iii) follows from Lemma \ref{lem:bernoulli}. 
\item Next we turn attention to $\bm{\Theta}_{5}$ which clearly obeys 
\[
\left\Vert \frac{1}{p}\mathcal{P}_{\Omega_{-l}}\left(\varepsilon_{i,j}\ind_{\left|\left(\bm{X}^{t,\left(l\right)}\bm{Y}^{t,\left(l\right)\top}\right)_{i,j}-M_{i,j}\right|\leq\tau}\right)\right\Vert \leq\left\Vert \frac{1}{p}\mathcal{P}_{\Omega}\left(\varepsilon_{i,j}\ind_{\left|\left(\bm{X}^{t,\left(l\right)}\bm{Y}^{t,\left(l\right)\top}\right)_{i,j}-M_{i,j}\right|\leq\tau}\right)\right\Vert .
\]
Similarly as \eqref{eq:strongcvx-alpha2-bound}, one has 
\[
\left\Vert \frac{1}{p}\mathcal{P}_{\Omega}\left(\varepsilon_{i,j}\ind_{\left|\left(\bm{X}^{t,\left(l\right)}\bm{Y}^{t,\left(l\right)\top}\right)_{i,j}-M_{i,j}\right|\leq\tau}\right)\right\Vert \lesssim\sigma\sqrt{\frac{n}{p}}.
\]
\item Regarding $\bm{\Theta}_{2}$, it can be expressed as 
\[
\bm{\Theta}_{2}=\frac{1}{p}\sum_{j=1}^{p}\bm{v}_{j},
\]
where 
\begin{align*}
\bm{v}_{j} & \coloneqq\delta_{l,j}\left(\left(\bm{X}^{t,\left(l\right)}\bm{Y}^{t,\left(l\right)\top}\right)_{l,j}-M_{l,j}^{\star}\right)\ind_{\left|\left(\bm{X}^{t,\left(l\right)}\bm{Y}^{t,\left(l\right)\top}\right)_{l,j}-M_{l,j}\right|\leq\tau}\bm{e}_{j}\\
 & \qquad-p\left(\left(\bm{X}^{t,\left(l\right)}\bm{Y}^{t,\left(l\right)\top}\right)_{l,j}-M_{l,j}^{\star}\right)\ind_{\left|\left(\bm{X}^{t,\left(l\right)}\bm{Y}^{t,\left(l\right)\top}\right)_{l,j}-M_{l,j}^{\star}\right|\leq\tau}\bm{e}_{j}.
\end{align*}
It is easy to verify that 
\begin{align*}
L & \coloneqq\max_{1\leq j\leq n}\left\Vert \bm{v}_{j}\right\Vert _{2}\leq\left\Vert \bm{X}^{t,\left(l\right)}\bm{Y}^{t,\left(l\right)\top}-\bm{M}^{\star}\right\Vert _{\infty},\\
V & \coloneqq\left\Vert \sum_{j=1}^{n}\mathbb{E}\left[\bm{v}_{j}^{\top}\bm{v}_{j}\right]\right\Vert \leq np\left\Vert \bm{X}^{t,\left(l\right)}\bm{Y}^{t,\left(l\right)\top}-\bm{M}^{\star}\right\Vert _{\infty}^{2}.
\end{align*}
We are positioned to apply the matrix Bernstein inequality \citep[Theorem 6.1.1]{tropp2015introduction}.
One has 
\begin{align*}
\frac{1}{p}\left\Vert \sum_{j=1}^{n}\bm{v}_{j}\right\Vert _{2} & \lesssim\frac{1}{p}\left(\sqrt{V\log n}+L\log n\right)\\
 & \lesssim\frac{1}{p}\sqrt{np\left\Vert \bm{X}^{t,\left(l\right)}\bm{Y}^{t,\left(l\right)\top}-\bm{M}^{\star}\right\Vert _{\infty}^{2}\log n}+\frac{1}{p}\left\Vert \bm{X}^{t,\left(l\right)}\bm{Y}^{t,\left(l\right)\top}-\bm{M}^{\star}\right\Vert _{\infty}\log n\\
 & \lesssim\sqrt{\frac{n\log n}{p}}\left\Vert \bm{F}^{t,\left(l\right)}\bm{R}^{t,\left(l\right)}-\bm{F}^{\star}\right\Vert _{2,\infty}\left\Vert \bm{Y}^{\star}\right\Vert _{2,\infty}+\frac{1}{p}\left\Vert \bm{F}^{t,\left(l\right)}\bm{R}^{t,\left(l\right)}-\bm{F}^{\star}\right\Vert _{2,\infty}\left\Vert \bm{Y}^{\star}\right\Vert _{2,\infty}\log n\\
 & \lesssim\sqrt{\frac{\mu r\log n}{p}\sigma_{\max}}\left\Vert \bm{F}^{t,\left(l\right)}\bm{R}^{t,\left(l\right)\top}-\bm{F}^{\star}\right\Vert _{2,\infty},
\end{align*}
where the last inequality holds as long as $np\geq1$. 
\item Finally for $\bm{\Theta}_{4}$, since \eqref{eq:tau-tl} implies 
\[
\ind_{\left|\left(\bm{X}^{t,\left(l\right)}\bm{Y}^{t,\left(l\right)\top}\right)_{l,j}-M_{l,j}^{\star}\right|>\tau}=0,
\]
we arrive at 
\begin{align*}
\bm{\Theta}_{4} & =\frac{1}{p}\mathcal{P}_{\Omega_{l}}\left(\tau\mathrm{sgn}\left(\left(\bm{X}^{t,\left(l\right)}\bm{Y}^{t,\left(l\right)\top}\right)_{i,j}-M_{i,j}\right)\ind_{\left|\left(\bm{X}^{t,\left(l\right)}\bm{Y}^{t,\left(l\right)\top}\right)_{i,j}-M_{i,j}\right|>\tau}\right)\\
 & =\frac{\tau}{p}\sum_{j=1}^{n}\underbrace{\delta_{l,j}\mathrm{sgn}\left(\left(\bm{X}^{t,\left(l\right)}\bm{Y}^{t,\left(l\right)\top}\right)_{l,j}-M_{l,j}\right)\ind_{\left|\left(\bm{X}^{t,\left(l\right)}\bm{Y}^{t,\left(l\right)\top}\right)_{l,j}-M_{l,j}\right|>\tau}\bm{e}_{j}}_{\eqqcolon\bm{u}_{j}}.
\end{align*}
It is easy to check that 
\begin{align*}
L & \coloneqq\max_{1\leq j\leq n}\left\Vert \bm{u}_{j}\right\Vert \leq1\\
V & \coloneqq\left\Vert \sum_{j=1}^{n}\mathbb{E}\left[\bm{u}_{j}^{\top}\bm{u}_{j}\right]\right\Vert =\left\Vert \sum_{j=1}^{n}\mathbb{E}\left[\delta_{l,j}\mathrm{sgn}\left(\left(\bm{X}^{t,\left(l\right)}\bm{Y}^{t,\left(l\right)\top}\right)_{l,j}-M_{l,j}\right)\ind_{\left|\left(\bm{X}^{t,\left(l\right)}\bm{Y}^{t,\left(l\right)\top}\right)_{l,j}-M_{l,j}\right|>\tau}\right]\bm{e}_{j}^{\top}\bm{e}_{j}\right\Vert \\
 & \leq\max_{j}\mathbb{E}\left[\delta_{l,j}\ind_{\left|\left(\bm{X}^{t,\left(l\right)}\bm{Y}^{t,\left(l\right)\top}\right)_{l,j}-M_{l,j}\right|>\tau}\right]\left\Vert \sum_{j=1}^{n}\bm{e}_{j}^{\top}\bm{e}_{j}\right\Vert \\
 & \overset{\text{(i)}}{\lesssim}\max_{j}\mathbb{E}\left[\delta_{l,j}\ind_{\left|\varepsilon_{l,j}\right|>\tau/2}\right]\left\Vert \sum_{j=1}^{n}\bm{e}_{j}^{\top}\bm{e}_{j}\right\Vert \overset{\text{(ii)}}{\lesssim}np\frac{\sigma^{2}}{\left(\tau/2\right)^{2}}\ll1,
\end{align*}
where (i) applies \eqref{eq:ind-1} and (ii) invokes Markov inequality.
Apply the matrix Bernstein inequality \citep[Theorem 6.1.1]{tropp2015introduction}
to discover that with probability over $1-O(n^{-100})$, 
\begin{align*}
\frac{\tau}{p}\left\Vert \sum_{j=1}^{n}\bm{u}_{j}\right\Vert  & \lesssim\frac{\tau}{p}\left(\sqrt{V\log n}+L\log n\right)\lesssim\sigma\sqrt{\frac{n}{p}}\log n+\frac{\left\Vert \bm{M}^{\star}\right\Vert _{\infty}}{p}\log n.
\end{align*}
\end{enumerate}
Combining all the bounds above gives 
\begin{align}
\left\Vert \bm{B}\right\Vert  & \lesssim n\left\Vert \bm{F}^{t,\left(l\right)}\bm{R}^{t,\left(l\right)}-\bm{F}^{\star}\right\Vert _{2,\infty}\left\Vert \bm{F}^{\star}\right\Vert _{2,\infty}+\sigma\sqrt{\frac{n}{p}}+\sqrt{\frac{\mu r\log n}{p}\sigma_{\max}}\left\Vert \bm{F}^{t,\left(l\right)}\bm{R}^{t,\left(l\right)}-\bm{F}^{\star}\right\Vert _{2,\infty}\nonumber \\
 & \qquad+\sigma\sqrt{\frac{n}{p}}\log n+\frac{\left\Vert \bm{M}^{\star}\right\Vert _{\infty}}{p}\log n\nonumber \\
 & \lesssim\left\Vert \bm{F}^{t,\left(l\right)}\bm{R}^{t,\left(l\right)}-\bm{F}^{\star}\right\Vert _{2,\infty}\sqrt{\mu rn\sigma_{\max}}+\sigma\sqrt{\frac{n}{p}}\log n+\frac{\left\Vert \bm{M}^{\star}\right\Vert _{\infty}}{p}\log n,\label{eq:B}
\end{align}
where the last line invokes Lemma \ref{lem:aux} and holds provided
that $np\gg\log n$.

Turning attention to $\bm{C}$, we have 
\begin{align}
 & \left\Vert \bm{X}^{t,\left(l\right)\top}\bm{X}^{t,\left(l\right)}-\bm{Y}^{t,\left(l\right)\top}\bm{Y}^{t,\left(l\right)}\right\Vert _{\mathrm{F}}\nonumber \\
 & \overset{\text{(i)}}{=}\left\Vert \left(\bm{R}^{t,\left(l\right)}\right)^{\top}\left(\bm{X}^{t,\left(l\right)\top}\bm{X}^{t,\left(l\right)}-\bm{Y}^{t,\left(l\right)\top}\bm{Y}^{t,\left(l\right)}\right)\bm{R}^{t,\left(l\right)}\right\Vert _{\mathrm{F}}\nonumber \\
 & \overset{\text{(ii)}}{\leq}\left\Vert \bm{R}^{t,\left(l\right)\top}\bm{X}^{t,\left(l\right)\top}\bm{X}^{t,\left(l\right)}\bm{R}^{t,\left(l\right)}-\bm{H}^{t\top}\bm{X}^{t\top}\bm{X}^{t}\bm{H}^{t}\right\Vert _{\mathrm{F}}+\left\Vert \bm{H}^{t\top}\bm{X}^{t\top}\bm{X}^{t}\bm{H}^{t}-\bm{H}^{t\top}\bm{Y}^{t\top}\bm{Y}^{t}\bm{H}^{t}\right\Vert _{\mathrm{F}}\nonumber \\
 & \qquad+\left\Vert \bm{H}^{t\top}\bm{Y}^{t\top}\bm{Y}^{t}\bm{H}^{t}-\bm{R}^{t,\left(l\right)\top}\bm{Y}^{t,\left(l\right)\top}\bm{Y}^{t,\left(l\right)}\bm{R}^{t,\left(l\right)}\right\Vert _{\mathrm{F}}\nonumber \\
 & \overset{\text{(iii)}}{\leq}\left\Vert \bm{R}^{t,\left(l\right)\top}\bm{X}^{t,\left(l\right)\top}\bm{X}^{t,\left(l\right)}\bm{R}^{t,\left(l\right)}-\bm{H}^{t\top}\bm{X}^{t\top}\bm{X}^{t,\left(l\right)}\bm{R}^{t,\left(l\right)}\right\Vert _{\mathrm{F}}\nonumber \\
 & \qquad+\left\Vert \bm{H}^{t\top}\bm{X}^{t\top}\bm{X}^{t,\left(l\right)}\bm{R}^{t,\left(l\right)}-\bm{H}^{t\top}\bm{X}^{t\top}\bm{X}^{t}\bm{H}^{t}\right\Vert _{\mathrm{F}}+\left\Vert \bm{X}^{t\top}\bm{X}^{t}-\bm{Y}^{t\top}\bm{Y}^{t}\right\Vert _{\mathrm{F}}\nonumber \\
 & \qquad+\left\Vert \bm{R}^{t,\left(l\right)\top}\bm{Y}^{t,\left(l\right)\top}\bm{Y}^{t,\left(l\right)}\bm{R}^{t,\left(l\right)}-\bm{H}^{t\top}\bm{Y}^{t\top}\bm{Y}^{t,\left(l\right)}\bm{R}^{t,\left(l\right)}\right\Vert _{\mathrm{F}}\nonumber \\
 & \qquad+\left\Vert \bm{H}^{t\top}\bm{Y}^{t\top}\bm{Y}^{t,\left(l\right)}\bm{R}^{t,\left(l\right)}-\bm{H}^{t\top}\bm{Y}^{t\top}\bm{Y}^{t}\bm{H}^{t}\right\Vert _{\mathrm{F}}\nonumber \\
 & \leq\left(\left\Vert \bm{X}^{\star}\right\Vert +\left\Vert \bm{X}^{t,\left(l\right)}\right\Vert \right)\left\Vert \bm{X}^{t,\left(l\right)}\bm{R}^{t,\left(l\right)}-\bm{X}^{t}\bm{H}^{t}\right\Vert _{\mathrm{F}}+\left\Vert \bm{X}^{t\top}\bm{X}^{t}-\bm{Y}^{t\top}\bm{Y}^{t}\right\Vert _{\mathrm{F}}\nonumber \\
 & \qquad+\left(\left\Vert \bm{Y}^{\star}\right\Vert +\left\Vert \bm{Y}^{t,\left(l\right)}\right\Vert \right)\left\Vert \bm{Y}^{t,\left(l\right)}\bm{R}^{t,\left(l\right)}-\bm{Y}^{t}\bm{H}^{t}\right\Vert _{\mathrm{F}}\nonumber \\
 & \overset{\text{(iv)}}{\lesssim}\sqrt{\kappa}\left(\frac{\sigma}{\sigma_{\min}}\sqrt{\frac{n}{p}}+\frac{\left\Vert \bm{M}^{\star}\right\Vert _{\infty}}{\sigma_{\min}}\sqrt{\frac{n}{p}}\right)\left\Vert \bm{F}^{\star}\right\Vert _{2,\infty}\left\Vert \bm{F}^{\star}\right\Vert \log n\nonumber \\
 & \qquad+\eta\mu\kappa^{4}r^{3.5}\sigma_{\max}^{2}\left(\frac{\sigma}{\sigma_{\min}}\sqrt{\frac{n}{p}}+\frac{\left\Vert \bm{M}^{\star}\right\Vert _{\infty}}{\sigma_{\min}}\sqrt{\frac{n}{p}}\right)^{2}\log^{2}n,\label{eq:C}
\end{align}
where (i) follows from the fact that $\Vert\bm{X}\Vert_{\mathrm{F}}=\Vert\bm{X}\bm{O}\Vert_{\mathrm{F}}$
for any $\bm{O}\in\mathcal{O}^{r\times r}$; (ii) and (iii) come from
the triangle inequality; (iv) arises from Lemma \ref{lem:aux}, \ref{lem:bal}
and \eqref{subeq:add-2}. where (i) follows from the fact that $\Vert\bm{X}\Vert_{\mathrm{F}}=\Vert\bm{X}\bm{O}\Vert_{\mathrm{F}}$
for any $\bm{O}\in\mathcal{O}^{r\times r}$; (ii) is due to $\bm{X}^{\star\top}\bm{X}^{\star}=\bm{Y}^{\star\top}\bm{Y}^{\star}$;
(iii) comes from \eqref{subeq:aux-6}; (iv) arises from Lemma \ref{lem:loo-1}
and \ref{lem:contraction}.

Combining \eqref{eq:B} and \eqref{eq:C} with \eqref{eq:Ftl} yields
\begin{align*}
 & \left\Vert \bm{F}^{t+1,\left(l\right)}\bm{H}^{t,\left(l\right)}-\widetilde{\bm{F}}^{t+1,\left(l\right)}\right\Vert \\
 & \leq\eta\left\Vert \bm{B}\right\Vert \left\Vert \bm{\Delta}^{t,\left(l\right)}\right\Vert +\frac{\eta}{2}\left\Vert \bm{F}^{\star}\right\Vert \left\Vert \bm{C}\right\Vert _{\mathrm{F}}+\frac{\eta}{2}\left\Vert \bm{\Delta}^{t,\left(l\right)}\right\Vert \left\Vert \bm{C}\right\Vert \\
 & \overset{\text{(i)}}{\lesssim}\eta\left(\left\Vert \bm{F}^{t,\left(l\right)}\bm{R}^{t,\left(l\right)}-\bm{F}^{\star}\right\Vert _{2,\infty}\sqrt{\mu rn\sigma_{\max}}+\sigma\sqrt{\frac{n}{p}}\log n+\frac{\left\Vert \bm{M}^{\star}\right\Vert _{\infty}}{p}\log n\right)\\
 &\qquad\times\sqrt{r}\left(\frac{\sigma}{\sigma_{\min}}\sqrt{\frac{n}{p}}+\frac{\left\Vert \bm{M}^{\star}\right\Vert _{\infty}}{\sigma_{\min}}\sqrt{\frac{n}{p}}\right)\left\Vert \bm{F}^{\star}\right\Vert \\
 & \qquad+\frac{\eta^{2}}{2}\mu\kappa^{4}r^{3.5}\sigma_{\max}^{2}\left(\frac{\sigma}{\sigma_{\min}}\sqrt{\frac{n}{p}}+\frac{\left\Vert \bm{M}^{\star}\right\Vert _{\infty}}{\sigma_{\min}}\sqrt{\frac{n}{p}}\right)^{2}\left\Vert \bm{F}^{\star}\right\Vert \log^{2}n\\
 & \qquad+\frac{\eta}{2}\sqrt{\kappa}\left(\frac{\sigma}{\sigma_{\min}}\sqrt{\frac{n}{p}}+\frac{\left\Vert \bm{M}^{\star}\right\Vert _{\infty}}{\sigma_{\min}}\sqrt{\frac{n}{p}}\right)\left\Vert \bm{F}^{\star}\right\Vert _{2,\infty}\left\Vert \bm{F}^{\star}\right\Vert ^{2}\log n\\
 & \lesssim\eta\sqrt{r}\sigma_{\min}\left(\frac{\sigma}{\sigma_{\min}}\sqrt{\frac{n}{p}}+\frac{\left\Vert \bm{M}^{\star}\right\Vert _{\infty}}{\sigma_{\min}}\sqrt{\frac{n}{p}}\right)^{2}\left\Vert \bm{F}^{\star}\right\Vert \log n\\
 & \qquad+\frac{\eta}{2}\sqrt{\kappa}\left(\frac{\sigma}{\sigma_{\min}}\sqrt{\frac{n}{p}}+\frac{\left\Vert \bm{M}^{\star}\right\Vert _{\infty}}{\sigma_{\min}}\sqrt{\frac{n}{p}}\right)\left\Vert \bm{F}^{\star}\right\Vert _{2,\infty}\left\Vert \bm{F}^{\star}\right\Vert ^{2}\log n,
\end{align*}
where (i) is due to \ref{subeq:aux-5} and the last inequality holds
provided that $0\leq\eta\leq c/(\mu\kappa^{5}r^{3}\sigma_{\max}\log n)$
for some small constant $c>0$. Plugging this inequality into \eqref{eq:Htl}
gives 
\begin{align*}
 & \left\Vert \left(\bm{H}^{t,\left(l\right)}\right)^{-1}\bm{H}^{t+1,\left(l\right)}-\bm{I}_{r}\right\Vert \\
 & \leq\frac{2}{\sigma_{\min}}\left\Vert \bm{F}^{t+1,\left(l\right)}\bm{H}^{t,\left(l\right)}-\widetilde{\bm{F}}^{t+1,\left(l\right)}\right\Vert \left\Vert \bm{F}^{\star}\right\Vert \\
 & \lesssim\eta\sqrt{r}\left(\frac{\sigma}{\sigma_{\min}}\sqrt{\frac{n}{p}}+\frac{\left\Vert \bm{M}^{\star}\right\Vert _{\infty}}{\sigma_{\min}}\sqrt{\frac{n}{p}}\right)^{2}\left\Vert \bm{F}^{\star}\right\Vert ^{2}\log n\\
 & \qquad+\frac{\eta}{2}\kappa^{1.5}\left(\frac{\sigma}{\sigma_{\min}}\sqrt{\frac{n}{p}}+\frac{\left\Vert \bm{M}^{\star}\right\Vert _{\infty}}{\sigma_{\min}}\sqrt{\frac{n}{p}}\right)\left\Vert \bm{F}^{\star}\right\Vert _{2,\infty}\left\Vert \bm{F}^{\star}\right\Vert \log n.
\end{align*}

\subsection{Proof of Lemma \ref{lem:loo-1}}

We only consider $1\leq l\leq n$ here. When $n+1\leq l\leq2n$, the
bound can be derived analogously.

The definition of $\bm{R}^{t,\left(l\right)}$ (cf.~\eqref{defn-Rtl})
implies

\[
\left\Vert \bm{F}^{t+1}\bm{H}^{t+1}-\bm{F}^{t+1,\left(l\right)}\bm{R}^{t+1,\left(l\right)}\right\Vert _{\mathrm{F}}\leq\left\Vert \bm{F}^{t+1}\bm{H}^{t}-\bm{F}^{t+1,\left(l\right)}\bm{R}^{t,\left(l\right)}\right\Vert _{\mathrm{F}}.
\]
Then the gradient update rules \eqref{subeq:gradient_update_ncvx}
reveal that 
\begin{align}
 & \bm{F}^{t+1}\bm{H}^{t}-\bm{F}^{t+1,\left(l\right)}\bm{R}^{t,\left(l\right)}\nonumber \\
 & =\underbrace{\bm{F}^{t}\bm{H}^{t}-\bm{F}^{t,\left(l\right)}\bm{R}^{t,\left(l\right)}-\eta\left[\nabla f\left(\bm{F}^{t}\bm{H}^{t}\right)-\nabla f\left(\bm{F}^{t,\left(l\right)}\bm{R}^{t,\left(l\right)}\right)\right]}_{\eqqcolon\bm{D}_{1}}\nonumber \\
 & \qquad+\underbrace{\eta\left[\nabla f^{\left(l\right)}\left(\bm{F}^{t,\left(l\right)}\bm{R}^{t,\left(l\right)}\right)-\nabla f\left(\bm{F}^{t,\left(l\right)}\bm{R}^{t,\left(l\right)}\right)\right]}_{\eqqcolon\bm{D}_{2}}.\label{eq:decompose-loo-2}
\end{align}
To start with, $\bm{D}_{1}$ can be controlled similarly as the proof
of Lemma \ref{lem:contraction} and gives 
\[
\left\Vert \bm{D}_{1}\right\Vert _{\mathrm{F}}\leq\left(1-\frac{\sigma_{\min}}{20}\eta\right)\left\Vert \bm{F}^{t}\bm{H}^{t}-\bm{F}^{t,\left(l\right)}\bm{R}^{t,\left(l\right)}\right\Vert _{\mathrm{F}},
\]
provided that $ $Then we are left to consider $\bm{D}_{2}$. In view
of the definition of gradients (cf.~\eqref{eq:obj} and \eqref{defn-floo-1})
and \eqref{eq:tau-tl}, one has 
\begin{align*}
 & \nabla f^{\left(l\right)}\left(\bm{F}^{t,\left(l\right)}\bm{R}^{t,\left(l\right)}\right)-\nabla f\left(\bm{F}^{t,\left(l\right)}\bm{R}^{t,\left(l\right)}\right)\\
 & =\left[\begin{array}{c}
\left[\mathcal{P}_{l,\cdot}\left(\psi_{\tau}\left(\bm{X}^{t,\left(l\right)}\bm{Y}^{t,\left(l\right)\top}-\bm{M}^{\star}\right)\right)-\frac{1}{p}\mathcal{P}_{\Omega_{l},\cdot}\left(\psi_{\tau}\left(\bm{X}^{t,\left(l\right)}\bm{Y}^{t,\left(l\right)\top}-\bm{M}\right)\right)\right]\bm{Y}^{t,\left(l\right)}\bm{R}^{t,\left(l\right)}\\
\left[\mathcal{P}_{l,\cdot}\left(\psi_{\tau}\left(\bm{X}^{t,\left(l\right)}\bm{Y}^{t,\left(l\right)\top}-\bm{M}^{\star}\right)\right)-\frac{1}{p}\mathcal{P}_{\Omega_{l},\cdot}\left(\psi_{\tau}\left(\bm{X}^{t,\left(l\right)}\bm{Y}^{t,\left(l\right)\top}-\bm{M}\right)\right)\right]^{\top}\bm{X}^{t,\left(l\right)}\bm{R}^{t,\left(l\right)}
\end{array}\right]\\
 & =\left[\begin{array}{c}
\bm{W}_{1}\\
\bm{W}_{2}
\end{array}\right]+\left[\begin{array}{c}
\bm{Z}_{1}\\
\bm{Z}_{2}
\end{array}\right]+\left[\begin{array}{c}
\bm{V}_{1}\\
\bm{V}_{2}
\end{array}\right],
\end{align*}
where 
\begin{align*}
\bm{W}_{1} & \coloneqq\mathcal{P}_{l,\cdot}\left(\left(\left(\bm{X}^{t,\left(l\right)}\bm{Y}^{t,\left(l\right)\top}\right)_{i,j}-M_{i,j}^{\star}\right)\right)\bm{Y}^{t,\left(l\right)}\bm{R}^{t,\left(l\right)}\\
 & \qquad-\frac{1}{p}\mathcal{P}_{\Omega_{l},\cdot}\left(\left(\left(\bm{X}^{t,\left(l\right)}\bm{Y}^{t,\left(l\right)\top}\right)_{i,j}-M_{i,j}^{\star}\right)\ind_{\left|\left(\bm{X}^{t,\left(l\right)}\bm{Y}^{t,\left(l\right)\top}\right)_{i,j}-M_{i,j}\right|\leq\tau}\right)\bm{Y}^{t,\left(l\right)}\bm{R}^{t,\left(l\right)},\\
\bm{W}_{2} & \coloneqq\left[\mathcal{P}_{l,\cdot}\left(\left(\left(\bm{X}^{t,\left(l\right)}\bm{Y}^{t,\left(l\right)\top}\right)_{i,j}-M_{i,j}^{\star}\right)\right)\right]^{\top}\bm{X}_{l,\cdot}^{t,\left(l\right)}\bm{R}^{t,\left(l\right)}\\
 & \qquad-\left[\frac{1}{p}\mathcal{P}_{\Omega_{l},\cdot}\left(\left(\left(\bm{X}^{t,\left(l\right)}\bm{Y}^{t,\left(l\right)\top}\right)_{i,j}-M_{i,j}^{\star}\right)\ind_{\left|\left(\bm{X}^{t,\left(l\right)}\bm{Y}^{t,\left(l\right)\top}\right)_{i,j}-M_{i,j}\right|\leq\tau}\right)\right]^{\top}\bm{X}_{l,\cdot}^{t,\left(l\right)}\bm{R}^{t,\left(l\right)},\\
\bm{Z}_{1} & \coloneqq-\frac{1}{p}\mathcal{P}_{\Omega_{l},\cdot}\left(\tau\mathrm{sgn}\left(\left(\bm{X}^{t,\left(l\right)}\bm{Y}^{t,\left(l\right)\top}\right)_{i,j}-M_{i,j}\right)\ind_{\left|\left(\bm{X}^{t,\left(l\right)}\bm{Y}^{t,\left(l\right)\top}\right)_{i,j}-M_{i,j}\right|>\tau}\right)\bm{Y}^{t,\left(l\right)}\bm{R}^{t,\left(l\right)},\\
\bm{Z}_{2} & \coloneqq-\left[\frac{1}{p}\mathcal{P}_{\Omega_{l},\cdot}\left(\tau\mathrm{sgn}\left(\left(\bm{X}^{t,\left(l\right)}\bm{Y}^{t,\left(l\right)\top}\right)_{i,j}-M_{i,j}\right)\ind_{\left|\left(\bm{X}^{t,\left(l\right)}\bm{Y}^{t,\left(l\right)\top}\right)_{i,j}-M_{i,j}\right|>\tau}\right)\right]^{\top}\bm{X}_{l,\cdot}^{t,\left(l\right)}\bm{R}^{t,\left(l\right)},\\
\bm{V}_{1} & \coloneqq\frac{1}{p}\mathcal{P}_{\Omega_{l},\cdot}\left(\varepsilon_{i,j}\ind_{\left|\left(\bm{X}^{t,\left(l\right)}\bm{Y}^{t,\left(l\right)\top}\right)_{i,j}-M_{i,j}\right|\leq\tau}\right)\bm{Y}^{t,\left(l\right)}\bm{R}^{t,\left(l\right)},\\
\bm{V}_{2} & \coloneqq\frac{1}{p}\mathcal{P}_{\Omega_{l},\cdot}\left(\varepsilon_{i,j}\ind_{\left|\left(\bm{X}^{t,\left(l\right)}\bm{Y}^{t,\left(l\right)\top}\right)_{i,j}-M_{i,j}\right|\leq\tau}\right)^{\top}\bm{X}_{l,\cdot}^{t,\left(l\right)}\bm{R}^{t,\left(l\right)}.
\end{align*}
We would control these terms one by one. 
\begin{enumerate}
\item The first term $\bm{W}_{1}$ obeys 
\[
p\left\Vert \bm{W}_{1}\right\Vert _{\mathrm{F}}=\left\Vert \sum_{j=1}^{n}\bm{u}_{j}\right\Vert _{\mathrm{F}},
\]
with 
\[
\bm{u}_{j}\coloneqq\left(p-\delta_{l,j}\ind_{\left|\left(\bm{X}^{t,\left(l\right)}\bm{Y}^{t,\left(l\right)\top}\right)_{l,j}-M_{l,j}\right|\leq\tau}\right)\left(\left(\bm{X}^{t,\left(l\right)}\bm{Y}^{t,\left(l\right)\top}\right)_{l,j}-M_{l,j}^{\star}\right)\bm{Y}_{j,\cdot}^{t,\left(l\right)}.
\]
We note that $\{\bm{u}_{j}\}_{j=1}^{n}$ is a set of independent vectors
when conditional on $\bm{X}^{t,(l)}$ and $\bm{Y}^{t,(l)}$. It follows
that 
\begin{align*}
L & \triangleq\max_{1\leq j\leq n}\left\Vert \bm{u}_{j}\right\Vert _{2}\leq\left\Vert \bm{X}^{t,\left(l\right)}\bm{Y}^{t,\left(l\right)\top}-\bm{M}^{\star}\right\Vert _{\infty}\left\Vert \bm{Y}^{t,\left(l\right)}\right\Vert _{2,\infty},
\end{align*}
and 
\begin{align*}
V & \triangleq\left\Vert \sum_{j=1}^{n}\mathbb{E}\left[\bm{u}_{j}^{\top}\bm{u}_{j}\right]\right\Vert \\
 & =\left\Vert \sum_{j=1}^{n}\mathbb{E}\left[\left(p-\delta_{l,j}\ind_{\left|\left(\bm{X}^{t,\left(l\right)}\bm{Y}^{t,\left(l\right)\top}\right)_{i,j}-M_{i,j}\right|\leq\tau}\right)^{2}\right]\left(\left(\bm{X}^{t,\left(l\right)}\bm{Y}^{t,\left(l\right)\top}\right)_{l,j}-M_{l,j}^{\star}\right)^{2}\bm{Y}_{j,\cdot}^{t,\left(l\right)}\bm{Y}_{j,\cdot}^{t,\left(l\right)\top}\right\Vert \\
 & \leq2\left\Vert \sum_{j=1}^{n}\mathbb{E}\left[p^{2}+\left(\delta_{l,j}\ind_{\left|\left(\bm{X}^{t,\left(l\right)}\bm{Y}^{t,\left(l\right)\top}\right)_{i,j}-M_{i,j}\right|\leq\tau}\right)^{2}\right]\left(\left(\bm{X}^{t,\left(l\right)}\bm{Y}^{t,\left(l\right)\top}\right)_{l,j}-M_{l,j}^{\star}\right)^{2}\bm{Y}_{j,\cdot}^{t,\left(l\right)}\bm{Y}_{j,\cdot}^{t,\left(l\right)\top}\right\Vert \\
 & \leq2\left\Vert \sum_{j=1}^{n}\mathbb{E}\left[p^{2}+\delta_{l,j}\right]\left(\left(\bm{X}^{t,\left(l\right)}\bm{Y}^{t,\left(l\right)\top}\right)_{l,j}-M_{l,j}^{\star}\right)^{2}\bm{Y}_{j,\cdot}^{t,\left(l\right)}\bm{Y}_{j,\cdot}^{t,\left(l\right)\top}\right\Vert \\
 & \leq4p\left\Vert \bm{X}^{t,\left(l\right)}\bm{Y}^{t,\left(l\right)\top}-\bm{M}^{\star}\right\Vert _{\infty}^{2}\left\Vert \bm{Y}^{t,\left(l\right)}\right\Vert _{\mathrm{F}}^{2}.
\end{align*}
The matrix Bernstein's inequality gives 
\begin{align*}
\left\Vert \bm{W}_{1}\right\Vert _{\mathrm{F}} & \lesssim\frac{1}{p}\left(\sqrt{V\log n}+L\log n\right)\\
 & \lesssim\sqrt{\frac{1}{p}\left\Vert \bm{X}^{t,\left(l\right)}\bm{Y}^{t,\left(l\right)\top}-\bm{M}^{\star}\right\Vert _{\infty}^{2}\left\Vert \bm{Y}^{t,\left(l\right)}\right\Vert _{\mathrm{F}}^{2}\log n}+\frac{1}{p}\left\Vert \bm{X}^{t,\left(l\right)}\bm{Y}^{t,\left(l\right)\top}-\bm{M}^{\star}\right\Vert _{\infty}\left\Vert \bm{Y}^{t,\left(l\right)}\right\Vert _{2,\infty}\log n\\
 & \lesssim\sqrt{\frac{\log n}{p}}\left\Vert \bm{F}^{t,\left(l\right)}\bm{R}^{t,\left(l\right)}-\bm{F}^{\star}\right\Vert _{2,\infty}\left\Vert \bm{F}^{\star}\right\Vert _{2,\infty}\left\Vert \bm{F}^{\star}\right\Vert _{\mathrm{F}}+\frac{1}{p}\left\Vert \bm{F}^{t,\left(l\right)}\bm{R}^{t,\left(l\right)}-\bm{F}^{\star}\right\Vert _{2,\infty}\left\Vert \bm{F}^{\star}\right\Vert _{2,\infty}^{2}\log n\\
 & \lesssim\sqrt{\frac{\mu r^{2}\log n}{np}}\sigma_{\max}\left\Vert \bm{F}^{t,\left(l\right)}\bm{R}^{t,\left(l\right)}-\bm{F}^{\star}\right\Vert _{2,\infty},
\end{align*}
where the second line follows from Lemma \ref{lem:aux}. 
\item Regarding $\bm{W}_{2}$, one has 
\[
\frac{p}{2}\left\Vert \bm{W}_{2}\right\Vert _{\mathrm{F}}=\left\Vert \left[\sum_{j=1}^{n}\bm{v}_{j}\right]^{\top}\bm{X}_{l,\cdot}^{t,\left(l\right)}\right\Vert _{\mathrm{F}}\le\left\Vert \sum_{j=1}^{n}\bm{v}_{j}\right\Vert _{2}\left\Vert \bm{X}_{l,\cdot}^{t,\left(l\right)}\right\Vert _{2},
\]
with 
\begin{align*}
\bm{v}_{j} & \coloneqq\bm{e}_{j}\left(p-\delta_{l,j}\ind_{\left|\left(\bm{X}^{t,\left(l\right)}\bm{Y}^{t,\left(l\right)\top}\right)_{l,j}-M_{l,j}\right|\leq\tau}\right)\left(\left(\bm{X}^{t,\left(l\right)}\bm{Y}^{t,\left(l\right)\top}\right)_{l,j}-M_{l,j}^{\star}\right).
\end{align*}
It is easy to obtain that 
\begin{align*}
L & \triangleq\max_{1\leq j\leq n}\left\Vert \bm{v}_{j}\right\Vert _{2}\leq\left\Vert \bm{X}^{t,\left(l\right)}\bm{Y}^{t,\left(l\right)\top}-\bm{M}^{\star}\right\Vert _{\infty},\\
V & \triangleq\left\Vert \sum_{j=1}^{n}\mathbb{E}\left[\left(p-\delta_{l,j}\ind_{\left|\left(\bm{X}^{t,\left(l\right)}\bm{Y}^{t,\left(l\right)\top}\right)_{l,j}-M_{l,j}\right|\leq\tau}\right)^{2}\left(\left(\bm{X}^{t,\left(l\right)}\bm{Y}^{t,\left(l\right)\top}\right)_{l,j}-M_{l,j}^{\star}\right)^{2}\bm{e}_{j}^{\top}\bm{e}_{j}\right]\right\Vert \\
 & \leq2\left\Vert \sum_{j=1}^{n}\mathbb{E}\left[p^{2}+\left(\delta_{l,j}\ind_{\left|\left(\bm{X}^{t,\left(l\right)}\bm{Y}^{t,\left(l\right)\top}\right)_{l,j}-M_{l,j}\right|\leq\tau}\right)^{2}\right]\left(\left(\bm{X}^{t,\left(l\right)}\bm{Y}^{t,\left(l\right)\top}\right)_{l,j}-M_{l,j}^{\star}\right)^{2}\bm{e}_{j}^{\top}\bm{e}_{j}\right\Vert \\
 & \le4p\left\Vert \bm{X}^{t,\left(l\right)}\bm{Y}^{t,\left(l\right)\top}-\bm{M}^{\star}\right\Vert _{\infty}^{2}\left\Vert \sum_{j=1}^{n}\bm{e}_{j}^{\top}\bm{e}_{j}\right\Vert =4np\left\Vert \bm{X}^{t,\left(l\right)}\bm{Y}^{t,\left(l\right)\top}-\bm{M}^{\star}\right\Vert _{\infty}^{2}.
\end{align*}
Invoke matrix Bernstein's inequality \citep[Theorem 6.1.1]{tropp2015introduction}
gives 
\begin{align*}
\left\Vert \bm{W}_{2}\right\Vert _{\mathrm{F}} & \lesssim\frac{1}{p}\left(\sqrt{V\log n}+L\log n\right)\left\Vert \bm{X}_{l,\cdot}^{t,\left(l\right)}\right\Vert _{2}\\
 & \lesssim\frac{1}{p}\left(\sqrt{np\left\Vert \bm{X}^{t,\left(l\right)}\bm{Y}^{t,\left(l\right)\top}-\bm{M}^{\star}\right\Vert _{\infty}^{2}\log n}+\left\Vert \bm{X}^{t,\left(l\right)}\bm{Y}^{t,\left(l\right)\top}-\bm{M}^{\star}\right\Vert _{\infty}\log n\right)\left\Vert \bm{X}_{l,\cdot}^{t,\left(l\right)}\right\Vert _{2}\\
 & \lesssim\sqrt{\frac{n\log n}{p}}\left\Vert \bm{X}^{t,\left(l\right)}\bm{Y}^{t,\left(l\right)\top}-\bm{M}^{\star}\right\Vert _{\infty}\left\Vert \bm{X}_{l,\cdot}^{t,\left(l\right)}\right\Vert _{2}\\
 & \lesssim\sqrt{\frac{\mu^{2}r^{2}\sigma_{\max}^{2}\log n}{np}}\left\Vert \bm{F}^{t,\left(l\right)}\bm{R}^{t,\left(l\right)}-\bm{F}^{\star}\right\Vert _{2,\infty},
\end{align*}
where the third line holds as long as $np\gg\log n$ and the last
inequality holds due to Lemma \ref{lem:aux}. 
\item Turn attention to $\bm{Z}_{1}$, we have 
\[
p\left\Vert \bm{Z}_{1}\right\Vert _{\mathrm{F}}=\tau\left\Vert \sum_{j=1}^{n}\bm{r}_{j}\right\Vert _{\mathrm{F}},
\]
with 
\begin{align*}
\bm{r}_{j} & \coloneqq-\delta_{l,j}\mathrm{sgn}\left(\left(\bm{X}^{t,\left(l\right)}\bm{Y}^{t,\left(l\right)\top}\right)_{l,j}-M_{l,j}\right)\ind_{\left|\left(\bm{X}^{t,\left(l\right)}\bm{Y}^{t,\left(l\right)\top}\right)_{l,j}-M_{l,j}\right|>\tau}\bm{Y}_{j,\cdot}^{t,\left(l\right)},
\end{align*}
where the last equality is due to \eqref{eq:tau-tl}. Conditional
on $\bm{X}^{t,(l)}$ and $\bm{Y}^{t,(l)}$, we can easily show that
\begin{align*}
L & \triangleq\max_{1\leq j\leq n}\left\Vert \bm{r}_{j}\right\Vert _{2}\leq\left\Vert \bm{Y}^{t,\left(l\right)}\right\Vert _{2,\infty},\\
V & \triangleq\left\Vert \sum_{j=1}^{n}\mathbb{E}\left[\bm{r}_{j}\bm{r}_{j}^{\top}\right]\right\Vert \\
 & =\left\Vert \sum_{j=1}^{n}\mathbb{E}\left[\delta_{l,j}\mathrm{sgn}\left(\left(\bm{X}^{t,\left(l\right)}\bm{Y}^{t,\left(l\right)\top}\right)_{l,j}-M_{l,j}\right)\ind_{\left|\left(\bm{X}^{t,\left(l\right)}\bm{Y}^{t,\left(l\right)\top}\right)_{l,j}-M_{l,j}\right|>\tau}\bm{Y}_{j,\cdot}^{t,\left(l\right)}\left(\bm{Y}_{j,\cdot}^{t,\left(l\right)}\right)^{\top}\right]\right\Vert \\
 & \leq\max_{j}\mathbb{E}\left[\left(\delta_{l,j}\mathrm{sgn}\left(\left(\bm{X}^{t,\left(l\right)}\bm{Y}^{t,\left(l\right)\top}\right)_{l,j}-M_{l,j}\right)\ind_{\left|\left(\bm{X}^{t,\left(l\right)}\bm{Y}^{t,\left(l\right)\top}\right)_{l,j}-M_{l,j}\right|>\tau}\right)^{2}\right]\left\Vert \bm{Y}^{t,\left(l\right)}\right\Vert _{\mathrm{F}}^{2}\\
 & \overset{\text{(i)}}{\leq}\max_{j}\mathbb{E}\left[\delta_{l,j}\ind_{\left|\varepsilon_{l,j}\right|>\tau/2}\right]\left\Vert \bm{Y}^{t,\left(l\right)}\right\Vert _{\mathrm{F}}^{2}\\
 & \overset{\text{(ii)}}{\lesssim}\frac{p\sigma^{2}}{\left(\tau/2\right)^{2}}\left\Vert \bm{Y}^{t,\left(l\right)}\right\Vert _{\mathrm{F}}^{2},
\end{align*}
where (i) follows from \eqref{eq:ind-1} and \eqref{eq:tau-tl} and
(ii) applies Markov inequality. The matrix Bernstein inequality \citep[Theorem 6.1.1]{tropp2015introduction}
implies that 
\[
\left\Vert \bm{Z}_{1}\right\Vert _{\mathrm{F}}\lesssim\frac{\tau}{p}\left(\sqrt{V\log n}+L\log n\right)\lesssim\frac{\tau}{p}\left(\sqrt{\frac{p\sigma^{2}}{\tau^{2}}\log n}\left\Vert \bm{Y}^{t,\left(l\right)}\right\Vert _{\mathrm{F}}+\left\Vert \bm{Y}^{t,\left(l\right)}\right\Vert _{2,\infty}\log n\right).
\]
\item Similarly, in view of \eqref{eq:tau-tl}, we have 
\begin{align}
\left\Vert \bm{Z}_{2}\right\Vert _{\mathrm{F}} & =\left\Vert \left[-\frac{1}{p}\mathcal{P}_{\Omega_{l},\cdot}\left(\tau\mathrm{sgn}\left(\left(\bm{X}^{t,\left(l\right)}\bm{Y}^{t,\left(l\right)\top}\right)_{i,j}-M_{i,j}\right)\ind_{\left|\left(\bm{X}^{t,\left(l\right)}\bm{Y}^{t,\left(l\right)\top}\right)_{i,j}-M_{i,j}\right|>\tau}\right)\right]^{\top}\bm{X}_{l,\cdot}^{t,\left(l\right)}\bm{R}^{t,\left(l\right)}\right\Vert _{\mathrm{F}}\nonumber \\
 & \lesssim\tau\left\Vert \sum_{j=1}^{n}\bm{b}_{j}\right\Vert _{2}\left\Vert \bm{X}_{l,\cdot}^{t,\left(l\right)}\right\Vert _{2},\label{eq:Z2}
\end{align}
with 
\[
\bm{b}_{j}\coloneqq\bm{e}_{j}\left(\frac{1}{p}\delta_{l,j}\mathrm{sgn}\left(\left(\bm{X}^{t,\left(l\right)}\bm{Y}^{t,\left(l\right)\top}\right)_{l,j}-M_{l,j}\right)\ind_{\left|\left(\bm{X}^{t,\left(l\right)}\bm{Y}^{t,\left(l\right)\top}\right)_{l,j}-M_{l,j}\right|>\tau}\right).
\]
It is easy to verify that 
\begin{align*}
L & \triangleq\max_{1\leq j\leq n}\left\Vert \bm{b}_{j}\right\Vert _{2}\leq\frac{1}{p},\\
V & \triangleq\left\Vert \sum_{j=1}^{n}\mathbb{E}\left[\left(\frac{1}{p}\delta_{l,j}\mathrm{sgn}\left(\left(\bm{X}^{t,\left(l\right)}\bm{Y}^{t,\left(l\right)\top}\right)_{l,j}-M_{l,j}\right)\ind_{\left|\left(\bm{X}^{t,\left(l\right)}\bm{Y}^{t,\left(l\right)\top}\right)_{l,j}-M_{l,j}\right|>\tau}\right)^{2}\bm{e}_{j}^{\top}\bm{e}_{j}\right]\right\Vert \\
 & \leq2\left\Vert \sum_{j=1}^{n}\mathbb{E}\left[\frac{1}{p^{2}}\delta_{l,j}\ind_{\left|\left(\bm{X}^{t,\left(l\right)}\bm{Y}^{t,\left(l\right)\top}\right)_{l,j}-M_{l,j}\right|>\tau}\right]\bm{e}_{j}^{\top}\bm{e}_{j}\right\Vert \\
 & \le2\max_{j}\mathbb{E}\left[\frac{1}{p^{2}}\delta_{l,j}\ind_{\left|\left(\bm{X}^{t,\left(l\right)}\bm{Y}^{t,\left(l\right)\top}\right)_{l,j}-M_{l,j}\right|>\tau}\right]\left\Vert \sum_{j=1}^{n}\bm{e}_{j}^{\top}\bm{e}_{j}\right\Vert \\
 & \overset{\text{(i)}}{\leq}2\max_{j}\mathbb{E}\left[\frac{1}{p^{2}}\delta_{l,j}\ind_{\left|\varepsilon_{l,j}\right|>\tau/2}\right]\left\Vert \sum_{j=1}^{n}\bm{e}_{j}^{\top}\bm{e}_{j}\right\Vert \overset{\text{(ii)}}{\leq}2n\left(\frac{1}{p}\cdot\frac{\sigma^{2}}{\left(\tau/2\right)^{2}}\right)\lesssim\frac{1}{p^{2}},
\end{align*}
where (i) arises from \eqref{eq:ind-1} and (ii) is due to Markov
inequality. The matrix Bernstein inequality \citep[Theorem 6.1.1]{tropp2015introduction}
reveals that 
\[
\left\Vert \sum_{j=1}^{n}\bm{b}_{j}\right\Vert _{\mathrm{F}}\lesssim\sqrt{V\log n}+L\log n\lesssim\frac{\log n}{p}.
\]
Plugging this bound into \eqref{eq:ind-1} yields 
\[
\left\Vert \bm{Z}_{2}\right\Vert _{\mathrm{F}}\lesssim\frac{\tau\log n}{p}\left\Vert \bm{X}_{l,\cdot}^{t,\left(l\right)}\right\Vert _{2}\lesssim\frac{\tau\log n}{p}\left\Vert \bm{F}^{\star}\right\Vert _{2,\infty},
\]
where the last inequality comes from \eqref{subeq:aux-5}. 
\item In terms of $\bm{V}_{1}$, one has 
\begin{align*}
\frac{p}{2}\left\Vert \bm{V}_{1}\right\Vert _{\mathrm{F}} & =\left\Vert \mathcal{P}_{\Omega_{l},\cdot}\left(\varepsilon_{i,j}\ind_{\left|\left(\bm{X}^{t,\left(l\right)}\bm{Y}^{t,\left(l\right)\top}\right)_{i,j}-M_{i,j}\right|\leq\tau}\right)\bm{Y}^{t,\left(l\right)}\right\Vert _{\mathrm{F}}\\
 & =\left\Vert \sum_{j=1}^{n}\delta_{l,j}\varepsilon_{l,j}\ind_{\left|\left(\bm{X}^{t,\left(l\right)}\bm{Y}^{t,\left(l\right)\top}\right)_{l,j}-M_{l,j}\right|\leq\tau}\bm{Y}_{j,\cdot}^{t,\left(l\right)}\right\Vert _{2}.
\end{align*}
Conditional on $\bm{Y}^{t,(l)}$, one has 
\begin{align*}
\left\Vert \left\Vert \delta_{l,j}\varepsilon_{l,j}\ind_{\left|\left(\bm{X}^{t,\left(l\right)}\bm{Y}^{t,\left(l\right)\top}\right)_{l,j}-M_{l,j}\right|\leq\tau}\bm{Y}_{j,\cdot}^{t,\left(l\right)}\right\Vert _{2}\right\Vert _{\psi_{1}} & \leq\left\Vert \bm{Y}^{t,\left(l\right)}\right\Vert _{2,\infty}\left\Vert \delta_{l,j}\varepsilon_{l,j}\ind_{\left|\left(\bm{X}^{t,\left(l\right)}\bm{Y}^{t,\left(l\right)\top}\right)_{l,j}-M_{l,j}\right|\leq\tau}\right\Vert _{\psi_{1}}\\
 & \overset{\text{(i)}}{\lesssim}\left\Vert \bm{Y}^{\star}\right\Vert _{2,\infty}\left\Vert \delta_{l,j}\varepsilon_{l,j}\ind_{\left|\varepsilon_{l,j}\right|\leq2\tau}\right\Vert _{\psi_{1}}\\
 & \lesssim\tau\left\Vert \bm{Y}^{\star}\right\Vert _{2,\infty}.
\end{align*}
Here (i) uses \eqref{subeq:aux-6} and the fact that 
\begin{align*}
\ind_{\left|\left(\bm{X}^{t,\left(l\right)}\bm{Y}^{t,\left(l\right)\top}\right)_{l,j}-M_{l,j}\right|\leq\tau} & =\ind_{\left|\left(\bm{X}^{t,\left(l\right)}\bm{Y}^{t,\left(l\right)\top}\right)_{l,j}-M_{l,j}^{\star}-\varepsilon_{l,j}\right|\leq\tau}\\
 & \leq\ind_{\left|\varepsilon_{l,j}\right|\leq\tau+\left|\left(\bm{X}^{t,\left(l\right)}\bm{Y}^{t,\left(l\right)\top}\right)_{l,j}-M_{l,j}^{\star}\right|}\leq\ind_{\left|\varepsilon_{i,j}\right|\leq2\tau},
\end{align*}
where the last inequality is due to the consequence of \eqref{subeq:add-1}
and \eqref{subeq:aux-5} as below, 
\begin{align*}
\max_{j}\left|\left(\bm{X}^{t,\left(l\right)}\bm{Y}^{t,\left(l\right)\top}\right)_{l,j}-M_{l,j}^{\star}\right| & \leq\left\Vert \left(\bm{F}^{t,\left(l\right)}\bm{H}^{t,\left(l\right)}-\bm{F}^{\star}\right)_{l,\cdot}\right\Vert _{2}\left(\left\Vert \bm{F}^{t,\left(l\right)}\right\Vert _{2,\infty}+\left\Vert \bm{F}^{\star}\right\Vert _{2,\infty}\right)\\
 & \leq3\left\Vert \left(\bm{F}^{t,\left(l\right)}\bm{H}^{t,\left(l\right)}-\bm{F}^{\star}\right)_{l,\cdot}\right\Vert _{2}\left\Vert \bm{F}^{\star}\right\Vert _{2,\infty}\\
 & \lesssim\kappa\sqrt{r}\left\Vert \bm{F}^{\star}\right\Vert _{2,\infty}^{2}\left(\frac{\sigma}{\sigma_{\min}}\sqrt{\frac{n\log n}{p}}+\frac{\left\Vert \bm{M}^{\star}\right\Vert _{\infty}}{\sigma_{\min}}\sqrt{\frac{n}{p}}\right)\ll\tau.
\end{align*}
In addition, one has 
\[
V\triangleq\left\Vert \mathbb{E}\left[\sum_{j=1}^{n}\delta_{l,j}\varepsilon_{l,j}^{2}\ind_{\left|\left(\bm{X}^{t,\left(l\right)}\bm{Y}^{t,\left(l\right)\top}\right)_{l,j}-M_{l,j}\right|\leq\tau}\bm{Y}_{j,\cdot}^{t,\left(l\right)}\bm{Y}_{j,\cdot}^{t,\left(l\right)\top}\right]\right\Vert \leq p\sigma^{2}\left\Vert \bm{Y}^{t,\left(l\right)}\right\Vert _{\mathrm{F}}^{2}.
\]
Hence, invoking the matrix Bernstein inequality \citep[Proposition 2]{koltchinskii2011nuclear}
gives that 
\begin{align*}
\left\Vert \sum_{j=1}^{n}\delta_{l,j}\varepsilon_{l,j}\ind_{\left|\left(\bm{X}^{t,\left(l\right)}\bm{Y}^{t,\left(l\right)\top}\right)_{l,j}-M_{l,j}\right|\leq\tau}\bm{Y}_{j,\cdot}^{t,\left(l\right)}\right\Vert _{2} & \lesssim\sqrt{V\log n}+\tau\left\Vert \bm{Y}^{\star}\right\Vert _{2,\infty}\log n\\
 & \lesssim\tau\left\Vert \bm{Y}^{\star}\right\Vert _{2,\infty}\log n.
\end{align*}
\item For the last term $\bm{V}_{2}$, one has 
\begin{align*}
\frac{p}{2}\left\Vert \bm{V}_{2}\right\Vert _{\mathrm{F}} & =\left\Vert \mathcal{P}_{\Omega_{l},\cdot}\left(\varepsilon_{i,j}\ind_{\left|\left(\bm{X}^{t,\left(l\right)}\bm{Y}^{t,\left(l\right)\top}\right)_{i,j}-M_{i,j}\right|\leq\tau}\right)^{\top}\bm{X}_{l,\cdot}^{t,\left(l\right)}\bm{R}^{t,\left(l\right)}\right\Vert _{\mathrm{F}}\\
 & =\left\Vert \mathcal{P}_{\Omega_{l},\cdot}\left(\varepsilon_{i,j}\ind_{\left|\left(\bm{X}^{t,\left(l\right)}\bm{Y}^{t,\left(l\right)\top}\right)_{i,j}-M_{i,j}\right|\leq\tau}\right)\right\Vert _{2}\left\Vert \bm{X}_{l,\cdot}^{t,\left(l\right)}\right\Vert _{2}\\
 & =\left\Vert \sum_{j=1}^{n}\bm{e}_{j}\delta_{i,j}\varepsilon_{l,j}\ind_{\left|\left(\bm{X}^{t,\left(l\right)}\bm{Y}^{t,\left(l\right)\top}\right)_{l,j}-M_{l,j}\right|\leq\tau}\right\Vert _{2}\left\Vert \bm{X}_{l,\cdot}^{t,\left(l\right)}\right\Vert _{2}.
\end{align*}
Similar to the bound of $\bm{V}_{1}$, one has 
\begin{align*}
L & \coloneqq\left\Vert \left\Vert \bm{e}_{j}\delta_{i,j}\varepsilon_{l,j}\ind_{\left|\left(\bm{X}^{t,\left(l\right)}\bm{Y}^{t,\left(l\right)\top}\right)_{l,j}-M_{l,j}\right|\leq\tau}\right\Vert _{2}\right\Vert _{\psi_{1}}\lesssim\tau,\\
V & =\left\Vert \sum_{j=1}^{n}\mathbb{E}\left[\delta_{i,j}\varepsilon_{l,j}^{2}\ind_{\left|\left(\bm{X}^{t,\left(l\right)}\bm{Y}^{t,\left(l\right)\top}\right)_{l,j}-M_{l,j}\right|\leq\tau}\right]\bm{e}_{j}^{\top}\bm{e}_{j}\right\Vert \le\sigma^{2}pn.
\end{align*}
The matrix Bernstein inequality \citep[Proposition 2]{koltchinskii2011nuclear}
reveals that 
\[
\left\Vert \sum_{j=1}^{n}\bm{e}_{j}\delta_{i,j}\varepsilon_{l,j}\ind_{\left|\left(\bm{X}^{t,\left(l\right)}\bm{Y}^{t,\left(l\right)\top}\right)_{l,j}-M_{l,j}\right|\leq\tau}\right\Vert _{2}\lesssim\sigma\sqrt{pn\log n}+\tau\log n\lesssim\tau\log n.
\]
and consequently 
\[
\left\Vert \bm{V}_{2}\right\Vert _{\mathrm{F}}\lesssim\frac{\tau\log n}{p}\left\Vert \bm{X}_{l,\cdot}^{t,\left(l\right)}\right\Vert _{2}\lesssim\frac{\tau\log n}{p}\left\Vert \bm{F}^{\star}\right\Vert _{2,\infty},
\]
where the last inequality is due to \eqref{subeq:aux-6}. 
\end{enumerate}
Taking all the bounds above together, we arrive at 
\begin{align*}
 & \left\Vert \nabla f^{\left(l\right)}\left(\bm{F}^{t,\left(l\right)}\bm{R}^{t,\left(l\right)}\right)-\nabla f\left(\bm{F}^{t,\left(l\right)}\bm{R}^{t,\left(l\right)}\right)\right\Vert _{\mathrm{F}}\\
 & \lesssim\sqrt{\frac{\mu r^{2}\log n}{np}}\sigma_{\max}\left\Vert \bm{F}^{t,\left(l\right)}\bm{R}^{t,\left(l\right)}-\bm{F}^{\star}\right\Vert _{2,\infty}+\sqrt{\frac{\mu^{2}r^{2}\sigma_{\max}^{2}\log n}{np}}\left\Vert \bm{F}^{t,\left(l\right)}\bm{R}^{t,\left(l\right)}-\bm{F}^{\star}\right\Vert _{2,\infty}\\
 & \qquad+\frac{\tau}{p}\left(\sqrt{\frac{p\sigma^{2}}{\tau^{2}}\log n}\left\Vert \bm{Y}^{t,\left(l\right)}\right\Vert _{\mathrm{F}}+\left\Vert \bm{Y}^{t,\left(l\right)}\right\Vert _{2,\infty}\log n\right)+\frac{\tau}{p}\left\Vert \bm{F}^{\star}\right\Vert _{2,\infty}\log n\\
 & \lesssim\left(\sigma\sqrt{\frac{n}{p}}+\left\Vert \bm{M}^{\star}\right\Vert _{\infty}\sqrt{\frac{n}{p}}\right)\log n\left\Vert \bm{F}^{\star}\right\Vert _{2,\infty},
\end{align*}
where the second line uses Lemma \ref{lem:aux}. Plugging the equation
above into \eqref{eq:decompose-loo-2} yields 
\begin{align*}
 & \left\Vert \bm{F}^{t+1}\bm{H}^{t}-\bm{F}^{t+1,\left(l\right)}\bm{R}^{t,\left(l\right)}\right\Vert _{\mathrm{F}}\\
 & \leq\left(1-\frac{\sigma_{\min}}{20}\eta\right)\left\Vert \bm{F}^{t}\bm{H}^{t}-\bm{F}^{t,\left(l\right)}\bm{R}^{t,\left(l\right)}\right\Vert _{\mathrm{F}}+C\left(\frac{\sigma}{\sigma_{\min}}\sqrt{\frac{n}{p}}+\frac{\left\Vert \bm{M}^{\star}\right\Vert _{\infty}}{\sigma_{\min}}\sqrt{\frac{n}{p}}\right)\left\Vert \bm{F}^{\star}\right\Vert _{2,\infty}\log n\\
 & \lesssim\sqrt{\kappa}\left(\frac{\sigma}{\sigma_{\min}}\sqrt{\frac{n}{p}}+\frac{\left\Vert \bm{M}^{\star}\right\Vert _{\infty}}{\sigma_{\min}}\sqrt{\frac{n}{p}}\right)\left\Vert \bm{F}^{\star}\right\Vert _{2,\infty}\log n.
\end{align*}

\subsection{Proof of Lemma \ref{lem:loo-2}}

For any $1\leq l\leq2n$, we have

\begin{align*}
 & \left\Vert \left(\bm{F}^{t+1}\bm{H}^{t+1}-\bm{F}^{\star}\right)_{l,\cdot}\right\Vert _{2}\\
 & \leq\left\Vert \left(\bm{F}^{t+1}\bm{H}^{t+1}-\bm{F}^{t+1,\left(l\right)}\bm{H}^{t+1,\left(l\right)}\right)_{l,\cdot}\right\Vert _{2}+\left\Vert \left(\bm{F}^{t+1,\left(l\right)}\bm{H}^{t+1,\left(l\right)}-\bm{F}^{\star}\right)_{l,\cdot}\right\Vert _{2}\\
 & \leq\left\Vert \bm{F}^{t+1}\bm{H}^{t+1}-\bm{F}^{t+1,\left(l\right)}\bm{H}^{t+1,\left(l\right)}\right\Vert _{\mathrm{F}}+\left\Vert \left(\bm{F}^{t+1,\left(l\right)}\bm{H}^{t+1,\left(l\right)}-\bm{F}^{\star}\right)_{l,\cdot}\right\Vert _{2}\\
 & \overset{\mathrm{(i)}}{\leq}5\kappa\left\Vert \bm{F}^{t+1}\bm{H}^{t+1}-\bm{F}^{t+1,\left(l\right)}\bm{R}^{t+1,\left(l\right)}\right\Vert _{\mathrm{F}}+\left\Vert \left(\bm{F}^{t+1,\left(l\right)}\bm{H}^{t+1,\left(l\right)}-\bm{F}^{\star}\right)_{l,\cdot}\right\Vert _{2}\\
 & \overset{\mathrm{(ii)}}{\lesssim}\kappa^{1.5}\left(\frac{\sigma}{\sigma_{\min}}\sqrt{\frac{n}{p}}+\frac{\left\Vert \bm{M}^{\star}\right\Vert _{\infty}}{\sigma_{\min}}\sqrt{\frac{n}{p}}\right)\left\Vert \bm{F}^{\star}\right\Vert _{2,\infty}\log n+\kappa\sqrt{r}\left\Vert \bm{F}^{\star}\right\Vert _{2,\infty}\left(\frac{\sigma}{\sigma_{\min}}\sqrt{\frac{n\log n}{p}}+\frac{\left\Vert \bm{M}^{\star}\right\Vert _{\infty}}{\sigma_{\min}}\sqrt{\frac{n}{p}}\right)\\
 & \lesssim\kappa^{1.5}\sqrt{r}\left(\frac{\sigma}{\sigma_{\min}}\sqrt{\frac{n}{p}}+\frac{\left\Vert \bm{M}^{\star}\right\Vert _{\infty}}{\sigma_{\min}}\sqrt{\frac{n}{p}}\right)\left\Vert \bm{F}^{\star}\right\Vert _{2,\infty}\log n,
\end{align*}
where (i) follows from \eqref{subeq:aux-4} and (ii) is due to Lemma
\ref{lem:loo} and \ref{lem:loo-1}.

\subsection{Proof of Lemma \ref{lem:bal}\label{subsec:Proof-of-Lemma-bal}}

For notational simplicity, we define

\begin{equation}
\bm{A}^{t}\coloneqq\bm{X}^{t\top}\bm{X}^{t}-\bm{Y}^{t\top}\bm{Y}^{t},\qquad\mathrm{and}\qquad\bm{D}^{t}\coloneqq\frac{1}{2p}\mathcal{P}_{\Omega}\left(\left\{ \psi_{\tau}\left(\left(\bm{X}\bm{Y}^{\top}\right)_{i,j}-M_{i,j}\right)\right\} _{i,j}\right).\label{defn-D}
\end{equation}
The gradient descent update rules \eqref{subeq:gradient_update_ncvx}
reveals that 
\begin{align}
\bm{A}^{t+1} & =\bm{A}^{t}-\eta\bm{B}^{t}+\eta^{2}\bm{C}^{t}\label{eq:A-update}
\end{align}
where 
\begin{align*}
\bm{B}^{t} & \coloneqq\bm{X}^{t\top}\nabla_{\bm{X}}f\left(\bm{X}^{t},\bm{Y}^{t}\right)+\nabla_{\bm{X}}f\left(\bm{X}^{t},\bm{Y}^{t}\right)^{\top}\bm{X}^{t}-\bm{Y}^{t\top}\nabla_{\bm{Y}}f\left(\bm{X}^{t},\bm{Y}^{t}\right)-\nabla_{\bm{Y}}f\left(\bm{X}^{t},\bm{Y}^{t}\right)^{\top}\bm{Y}^{t}\\
\bm{C}^{t} & \coloneqq\nabla_{\bm{X}}f\left(\bm{X}^{t},\bm{Y}^{t}\right)^{\top}\nabla_{\bm{X}}f\left(\bm{X}^{t},\bm{Y}^{t}\right)-\nabla_{\bm{Y}}f\left(\bm{X}^{t},\bm{Y}^{t}\right)^{\top}\nabla_{\bm{Y}}f\left(\bm{X}^{t},\bm{Y}^{t}\right).
\end{align*}
Due to the definition of $\bm{D}^{t}$ (cf.~\eqref{defn-D}), one
has 
\begin{align*}
\nabla_{\bm{X}}f\left(\bm{X}^{t},\bm{Y}^{t}\right) & =\bm{D}^{t}\bm{Y}^{t}+\frac{1}{2}\bm{X}^{t}\left(\bm{X}^{t\top}\bm{X}^{t}-\bm{Y}^{t\top}\bm{Y}^{t}\right),\\
\nabla_{\bm{Y}}f\left(\bm{X}^{t},\bm{Y}^{t}\right) & =\bm{D}^{t\top}\bm{X}^{t}+\frac{1}{2}\bm{Y}^{t}\left(\bm{Y}^{t\top}\bm{Y}^{t}-\bm{X}^{t\top}\bm{X}^{t}\right).
\end{align*}
Simple calculation gives that 
\begin{align*}
\bm{B}^{t} & =\bm{X}^{t\top}\left[\bm{D}^{t}\bm{Y}^{t}+\frac{1}{2}\bm{X}^{t}\left(\bm{X}^{t\top}\bm{X}^{t}-\bm{Y}^{t\top}\bm{Y}^{t}\right)\right]+\left[\bm{D}^{t}\bm{Y}^{t}+\frac{1}{2}\bm{X}^{t}\left(\bm{X}^{t\top}\bm{X}^{t}-\bm{Y}^{t\top}\bm{Y}^{t}\right)\right]^{\top}\bm{X}^{t}\\
 & \qquad-\bm{Y}^{t\top}\left[\bm{D}^{t\top}\bm{X}^{t}+\frac{1}{2}\bm{Y}^{t}\left(\bm{Y}^{t\top}\bm{Y}^{t}-\bm{X}^{t\top}\bm{X}^{t}\right)\right]-\left[\bm{D}^{t\top}\bm{X}^{t}+\frac{1}{2}\bm{Y}^{t}\left(\bm{Y}^{t\top}\bm{Y}^{t}-\bm{X}^{t\top}\bm{X}^{t}\right)\right]^{\top}\bm{Y}^{t}\\
 & =\frac{1}{2}\left(\bm{X}^{t\top}\bm{X}^{t}+\bm{Y}^{t\top}\bm{Y}^{t}\right)\left(\bm{X}^{t\top}\bm{X}^{t}-\bm{Y}^{t\top}\bm{Y}^{t}\right)+\frac{1}{2}\left(\bm{X}^{t\top}\bm{X}^{t}-\bm{Y}^{t\top}\bm{Y}^{t}\right)\left(\bm{X}^{t\top}\bm{X}^{t}+\bm{Y}^{t\top}\bm{Y}^{t}\right),\\
\bm{C}^{t} & =\left[\bm{D}^{t}\bm{Y}^{t}+\frac{1}{2}\bm{X}^{t}\left(\bm{X}^{t\top}\bm{X}^{t}-\bm{Y}^{t\top}\bm{Y}^{t}\right)\right]^{\top}\left[\bm{D}^{t}\bm{Y}^{t}+\frac{1}{2}\bm{X}^{t}\left(\bm{X}^{t\top}\bm{X}^{t}-\bm{Y}^{t\top}\bm{Y}^{t}\right)\right]\\
 & \qquad-\left[\bm{D}^{t\top}\bm{X}^{t}+\frac{1}{2}\bm{Y}^{t}\left(\bm{Y}^{t\top}\bm{Y}^{t}-\bm{X}^{t\top}\bm{X}^{t}\right)\right]^{\top}\left[\bm{D}^{t\top}\bm{X}^{t}+\frac{1}{2}\bm{Y}^{t}\left(\bm{Y}^{t\top}\bm{Y}^{t}-\bm{X}^{t\top}\bm{X}^{t}\right)\right]\\
 & =\bm{Y}^{t\top}\bm{D}^{t\top}\bm{D}^{t}\bm{Y}^{t}-\bm{X}^{t\top}\bm{D}^{t\top}\bm{D}^{t}\bm{X}^{t}+\frac{1}{4}\left(\bm{A}^{t}\right)^{3}\\
 & \qquad+\frac{1}{2}\left(\bm{Y}^{t\top}\bm{D}^{t\top}\bm{X}^{t}+\bm{X}^{t\top}\bm{D}^{t}\bm{Y}^{t}\right)\bm{A}^{t}+\frac{1}{2}\bm{A}^{t}\left(\bm{X}^{t\top}\bm{D}^{t}\bm{Y}^{t}+\bm{Y}^{t\top}\bm{D}^{t\top}\bm{X}^{t}\right).
\end{align*}
Plugging these into \eqref{eq:A-update}, one has 
\begin{align*}
\bm{A}^{t+1} & =\bm{A}^{t}-\eta\bm{B}^{t}+\eta^{2}\bm{C}^{t}\\
 & =\left[\frac{1}{2}-\frac{\eta}{2}\left(\bm{X}^{t\top}\bm{X}^{t}+\bm{Y}^{t\top}\bm{Y}^{t}\right)+\frac{\eta^{2}}{8}\left(\bm{A}^{t}\right)^{2}+\frac{\eta^{2}}{2}\left(\bm{Y}^{t\top}\bm{D}^{t\top}\bm{X}^{t}+\bm{X}^{t\top}\bm{D}^{t}\bm{Y}^{t}\right)\right]\bm{A}^{t}\\
 & \qquad+\bm{A}^{t}\left[\frac{1}{2}-\frac{\eta}{2}\left(\bm{X}^{t\top}\bm{X}^{t}+\bm{Y}^{t\top}\bm{Y}^{t}\right)+\frac{\eta^{2}}{8}\left(\bm{A}^{t}\right)^{2}+\frac{\eta^{2}}{2}\left(\bm{Y}^{t\top}\bm{D}^{t\top}\bm{X}^{t}+\bm{X}^{t\top}\bm{D}^{t}\bm{Y}^{t}\right)\right]+\eta^{2}\bm{E}^{t},
\end{align*}
and thus 
\begin{align}
\left\Vert \bm{A}^{t+1}\right\Vert _{\mathrm{F}} & \leq2\left\Vert \frac{1}{2}\bm{I}_{r}-\frac{\eta}{2}\left(\bm{X}^{t\top}\bm{X}^{t}+\bm{Y}^{t\top}\bm{Y}^{t}\right)+\frac{\eta^{2}}{8}\left(\bm{A}^{t}\right)^{2}+\frac{\eta^{2}}{2}\left(\bm{Y}^{t\top}\bm{D}^{t\top}\bm{X}^{t}+\bm{X}^{t\top}\bm{D}^{t}\bm{Y}^{t}\right)\right\Vert \left\Vert \bm{A}^{t}\right\Vert _{\mathrm{F}}\nonumber\\
&\qquad+\eta^{2}\left\Vert \bm{E}^{t}\right\Vert _{\mathrm{F}}.\label{eq:A-upper}
\end{align}
To control this upper bound, one has Lemma \ref{lem:aux} shows that
\[
\sigma_{\min}\left(\bm{X}^{t\top}\bm{X}^{t}+\bm{Y}^{t\top}\bm{Y}^{t}\right)\geq\sigma_{\min}\left(\bm{X}^{t\top}\bm{X}^{t}\right)\geq\frac{3\sigma_{\min}}{4}.
\]
Furthermore, one has 
\begin{align*}
\left\Vert \bm{D}^{t}\right\Vert  & =\left\Vert \frac{1}{2p}\mathcal{P}_{\Omega}\left(\left\{ \psi_{\tau}\left(\left(\bm{X}\bm{Y}^{\top}\right)_{i,j}-M_{i,j}\right)\right\} _{i,j}\right)\right\Vert \\
 & \leq\left\Vert \frac{1}{p}\mathcal{P}_{\Omega}\left(\left(\left(\bm{X}\bm{Y}^{\top}\right)_{i,j}-M_{i,j}^{\star}\right)\ind_{\left|\left(\bm{X}\bm{Y}^{\top}\right)_{i,j}-M_{i,j}\right|\leq\tau}\right)\right\Vert +\left\Vert \frac{1}{p}\mathcal{P}_{\Omega}\left(\varepsilon_{i,j}\ind_{\left|\left(\bm{X}\bm{Y}^{\top}\right)_{i,j}-M_{i,j}\right|\leq\tau}\right)\right\Vert \\
 & \qquad+\left\Vert \frac{1}{2p}\mathcal{P}_{\Omega}\left(\tau\mathrm{sgn}\left(\left(\bm{X}\bm{Y}^{\top}\right)_{i,j}-M_{i,j}\right)\ind_{\left|\left(\bm{X}\bm{Y}^{\top}\right)_{i,j}-M_{i,j}\right|>\tau}\right)\right\Vert \\
 & \overset{\text{(i)}}{\lesssim}\left\Vert \frac{1}{p}\mathcal{P}_{\Omega}\left(\bm{1}\bm{1}^{\top}\right)\right\Vert \max_{i,j}\left|\left(\left(\bm{X}\bm{Y}^{\top}\right)_{i,j}-M_{i,j}^{\star}\right)\ind_{\left|\left(\bm{X}\bm{Y}^{\top}\right)_{i,j}-M_{i,j}\right|\leq\tau}\right|+\sigma\sqrt{\frac{n}{p}}+\sigma\sqrt{\frac{n}{p}}\\
 & \overset{\text{(ii)}}{\lesssim}n\left\Vert \bm{X}\bm{Y}^{\top}-\bm{M}^{\star}\right\Vert _{\infty}+\sigma\sqrt{\frac{n}{p}}\\
 & \lesssim n\left\Vert \bm{F}^{t}\bm{H}^{t}-\bm{F}^{\star}\right\Vert _{2,\infty}\left\Vert \bm{F}^{\star}\right\Vert _{2,\infty}+\sigma\sqrt{\frac{n}{p}}\\
 & \lesssim\mu\kappa^{3/2}r^{3/2}\sigma_{\max}\left(\frac{\sigma}{\sigma_{\min}}\sqrt{\frac{n}{p}}+\frac{\left\Vert \bm{M}^{\star}\right\Vert _{\infty}}{\sigma_{\min}}\sqrt{\frac{n}{p}}\right)\log n,
\end{align*}
where (i) is due to \eqref{eq:operator-ob}, \eqref{eq:noise} and
\eqref{eq:tau}; (ii) comes from Lemma \ref{lem:bernoulli}; the last
line relies on Lemma \ref{lem:loo-2}. Then we turn back to consider
\eqref{eq:A-upper}. We have

\begin{align*}
 & \left\Vert \frac{\eta^{2}}{8}\left(\bm{A}^{t}\right)^{2}+\frac{\eta^{2}}{2}\left(\bm{Y}^{t\top}\bm{D}^{t\top}\bm{X}^{t}+\bm{X}^{t\top}\bm{D}^{t}\bm{Y}^{t}\right)\right\Vert \\
 & \lesssim\eta^{2}\left(\left\Vert \bm{A}^{t}\right\Vert ^{2}+\left\Vert \bm{F}^{t}\right\Vert ^{2}\left\Vert \bm{D}^{t}\right\Vert \right)\\
 & \lesssim\eta^{2}\left\{ \left[\eta\sqrt{r}\kappa\sigma_{\max}^{2}\left(\frac{\sigma}{\sigma_{\min}}\sqrt{\frac{n\log n}{p}}+\frac{\mu r}{\sqrt{np}}\right)\right]^{2}+\mu\kappa^{3/2}r^{3/2}\sigma_{\max}^{2}\left(\frac{\sigma}{\sigma_{\min}}\sqrt{\frac{n}{p}}+\frac{\left\Vert \bm{M}^{\star}\right\Vert _{\infty}}{\sigma_{\min}}\sqrt{\frac{n}{p}}\right)\log n\right\} \\
 & \ll\eta\sigma_{\min},
\end{align*}
provided $\eta\mu\kappa^{3}r^{2}\sigma_{\max}\log n\ll1$, and then
we arrive at 
\begin{align*}
 & \sigma_{\min}\left(\frac{\eta}{2}\left(\bm{X}^{t\top}\bm{X}^{t}+\bm{Y}^{t\top}\bm{Y}^{t}\right)-\frac{\eta^{2}}{8}\left(\bm{A}^{t}\right)^{2}-\frac{\eta^{2}}{2}\left(\bm{Y}^{t\top}\bm{D}^{t\top}\bm{X}^{t}+\bm{X}^{t\top}\bm{D}^{t}\bm{Y}^{t}\right)\right)\\
 & \geq\sigma_{\min}\left(\frac{\eta}{2}\left(\bm{X}^{t\top}\bm{X}^{t}+\bm{Y}^{t\top}\bm{Y}^{t}\right)\right)-\left\Vert \frac{\eta^{2}}{8}\left(\bm{A}^{t}\right)^{2}+\frac{\eta^{2}}{2}\left(\bm{Y}^{t\top}\bm{D}^{t\top}\bm{X}^{t}+\bm{X}^{t\top}\bm{D}^{t}\bm{Y}^{t}\right)\right\Vert \\
 & \geq\frac{3\eta\sigma_{\min}}{8}-\frac{\eta\sigma_{\min}}{8}\geq\frac{\eta\sigma_{\min}}{4}.
\end{align*}
Combine \eqref{eq:A-upper} with the previous bound to obtain 
\begin{align}
\left\Vert \bm{A}^{t+1}\right\Vert _{\mathrm{F}} & \leq2\left(\frac{1}{2}-\frac{\eta}{4}\sigma_{\min}\right)\left\Vert \bm{A}^{t}\right\Vert _{\mathrm{F}}+\eta^{2}\left\Vert \bm{E}^{t}\right\Vert _{\mathrm{F}}.\label{eq:A}
\end{align}
It remains to bound $\bm{E}^{t}$. We have 
\begin{align*}
\left\Vert \bm{E}^{t}\right\Vert _{\mathrm{F}} & \leq\left\Vert \bm{Y}^{t\top}\bm{D}^{t\top}\bm{D}^{t}\bm{Y}^{t}-\bm{X}^{t\top}\bm{D}^{t\top}\bm{D}^{t}\bm{X}^{t}\right\Vert _{\mathrm{F}}\\
 & \leq8\left\Vert \bm{D}^{t}\right\Vert ^{2}\left\Vert \bm{F}^{\star}\right\Vert \left\Vert \bm{F}^{\star}\right\Vert _{\mathrm{F}},
\end{align*}
which utilizes \eqref{subeq:aux-5}. Plugging this back into \eqref{eq:A}
yields 
\begin{align*}
\left\Vert \bm{A}^{t+1}\right\Vert _{\mathrm{F}} & \lesssim\frac{\eta}{\sigma_{\min}}\left\Vert \bm{D}^{t}\right\Vert ^{2}\left\Vert \bm{F}^{\star}\right\Vert \left\Vert \bm{F}^{\star}\right\Vert _{\mathrm{F}}\lesssim\eta\sqrt{r}\kappa\left\Vert \bm{D}^{t}\right\Vert ^{2}\\
 & \lesssim\eta\mu\kappa^{4}r^{3.5}\sigma_{\max}^{2}\left(\frac{\sigma}{\sigma_{\min}}\sqrt{\frac{n}{p}}+\frac{\left\Vert \bm{M}^{\star}\right\Vert _{\infty}}{\sigma_{\min}}\sqrt{\frac{n}{p}}\right)^{2}\log^{2}n.
\end{align*}

\section{Proofs for spectral initialization}

Before embarking on the proof, let us introduce a few matrices as follows

\begin{align}
\bm{H}_{\bm{X}}^{0} & \coloneqq\mathsf{sgn}\left(\left(\bm{X}^{0}\right)^{\top}\bm{X}^{\star}\right)\qquad\text{and}\qquad\bm{H}_{\bm{Y}}^{0}\coloneqq\mathsf{sgn}\left(\left(\bm{Y}^{0}\right)^{\top}\bm{Y}^{\star}\right),\nonumber \\
\bm{Q}_{\bm{U}} & \coloneqq\mathsf{sgn}\left(\left(\bm{U}^{0}\right)^{\top}\bm{U}^{\star}\right)\qquad\text{and}\qquad\bm{Q}_{\bm{V}}\coloneqq\mathsf{sgn}\left(\left(\bm{V}^{0}\right)^{\top}\bm{V}^{\star}\right),\label{defn-Quv}\\
\bm{H}_{\bm{U}} & \coloneqq\left(\bm{U}^{0}\right)^{\top}\bm{U}^{\star}\qquad\text{and}\qquad\bm{H}_{\bm{V}}\coloneqq\left(\bm{V}^{0}\right)^{\top}\bm{V}^{\star}.\label{defn-Huv}
\end{align}
In addition, we give an useful lemma which will facilitate our proof.
\begin{lemma}\label{lem:1}Suppose the sample size obeys $n^{2}p\geq C\mu^{2}r^{2}\kappa^{2}n\log n$
for some sufficiently large constant $C>0$, the noise satisfies $\frac{\sigma}{\sigma_{\min}}\sqrt{\frac{n}{p}}\leq c,$for
some sufficiently small constant $c>0$. Then with probability at
least $1-O(n^{-10})$, one has

\begin{subequations}
\label{subeq:lem1} 
\begin{align}
\left\Vert \bm{M}^{0}-\bm{M}^{\star}\right\Vert  & \lesssim\sigma\sqrt{\frac{n}{p}}+\sqrt{\frac{n}{p}}\left\Vert \bm{M}^{\star}\right\Vert _{\infty},\label{subeq:lem1-1}\\
\max_{1\leq l\leq n}\left\Vert \bm{M}^{0,\left(l\right)}-\bm{M}^{\star}\right\Vert  & \lesssim\sigma\sqrt{\frac{n}{p}}+\sqrt{\frac{n}{p}}\left\Vert \bm{M}^{\star}\right\Vert _{\infty}\label{subeq:lem1-1l}
\end{align}
\begin{equation}
\max\left\{ \left\Vert \bm{U}\bm{Q}_{\bm{U}}-\bm{U}^{\star}\right\Vert ,\left\Vert \bm{V}\bm{Q}_{\bm{V}}-\bm{V}^{\star}\right\Vert \right\} \lesssim\frac{\sigma}{\sigma_{\min}}\sqrt{\frac{n}{p}}+\frac{\left\Vert \bm{M}^{\star}\right\Vert _{\infty}}{\sigma_{\min}}\sqrt{\frac{n}{p}},\label{subeq:lem1-2}
\end{equation}
\begin{equation}
\max\left\{ \left\Vert \bm{Q}_{\bm{U}}-\bm{H}_{\bm{U}}\right\Vert ,\left\Vert \bm{Q}_{\bm{V}}-\bm{H}_{\bm{V}}\right\Vert \right\} \lesssim\frac{n}{\sigma_{\min}^{2}}\frac{\left\Vert \bm{M}^{\star}\right\Vert _{\infty}^{2}+\sigma^{2}}{p},\label{subeq:lem1-3}
\end{equation}
\begin{equation}
\frac{1}{2}\leq\min\left\{ \left\Vert \bm{H}_{\bm{U}}\right\Vert ,\left\Vert \bm{H}_{\bm{V}}\right\Vert \right\} \leq\max\left\{ \left\Vert \bm{H}_{\bm{U}}\right\Vert ,\left\Vert \bm{H}_{\bm{V}}\right\Vert \right\} \leq2.\label{subeq:lem1-4}
\end{equation}
\end{subequations}

\end{lemma}

\subsection{Proof of Lemma \ref{lem:2}}

Let us introduce the symmetric versions of $\bm{M}^{\star}$ and $\bm{M}^{0}$,
denoted by $\widetilde{\bm{M}}^{\star}$ and $\widetilde{\bm{M}}^{0}$
\[
\widetilde{\bm{M}}^{\star}\coloneqq\left[\begin{array}{cc}
\bm{0} & \bm{M}^{\star}\\
\left(\bm{M}^{\star}\right)^{\top} & \bm{0}
\end{array}\right],\qquad\widetilde{\bm{M}}^{0}\coloneqq\left[\begin{array}{cc}
\bm{0} & \bm{M}^{0}\\
\left(\bm{M}^{0}\right)^{\top} & \bm{0}
\end{array}\right].
\]
Recall that the top-$r$ SVD of $\bm{M}^{0}$ is $\bm{U}^{0}\bm{\Sigma}^{0}\bm{V}^{0\top}$.
It follows that the top-$r$ SVD of $\widetilde{\bm{M}}^{0}$ would
be 
\[
\left(\frac{1}{\sqrt{2}}\left[\begin{array}{c}
\bm{U}^{0}\\
\bm{V}^{0}
\end{array}\right]\right)\bm{\Sigma}^{0}\left(\frac{1}{\sqrt{2}}\left[\begin{array}{c}
\bm{U}^{0}\\
\bm{V}^{0}
\end{array}\right]\right)^{\top}.
\]
Define 
\begin{align*}
\bm{Z}^{0} & \coloneqq\frac{1}{\sqrt{2}}\left[\begin{array}{c}
\bm{U}^{0}\\
\bm{V}^{0}
\end{array}\right],\qquad\bm{Z}^{\star}\coloneqq\frac{1}{\sqrt{2}}\left[\begin{array}{c}
\bm{U}^{\star}\\
\bm{V}^{\star}
\end{array}\right],\\
\bm{Q} & \coloneqq\arg\min_{\bm{R}\in\mathcal{O}^{r\times r}}\left\Vert \bm{Z}^{0}\bm{R}-\bm{Z}^{\star}\right\Vert _{\mathrm{F}}.
\end{align*}

To prove Lemma \ref{lem:2}, we start from an application of \citet[Lemma 45, 46, 47]{ma2017implicit}
on $\widetilde{\bm{M}}^{0}$ and obtain:

\begin{subequations}
\label{eq:aux} 
\begin{align}
\left\Vert \bm{H}^{0}-\bm{Q}\right\Vert  & \lesssim\frac{1}{\sigma_{\min}}\left\Vert \widetilde{\bm{M}}^{0}-\widetilde{\bm{M}}^{\star}\right\Vert ,\label{eq:aux-1}\\
\left\Vert \left(\bm{\Sigma}^{0}\right)^{1/2}\bm{Q}-\bm{Q}\left(\bm{\Sigma}^{\star}\right)^{1/2}\right\Vert  & \lesssim\frac{1}{\sqrt{\sigma_{\min}}}\left\Vert \widetilde{\bm{M}}^{0}-\widetilde{\bm{M}}^{\star}\right\Vert ,\label{eq:aux-2}\\
\left\Vert \bm{Z}^{0}\bm{Q}-\bm{Z}^{\star}\right\Vert  & \leq\frac{1}{\sigma_{\min}}\left\Vert \widetilde{\bm{M}}^{0}-\widetilde{\bm{M}}^{\star}\right\Vert .\label{eq:aux-3}
\end{align}
\end{subequations}
Then, we turn attention to the decomposition 
\begin{align}
\left[\begin{array}{c}
\bm{X}^{0}\bm{H}^{0}-\bm{X}^{\star}\\
\bm{Y}^{0}\bm{H}^{0}-\bm{Y}^{\star}
\end{array}\right] & =\left[\begin{array}{c}
\bm{U}^{0}\\
\bm{V}^{0}
\end{array}\right]\left(\bm{\Sigma}^{0}\right)^{1/2}\left(\bm{H}^{0}-\bm{Q}\right)+\left[\begin{array}{c}
\bm{U}^{0}\\
\bm{V}^{0}
\end{array}\right]\left[\left(\bm{\Sigma}^{0}\right)^{1/2}\bm{Q}-\bm{Q}\left(\bm{\Sigma}^{\star}\right)^{1/2}\right]\nonumber \\
 & \qquad+\left(\left[\begin{array}{c}
\bm{U}^{0}\\
\bm{V}^{0}
\end{array}\right]\bm{Q}-\left[\begin{array}{c}
\bm{U}^{\star}\\
\bm{V}^{\star}
\end{array}\right]\right)\left(\bm{\Sigma}^{\star}\right)^{1/2}.\label{eq:decomp}
\end{align}
Taking \eqref{eq:decomp} collectively with \eqref{eq:aux} reveals
that 
\begin{align*}
\left\Vert \left[\begin{array}{c}
\bm{X}^{0}\bm{H}^{0}-\bm{X}^{\star}\\
\bm{Y}^{0}\bm{H}^{0}-\bm{Y}^{\star}
\end{array}\right]\right\Vert  & \leq\left\Vert \left(\bm{\Sigma}^{0}\right)^{1/2}\right\Vert \left\Vert \bm{H}^{0}-\bm{Q}\right\Vert +\left\Vert \left(\bm{\Sigma}^{0}\right)^{1/2}\bm{Q}-\bm{Q}\left(\bm{\Sigma}^{\star}\right)^{1/2}\right\Vert \\
 & \qquad+\left\Vert \left[\begin{array}{c}
\bm{U}^{0}\\
\bm{V}^{0}
\end{array}\right]\bm{Q}-\left[\begin{array}{c}
\bm{U}^{\star}\\
\bm{V}^{\star}
\end{array}\right]\right\Vert \left\Vert \left(\bm{\Sigma}^{\star}\right)^{1/2}\right\Vert \\
 & \lesssim\sqrt{\sigma_{\max}}\left\Vert \bm{H}^{0}-\bm{Q}\right\Vert +\left\Vert \left(\bm{\Sigma}^{0}\right)^{1/2}\bm{Q}-\bm{Q}\left(\bm{\Sigma}^{\star}\right)^{1/2}\right\Vert +\sqrt{\sigma_{\max}}\left\Vert \bm{Z}^{0}\bm{Q}-\bm{Z}^{\star}\right\Vert \\
 & \lesssim\left(\frac{\sqrt{\sigma_{\max}}}{\sigma_{\min}}+\frac{1}{\sqrt{\sigma_{\min}}}+\frac{\sqrt{\sigma_{\max}}}{\sigma_{\min}}\right)\left\Vert \widetilde{\bm{M}}^{0}-\widetilde{\bm{M}}^{\star}\right\Vert \\
 & \lesssim\left(\frac{\sigma}{\sigma_{\min}}\sqrt{\frac{n}{p}}+\frac{\left\Vert \bm{M}^{\star}\right\Vert _{\infty}}{\sigma_{\min}}\sqrt{\frac{n}{p}}\right)\left\Vert \bm{X}^{\star}\right\Vert ,
\end{align*}
where the last line utilizes the immediate consequence of Lemma \ref{lem:1}
that 
\begin{equation}
\left\Vert \widetilde{\bm{M}}^{0}-\widetilde{\bm{M}}^{\star}\right\Vert =\left\Vert \bm{M}^{0}-\bm{M}^{\star}\right\Vert \lesssim\sigma\sqrt{\frac{n}{p}}+\sqrt{\frac{n}{p}}\left\Vert \bm{M}^{\star}\right\Vert _{\infty}.\label{eq:tilde-M}
\end{equation}

\subsection{Proof of Lemma \ref{lem:3}}

To start with, we prove an useful lemma.

\begin{lemma}\label{lem:3-5}Suppose $\frac{\sigma}{\sigma_{\min}}\sqrt{\frac{\kappa\mu rn\log n}{p}}\leq c$
for some small enough constant $c>0$ and $n^{2}p\geq C\kappa^{3}\mu^{2}r^{3}n\log n$
for some large enough constant $C>0$. Then with probability at least
$1-O(n^{-10})$, one has 
\begin{align*}
\left\Vert \bm{U}^{0}\bm{Q}_{\bm{U}}-\bm{U}^{\star}\right\Vert _{2,\infty} & \lesssim\left\Vert \bm{U}^{\star}\right\Vert _{2,\infty}\left[\kappa\left(\frac{\sigma}{\sigma_{\min}}\sqrt{\frac{n}{p}}+\frac{\left\Vert \bm{M}^{\star}\right\Vert _{\infty}}{\sigma_{\min}}\sqrt{\frac{n}{p}}\right)^{2}+\sqrt{\frac{1}{p}\left(\left\Vert \bm{M}^{\star}\right\Vert _{\infty}^{2}+\sigma^{2}\right)}\frac{\sqrt{r\log n}}{\sigma_{\min}}\right]\\
 & \qquad+\frac{\left\Vert \bm{M}^{\star}\right\Vert _{\infty}}{\sigma_{\min}}\sqrt{\frac{r\log n}{p}}+\frac{\sigma}{\sigma_{\min}}\sqrt{\frac{\mu r}{p}}\log n
\end{align*}

\end{lemma}

The decomposition \eqref{eq:decomp} together with \eqref{eq:aux},
\eqref{eq:tilde-M} and Lemma \ref{lem:3-5} yields

\begin{align*}
\left\Vert \bm{X}^{0}\bm{H}^{0}-\bm{X}^{\star}\right\Vert _{2,\infty} & \leq\left\Vert \bm{U}^{0}\right\Vert _{2,\infty}\left\{ \left\Vert \left(\bm{\Sigma}^{0}\right)^{1/2}\right\Vert \left\Vert \bm{H}^{0}-\bm{Q}_{\bm{U}}\right\Vert +\left\Vert \left(\bm{\Sigma}^{0}\right)^{1/2}\bm{Q}_{\bm{U}}-\bm{Q}_{\bm{U}}\left(\bm{\Sigma}^{\star}\right)^{1/2}\right\Vert \right\} \\
 & \qquad+\sqrt{\sigma_{\max}}\left\Vert \bm{U}^{0}\bm{Q}_{\bm{U}}-\bm{U}^{\star}\right\Vert _{2,\infty}\\
 & \lesssim\left\Vert \bm{U}^{\star}\right\Vert _{2,\infty}\left(\sqrt{\sigma_{\max}}\frac{1}{\sigma_{\min}}\left\Vert \widetilde{\bm{M}}^{0}-\widetilde{\bm{M}}^{\star}\right\Vert +\frac{1}{\sqrt{\sigma_{\min}}}\left\Vert \widetilde{\bm{M}}^{0}-\widetilde{\bm{M}}^{\star}\right\Vert \right)\\
 & \qquad+\sqrt{\sigma_{\max}}\left\Vert \bm{U}^{0}\bm{Q}_{\bm{U}}-\bm{U}^{\star}\right\Vert _{2,\infty}\\
 & \lesssim\left\Vert \bm{U}^{\star}\right\Vert _{2,\infty}\sqrt{\sigma_{\max}}\left(\frac{\sigma}{\sigma_{\min}}\sqrt{\frac{n}{p}}+\frac{\left\Vert \bm{M}^{\star}\right\Vert _{\infty}}{\sigma_{\min}}\sqrt{\frac{n}{p}}\right)\\
 & \qquad+\sqrt{\sigma_{\max}}\left(\frac{\left\Vert \bm{M}^{\star}\right\Vert _{\infty}}{\sigma_{\min}}\sqrt{\frac{r\log n}{p}}+\frac{\sigma}{\sigma_{\min}}\sqrt{\frac{\mu r}{p}}\log n\right).
\end{align*}

\subsubsection{Proof of Lemma \ref{lem:3-5}}

To prove this, we begin from a collection of useful lemmas.

\begin{lemma}\label{lem:3-2}Suppose that $n^{2}p\geq C\mu^{2}rn\log n$
for some sufficiently large constant $C>0$. Then with probability
at least $1-O(n^{-10})$, one has 
\begin{align*}
\left\Vert \bm{Q}_{\bm{U}}^{\top}\bm{\Sigma}^{0}\bm{Q}_{\bm{V}}-\bm{\Sigma}^{\star}\right\Vert  & \lesssim\sigma_{\max}\left(\frac{n}{\sigma_{\min}^{2}}\frac{\left\Vert \bm{M}^{\star}\right\Vert _{\infty}^{2}+\sigma^{2}}{p}\right)+\sqrt{\frac{1}{p}\left(\left\Vert \bm{M}^{\star}\right\Vert _{\infty}^{2}+\sigma^{2}\right)}\sqrt{r\log n}+n\frac{\sigma^{2}}{\tau},\\
\left\Vert \bm{H}_{\bm{U}}^{\top}\bm{\Sigma}^{0}\bm{H}_{\bm{V}}-\bm{\Sigma}^{\star}\right\Vert  & \lesssim\left(\sigma\sqrt{\frac{n}{p}}+\sqrt{\frac{n}{p}}\left\Vert \bm{M}^{\star}\right\Vert _{\infty}\right)^{3}\frac{1}{\sigma_{\min}^{2}}+\sqrt{\frac{1}{p}\left(\left\Vert \bm{M}^{\star}\right\Vert _{\infty}^{2}+\sigma^{2}\right)}\sqrt{r\log n}+n\frac{\sigma^{2}}{\tau}.
\end{align*}

\end{lemma}

\begin{lemma}\label{lem:3-3}Suppose the sample size obeys $n^{2}p\geq C\mu\kappa r\log n$
for some sufficiently large constant $C>0$, the noise satisfies $\frac{\sigma}{\sigma_{\min}}\sqrt{\frac{n}{p}}\leq\frac{c}{\log n}$
for some sufficiently small constant $c>0$. Then with probability
at least $1-O(n^{-10})$, one has 
\begin{align*}
\left\Vert \bm{U}^{0}\bm{\Sigma}^{0}\bm{H}_{\bm{V}}-\bm{M}^{0}\bm{V}^{\star}\right\Vert _{2,\infty} & \lesssim\frac{\log^{2}n}{\sigma_{\min}}\sqrt{\mu rn}\frac{\sigma^{2}+\left\Vert \bm{M}^{\star}\right\Vert _{\infty}^{2}}{p}+\frac{\tau+\left\Vert \bm{M}^{\star}\right\Vert _{\infty}}{p}\log n\left\Vert \bm{V}^{0}\bm{H}_{\bm{V}}-\bm{V}^{\star}\right\Vert _{2,\infty}\\
 & \qquad+\frac{n}{\sigma_{\min}}\frac{\sigma^{2}+\left\Vert \bm{M}^{\star}\right\Vert _{\infty}^{2}}{p}\left\Vert \bm{U}^{0}\bm{H}_{\bm{U}}-\bm{U}^{\star}\right\Vert _{2,\infty}\\
 & \qquad+\left\Vert \bm{U}^{\star}\right\Vert _{2,\infty}\sigma_{\max}\left(\frac{\sigma}{\sigma_{\min}}\sqrt{\frac{n}{p}}+\frac{\left\Vert \bm{M}^{\star}\right\Vert _{\infty}}{\sigma_{\min}}\sqrt{\frac{n}{p}}\right)^{2}.
\end{align*}
\end{lemma}

\begin{lemma}\label{lem:3-4}Suppose the sample size obeys $np\geq C\kappa^{2}\mu^{3}r^{2}\log^{3}n$
for some sufficiently large constant $C>0$, the noise satisfies $\frac{\sigma}{\sigma_{\min}}\sqrt{\frac{n}{p}}\leq\frac{c}{\sqrt{\mu\log^{3}n}}$
for some sufficiently small constant $c>0$. Then with probability
at least $1-O(n^{-10})$, one has 
\begin{align*}
\left\Vert \bm{U}^{0}\bm{H}_{\bm{U}}-\bm{U}^{\star}\right\Vert _{2,\infty} & \lesssim\frac{\left(\sigma+\left\Vert \bm{M}^{\star}\right\Vert _{\infty}\right)}{\sigma_{\min}}\sqrt{\frac{r\log n}{p}}+\left\Vert \bm{U}^{\star}\right\Vert _{2,\infty}\kappa\left(\frac{\sigma}{\sigma_{\min}}\sqrt{\frac{n}{p}}+\frac{\left\Vert \bm{M}^{\star}\right\Vert _{\infty}}{\sigma_{\min}}\sqrt{\frac{n}{p}}\right)^{2}\\
 & \qquad+\frac{\left(\sigma\sqrt{np}+\left\Vert \bm{M}^{\star}\right\Vert _{\infty}\right)\log n}{\sigma_{\min}p}\sqrt{\frac{\mu r}{n}}.
\end{align*}
\end{lemma}

We start from the decomposition that 
\begin{equation}
\bm{U}^{0}\bm{Q}_{\bm{U}}-\bm{U}^{\star}=\underbrace{\left(\bm{M}^{0}-\bm{M}^{\star}\right)\bm{V}^{\star}\left(\bm{\Sigma}^{\star}\right)^{-1}}_{\eqqcolon\beta_{1}}+\underbrace{\bm{\Delta}_{\bm{U}}\left(\bm{\Sigma}^{\star}\right)^{-1}}_{\eqqcolon\beta_{2}}+\underbrace{\bm{U}^{0}\bm{Q}_{\bm{U}}\bm{\Delta}_{\bm{\Sigma}}\left(\bm{\Sigma}^{\star}\right)^{-1}}_{\eqqcolon\beta_{3}},\label{eq:lem3-5-1}
\end{equation}

where 
\[
\bm{\Delta}_{\bm{U}}=\bm{U}^{0}\bm{\Sigma}^{0}\bm{Q}_{\bm{V}}-\bm{M}^{0}\bm{V}^{\star}\qquad\mathrm{and}\qquad\bm{\Delta}_{\bm{\Sigma}}=\bm{\Sigma}^{\star}-\bm{Q}_{\bm{U}}^{\top}\bm{\Sigma}^{0}\bm{Q}_{\bm{V}}.
\]

\begin{enumerate}
\item For the first term $\beta_{1}$, one has 
\begin{align*}
\left\Vert \left(\bm{M}^{0}-\bm{M}^{\star}\right)\bm{V}^{\star}\left(\bm{\Sigma}^{\star}\right)^{-1}\right\Vert _{2,\infty} & \leq\left\Vert \left(\bm{E}-\mathbb{E}\left[\bm{E}\right]\right)\bm{V}^{\star}\left(\bm{\Sigma}^{\star}\right)^{-1}\right\Vert _{2,\infty}+\left\Vert \mathbb{E}\left[\bm{E}\right]\bm{V}^{\star}\left(\bm{\Sigma}^{\star}\right)^{-1}\right\Vert _{2,\infty},
\end{align*}
where we denote 
\begin{equation}
\bm{E}=\bm{M}^{0}-\bm{M}^{\star}.\label{defn-E}
\end{equation}
Note that the entries of $\bm{E}$ are independent and 
\begin{align*}
E_{i,j} & =\left[\frac{1}{p}\delta_{i,j}\left(M_{i,j}^{\star}+\varepsilon_{i,j}\right)-M_{i,j}^{\star}\right]\ind_{\left|M_{i,j}^{0}\right|\leq\tau}+\left[\frac{\tau}{p}\delta_{i,j}\mathrm{sign}\left(M_{i,j}^{0}\right)-M_{i,j}^{\star}\right]\ind_{\left|M_{i,j}^{0}\right|>\tau}.
\end{align*}
It then follows that 
\begin{align}
\left|\mathbb{E}\left[E_{i,j}\right]\right| & =\left|\mathbb{E}\left[\varepsilon_{i,j}\ind_{\left|M_{i,j}^{0}\right|\leq\tau}\right]+\mathbb{E}\left[\left(\tau\mathrm{sign}\left(M_{i,j}^{0}\right)-M_{i,j}^{\star}\right)\ind_{\left|M_{i,j}^{0}\right|>\tau}\right]\right|\nonumber \\
 & \overset{\text{(i)}}{=}\left|-\mathbb{E}\left[\varepsilon_{i,j}\ind_{\left|M_{i,j}^{0}\right|>\tau}\right]+\mathbb{E}\left[\left(\tau\mathrm{sign}\left(M_{i,j}^{0}\right)-M_{i,j}^{\star}\right)\ind_{\left|M_{i,j}^{0}\right|>\tau}\right]\right|\nonumber \\
 & \leq\left|\mathbb{E}\left[\varepsilon_{i,j}\ind_{\left|M_{i,j}^{0}\right|>\tau}\right]\right|+\mathbb{E}\left[\left|\tau\mathrm{sign}\left(M_{i,j}^{0}\right)-M_{i,j}^{\star}\right|\ind_{\left|M_{i,j}^{0}\right|>\tau}\right]\nonumber \\
 & \leq\sqrt{\mathbb{E}\left[\varepsilon_{i,j}^{2}\right]\mathbb{E}\left[\ind_{\left|M_{i,j}^{0}\right|>\tau}\right]}+\left(\tau+\left\Vert \bm{M}^{\star}\right\Vert _{\infty}\right)\mathbb{E}\left[\ind_{\left|M_{i,j}^{0}\right|>\tau}\right]\nonumber \\
 & \overset{\text{(ii)}}{\lesssim}\frac{\sigma^{2}}{\tau},\label{eq:mean-E}
\end{align}
where (i) comes from the fact that 
\[
0=\mathbb{E}\left[\varepsilon_{i,j}\right]=\mathbb{E}\left[\varepsilon_{i,j}\ind_{\left|M_{i,j}^{0}\right|\leq\tau}\right]+\mathbb{E}\left[\varepsilon_{i,j}\ind_{\left|M_{i,j}^{0}\right|>\tau}\right],
\]
and (ii) is due to an application of Markov inequality, 
\begin{equation}
\mathbb{E}\left[\ind_{\left|M_{i,j}^{0}\right|>\tau}\right]\le\mathbb{E}\left[\ind_{\left|\varepsilon_{i,j}\right|>\tau-\left\Vert \bm{M}^{\star}\right\Vert _{\infty}}\right]\leq\mathbb{E}\left[\ind_{\left|\varepsilon_{i,j}\right|>\tau/2}\right]\le\frac{\sigma^{2}}{\left(\tau/2\right)^{2}}.\label{eq:initial-ind}
\end{equation}
Furthermore, one has 
\begin{align}
	\mathbb{V}\left[E_{i,j}\right] & =\mathbb{V}\left[\frac{1}{p}\delta_{i,j}\left(\left(M_{i,j}^{\star}+\varepsilon_{i,j}\right)\ind_{\left|M_{i,j}^{0}\right|\leq\tau}+\tau\mathrm{sign}\left(M_{i,j}^{0}\right)\ind_{\left|M_{i,j}^{0}\right|>\tau}\right)\right]\nonumber \\
	& \leq\frac{1}{p}\mathbb{E}\left[\left(\left(M_{i,j}^{\star}+\varepsilon_{i,j}\right)\ind_{\left|M_{i,j}^{0}\right|\leq\tau}+\tau\mathrm{sign}\left(M_{i,j}^{0}\right)\ind_{\left|M_{i,j}^{0}\right|>\tau}\right)^{2}\right]\nonumber \\
	& \overset{\text{(i)}}{\leq}\frac{2}{p}\mathbb{E}\left[\left(M_{i,j}^{\star}+\varepsilon_{i,j}\right)^{2}\ind_{\left|M_{i,j}^{0}\right|\leq\tau}+\tau^{2}\ind_{\left|M_{i,j}^{0}\right|>\tau}\right]\nonumber \\
	& \leq\frac{2}{p}\left(\mathbb{E}\left[\left(M_{i,j}^{\star}+\varepsilon_{i,j}\right)^{2}\right]+\tau^{2}\mathbb{E}\left[\ind_{\left|M_{i,j}^{0}\right|>\tau}\right]\right)\nonumber \\
	& \overset{\text{(ii)}}{\leq}\frac{6}{p}\left(\left\Vert \bm{M}^{\star}\right\Vert _{\infty}^{2}+\sigma^{2}\right)\eqqcolon\widetilde{\sigma}^{2},\label{eq:defn-sigmatilde}
\end{align}
where (i) comes from the elementary fact that $(a+b)^{2}\leq2(a^{2}+b^{2})$
and (ii) makes use of \eqref{eq:initial-ind}. Additionally, we have
a simple upper bound that 
\begin{equation}
	B\coloneqq\max_{i,j}\left|E_{i,j}-\mathbb{E}\left[E_{i,j}\right]\right|\leq\frac{\tau+2\left\Vert \bm{M}^{\star}\right\Vert _{\infty}}{p}.\label{eq:defn-B}\end{equation}
Therefore, Lemma \ref{lem:bernstein} gives rise to 
\begin{align*}
\left\Vert \left(\bm{E}-\mathbb{E}\left[\bm{E}\right]\right)\bm{V}^{\star}\left(\bm{\Sigma}^{\star}\right)^{-1}\right\Vert _{2,\infty} & \lesssim\widetilde{\sigma}\left\Vert \bm{V}^{\star}\left(\bm{\Sigma}^{\star}\right)^{-1}\right\Vert _{\mathrm{F}}\sqrt{\log n}+B\left\Vert \bm{V}^{\star}\left(\bm{\Sigma}^{\star}\right)^{-1}\right\Vert _{2.\infty}\log n\\
 & \lesssim\frac{\widetilde{\sigma}\sqrt{r}}{\sigma_{\min}}\sqrt{\log n}+\frac{B}{\sigma_{\min}}\left\Vert \bm{V}^{\star}\right\Vert _{2,\infty}\log n.
\end{align*}
Regarding $\Vert\mathbb{E}[\bm{E}]\bm{V}^{\star}(\bm{\Sigma}^{\star})^{-1}\Vert_{2,\infty}$,
one has 
\[
\left\Vert \mathbb{E}\left[\bm{E}\right]\bm{V}^{\star}\left(\bm{\Sigma}^{\star}\right)^{-1}\right\Vert _{2,\infty}\leq\left\Vert \mathbb{E}\left[\bm{E}\right]\right\Vert _{2,\infty}\left\Vert \bm{V}^{\star}\left(\bm{\Sigma}^{\star}\right)^{-1}\right\Vert \leq\frac{\sqrt{n}}{\sigma_{\min}}\max_{i,j}\left|\mathbb{E}\left[E_{i,j}\right]\right|\lesssim\frac{1}{\sigma_{\min}}\frac{\sigma^{2}\sqrt{n}}{\tau},
\]
where the last inequality follows from \eqref{eq:mean-E}. 
\item Next, we turning attention to $\beta_{2}$, which can be further decomposed
as 
\begin{align*}
\left\Vert \bm{\Delta}_{\bm{U}}\right\Vert _{2,\infty} & \leq\left\Vert \bm{U}^{0}\bm{\Sigma}^{0}\bm{H}_{\bm{V}}-\bm{M}^{0}\bm{V}^{\star}\right\Vert _{2,\infty}+\left\Vert \bm{U}^{0}\bm{\Sigma}^{0}\left(\bm{H}_{\bm{V}}-\bm{Q}_{\bm{V}}\right)\right\Vert _{2,\infty}\\
 & \leq\underbrace{\left\Vert \bm{U}^{0}\bm{\Sigma}^{0}\bm{H}_{\bm{V}}-\bm{M}^{0}\bm{V}^{\star}\right\Vert _{2,\infty}}_{\eqqcolon\gamma_{1}}+\underbrace{\left\Vert \bm{U}^{0}\right\Vert _{2,\infty}\left\Vert \bm{\Sigma}^{0}\right\Vert \left\Vert \bm{H}_{\bm{V}}-\bm{Q}_{\bm{V}}\right\Vert }_{\eqqcolon\gamma_{2}}.
\end{align*}
Combining Lemma \ref{lem:3-3} and Lemma \ref{lem:3-4} yields the
bound of $\gamma_{1}$ 
\begin{align}
\left\Vert \bm{U}^{0}\bm{\Sigma}^{0}\bm{H}_{\bm{V}}-\bm{M}^{0}\bm{V}^{\star}\right\Vert _{2,\infty} & \lesssim\frac{\log^{2}n}{\sigma_{\min}}\sqrt{\mu rn}\frac{\sigma^{2}+\left\Vert \bm{M}^{\star}\right\Vert _{\infty}^{2}}{p}\nonumber \\
 & \qquad+\left\Vert \bm{U}^{\star}\right\Vert _{2,\infty}\sigma_{\max}\left(\frac{\sigma}{\sigma_{\min}}\sqrt{\frac{n}{p}}+\frac{\left\Vert \bm{M}^{\star}\right\Vert _{\infty}}{\sigma_{\min}}\sqrt{\frac{n}{p}}\right)^{2}.\label{eq:initial-gamma1}
\end{align}
Regarding $\gamma_{2}$, one has 
\begin{align}
\left\Vert \bm{U}^{0}\right\Vert _{2,\infty} & \leq\left\Vert \bm{U}^{0}\bm{H}_{\bm{U}}-\bm{U}^{\star}\right\Vert _{2,\infty}+\left\Vert \bm{U}^{\star}\right\Vert _{2,\infty}\nonumber \\
 & \leq\widetilde{C}\frac{\left(\sigma+\left\Vert \bm{M}^{\star}\right\Vert _{\infty}\right)}{\sigma_{\min}}\sqrt{\frac{r\log n}{p}}+\widetilde{C}\left\Vert \bm{U}^{\star}\right\Vert _{2,\infty}\kappa\left(\frac{\sigma}{\sigma_{\min}}\sqrt{\frac{n}{p}}+\frac{\left\Vert \bm{M}^{\star}\right\Vert _{\infty}}{\sigma_{\min}}\sqrt{\frac{n}{p}}\right)^{2}\nonumber \\
 & \qquad+\widetilde{C}\frac{\left(\sigma\sqrt{np}+\left\Vert \bm{M}^{\star}\right\Vert _{\infty}\right)\log n}{\sigma_{\min}p}\sqrt{\frac{\mu r}{n}}+\left\Vert \bm{U}^{\star}\right\Vert _{2,\infty}\nonumber \\
 & \leq2\left\Vert \bm{U}^{\star}\right\Vert _{2,\infty},\label{eq:initial-gamma2-2}
\end{align}
as long as $\frac{\sigma}{\sigma_{\min}}\sqrt{\frac{\kappa\mu rn\log n}{p}}\ll1$
and $n^{2}p\gg\kappa^{3}\mu^{2}r^{3}n\log n$. Furthermore, we have
\begin{align}
\left\Vert \bm{\Sigma}^{0}\right\Vert  & =\left\Vert \bm{Q}_{\bm{U}}^{\top}\bm{\Sigma}^{0}\bm{Q}_{\bm{V}}-\bm{\Sigma}^{\star}\right\Vert \leq\left\Vert \bm{Q}_{\bm{U}}^{\top}\bm{\Sigma}^{0}\bm{Q}_{\bm{V}}-\bm{\Sigma}^{\star}\right\Vert +\left\Vert \bm{\Sigma}^{\star}\right\Vert \nonumber \\
 & \leq\widetilde{C}\sigma_{\max}\left(\frac{n}{\sigma_{\min}^{2}}\frac{\left\Vert \bm{M}^{\star}\right\Vert _{\infty}^{2}+\sigma^{2}}{p}\right)+\widetilde{C}\sqrt{\frac{1}{p}\left(\left\Vert \bm{M}^{\star}\right\Vert _{\infty}^{2}+\sigma^{2}\right)}\sqrt{r\log n}+\widetilde{C}n\frac{\sigma^{2}}{\tau}+\left\Vert \bm{\Sigma}^{\star}\right\Vert \nonumber \\
 & \leq2\left\Vert \bm{\Sigma}^{\star}\right\Vert ,\label{eq:initial-gamma2-1}
\end{align}
provided $\frac{\sigma}{\sigma_{\min}}\sqrt{\frac{n}{p}}\ll1$ and
$n^{2}p\gg\kappa^{2}\mu^{2}r^{2}n\log n$. Taking \eqref{eq:initial-gamma2-2},
\eqref{eq:initial-gamma2-1} and Lemma \ref{lem:1} collectively gives
\begin{equation}
\gamma_{2}\lesssim\sigma_{\max}\left\Vert \bm{U}^{\star}\right\Vert _{2,\infty}\frac{n}{\sigma_{\min}^{2}}\frac{\left\Vert \bm{M}^{\star}\right\Vert _{\infty}^{2}+\sigma^{2}}{p}.\label{eq:initial-gamma2}
\end{equation}
Consequently, \eqref{eq:initial-gamma2} combined with \eqref{eq:initial-gamma1}
reveals that 
\begin{align*}
\left\Vert \bm{\Delta}_{\bm{U}}\right\Vert _{2,\infty} & \lesssim\frac{\log^{2}n}{\sigma_{\min}}\sqrt{\mu rn}\frac{\sigma^{2}+\left\Vert \bm{M}^{\star}\right\Vert _{\infty}^{2}}{p}+\left\Vert \bm{U}^{\star}\right\Vert _{2,\infty}\sigma_{\max}\left(\frac{\sigma}{\sigma_{\min}}\sqrt{\frac{n}{p}}+\frac{\left\Vert \bm{M}^{\star}\right\Vert _{\infty}}{\sigma_{\min}}\sqrt{\frac{n}{p}}\right)^{2}\\
 & \qquad+\sigma_{\max}\left\Vert \bm{U}^{\star}\right\Vert _{2,\infty}\frac{n}{\sigma_{\min}^{2}}\frac{\left\Vert \bm{M}^{\star}\right\Vert _{\infty}^{2}+\sigma^{2}}{p}\\
 & \lesssim\frac{\log^{2}n}{\sigma_{\min}}\sqrt{\mu rn}\frac{\sigma^{2}+\left\Vert \bm{M}^{\star}\right\Vert _{\infty}^{2}}{p}+\left\Vert \bm{U}^{\star}\right\Vert _{2,\infty}\sigma_{\max}\left(\frac{\sigma}{\sigma_{\min}}\sqrt{\frac{n}{p}}+\frac{\left\Vert \bm{M}^{\star}\right\Vert _{\infty}}{\sigma_{\min}}\sqrt{\frac{n}{p}}\right)^{2}.
\end{align*}
\item The last term $\beta_{3}$ can be controlled by utilizing \eqref{eq:initial-gamma2-2}
and Lemma \ref{lem:3-2}: 
\begin{align*}
 & \left\Vert \bm{U}^{0}\bm{Q}_{\bm{U}}\bm{\Delta}_{\bm{\Sigma}}\left(\bm{\Sigma}^{\star}\right)^{-1}\right\Vert _{2,\infty}\\
 & \leq\left\Vert \bm{U}^{0}\right\Vert _{2,\infty}\left\Vert \bm{Q}_{\bm{U}}\right\Vert \left\Vert \bm{\Delta}_{\bm{\Sigma}}\right\Vert \left\Vert \left(\bm{\Sigma}^{\star}\right)^{-1}\right\Vert \\
 & \lesssim\left\Vert \bm{U}^{\star}\right\Vert _{2,\infty}\frac{1}{\sigma_{\min}}\left[\sigma_{\max}\left(\frac{n}{\sigma_{\min}^{2}}\frac{\left\Vert \bm{M}^{\star}\right\Vert _{\infty}^{2}+\sigma^{2}}{p}\right)+\sqrt{\frac{1}{p}\left(\left\Vert \bm{M}^{\star}\right\Vert _{\infty}^{2}+\sigma^{2}\right)}\sqrt{r\log n}+n\frac{\sigma^{2}}{\tau}\right].
\end{align*}
\end{enumerate}
Plugging all these bounds into the decomposition \eqref{eq:lem3-5-1},
we arrive at 
\begin{align*}
& \left\Vert \bm{U}^{0}\bm{Q}_{\bm{U}}-\bm{U}^{\star}\right\Vert _{2,\infty}  \\
& \leq\left\Vert \left(\bm{M}^{0}-\bm{M}^{\star}\right)\bm{V}^{\star}\left(\bm{\Sigma}^{\star}\right)^{-1}\right\Vert _{2,\infty}+\left\Vert \bm{\Delta}_{\bm{U}}\left(\bm{\Sigma}^{\star}\right)^{-1}\right\Vert _{2,\infty}+\left\Vert \bm{U}^{0}\bm{Q}_{\bm{U}}\bm{\Delta}_{\bm{\Sigma}}\left(\bm{\Sigma}^{\star}\right)^{-1}\right\Vert _{2,\infty}\\
 & \lesssim\frac{\widetilde{\sigma}\sqrt{r}}{\sigma_{\min}}\sqrt{\log n}+\frac{B}{\sigma_{\min}}\left\Vert \bm{V}^{\star}\right\Vert _{2,\infty}\log n+\frac{1}{\sigma_{\min}}\frac{\sigma^{2}\sqrt{n}}{\tau}\\
 & \qquad+\frac{1}{\sigma_{\min}}\left[\frac{\log^{2}n}{\sigma_{\min}}\sqrt{\mu rn}\frac{\sigma^{2}+\left\Vert \bm{M}^{\star}\right\Vert _{\infty}^{2}}{p}+\left\Vert \bm{U}^{\star}\right\Vert _{2,\infty}\sigma_{\max}\left(\frac{\sigma}{\sigma_{\min}}\sqrt{\frac{n}{p}}+\frac{\left\Vert \bm{M}^{\star}\right\Vert _{\infty}}{\sigma_{\min}}\sqrt{\frac{n}{p}}\right)^{2}\right]\\
 & \qquad+\left\Vert \bm{U}^{\star}\right\Vert _{2,\infty}\frac{1}{\sigma_{\min}}\left[\sigma_{\max}\left(\frac{n}{\sigma_{\min}^{2}}\frac{\left\Vert \bm{M}^{\star}\right\Vert _{\infty}^{2}+\sigma^{2}}{p}\right)+\sqrt{\frac{1}{p}\left(\left\Vert \bm{M}^{\star}\right\Vert _{\infty}^{2}+\sigma^{2}\right)}\sqrt{r\log n}+n\frac{\sigma^{2}}{\tau}\right]\\
 & \lesssim\frac{\left\Vert \bm{M}^{\star}\right\Vert _{\infty}}{\sigma_{\min}}\sqrt{\frac{\mu r\log n}{p}}+\frac{\sigma}{\sigma_{\min}}\sqrt{\frac{\mu r}{p}}\log n\\
 & \qquad+\left\Vert \bm{U}^{\star}\right\Vert _{2,\infty}\left[\kappa\left(\frac{\sigma}{\sigma_{\min}}\sqrt{\frac{n}{p}}+\frac{\left\Vert \bm{M}^{\star}\right\Vert _{\infty}}{\sigma_{\min}}\sqrt{\frac{n}{p}}\right)^{2}+\sqrt{\frac{1}{p}\left(\left\Vert \bm{M}^{\star}\right\Vert _{\infty}^{2}+\sigma^{2}\right)}\frac{\sqrt{r\log n}}{\sigma_{\min}}\right],
\end{align*}
whereas the last inequality holds as long as $np\gg\mu^{2}\kappa^{2}r^{2}\log^{2}n$.

\subsubsection{Proof of Lemma \ref{lem:3-2}}

To begin with, we can decompose $\Vert\bm{Q}_{\bm{U}}^{\top}\bm{\Sigma}^{0}\bm{Q}_{\bm{V}}-\bm{\Sigma}^{\star}\Vert$
as 
\begin{align}
\left\Vert \bm{Q}_{\bm{U}}^{\top}\bm{\Sigma}^{0}\bm{Q}_{\bm{V}}-\bm{\Sigma}^{\star}\right\Vert &\leq\underbrace{\left\Vert \bm{Q}_{\bm{U}}^{\top}\bm{\Sigma}^{0}\bm{Q}_{\bm{V}}-\bm{H}_{\bm{U}}^{\top}\bm{\Sigma}^{0}\bm{H}_{\bm{V}}\right\Vert }_{\eqqcolon\alpha_{1}}+\underbrace{\left\Vert \bm{H}_{\bm{U}}^{\top}\bm{\Sigma}^{0}\bm{H}_{\bm{V}}-\bm{U}^{\star\top}\bm{M}^{0}\bm{V}^{\star}\right\Vert }_{\eqqcolon\alpha_{2}}\nonumber\\
&\qquad+\underbrace{\left\Vert \bm{U}^{\star\top}\bm{M}^{0}\bm{V}^{\star}-\bm{\Sigma}^{\star}\right\Vert }_{\eqqcolon\alpha_{3}}.
\end{align}
In the sequel, we shall establish the bounds on $\alpha_{1}$, $\alpha_{2}$,
$\alpha_{3}$ separately. 
\begin{enumerate}
\item Regarding the first term $\alpha_{1}$, one has 
\begin{align*}
\alpha_{1} & \leq\left\Vert \left(\bm{Q}_{\bm{U}}-\bm{H}_{\bm{U}}\right)^{\top}\bm{\Sigma}^{0}\bm{Q}_{\bm{V}}\right\Vert +\left\Vert \bm{H}_{\bm{U}}^{\top}\bm{\Sigma}^{0}\left(\bm{Q}_{\bm{V}}-\bm{H}_{\bm{V}}\right)\right\Vert \\
 & \leq\left\Vert \bm{Q}_{\bm{U}}-\bm{H}_{\bm{U}}\right\Vert \left\Vert \bm{\Sigma}^{0}\right\Vert \left\Vert \bm{H}_{\bm{V}}\right\Vert +\left\Vert \bm{H}_{\bm{U}}\right\Vert \left\Vert \bm{\Sigma}^{0}\right\Vert \left\Vert \bm{Q}_{\bm{V}}-\bm{H}_{\bm{V}}\right\Vert \\
 & \lesssim\sigma_{\max}\left(\frac{n}{\sigma_{\min}^{2}}\frac{\left\Vert \bm{M}^{\star}\right\Vert _{\infty}^{2}+\sigma^{2}}{p}\right),
\end{align*}
where the last line follows from Lemma \ref{lem:1} and \eqref{eq:initial-gamma2-1}. 
\item Next, since $\alpha_{2}$ is exactly the same as the term in \citet[Section C.3.2]{yan2021inference},
we can invoke the results therein to obtain 
\[
\alpha_{2}\lesssim\left\Vert \bm{U}^{0}\bm{Q}_{\bm{U}}-\bm{U}^{\star}\right\Vert \left\Vert \bm{E}\right\Vert \left\Vert \bm{V}^{0}\bm{Q}_{\bm{V}}-\bm{V}^{\star}\right\Vert \lesssim\left(\sigma\sqrt{\frac{n}{p}}+\sqrt{\frac{n}{p}}\left\Vert \bm{M}^{\star}\right\Vert _{\infty}\right)^{3}\frac{1}{\sigma_{\min}^{2}},
\]
where the last line arises from Lemma \ref{lem:1}. 
\item Turning to the last term $\alpha_{3}$, one has 
\begin{align*}
\alpha_{3} & =\left\Vert \bm{U}^{\star\top}\bm{M}^{0}\bm{V}^{\star}-\bm{\Sigma}^{\star}\right\Vert =\left\Vert \bm{U}^{\star\top}\left(\bm{M}^{0}-\bm{M}^{\star}\right)\bm{V}^{\star}\right\Vert \\
 & \leq\underbrace{\left\Vert \bm{U}^{\star\top}\left(\bm{M}^{0}-\bm{M}^{\star}-\mathbb{E}\left[\bm{M}^{0}-\bm{M}^{\star}\right]\right)\bm{V}^{\star}\right\Vert }_{\eqqcolon\alpha_{31}}+\underbrace{\left\Vert \bm{U}^{\star\top}\mathbb{E}\left[\bm{M}^{0}-\bm{M}^{\star}\right]\bm{V}^{\star}\right\Vert }_{\eqqcolon\alpha_{32}}
\end{align*}
The bound of $\alpha_{31}$ can be derived in the same way as \citet[Equation (C.11)]{yan2021inference}
\begin{align*}
\alpha_{31} & \lesssim\widetilde{\sigma}\sqrt{r\log n}+\frac{B\mu r\log n}{n}\lesssim\sqrt{\frac{1}{p}\left(\left\Vert \bm{M}^{\star}\right\Vert _{\infty}^{2}+\sigma^{2}\right)}\sqrt{r\log n}+\frac{\left(\tau+\left\Vert \bm{M}^{\star}\right\Vert _{\infty}\right)\mu r\log n}{np}\\
 & \lesssim\sqrt{\frac{1}{p}\left(\left\Vert \bm{M}^{\star}\right\Vert _{\infty}^{2}+\sigma^{2}\right)}\sqrt{r\log n},
\end{align*}
where $\widetilde{\sigma}$ and $B$ are defined in \eqref{eq:defn-sigmatilde}
and \eqref{eq:defn-B}, and the last inequality holds as long as $n\gg\mu^{2}r\log n$
and $np\gg1$. Regarding $\alpha_{32}$, one has 
\[
\alpha_{32}\leq\left\Vert \bm{U}^{\star}\right\Vert \left\Vert \bm{V}^{\star}\right\Vert \left\Vert \mathbb{E}\left[\bm{M}^{0}-\bm{M}^{\star}\right]\right\Vert \leq\left\Vert \mathbb{E}\left[\bm{M}^{0}-\bm{M}^{\star}\right]\right\Vert _{\mathrm{F}}\lesssim nB.
\]
\end{enumerate}
Finally, taking all the results above together, one has 
\begin{align*}
\left\Vert \bm{\Delta}_{\bm{\Sigma}}\right\Vert  & \lesssim\sigma_{\max}\left(\frac{n}{\sigma_{\min}^{2}}\frac{\left\Vert \bm{M}^{\star}\right\Vert _{\infty}^{2}+\sigma^{2}}{p}\right)+\sqrt{\frac{1}{p}\left(\left\Vert \bm{M}^{\star}\right\Vert _{\infty}^{2}+\sigma^{2}\right)}\sqrt{r\log n}+n\frac{\sigma^{2}}{\tau}.
\end{align*}
Additionally, we have 
\begin{align*}
\left\Vert \bm{H}_{\bm{U}}^{\top}\bm{\Sigma}^{0}\bm{H}_{\bm{V}}-\bm{\Sigma}^{\star}\right\Vert  & \leq\left\Vert \bm{H}_{\bm{U}}^{\top}\bm{\Sigma}^{0}\bm{H}_{\bm{V}}-\bm{U}^{\star\top}\bm{M}^{0}\bm{V}^{\star}\right\Vert +\left\Vert \bm{U}^{\star\top}\bm{M}^{0}\bm{V}^{\star}-\bm{\Sigma}^{\star}\right\Vert =\alpha_{2}+\alpha_{3}\\
 & \leq\sigma_{\max}\left(\frac{n}{\sigma_{\min}^{2}}\frac{\left\Vert \bm{M}^{\star}\right\Vert _{\infty}^{2}+\sigma^{2}}{p}\right)+\sqrt{\frac{1}{p}\left(\left\Vert \bm{M}^{\star}\right\Vert _{\infty}^{2}+\sigma^{2}\right)}\sqrt{r\log n}+n\frac{\sigma^{2}}{\tau}.
\end{align*}

\subsubsection{Proof of Lemma \ref{lem:3-3}}

Applying the triangle inequality enables us to obtain

\begin{align}
 & \left\Vert \bm{U}^{0}\bm{\Sigma}^{0}\bm{H}_{\bm{V}}-\bm{M}^{0}\bm{V}^{\star}\right\Vert _{2,\infty}\nonumber \\
 & \leq\left\Vert \left(\bm{M}^{0}-\bm{M}^{\star}\right)\left(\bm{V}^{0}\bm{H}_{\bm{V}}-\bm{V}^{\star}\right)\right\Vert _{2,\infty}+\left\Vert \bm{M}^{\star}\left(\bm{V}^{0}\bm{H}_{\bm{V}}-\bm{V}^{\star}\right)\right\Vert _{2,\infty}\nonumber \\
 & \leq\underbrace{\left\Vert \bm{M}^{\star}\left(\bm{V}^{0}\bm{H}_{\bm{V}}-\bm{V}^{\star}\right)\right\Vert _{2,\infty}}_{\eqqcolon\beta_{1}}+\underbrace{\left\Vert \left(\bm{M}^{0}-\bm{M}^{\star}-\mathbb{E}\left[\bm{M}^{0}-\bm{M}^{\star}\right]\right)\left(\bm{V}^{0}\bm{H}_{\bm{V}}-\bm{V}^{\star}\right)\right\Vert _{2,\infty}}_{\eqqcolon\beta_{2}}\nonumber \\
 & \qquad+\underbrace{\left\Vert \mathbb{E}\left[\bm{M}^{0}-\bm{M}^{\star}\right]\left(\bm{V}^{0}\bm{H}_{\bm{V}}-\bm{V}^{\star}\right)\right\Vert _{2,\infty}}_{\eqqcolon\beta_{3}}.\label{eq:lem3-3-decompose}
\end{align}
In what follows, we shall control these three terms separately. 
\begin{enumerate}
\item We start from $\beta_{2}$. In view of the leave-one-out sequences
defined in Algorithm \ref{alg:gd-rmc-loo}, one has 
\begin{align*}
 & \left\Vert \left(\bm{M}^{0}-\bm{M}^{\star}-\mathbb{E}\left[\bm{M}^{0}-\bm{M}^{\star}\right]\right)_{l,\cdot}\left(\bm{V}^{0}\bm{H}_{\bm{V}}-\bm{V}^{\star}\right)\right\Vert _{2}\\
 & \leq\underbrace{\left\Vert \left(\bm{M}^{0}-\bm{M}^{\star}-\mathbb{E}\left[\bm{M}^{0}-\bm{M}^{\star}\right]\right)_{l,\cdot}\left(\bm{V}^{0,\left(l\right)}\bm{H}_{\bm{V}}^{\left(l\right)}-\bm{V}^{\star}\right)\right\Vert _{2}}_{\eqqcolon\alpha_{1}}\\
 & \qquad+\underbrace{\left\Vert \left(\bm{M}^{0}-\bm{M}^{\star}-\mathbb{E}\left[\bm{M}^{0}-\bm{M}^{\star}\right]\right)_{l,\cdot}\left(\bm{V}^{0,\left(l\right)}\bm{H}_{\bm{V}}^{\left(l\right)}-\bm{V}^{0}\bm{H}_{\bm{V}}\right)\right\Vert _{2}}_{\eqqcolon\alpha_{2}}.
\end{align*}
Recall the definitions of $B$ and $\widetilde{\sigma}$ in \eqref{eq:defn-sigmatilde}
and \eqref{eq:defn-B}. Conditional on $\bm{V}^{\left(l\right)}$,
invoking Lemma \ref{lem:bernstein} yields 
\begin{align}
\alpha_{1} & \lesssim\widetilde{\sigma}\sqrt{\log n}\left\Vert \bm{V}^{0,\left(l\right)}\bm{H}_{\bm{V}}^{\left(l\right)}-\bm{V}^{\star}\right\Vert _{\mathrm{F}}+B\log n\left\Vert \bm{V}^{0,\left(l\right)}\bm{H}_{\bm{V}}^{\left(l\right)}-\bm{V}^{\star}\right\Vert _{2,\infty}\nonumber \\
 & \lesssim\widetilde{\sigma}\sqrt{\log n}\left\Vert \bm{V}^{0}\bm{H}_{\bm{V}}-\bm{V}^{\star}\right\Vert _{\mathrm{F}}+B\log n\left\Vert \bm{V}^{0}\bm{H}_{\bm{V}}-\bm{V}^{\star}\right\Vert _{2,\infty}\nonumber \\
 & \qquad+\left(\widetilde{\sigma}\sqrt{\log n}+B\log n\right)\left\Vert \bm{V}^{0,\left(l\right)}\bm{H}_{\bm{V}}^{\left(l\right)}-\bm{V}^{0}\bm{H}_{\bm{V}}\right\Vert _{\mathrm{F}}.\label{eq:alpha1}
\end{align}
Regarding $\alpha_{2}$, we have 
\begin{align}
\alpha_{2} & \leq\left\Vert \left(\bm{M}^{0}-\bm{M}^{\star}-\mathbb{E}\left[\bm{M}^{0}-\bm{M}^{\star}\right]\right)_{l,\cdot}\right\Vert _{2}\left\Vert \bm{V}^{0,\left(l\right)}\bm{H}_{\bm{V}}^{\left(l\right)}-\bm{V}^{0}\bm{H}_{\bm{V}}\right\Vert _{\mathrm{F}}\nonumber \\
 & \lesssim\left(\widetilde{\sigma}\sqrt{n}+B\sqrt{\log n}\right)\left\Vert \bm{V}^{0,\left(l\right)}\bm{H}_{\bm{V}}^{\left(l\right)}-\bm{V}^{0}\bm{H}_{\bm{V}}\right\Vert _{\mathrm{F}},\label{eq:lem3-3-alpha2}
\end{align}
where the second line makes use of \citet[Theorem 3.4]{chen2021spectral}. 
To bound $\Vert\bm{V}^{0,\left(l\right)}\bm{H}_{\bm{V}}^{\left(l\right)}-\bm{V}^{0}\bm{H}_{\bm{V}}\Vert,$
one has 
\begin{align}
\left\Vert \bm{V}^{0,\left(l\right)}\bm{H}_{\bm{V}}^{\left(l\right)}-\bm{V}^{0}\bm{H}_{\bm{V}}\right\Vert _{\mathrm{F}}&=\left\Vert \left(\bm{V}^{0,\left(l\right)}\bm{V}^{0,\left(l\right)\top}-\bm{V}^{0}\bm{V}^{0\top}\right)\bm{V}^{\star}\right\Vert _{\mathrm{F}}\nonumber \\
&\leq\left\Vert \bm{V}^{0,\left(l\right)}\bm{V}^{0,\left(l\right)\top}-\bm{V}^{0}\bm{V}^{0\top}\right\Vert _{\mathrm{F}}.\label{eq:lem3-3-00}
\end{align}
Invoking Wedin's sin$\bm{\Theta}$ Theorem \citet[Theorem 2.9]{chen2021spectral}
yields 
\begin{align}
 & \max\left\{ \left\Vert \bm{U}^{0,\left(l\right)}\bm{U}^{0,\left(l\right)\top}-\bm{U}^{0}\bm{U}^{0\top}\right\Vert _{\mathrm{F}},\left\Vert \bm{V}^{0,\left(l\right)}\bm{V}^{0,\left(l\right)\top}-\bm{V}^{0}\bm{V}^{0\top}\right\Vert _{\mathrm{F}}\right\} \nonumber \\
 & \lesssim\frac{\max\left\{ \left\Vert \left(\bm{M}^{0,\left(l\right)}-\bm{M}^{0}\right)\bm{V}^{0,\left(l\right)}\right\Vert _{\mathrm{F}},\left\Vert \left(\bm{M}^{0,\left(l\right)}-\bm{M}^{0}\right)^{\top}\bm{U}^{0,\left(l\right)}\right\Vert _{\mathrm{F}}\right\} }{\sigma_{r}\left(\bm{M}^{0,\left(l\right)}\right)-\sigma_{r+1}\left(\bm{M}^{0}\right)-\left\Vert \bm{M}^{0,\left(l\right)}-\bm{M}^{0}\right\Vert }\nonumber \\
 & \lesssim\frac{\max\left\{ \left\Vert \left(\bm{M}^{0,\left(l\right)}-\bm{M}^{0}\right)\bm{V}^{0,\left(l\right)}\right\Vert _{\mathrm{F}},\left\Vert \left(\bm{M}^{0,\left(l\right)}-\bm{M}^{0}\right)^{\top}\bm{U}^{0,\left(l\right)}\right\Vert _{\mathrm{F}}\right\} }{\sigma_{r}\left(\bm{M}^{\star}\right)-\sigma_{r+1}\left(\bm{M}^{\star}\right)-\left\Vert \bm{M}^{0}-\bm{M}^{\star}\right\Vert -\left\Vert \bm{M}^{0,\left(l\right)}-\bm{M}^{\star}\right\Vert -\left\Vert \bm{M}^{0,\left(l\right)}-\bm{M}^{0}\right\Vert }\nonumber \\
 & \lesssim\frac{\max\left\{ \left\Vert \left(\bm{M}^{0,\left(l\right)}-\bm{M}^{0}\right)\bm{V}^{0,\left(l\right)}\right\Vert _{\mathrm{F}},\left\Vert \left(\bm{M}^{0,\left(l\right)}-\bm{M}^{0}\right)^{\top}\bm{U}^{0,\left(l\right)}\right\Vert _{\mathrm{F}}\right\} }{\sigma_{\min}},\label{eq:lem3-3-0}
\end{align}
where the last inequality utilizes Lemma \ref{lem:1}. Then we turn
attention to $\Vert(\bm{M}^{0,(l)}-\bm{M}^{0})\bm{V}^{0,(l)}\Vert_{\mathrm{F}}$
and $\Vert(\bm{M}^{0,(l)}-\bm{M}^{0})^{\top}\bm{U}^{0,(l)}\Vert_{\mathrm{F}}$.
In view of the definition of $\bm{M}^{0,(l)}$ (cf.~\eqref{eq:spectral-method-matrix-1}),
we can deduce that \eqref{eq:defn-sigmatilde} 
\begin{align}
&\left\Vert \left(\bm{M}^{0,\left(l\right)}-\bm{M}^{0}\right)\bm{V}^{0,\left(l\right)}\right\Vert _{\mathrm{F}} \nonumber \\
& \quad=\left\Vert \left(\frac{1}{2p}\mathcal{P}_{l}\left(\psi_{\tau}\left(\bm{M}^{0}\right)\right)-\mathcal{P}_{l}\left(\bm{M}^{\star}\right)\right)\bm{V}^{0,\left(l\right)}\right\Vert _{2}\nonumber \\
 &\quad \overset{\text{(i)}}{\leq}\left\Vert \bm{E}_{l,\cdot}\bm{V}^{0,\left(l\right)}\bm{H}_{\bm{V}}^{\left(l\right)}\right\Vert _{2}\nonumber \\
 &\quad \leq\left\Vert \left(\bm{E}-\mathbb{E}\left[\bm{E}\right]\right)_{l,\cdot}\bm{V}^{\star}\right\Vert _{2}+\left\Vert \left(\bm{E}-\mathbb{E}\left[\bm{E}\right]\right)_{l,\cdot}\left(\bm{V}^{0,\left(l\right)}\bm{H}_{\bm{V}}^{\left(l\right)}-\bm{V}^{\star}\right)\right\Vert _{2}\nonumber \\
 &\quad \qquad+\left\Vert \left(\mathbb{E}\left[\bm{E}\right]\right)_{l,\cdot}\bm{V}^{0,\left(l\right)}\bm{H}_{\bm{V}}^{\left(l\right)}\right\Vert _{2}\nonumber \\
 &\quad \overset{\text{(ii)}}{\lesssim}\widetilde{\sigma}\sqrt{\log n}\left\Vert \bm{V}^{\star}\right\Vert _{\mathrm{F}}+B\log n\left\Vert \bm{V}^{\star}\right\Vert _{2,\infty}+\alpha_{1}+\left(\frac{\sigma^{2}}{\tau}\right)\sqrt{n}\left\Vert \bm{V}^{\star}\right\Vert ,\label{eq:lem3-3-1}
\end{align}
where (i) is due to \eqref{subeq:lem1-4}; (ii) comes from \eqref{eq:mean-E},
Lemma \ref{lem:bernstein} and the fact that 
\[
\left\Vert \bm{V}^{0,\left(l\right)}\bm{H}_{\bm{V}}^{\left(l\right)}\right\Vert \lesssim\left\Vert \bm{V}^{0,\left(l\right)}\right\Vert \lesssim\left\Vert \bm{V}^{\star}\right\Vert .
\]
In terms of $\Vert(\bm{M}^{0,(l)}-\bm{M}^{0})^{\top}\bm{U}^{0,(l)}\Vert_{\mathrm{F}}$,
one has 
\begin{align}
\left\Vert \left(\bm{M}^{0,\left(l\right)}-\bm{M}^{0}\right)^{\top}\bm{U}^{0,\left(l\right)}\right\Vert _{\mathrm{F}} & =\left\Vert \bm{E}_{l,\cdot}^{\top}\bm{U}_{l,\cdot}^{0,\left(l\right)}\right\Vert _{\mathrm{F}}=\left\Vert \bm{E}_{l,\cdot}\right\Vert _{2}\left\Vert \bm{U}_{l,\cdot}^{0,\left(l\right)}\right\Vert _{2}\nonumber \\
 & \overset{\text{(i)}}{\leq}2\left(\left\Vert \left(\bm{E}-\mathbb{E}\left[\bm{E}\right]\right)_{l,\cdot}\right\Vert _{2}+\left\Vert \left(\mathbb{E}\left[\bm{E}\right]\right)_{l,\cdot}\right\Vert _{2}\right)\left\Vert \bm{U}_{l,\cdot}^{0,\left(l\right)}\bm{H}_{\bm{U}}^{\left(l\right)}\right\Vert _{2}\nonumber \\
 & \overset{\text{(ii)}}{\lesssim}2\left(\left\Vert \bm{E}-\mathbb{E}\left[\bm{E}\right]\right\Vert +\frac{\sigma^{2}}{\tau}\sqrt{n}\right)\left(\left\Vert \bm{U}^{\star}\right\Vert _{2,\infty}+\left\Vert \bm{U}^{0,\left(l\right)}\bm{H}_{\bm{U}}^{\left(l\right)}-\bm{U}^{\star}\right\Vert _{2,\infty}\right)\nonumber \\
 & \overset{\text{(iii)}}{\lesssim}2\left(\widetilde{\sigma}\sqrt{n}+\frac{\sigma^{2}}{\tau}\sqrt{n}\right)\left(\left\Vert \bm{U}^{\star}\right\Vert _{2,\infty}+\left\Vert \bm{U}^{0,\left(l\right)}\bm{H}_{\bm{U}}^{\left(l\right)}-\bm{U}^{\star}\right\Vert _{2,\infty}\right)\nonumber \\
 & \lesssim\widetilde{\sigma}\sqrt{n}\left(\left\Vert \bm{U}^{\star}\right\Vert _{2,\infty}+\left\Vert \bm{U}^{0,\left(l\right)}\bm{H}_{\bm{U}}^{\left(l\right)}-\bm{U}^{\star}\right\Vert _{2,\infty}\right),\label{eq:lem3-3-2}
\end{align}
where (i) arises from Lemma \ref{lem:1}; (ii) is due to \eqref{eq:mean-E};
(iii) follows from the standard matrix tail bounds \citep[Theorem 3.4]{chen2021spectral}.
Plugging \eqref{eq:lem3-3-1} and \eqref{eq:lem3-3-2} into \eqref{eq:lem3-3-0}
and \eqref{eq:lem3-3-00} gives 
\begin{align}
 & \max\left\{ \left\Vert \bm{U}^{0,\left(l\right)}\bm{H}_{\bm{U}}^{\left(l\right)}-\bm{U}^{0}\bm{H}_{\bm{U}}\right\Vert _{\mathrm{F}},\left\Vert \bm{V}^{0,\left(l\right)}\bm{H}_{\bm{V}}^{\left(l\right)}-\bm{V}^{0}\bm{H}_{\bm{V}}\right\Vert _{\mathrm{F}}\right\} \nonumber \\
 & \lesssim\frac{\widetilde{\sigma}}{\sigma_{\min}}\sqrt{\log n}\left\Vert \bm{V}^{\star}\right\Vert _{\mathrm{F}}+\frac{B}{\sigma_{\min}}\log n\left\Vert \bm{V}^{\star}\right\Vert _{2,\infty}+\frac{\alpha_{1}}{\sigma_{\min}}+\left(\frac{\sigma^{2}}{\tau\sigma_{\min}}\right)\sqrt{n}\left\Vert \bm{V}^{\star}\right\Vert \nonumber \\
 & \qquad+\frac{\widetilde{\sigma}\sqrt{n}}{\sigma_{\min}}\left(\left\Vert \bm{U}^{\star}\right\Vert _{2,\infty}+\left\Vert \bm{U}^{0,\left(l\right)}\bm{H}_{\bm{U}}^{\left(l\right)}-\bm{U}^{\star}\right\Vert _{2,\infty}\right).\label{eq:alpha2}
\end{align}
Then we turn to bound $\Vert\bm{U}^{0,(l)}\bm{H}_{\bm{U}}^{(l)}-\bm{U}^{\star}\Vert_{2,\infty}$.
One has 
\begin{align*}
 & \left\Vert \bm{U}^{0,\left(l\right)}\bm{H}_{\bm{U}}^{\left(l\right)}-\bm{U}^{\star}\right\Vert _{2,\infty}\\
 & \lesssim\left\Vert \bm{U}^{0}\bm{H}_{\bm{U}}-\bm{U}^{\star}\right\Vert _{2,\infty}+\left\Vert \bm{U}^{0,\left(l\right)}\bm{H}_{\bm{U}}^{\left(l\right)}-\bm{U}^{0}\bm{H}_{\bm{U}}\right\Vert _{\mathrm{F}}\\
 & \lesssim\left\Vert \bm{U}^{0}\bm{H}_{\bm{U}}-\bm{U}^{\star}\right\Vert _{2,\infty}+\frac{\widetilde{\sigma}}{\sigma_{\min}}\sqrt{\log n}\left\Vert \bm{V}^{\star}\right\Vert _{\mathrm{F}}+\frac{B}{\sigma_{\min}}\log n\left\Vert \bm{V}^{\star}\right\Vert _{2,\infty}+\frac{\alpha_{1}}{\sigma_{\min}}\\
 & \qquad+\left(\frac{\sigma^{2}}{\tau\sigma_{\min}}\right)\sqrt{n}\left\Vert \bm{V}^{\star}\right\Vert +\frac{\widetilde{\sigma}\sqrt{n}}{\sigma_{\min}}\left(\left\Vert \bm{U}^{\star}\right\Vert _{2,\infty}+\left\Vert \bm{U}^{0,\left(l\right)}\bm{H}_{\bm{U}}^{\left(l\right)}-\bm{U}^{\star}\right\Vert _{2,\infty}\right).
\end{align*}
Rearrange the terms containing $\Vert\bm{U}^{0,(l)}\bm{H}_{\bm{U}}^{(l)}-\bm{U}^{\star}\Vert_{2,\infty}$
to obtain 
\begin{align}
\left\Vert \bm{U}^{0,\left(l\right)}\bm{H}_{\bm{U}}^{\left(l\right)}-\bm{U}^{\star}\right\Vert _{2,\infty} & \lesssim\left\Vert \bm{U}^{0}\bm{H}_{\bm{U}}-\bm{U}^{\star}\right\Vert _{2,\infty}+\frac{\widetilde{\sigma}\sqrt{\log n}}{\sigma_{\min}}\left\Vert \bm{V}^{\star}\right\Vert _{\mathrm{F}}+\frac{B\log n}{\sigma_{\min}}\left\Vert \bm{V}^{\star}\right\Vert _{2,\infty}\nonumber \\
 & \qquad+\frac{\alpha_{1}}{\sigma_{\min}}+\left(\frac{\sigma^{2}}{\tau\sigma_{\min}}\right)\sqrt{n}\left\Vert \bm{V}^{\star}\right\Vert +\frac{\widetilde{\sigma}\sqrt{n}}{\sigma_{\min}}\left\Vert \bm{U}^{\star}\right\Vert _{2,\infty}.\label{eq:lem3-3-4}
\end{align}
Substitution of \eqref{eq:lem3-3-4} into \eqref{eq:lem3-3-alpha2}
yields 
\begin{align}
 & \max\left\{ \left\Vert \bm{U}^{0,\left(l\right)}\bm{H}_{\bm{U}}^{\left(l\right)}-\bm{U}^{0}\bm{H}_{\bm{U}}\right\Vert _{\mathrm{F}},\left\Vert \bm{V}^{0,\left(l\right)}\bm{H}_{\bm{V}}^{\left(l\right)}-\bm{V}^{0}\bm{H}_{\bm{V}}\right\Vert _{\mathrm{F}}\right\} \nonumber \\
 & \lesssim\frac{\widetilde{\sigma}\sqrt{\log n}}{\sigma_{\min}}\left\Vert \bm{V}^{\star}\right\Vert _{\mathrm{F}}+\frac{B\log n}{\sigma_{\min}}\left\Vert \bm{V}^{\star}\right\Vert _{2,\infty}+\frac{\alpha_{1}}{\sigma_{\min}}+\left(\frac{\sigma^{2}}{\tau\sigma_{\min}}\right)\sqrt{n}\left\Vert \bm{V}^{\star}\right\Vert \nonumber \\
 & \qquad+\frac{\widetilde{\sigma}\sqrt{n}}{\sigma_{\min}}\left\Vert \bm{U}^{\star}\right\Vert _{2,\infty}+\frac{\widetilde{\sigma}\sqrt{n}}{\sigma_{\min}}\left\Vert \bm{U}^{0}\bm{H}_{\bm{U}}-\bm{U}^{\star}\right\Vert _{2,\infty},\label{eq:Ul-U}
\end{align}
provided that $\frac{\sigma}{\sigma_{\min}}\sqrt{\frac{n}{p}}\ll1$.
Plugging this into \eqref{eq:lem3-3-alpha2} gives 
\begin{align}
\alpha_{2} & \lesssim\frac{\left(\widetilde{\sigma}\sqrt{n}+B\sqrt{\log n}\right)}{\sigma_{\min}}\left(\widetilde{\sigma}\sqrt{\log n}\left\Vert \bm{V}^{\star}\right\Vert _{\mathrm{F}}+B\log n\left\Vert \bm{V}^{\star}\right\Vert _{2,\infty}+\alpha_{1}+\left(\frac{\sigma^{2}}{\tau}\right)\sqrt{n}\left\Vert \bm{V}^{\star}\right\Vert \right)\nonumber \\
 & \qquad+\frac{\left(\widetilde{\sigma}\sqrt{n}+B\sqrt{\log n}\right)}{\sigma_{\min}}\left(\widetilde{\sigma}\sqrt{n}\left\Vert \bm{U}^{\star}\right\Vert _{2,\infty}+\widetilde{\sigma}\sqrt{n}\left\Vert \bm{U}^{0}\bm{H}_{\bm{U}}-\bm{U}^{\star}\right\Vert _{2,\infty}\right)\label{eq:lem3-3-alpha2-1}
\end{align}
Furthermore, substitution of \eqref{eq:Ul-U} into \eqref{eq:alpha1}
yields 
\begin{align*}
\alpha_{1} & \lesssim\widetilde{\sigma}\sqrt{\log n}\left\Vert \bm{V}^{0}\bm{H}_{\bm{V}}-\bm{V}^{\star}\right\Vert _{\mathrm{F}}+B\log n\left\Vert \bm{V}^{0}\bm{H}_{\bm{V}}-\bm{V}^{\star}\right\Vert _{2,\infty}\\
 & \qquad+\frac{\left(\widetilde{\sigma}\sqrt{\log n}+B\log n\right)}{\sigma_{\min}}\left(\widetilde{\sigma}\sqrt{\log n}\left\Vert \bm{V}^{\star}\right\Vert _{\mathrm{F}}+B\log n\left\Vert \bm{V}^{\star}\right\Vert _{2,\infty}+\alpha_{1}+\left(\frac{\sigma^{2}}{\tau}\right)\sqrt{n}\left\Vert \bm{V}^{\star}\right\Vert \right)\\
 & \qquad+\frac{\left(\widetilde{\sigma}\sqrt{\log n}+B\log n\right)}{\sigma_{\min}}\left(\widetilde{\sigma}\sqrt{n}\left\Vert \bm{U}^{\star}\right\Vert _{2,\infty}+\widetilde{\sigma}\sqrt{n}\left\Vert \bm{U}^{0}\bm{H}_{\bm{U}}-\bm{U}^{\star}\right\Vert _{2,\infty}\right).
\end{align*}
We can rearrange terms to derive the bound of $\alpha_{1}$ 
\begin{align}
\alpha_{1} & \lesssim\widetilde{\sigma}\sqrt{\log n}\left\Vert \bm{V}^{0}\bm{H}_{\bm{V}}-\bm{V}^{\star}\right\Vert _{\mathrm{F}}+B\log n\left\Vert \bm{V}^{0}\bm{H}_{\bm{V}}-\bm{V}^{\star}\right\Vert _{2,\infty}\nonumber \\
 & \qquad+\frac{\left(\widetilde{\sigma}\sqrt{\log n}+B\log n\right)}{\sigma_{\min}}\left(\widetilde{\sigma}\sqrt{\log n}\left\Vert \bm{V}^{\star}\right\Vert _{\mathrm{F}}+B\log n\left\Vert \bm{V}^{\star}\right\Vert _{2,\infty}+\left(\frac{\sigma^{2}}{\tau}\right)\sqrt{n}\left\Vert \bm{V}^{\star}\right\Vert \right)\nonumber \\
 & \qquad+\frac{\left(\widetilde{\sigma}\sqrt{\log n}+B\log n\right)}{\sigma_{\min}}\left(\widetilde{\sigma}\sqrt{n}\left\Vert \bm{U}^{\star}\right\Vert _{2,\infty}+\widetilde{\sigma}\sqrt{n}\left\Vert \bm{U}^{0}\bm{H}_{\bm{U}}-\bm{U}^{\star}\right\Vert _{2,\infty}\right)\nonumber \\
 & \lesssim\frac{\widetilde{\sigma}\sqrt{r\log n}}{\sigma_{\min}}\left(\sigma\sqrt{\frac{n}{p}}+\left\Vert \bm{M}^{\star}\right\Vert _{\infty}\sqrt{\frac{n}{p}}\right)+B\frac{\log n}{\sigma_{\min}}\left(B\sqrt{\frac{\mu r}{n}}\log n+\widetilde{\sigma}\sqrt{\mu r}\right)\nonumber \\
 & \qquad+B\log n\left\Vert \bm{V}^{0}\bm{H}_{\bm{V}}-\bm{V}^{\star}\right\Vert _{2,\infty}.\label{eq:lem3-3-alpha1}
\end{align}
Here the second line utilizes the fact that 
\begin{align*}
\left\Vert \bm{V}^{0}\bm{H}_{\bm{V}}-\bm{V}^{\star}\right\Vert _{\mathrm{F}} & =\left\Vert \bm{V}^{0}\bm{V}^{0\top}\bm{V}^{\star}-\bm{V}^{\star}\right\Vert _{\mathrm{F}}=\left\Vert \left(\bm{V}^{0}\bm{V}^{0\top}-\bm{V}^{\star}\bm{V}^{\star\top}\right)\bm{V}^{\star}\right\Vert _{\mathrm{F}}\\
 & \leq\left\Vert \bm{V}^{0}\bm{V}^{0\top}-\bm{V}^{\star}\bm{V}^{\star\top}\right\Vert _{\mathrm{F}}\lesssim\left\Vert \bm{V}^{0}\bm{Q}_{\bm{V}}-\bm{V}^{\star}\right\Vert _{\mathrm{F}}\\
 & \lesssim\frac{\sigma}{\sigma_{\min}}\sqrt{\frac{n}{p}}+\frac{\left\Vert \bm{M}^{\star}\right\Vert _{\infty}}{\sigma_{\min}}\sqrt{\frac{n}{p}},
\end{align*}
where the last inequality follows from Lemma \ref{lem:1}. Combining
\eqref{eq:lem3-3-alpha1} and \eqref{eq:lem3-3-alpha2-1}, we obtain
\begin{align}
 & \left\Vert \left(\bm{M}^{0}-\bm{M}^{\star}-\mathbb{E}\left[\bm{M}^{0}-\bm{M}^{\star}\right]\right)\left(\bm{V}^{0}\bm{H}_{\bm{V}}-\bm{V}^{\star}\right)\right\Vert _{2,\infty}\nonumber \\
 & \lesssim\frac{\left(\widetilde{\sigma}\sqrt{n}+B\sqrt{\log n}\right)}{\sigma_{\min}}\left(\widetilde{\sigma}\sqrt{\log n}\left\Vert \bm{V}^{\star}\right\Vert _{\mathrm{F}}+B\log n\left\Vert \bm{V}^{\star}\right\Vert _{2,\infty}+\alpha_{1}+\left(\frac{\sigma^{2}}{\tau}\right)\sqrt{n}\left\Vert \bm{V}^{\star}\right\Vert \right)\nonumber \\
 & \qquad+\frac{\left(\widetilde{\sigma}\sqrt{n}+B\sqrt{\log n}\right)}{\sigma_{\min}}\left(\widetilde{\sigma}\sqrt{n}\left\Vert \bm{U}^{\star}\right\Vert _{2,\infty}+\widetilde{\sigma}\sqrt{n}\left\Vert \bm{U}^{0}\bm{H}_{\bm{U}}-\bm{U}^{\star}\right\Vert _{2,\infty}\right)\nonumber\\
 &\qquad+ \frac{\widetilde{\sigma}\sqrt{r\log n}}{\sigma_{\min}}\left(\sigma\sqrt{\frac{n}{p}}+\left\Vert \bm{M}^{\star}\right\Vert _{\infty}\sqrt{\frac{n}{p}}\right)+B\frac{\log n}{\sigma_{\min}}\left(B\sqrt{\frac{\mu r}{n}}\log n+\widetilde{\sigma}\sqrt{\mu r}\right)\nonumber \\
 &\qquad+B\log n\left\Vert \bm{V}^{0}\bm{H}_{\bm{V}}-\bm{V}^{\star}\right\Vert _{2,\infty}\nonumber \\
 & \lesssim\widetilde{\sigma}^{2}\sqrt{\mu rn}\frac{\log^{2}n}{\sigma_{\min}}+B\log n\left\Vert \bm{V}^{0}\bm{H}_{\bm{V}}-\bm{V}^{\star}\right\Vert _{2,\infty}+\frac{\widetilde{\sigma}^{2}n\sqrt{\log n}}{\sigma_{\min}}\left\Vert \bm{U}^{0}\bm{H}_{\bm{U}}-\bm{U}^{\star}\right\Vert _{2,\infty},\label{eq:lem3-3-7}
\end{align}
where the last inequality holds provided $\frac{\sigma}{\sigma_{\min}}\sqrt{\frac{n}{p}}\ll\frac{1}{\log n}$
and $np\gg\mu\kappa r\log n$. 
\item Next, we turn attention to $\beta_{3}$, which satisfies 
\begin{align*}
 & \left\Vert \mathbb{E}\left[\bm{M}^{0}-\bm{M}^{\star}\right]\left(\bm{V}\bm{H}_{\bm{V}}^{0}-\bm{V}^{\star}\right)\right\Vert _{2,\infty}\\
 & \leq\left\Vert \mathbb{E}\left[\bm{M}^{0}-\bm{M}^{\star}\right]\right\Vert _{2,\infty}\left\Vert \bm{V}^{0}\bm{H}_{\bm{V}}-\bm{V}^{\star}\right\Vert \\
 & \lesssim\sqrt{n}\frac{\sigma^{2}}{\tau}\left(\frac{\sigma}{\sigma_{\min}}\sqrt{\frac{n}{p}}+\frac{\left\Vert \bm{M}^{\star}\right\Vert _{\infty}}{\sigma_{\min}}\sqrt{\frac{n}{p}}\right),
\end{align*}
where the third line follows from \eqref{eq:mean-E} and a consequence
of Lemma \ref{lem:1} 
\begin{align}
\left\Vert \bm{V}^{0}\bm{H}_{\bm{V}}-\bm{V}^{\star}\right\Vert  & \leq\left\Vert \bm{V}^{0}\bm{Q}_{\bm{V}}-\bm{V}^{\star}\right\Vert +\left\Vert \bm{V}^{0}\bm{H}_{\bm{V}}-\bm{V}^{0}\bm{Q}_{\bm{V}}\right\Vert \nonumber \\
 & \leq\left\Vert \bm{V}^{0}\bm{Q}_{\bm{V}}-\bm{V}^{\star}\right\Vert +\left\Vert \bm{H}_{\bm{V}}-\bm{Q}_{\bm{V}}\right\Vert \nonumber \\
 & \lesssim\frac{\sigma}{\sigma_{\min}}\sqrt{\frac{n}{p}}+\frac{\left\Vert \bm{M}^{\star}\right\Vert _{\infty}}{\sigma_{\min}}\sqrt{\frac{n}{p}}.\label{eq:lem3-3-8}
\end{align}
\item Lastly, we can reuse the results in \citet[Section 3.3 Step 3]{yan2021inference}
to bound $\beta_{1}$ as 
\begin{align}
\left\Vert \bm{M}^{\star}\left(\bm{V}^{0}\bm{H}_{\bm{V}}-\bm{V}^{\star}\right)\right\Vert _{2,\infty} & \lesssim\left\Vert \bm{U}^{\star}\right\Vert _{2,\infty}\sigma_{\max}\left\Vert \bm{V}^{0}\bm{R}_{\bm{V}}-\bm{V}^{\star}\right\Vert ^{2}\nonumber \\
 & \lesssim\left\Vert \bm{U}^{\star}\right\Vert _{2,\infty}\sigma_{\max}\left(\frac{\sigma}{\sigma_{\min}}\sqrt{\frac{n}{p}}+\frac{\left\Vert \bm{M}^{\star}\right\Vert _{\infty}}{\sigma_{\min}}\sqrt{\frac{n}{p}}\right)^{2},\label{eq:lem3-3-9}
\end{align}
where the second inequality makes use of Lemma \ref{lem:1}. 
\end{enumerate}
Finally, plugging the bounds \eqref{eq:lem3-3-7}, \eqref{eq:lem3-3-8}
and \eqref{eq:lem3-3-9} into \eqref{eq:lem3-3-decompose} reveals
that 
\begin{align*}
\left\Vert \bm{U}^{0}\bm{\Sigma}^{0}\bm{H}_{\bm{V}}-\bm{M}^{0}\bm{V}^{\star}\right\Vert  & \lesssim\frac{\log^{2}n}{\sigma_{\min}}\sqrt{\mu rn}\frac{\sigma^{2}+\left\Vert \bm{M}^{\star}\right\Vert _{\infty}^{2}}{p}+\frac{\tau+\left\Vert \bm{M}^{\star}\right\Vert _{\infty}}{p}\log n\left\Vert \bm{V}^{0}\bm{H}_{\bm{V}}-\bm{V}^{\star}\right\Vert _{2,\infty}\\
 & \qquad+\frac{1}{\sigma_{\min}}\frac{\sigma^{2}+\left\Vert \bm{M}^{\star}\right\Vert _{\infty}^{2}}{p}n\left\Vert \bm{U}^{0}\bm{H}_{\bm{U}}-\bm{U}^{\star}\right\Vert _{2,\infty}\\
 & \qquad+\left\Vert \bm{U}^{\star}\right\Vert _{2,\infty}\sigma_{\max}\left(\frac{\sigma}{\sigma_{\min}}\sqrt{\frac{n}{p}}+\frac{\left\Vert \bm{M}^{\star}\right\Vert _{\infty}}{\sigma_{\min}}\sqrt{\frac{n}{p}}\right)^{2}.
\end{align*}

\subsubsection{Proof of Lemma \ref{lem:3-4}}

We start from a crucial decomposition 
\begin{align*}
\left\Vert \bm{U}^{0}\bm{H}_{\bm{U}}-\bm{U}^{\star}\right\Vert _{2,\infty} & \leq\frac{1}{\sigma_{\min}}\left\Vert \bm{U}^{0}\bm{H}_{\bm{U}}\bm{\Sigma}^{\star}-\bm{U}^{\star}\bm{\Sigma}^{\star}\right\Vert _{2,\infty}\\
 & \leq\underbrace{\frac{1}{\sigma_{\min}}\left\Vert \bm{U}^{0}\bm{H}_{\bm{U}}\bm{\Sigma}^{\star}-\bm{U}^{0}\bm{\Sigma}^{0}\bm{H}_{\bm{V}}\right\Vert _{2,\infty}}_{\eqqcolon\omega_{1}}+\underbrace{\frac{1}{\sigma_{\min}}\left\Vert \bm{U}^{0}\bm{\Sigma}^{0}\bm{H}_{\bm{V}}-\bm{M}^{0}\bm{V}^{\star}\right\Vert _{2,\infty}}_{\eqqcolon\omega_{2}}\\
 & \qquad+\underbrace{\frac{1}{\sigma_{\min}}\left\Vert \left(\bm{M}^{0}-\bm{M}^{\star}\right)\bm{V}^{\star}\right\Vert _{2,\infty}}_{\eqqcolon\omega_{3}}.
\end{align*}
In the sequel, we shall bound these three terms separately. 
\begin{enumerate}
\item To bound $\omega_{1}$, we note that it is exactly the term $\beta_{3}$
in \citet[Section C.3.4]{yan2021inference}. Invoking the results
therein, one has 
\begin{align*}
\omega_{1} & \leq\frac{1}{\sigma_{\min}}\left(\left\Vert \bm{U}^{\star}\right\Vert _{2,\infty}+\left\Vert \bm{U}^{0}\bm{H}_{\bm{U}}-\bm{U}^{\star}\right\Vert _{2,\infty}\right)\left(\left\Vert \bm{H}_{\bm{U}}^{\top}\bm{\Sigma}^{0}\bm{H}_{\bm{V}}-\bm{\Sigma}^{\star}\right\Vert +\left\Vert \bm{\Sigma}^{\star}\right\Vert \left\Vert \bm{H}_{\bm{U}}-\bm{Q}_{\bm{U}}\right\Vert \right).
\end{align*}
\item Regarding $\omega_{2}$, Lemma \ref{lem:3-3} reveals that 
\begin{align*}
 & \frac{1}{\sigma_{\min}}\left\Vert \bm{U}^{0}\bm{\Sigma}^{0}\bm{H}_{\bm{V}}-\bm{M}^{0}\bm{V}^{\star}\right\Vert _{2,\infty}\\
 & \lesssim\frac{\log^{2}n}{\sigma_{\min}^{2}}\sqrt{\mu rn}\frac{\sigma^{2}+\left\Vert \bm{M}^{\star}\right\Vert _{\infty}^{2}}{p}+\frac{\tau+\left\Vert \bm{M}^{\star}\right\Vert _{\infty}}{p\sigma_{\min}}\log n\left\Vert \bm{V}^{0}\bm{H}_{\bm{V}}-\bm{V}^{\star}\right\Vert _{2,\infty}\\
 & \qquad+\frac{n}{\sigma_{\min}^{2}}\frac{\sigma^{2}+\left\Vert \bm{M}^{\star}\right\Vert _{\infty}^{2}}{p}\left\Vert \bm{U}^{0}\bm{H}_{\bm{U}}-\bm{U}^{\star}\right\Vert _{2,\infty}+\left\Vert \bm{U}^{\star}\right\Vert _{2,\infty}\kappa\left(\frac{\sigma}{\sigma_{\min}}\sqrt{\frac{n}{p}}+\frac{\left\Vert \bm{M}^{\star}\right\Vert _{\infty}}{\sigma_{\min}}\sqrt{\frac{n}{p}}\right)^{2}.
\end{align*}
\item Turning attention to the last term $\omega_{3}$, we have 
\begin{align*}
 & \frac{1}{\sigma_{\min}}\left\Vert \left(\bm{M}^{0}-\bm{M}^{\star}\right)\bm{V}^{\star}\right\Vert _{2,\infty}\\
 & \lesssim\frac{1}{\sigma_{\min}}\left\Vert \left(\bm{M}^{0}-\bm{M}^{\star}-\mathbb{E}\left[\bm{M}^{0}-\bm{M}^{\star}\right]\right)\bm{V}^{\star}\right\Vert _{2,\infty}+\frac{1}{\sigma_{\min}}\left\Vert \mathbb{E}\left[\bm{M}^{0}-\bm{M}^{\star}\right]\bm{V}^{\star}\right\Vert _{2,\infty}\\
 & \overset{\text{(i)}}{\lesssim}\frac{1}{\sigma_{\min}}\left(\widetilde{\sigma}\left\Vert \bm{V}^{\star}\right\Vert _{\mathrm{F}}\sqrt{\log n}+B\left\Vert \bm{V}^{\star}\right\Vert _{2,\infty}\log n\right)+\frac{\sqrt{n}}{\sigma_{\min}}B\\
 & \overset{\text{(ii)}}{\lesssim}\frac{\left(\sigma+\left\Vert \bm{M}^{\star}\right\Vert _{\infty}\right)}{\sigma_{\min}}\sqrt{\frac{r\log n}{p}}+\frac{\left(\tau+\left\Vert \bm{M}^{\star}\right\Vert _{\infty}\right)\log n}{\sigma_{\min}p}\left\Vert \bm{V}^{\star}\right\Vert _{2,\infty},
\end{align*}
where (i) is due to Lemma \ref{lem:bernstein} and (ii) plugs in the
definitions of $\widetilde{\sigma}$ and $B$ (cf.~\eqref{eq:defn-sigmatilde}
and \eqref{eq:defn-B}). 
\end{enumerate}
Finally, taking the above bounds in $\omega_{1}$, $\omega_{2}$ and
$\omega_{3}$ collectively, we arrive at 
\begin{align*}
&\left\Vert \bm{U}^{0}\bm{H}_{\bm{U}}-\bm{U}^{\star}\right\Vert _{2,\infty} \\
& \quad\lesssim\frac{1}{\sigma_{\min}}\left(\left\Vert \bm{U}^{\star}\right\Vert _{2,\infty}+\left\Vert \bm{U}^{0}\bm{H}_{\bm{U}}-\bm{U}^{\star}\right\Vert _{2,\infty}\right)\left[\left\Vert \bm{H}_{\bm{U}}^{\top}\bm{\Sigma}^{0}\bm{H}_{\bm{V}}-\bm{\Sigma}^{\star}\right\Vert +\left\Vert \bm{\Sigma}^{\star}\right\Vert \left\Vert \bm{H}_{\bm{U}}-\bm{Q}_{\bm{U}}\right\Vert \right]\\
 & \qquad+\frac{\log^{2}n}{\sigma_{\min}^{2}}\sqrt{\mu rn}\frac{\sigma^{2}+\left\Vert \bm{M}^{\star}\right\Vert _{\infty}^{2}}{p}+\frac{\tau+\left\Vert \bm{M}^{\star}\right\Vert _{\infty}}{p\sigma_{\min}}\log n\left\Vert \bm{V}^{0}\bm{H}_{\bm{V}}-\bm{V}^{\star}\right\Vert _{2,\infty}\\
 & \qquad+\frac{n}{\sigma_{\min}}\frac{\sigma^{2}+\left\Vert \bm{M}^{\star}\right\Vert _{\infty}^{2}}{p}\left\Vert \bm{U}^{0}\bm{H}_{\bm{U}}-\bm{U}^{\star}\right\Vert _{2,\infty}+\left\Vert \bm{U}^{\star}\right\Vert _{2,\infty}\kappa\left(\frac{\sigma}{\sigma_{\min}}\sqrt{\frac{n}{p}}+\frac{\left\Vert \bm{M}^{\star}\right\Vert _{\infty}}{\sigma_{\min}}\sqrt{\frac{n}{p}}\right)^{2}\\
 & \qquad+\frac{\left(\sigma+\left\Vert \bm{M}^{\star}\right\Vert _{\infty}\right)}{\sigma_{\min}}\sqrt{\frac{r\log n}{p}}+\frac{\left(\tau+\left\Vert \bm{M}^{\star}\right\Vert _{\infty}\right)\log n}{\sigma_{\min}p}\left\Vert \bm{V}^{\star}\right\Vert _{2,\infty}.
\end{align*}
Rearranging terms containing $\Vert\bm{U}^{0}\bm{H}_{\bm{U}}-\bm{U}^{\star}\Vert_{2,\infty}$
and making use of Lemma \ref{lem:1} and \ref{lem:3-2}, we can obtain
\begin{align}
 & \left\Vert \bm{U}^{0}\bm{H}_{\bm{U}}-\bm{U}^{\star}\right\Vert _{2,\infty}\nonumber \\
 & \lesssim\frac{\left(\sigma+\left\Vert \bm{M}^{\star}\right\Vert _{\infty}\right)}{\sigma_{\min}}\sqrt{\frac{r\log n}{p}}+\frac{\tau+\left\Vert \bm{M}^{\star}\right\Vert _{\infty}}{p\sigma_{\min}}\log n\left\Vert \bm{V}^{0}\bm{H}_{\bm{V}}-\bm{V}^{\star}\right\Vert _{2,\infty}\nonumber \\
 & \qquad+\left\Vert \bm{U}^{\star}\right\Vert _{2,\infty}\kappa\left(\frac{\sigma}{\sigma_{\min}}\sqrt{\frac{n}{p}}+\frac{\left\Vert \bm{M}^{\star}\right\Vert _{\infty}}{\sigma_{\min}}\sqrt{\frac{n}{p}}\right)^{2}+\frac{\left(\tau+\left\Vert \bm{M}^{\star}\right\Vert _{\infty}\right)\log n}{\sigma_{\min}p}\left\Vert \bm{V}^{\star}\right\Vert _{2,\infty},\label{eq:U-2infty}
\end{align}
provided that $\frac{\sigma}{\sigma_{\min}}\sqrt{\frac{n}{p}}\ll\frac{1}{\sqrt{\mu\log^{3}n}}$
and $np\gg\kappa^{2}\mu^{3}r^{2}\log^{3}n$. Similarly, one has 
\begin{align}
 & \left\Vert \bm{V}^{0}\bm{H}_{\bm{V}}-\bm{V}^{\star}\right\Vert _{2,\infty}\nonumber \\
 & \lesssim\frac{\left(\sigma+\left\Vert \bm{M}^{\star}\right\Vert _{\infty}\right)}{\sigma_{\min}}\sqrt{\frac{r\log n}{p}}+\frac{\tau+\left\Vert \bm{M}^{\star}\right\Vert _{\infty}}{p\sigma_{\min}}\log n\left\Vert \bm{U}^{0}\bm{H}_{\bm{U}}-\bm{U}^{\star}\right\Vert _{2,\infty}\nonumber \\
 & \qquad+\left\Vert \bm{V}^{\star}\right\Vert _{2,\infty}\kappa\left(\frac{\sigma}{\sigma_{\min}}\sqrt{\frac{n}{p}}+\frac{\left\Vert \bm{M}^{\star}\right\Vert _{\infty}}{\sigma_{\min}}\sqrt{\frac{n}{p}}\right)^{2}+\frac{\left(\tau+\left\Vert \bm{M}^{\star}\right\Vert _{\infty}\right)\log n}{\sigma_{\min}p}\left\Vert \bm{U}^{\star}\right\Vert _{2,\infty}.\label{eq:V-2infty}
\end{align}
Plugging \eqref{eq:V-2infty} into \eqref{eq:U-2infty} and rearranging
terms reveal that 
\begin{align*}
\left\Vert \bm{U}^{0}\bm{H}_{\bm{U}}-\bm{U}^{\star}\right\Vert _{2,\infty} & \lesssim\frac{\left(\sigma+\left\Vert \bm{M}^{\star}\right\Vert _{\infty}\right)}{\sigma_{\min}}\sqrt{\frac{r\log n}{p}}+\left\Vert \bm{U}^{\star}\right\Vert _{2,\infty}\kappa\left(\frac{\sigma}{\sigma_{\min}}\sqrt{\frac{n}{p}}+\frac{\left\Vert \bm{M}^{\star}\right\Vert _{\infty}}{\sigma_{\min}}\sqrt{\frac{n}{p}}\right)^{2}\\
 & \qquad+\frac{\left(\sigma\sqrt{np}+\left\Vert \bm{M}^{\star}\right\Vert _{\infty}\right)\log n}{\sigma_{\min}p}\left\Vert \bm{V}^{\star}\right\Vert _{2,\infty}.
\end{align*}
The bound of $\Vert\bm{V}^{0}\bm{H}_{\bm{V}}-\bm{V}^{\star}\Vert_{2,\infty}$
can be established by a similar argument.

\subsection{Proof of Lemma \ref{lem:4}}

For any fixed $1\leq l\leq2n$, the triangle inequality enables us
to obtain 
\begin{align}
\left\Vert \left[\begin{array}{c}
\bm{X}^{0,\left(l\right)}\bm{H}^{0,\left(l\right)}-\bm{X}^{\star}\\
\bm{Y}^{0,\left(l\right)}\bm{H}^{0,\left(l\right)}-\bm{Y}^{\star}
\end{array}\right]\right\Vert _{2} & \leq\left\Vert \left[\begin{array}{c}
\bm{X}^{0,\left(l\right)}\left(\bm{H}^{0,\left(l\right)}-\bm{Q}^{\left(l\right)}\right)\\
\bm{Y}^{0,\left(l\right)}\left(\bm{H}^{0,\left(l\right)}-\bm{Q}^{\left(l\right)}\right)
\end{array}\right]\right\Vert _{2}+\left\Vert \left[\begin{array}{c}
\bm{X}^{0,\left(l\right)}\bm{Q}^{\left(l\right)}-\bm{X}^{\star}\\
\bm{Y}^{0,\left(l\right)}\bm{Q}^{\left(l\right)}-\bm{Y}^{\star}
\end{array}\right]\right\Vert _{2}\nonumber \\
 & \leq\left\Vert \left[\begin{array}{c}
\bm{X}^{0,\left(l\right)}\\
\bm{Y}^{0,\left(l\right)}
\end{array}\right]\right\Vert _{2}\left\Vert \bm{H}^{0,\left(l\right)}-\bm{Q}^{\left(l\right)}\right\Vert +\left\Vert \left[\begin{array}{c}
\bm{X}^{0,\left(l\right)}\bm{Q}^{\left(l\right)}-\bm{X}^{\star}\\
\bm{Y}^{0,\left(l\right)}\bm{Q}^{\left(l\right)}-\bm{Y}^{\star}
\end{array}\right]\right\Vert _{2}\label{eq:lem4-00}
\end{align}
In view of 
\[
\left[\begin{array}{c}
\bm{X}^{0,\left(l\right)}\\
\bm{Y}^{0,\left(l\right)}
\end{array}\right]=\left[\begin{array}{cc}
\bm{0} & \bm{M}^{0,\left(l\right)}\\
\left(\bm{M}^{0,\left(l\right)}\right)^{\top} & \bm{0}
\end{array}\right]\left[\begin{array}{c}
\bm{U}^{0,\left(l\right)}\\
\bm{V}^{0,\left(l\right)}
\end{array}\right]\left(\bm{\Sigma}^{\left(l\right)}\right)^{-1/2}
\]
and the definition of $\widetilde{\bm{M}}^{\star}$ (cf.~\eqref{eq:tilde-M}),
we can obtain 
\[
\left[\begin{array}{c}
\bm{X}^{0,\left(l\right)}\bm{Q}^{\left(l\right)}-\bm{X}^{\star}\\
\bm{Y}^{0,\left(l\right)}\bm{Q}^{\left(l\right)}-\bm{Y}^{\star}
\end{array}\right]=\widetilde{\bm{M}}^{0,\left(l\right)}\left[\begin{array}{c}
\bm{U}^{0,\left(l\right)}\\
\bm{V}^{0,\left(l\right)}
\end{array}\right]\left(\bm{\Sigma}^{\left(l\right)}\right)^{-1/2}\bm{Q}^{\left(l\right)}-\widetilde{\bm{M}}^{\star}\left[\begin{array}{c}
\bm{U}^{\star}\\
\bm{V}^{\star}
\end{array}\right]\left(\bm{\Sigma}^{\star}\right)^{-1/2}.
\]
Due to the fact that $\widetilde{\bm{M}}_{l,\cdot}^{0,(l)}=\widetilde{\bm{M}}_{l,\cdot}^{\star}$,
one has 
\begin{align*}
\left[\begin{array}{c}
\bm{X}^{0,\left(l\right)}\bm{Q}^{\left(l\right)}-\bm{X}^{\star}\\
\bm{Y}^{0,\left(l\right)}\bm{Q}^{\left(l\right)}-\bm{Y}^{\star}
\end{array}\right]_{l,\cdot} & =\widetilde{\bm{M}}_{l,\cdot}^{\star}\left\{ \left[\begin{array}{c}
\bm{U}^{0,\left(l\right)}\\
\bm{V}^{0,\left(l\right)}
\end{array}\right]\left(\bm{\Sigma}^{\left(l\right)}\right)^{-1/2}\bm{Q}^{\left(l\right)}-\left[\begin{array}{c}
\bm{U}^{\star}\\
\bm{V}^{\star}
\end{array}\right]\left(\bm{\Sigma}^{\star}\right)^{-1/2}\right\} \\
 & =\widetilde{\bm{M}}_{l,\cdot}^{\star}\left[\begin{array}{c}
\bm{U}^{0,\left(l\right)}\\
\bm{V}^{0,\left(l\right)}
\end{array}\right]\left[\left(\bm{\Sigma}^{\left(l\right)}\right)^{-1/2}\bm{Q}^{\left(l\right)}-\bm{Q}^{\left(l\right)}\left(\bm{\Sigma}^{\star}\right)^{-1/2}\right]\\
 & \qquad+\widetilde{\bm{M}}_{l,\cdot}^{\star}\left[\begin{array}{c}
\bm{U}^{0,\left(l\right)}\bm{Q}^{\left(l\right)}-\bm{U}^{\star}\\
\bm{V}^{0,\left(l\right)}\bm{Q}^{\left(l\right)}-\bm{V}^{\star}
\end{array}\right]\left(\bm{\Sigma}^{\star}\right)^{-1/2}.
\end{align*}
Therefore, it follows that 
\begin{align}
\left\Vert \left[\begin{array}{c}
\bm{X}^{0,\left(l\right)}\bm{Q}^{\left(l\right)}-\bm{X}^{\star}\\
\bm{Y}^{0,\left(l\right)}\bm{Q}^{\left(l\right)}-\bm{Y}^{\star}
\end{array}\right]_{l,\cdot}\right\Vert _{2} & \leq\left\Vert \bm{M}^{\star}\right\Vert _{2,\infty}\left\Vert \left(\bm{\Sigma}^{\left(l\right)}\right)^{-1/2}\bm{Q}^{\left(l\right)}-\bm{Q}^{\left(l\right)}\left(\bm{\Sigma}^{\star}\right)^{-1/2}\right\Vert \nonumber \\
 & \qquad+\frac{\left\Vert \bm{M}^{\star}\right\Vert _{2,\infty}}{\sqrt{\sigma_{\min}}}\left\Vert \left[\begin{array}{c}
\bm{U}^{0,\left(l\right)}\bm{Q}^{\left(l\right)}-\bm{U}^{\star}\\
\bm{V}^{0,\left(l\right)}\bm{Q}^{\left(l\right)}-\bm{V}^{\star}
\end{array}\right]\right\Vert .\label{eq:lem4-0}
\end{align}
To control this bound, \citet[Lemma 46]{ma2017implicit} gives 
\begin{align}
\left\Vert \left(\bm{\Sigma}^{\left(l\right)}\right)^{-1/2}\bm{Q}^{\left(l\right)}-\bm{Q}^{\left(l\right)}\left(\bm{\Sigma}^{\star}\right)^{-1/2}\right\Vert  & =\left\Vert \left(\bm{\Sigma}^{\left(l\right)}\right)^{-1/2}\left[\bm{Q}^{\left(l\right)}\left(\bm{\Sigma}^{\star}\right)^{1/2}-\left(\bm{\Sigma}^{\left(l\right)}\right)^{1/2}\bm{Q}^{\left(l\right)}\right]\left(\bm{\Sigma}^{\star}\right)^{-1/2}\right\Vert \nonumber \\
 & \lesssim\frac{1}{\sigma_{\min}}\left\Vert \bm{Q}^{\left(l\right)}\left(\bm{\Sigma}^{\star}\right)^{1/2}-\left(\bm{\Sigma}^{\left(l\right)}\right)^{1/2}\bm{Q}^{\left(l\right)}\right\Vert \nonumber \\
 & \lesssim\frac{1}{\sigma_{\min}^{3/2}}\left\Vert \widetilde{\bm{M}}^{0,\left(l\right)}-\widetilde{\bm{M}}^{\star}\right\Vert .\label{eq:lem4-1}
\end{align}
Furthermore, \citet[Lemma 45, 47]{ma2017implicit} imply that 
\begin{align}
\left\Vert \left[\begin{array}{c}
\bm{U}^{0,\left(l\right)}\bm{Q}^{\left(l\right)}-\bm{U}^{\star}\\
\bm{V}^{0,\left(l\right)}\bm{Q}^{\left(l\right)}-\bm{V}^{\star}
\end{array}\right]\right\Vert  & \lesssim\frac{1}{\sigma_{\min}}\left\Vert \widetilde{\bm{M}}^{0,\left(l\right)}-\widetilde{\bm{M}}^{\star}\right\Vert ,\label{eq:lem4-2}\\
\left\Vert \bm{H}^{0,\left(l\right)}-\bm{Q}^{\left(l\right)}\right\Vert  & \lesssim\frac{1}{\sigma_{\min}}\left\Vert \widetilde{\bm{M}}^{0,\left(l\right)}-\widetilde{\bm{M}}^{\star}\right\Vert .\label{eq:lem4-3}
\end{align}
Substitution of \eqref{eq:lem4-1}-\eqref{eq:lem4-3} into \eqref{eq:lem4-0}
yields 
\begin{equation}
\left\Vert \left[\begin{array}{c}
\bm{X}^{0,\left(l\right)}\bm{Q}^{\left(l\right)}-\bm{X}^{\star}\\
\bm{Y}^{0,\left(l\right)}\bm{Q}^{\left(l\right)}-\bm{Y}^{\star}
\end{array}\right]_{l,\cdot}\right\Vert _{2}\leq\frac{1}{\sigma_{\min}^{3/2}}\left\Vert \widetilde{\bm{M}}^{0,\left(l\right)}-\widetilde{\bm{M}}^{\star}\right\Vert \left\Vert \bm{M}^{\star}\right\Vert _{2,\infty}.\label{eq:lem4-4}
\end{equation}
Plugging \eqref{eq:lem4-3} and \eqref{eq:lem4-4} into \eqref{eq:lem4-00}
gives

\begin{align*}
\left\Vert \left[\begin{array}{c}
\bm{X}^{0,\left(l\right)}\bm{H}^{0,\left(l\right)}-\bm{X}^{\star}\\
\bm{Y}^{0,\left(l\right)}\bm{H}^{0,\left(l\right)}-\bm{Y}^{\star}
\end{array}\right]_{l,\cdot}\right\Vert _{2} & \leq\frac{1}{\sigma_{\min}}\left\Vert \widetilde{\bm{M}}^{0,\left(l\right)}-\widetilde{\bm{M}}^{\star}\right\Vert \left\Vert \left[\begin{array}{c}
\bm{X}^{0,\left(l\right)}\\
\bm{Y}^{0,\left(l\right)}
\end{array}\right]\right\Vert _{2,\infty}+\frac{1}{\sigma_{\min}^{3/2}}\left\Vert \widetilde{\bm{M}}^{0,\left(l\right)}-\widetilde{\bm{M}}^{\star}\right\Vert \left\Vert \bm{M}^{\star}\right\Vert _{2,\infty}\\
 & \lesssim\frac{\sqrt{\kappa}}{\sigma_{\min}}\left(\sigma\sqrt{\frac{n}{p}}+\sqrt{\frac{n}{p}}\left\Vert \bm{M}^{\star}\right\Vert _{\infty}\right)\left\Vert \bm{X}^{\star}\right\Vert _{2,\infty},
\end{align*}
where the last inequality makes use of \eqref{subeq:aux-6} and \eqref{subeq:lem1-1l}
\[
\left\Vert \widetilde{\bm{M}}^{0,\left(l\right)}-\widetilde{\bm{M}}^{\star}\right\Vert =\left\Vert \bm{M}^{0,\left(l\right)}-\bm{M}^{\star}\right\Vert \lesssim\sigma\sqrt{\frac{n}{p}}+\sqrt{\frac{n}{p}}\left\Vert \bm{M}^{\star}\right\Vert _{\infty}.
\]

\subsection{Proof of Lemma \ref{lem:5}}
To start with, we define 
\[
\bm{Q}^{0,\left(l\right)}=\arg\min_{\bm{R}\in\mathcal{O}^{r\times r}}\left\Vert \bm{U}^{0,\left(l\right)}\bm{R}-\bm{U}^{0}\right\Vert _{\mathrm{F}}.
\]
The definition of $\bm{R}^{0,\left(l\right)}$ (cf.~\eqref{defn-Rtl})
implies that 
\begin{equation}
\left\Vert \bm{X}^{0}\bm{H}^{0}-\bm{X}^{0,\left(l\right)}\bm{R}^{0,\left(l\right)}\right\Vert _{\mathrm{F}}\leq\left\Vert \bm{X}^{0}-\bm{X}^{0,\left(l\right)}\bm{Q}^{0,\left(l\right)}\right\Vert _{\mathrm{F}}.\label{eq:lem5-1}
\end{equation}
Then one has the following decomposition of $\bm{X}^{0,(l)}\bm{Q}^{0,(l)}-\bm{X}^{0}$,
\[
\bm{X}^{0,\left(l\right)}\bm{Q}^{0,\left(l\right)}-\bm{X}^{0}=\bm{U}^{0,\left(l\right)}\left[\left(\bm{\Sigma}^{0,\left(l\right)}\right)^{1/2}\bm{Q}^{0,\left(l\right)}-\bm{Q}^{0,\left(l\right)}\left(\bm{\Sigma}^{0}\right)^{1/2}\right]+\left(\bm{U}^{0,\left(l\right)}\bm{Q}^{0,\left(l\right)}-\bm{U}^{0}\right)\left(\bm{\Sigma}^{0}\right)^{1/2}.
\]
The triangle inequality reveals that 
\begin{align*}
\left\Vert \bm{X}^{0,\left(l\right)}\bm{Q}^{0,\left(l\right)}-\bm{X}^{0}\right\Vert _{\mathrm{F}} & \leq\left\Vert \left(\bm{\Sigma}^{0,\left(l\right)}\right)^{1/2}\bm{Q}^{0,\left(l\right)}-\bm{Q}^{0,\left(l\right)}\left(\bm{\Sigma}^{0}\right)^{1/2}\right\Vert _{\mathrm{F}}+\left\Vert \bm{U}^{0,\left(l\right)}\bm{Q}^{0,\left(l\right)}-\bm{U}^{0}\right\Vert _{\mathrm{F}}\left\Vert \left(\bm{\Sigma}^{0}\right)^{1/2}\right\Vert .
\end{align*}
Invoking \citet[Lemma 46]{ma2017implicit} gives 
\begin{align*}
 & \left\Vert \left(\bm{\Sigma}^{0,\left(l\right)}\right)^{1/2}\bm{Q}^{0,\left(l\right)}-\bm{Q}^{0,\left(l\right)}\left(\bm{\Sigma}^{0}\right)^{1/2}\right\Vert _{\mathrm{F}}\\
 & \lesssim\frac{1}{\sqrt{\sigma_{\min}}}\left\Vert \left(\widetilde{\bm{M}}^{0}-\widetilde{\bm{M}}^{0,\left(l\right)}\right)\left[\begin{array}{c}
\bm{U}^{0,\left(l\right)}\\
\bm{V}^{0,\left(l\right)}
\end{array}\right]\right\Vert _{\mathrm{F}}\\
 & \lesssim\frac{1}{\sqrt{\sigma_{\min}}}\left\Vert \left(\bm{M}^{0}-\bm{M}^{0,\left(l\right)}\right)\bm{V}^{0,\left(l\right)}\right\Vert _{\mathrm{F}}+\frac{1}{\sqrt{\sigma_{\min}}}\left\Vert \left(\bm{M}^{0}-\bm{M}^{0,\left(l\right)}\right)^{\top}\bm{U}^{0,\left(l\right)}\right\Vert _{\mathrm{F}}.
\end{align*}
Furthermore, Davis-Kahan's sin$\Theta$ theorem \citep{davis1970rotation}
implies that 
\[
\left\Vert \bm{U}^{0,\left(l\right)}\bm{Q}^{0,\left(l\right)}-\bm{U}^{0}\right\Vert _{\mathrm{F}}\lesssim\frac{1}{\sigma_{\min}}\left\Vert \left(\bm{M}^{0}-\bm{M}^{0,\left(l\right)}\right)\bm{V}^{0,\left(l\right)}\right\Vert _{\mathrm{F}}.
\]
Taking the results above together gives rise to 
\begin{align}
\left\Vert \bm{X}^{0,\left(l\right)}\bm{Q}^{0,\left(l\right)}-\bm{X}^{0}\right\Vert _{\mathrm{F}} & \lesssim\frac{1}{\sqrt{\sigma_{\min}}}\left\Vert \left(\bm{M}^{0}-\bm{M}^{0,\left(l\right)}\right)\bm{V}^{0,\left(l\right)}\right\Vert _{\mathrm{F}}+\frac{1}{\sqrt{\sigma_{\min}}}\left\Vert \left(\bm{M}^{0}-\bm{M}^{0,\left(l\right)}\right)^{\top}\bm{U}^{0,\left(l\right)}\right\Vert _{\mathrm{F}}\nonumber \\
 & +\frac{1}{\sigma_{\min}}\left\Vert \left(\bm{M}^{0}-\bm{M}^{0,\left(l\right)}\right)^{\top}\bm{U}^{0,\left(l\right)}\right\Vert _{\mathrm{F}}\left\Vert \left(\bm{\Sigma}^{0}\right)^{1/2}\right\Vert \nonumber \\
 & \lesssim\frac{1}{\sqrt{\sigma_{\min}}}\left\Vert \left(\bm{M}^{0}-\bm{M}^{0,\left(l\right)}\right)\bm{V}^{0,\left(l\right)}\right\Vert _{\mathrm{F}}+\sqrt{\frac{\kappa}{\sigma_{\min}}}\left\Vert \left(\bm{M}^{0}-\bm{M}^{0,\left(l\right)}\right)^{\top}\bm{U}^{0,\left(l\right)}\right\Vert _{\mathrm{F}},\label{eq:lem5-2}
\end{align}
where the last inequality utilizes \eqref{eq:initial-gamma2-1}. Then
we resort to the following claim to control the right-hand side. 
\begin{claim}
\label{claim:lem5}With probability exceeding $1-O(n^{-100})$, one
has
\begin{align}
&\left\Vert \left(\bm{M}^{0}-\bm{M}^{0,\left(l\right)}\right)^{\top}\bm{U}^{0,\left(l\right)}\right\Vert _{\mathrm{F}}\nonumber \\
&\qquad\lesssim\left(\sqrt{\frac{n\left(\left\Vert \bm{M}^{\star}\right\Vert _{\infty}^{2}+\sigma^{2}\right)}{p}\log n}+\frac{\left(\tau+\left\Vert \bm{M}^{\star}\right\Vert _{\infty}\right)}{p}\log n\right)\left\Vert \bm{U}^{0,\left(l\right)}\right\Vert _{2,\infty},\label{eq:claim5-1}
\end{align}
and
\begin{align}
	&\left\Vert \left(\bm{M}^{0}-\bm{M}^{0,\left(l\right)}\right)\bm{V}^{0,\left(l\right)}\right\Vert _{\mathrm{F}}\nonumber \\
	&\qquad \lesssim\left(\sqrt{\frac{n\left(\left\Vert \bm{M}^{\star}\right\Vert _{\infty}^{2}+\sigma^{2}\right)}{p}\log n}+\frac{\left(\tau+\left\Vert \bm{M}^{\star}\right\Vert _{\infty}\right)}{p}\log n\right)\left\Vert \bm{V}^{0,\left(l\right)}\right\Vert _{2,\infty}.\label{eq:claim5-2}
\end{align}
\end{claim}

In terms of $\Vert\bm{U}^{0,\left(l\right)}\Vert_{2,\infty}$, applying
similar derivation as Lemma \ref{lem:3-5} enables us to obtain that
\begin{align*}
&\left\Vert \bm{U}^{0,\left(l\right)}\bm{Q}_{\bm{U}}^{\left(l\right)}-\bm{U}^{\star}\right\Vert _{2,\infty} \\
&\quad \lesssim\frac{\left\Vert \bm{M}^{\star}\right\Vert _{\infty}}{\sigma_{\min}}\sqrt{\frac{r\log n}{p}}+\frac{\sigma}{\sigma_{\min}}\sqrt{\frac{\mu r}{p}}\log n\\
 & \qquad+\left\Vert \bm{U}^{\star}\right\Vert _{2,\infty}\left[\kappa\left(\frac{\sigma}{\sigma_{\min}}\sqrt{\frac{n}{p}}+\frac{\left\Vert \bm{M}^{\star}\right\Vert _{\infty}}{\sigma_{\min}}\sqrt{\frac{n}{p}}\right)^{2}+\sqrt{\frac{1}{p}\left(\left\Vert \bm{M}^{\star}\right\Vert _{\infty}^{2}+\sigma^{2}\right)}\frac{\sqrt{r\log n}}{\sigma_{\min}}\right]\\
 &\quad \ll\left\Vert \bm{U}^{\star}\right\Vert _{2,\infty},
\end{align*}
as long as $\frac{\sigma}{\sigma_{\min}}\sqrt{\frac{\kappa rn\log^{2}n}{p}}\ll1$
and $n^{2}p\gg\kappa^{2}\mu^{2}r^{3}n\log n$. Analogously we have
\[
\left\Vert \bm{V}^{0,\left(l\right)}\bm{Q}_{\bm{V}}^{\left(l\right)}-\bm{V}^{\star}\right\Vert _{2,\infty}\lesssim\left\Vert \bm{V}^{\star}\right\Vert _{2,\infty}.
\]
It then follows that 
\begin{align*}
\left\Vert \bm{U}^{0,\left(l\right)}\right\Vert _{2,\infty} & =\left\Vert \bm{U}^{0,\left(l\right)}\bm{Q}_{\bm{U}}^{\left(l\right)}\right\Vert _{2,\infty}\leq\left\Vert \bm{U}^{\star}\right\Vert _{2,\infty}+\left\Vert \bm{U}^{0,\left(l\right)}\bm{Q}_{\bm{U}}^{\left(l\right)}-\bm{U}^{\star}\right\Vert _{2,\infty}\leq2\left\Vert \bm{U}^{\star}\right\Vert _{2,\infty},
\end{align*}
and similarly $\Vert\bm{V}^{0,\left(l\right)}\Vert_{2,\infty}\leq2\Vert\bm{V}^{\star}\Vert_{2,\infty}.$
Therefore, combining \eqref{eq:lem5-1}, \eqref{eq:lem5-2} and Claim
\ref{claim:lem5} yields

\begin{align*}
\left\Vert \bm{X}^{0}\bm{H}^{0}-\bm{X}^{0,\left(l\right)}\bm{R}^{0,\left(l\right)}\right\Vert _{\mathrm{F}} & \lesssim\sqrt{\frac{\kappa}{\sigma_{\min}}}\left(\left\Vert \left(\bm{M}^{0}-\bm{M}^{0,\left(l\right)}\right)^{\top}\bm{U}^{0,\left(l\right)}\right\Vert _{\mathrm{F}}+\left\Vert \left(\bm{M}^{0}-\bm{M}^{0,\left(l\right)}\right)\bm{V}^{0,\left(l\right)}\right\Vert _{\mathrm{F}}\right)\\
 & \lesssim\sqrt{\frac{\kappa}{\sigma_{\min}}}\left(\sqrt{\frac{n\left(\left\Vert \bm{M}^{\star}\right\Vert _{\infty}^{2}+\sigma^{2}\right)}{p}\log n}+\frac{\left(\tau+\left\Vert \bm{M}^{\star}\right\Vert _{\infty}\right)}{p}\log n\right)\left\Vert \left[\begin{array}{c}
\bm{U}^{\star}\\
\bm{V}^{\star}
\end{array}\right]\right\Vert _{2,\infty}\\
 & \lesssim\sqrt{\kappa}\left(\frac{\sigma}{\sigma_{\min}}\sqrt{\frac{n}{p}}+\frac{\left\Vert \bm{M}^{\star}\right\Vert _{\infty}}{\sigma_{\min}}\sqrt{\frac{n}{p}}\right)\log n\left\Vert \bm{F}^{\star}\right\Vert _{2,\infty}.
\end{align*}
This completes the proof.
\begin{proof}
[Proof of Claim \ref{claim:lem5}] We prove \eqref{eq:claim5-2}
here, and \eqref{eq:claim5-1} would follow in an analogous way. From
the definition of $\bm{M}^{0,(l)}$ (cf.~\eqref{eq:spectral-method-matrix-1}),
one has 
\begin{align*}
p\left(\bm{M}^{0}-\bm{M}^{0,\left(l\right)}\right)\bm{V}^{0,\left(l\right)} & =\left(\mathcal{P}_{\Omega_{l}}\left(\psi_{\tau}\left(\bm{M}\right)\right)-p\mathcal{P}_{l}\left(\bm{M}^{\star}\right)\right)\bm{V}^{0,\left(l\right)}\\
 & =\sum_{j=1}^{n}\left(\delta_{l,j}\psi_{\tau}\left(M_{l,j}\right)-pM_{l,j}^{\star}\right)\bm{V}_{j,\cdot}^{0,\left(l\right)}.
\end{align*}
It is then easy to check that 
\begin{align*}
L & \coloneqq\max_{1\leq j\leq n}\left\Vert \left(\delta_{l,j}\left(M_{l,j}\ind_{\left|M_{l,j}\right|\leq\tau}+\tau\ind_{\left|M_{l,j}\right|>\tau}\right)-pM_{l,j}^{\star}\right)\bm{V}_{j,\cdot}^{0,\left(l\right)}\right\Vert _{2}\leq\left(\tau+\left\Vert \bm{M}^{\star}\right\Vert _{\infty}\right)\left\Vert \bm{V}^{\star}\right\Vert _{2,\infty}\\
V & \coloneqq\left\Vert \sum_{j=1}^{n}\mathbb{E}\left[\left(\delta_{l,j}\psi_{\tau}\left(M_{l,j}\right)-pM_{l,j}^{\star}\right)^{2}\right]\bm{V}_{j,\cdot}^{0,\left(l\right)}\bm{V}_{j,\cdot}^{0,\left(l\right)\top}\right\Vert \\
 & \leq\max_{j}\mathbb{E}\left[\left(\left(\delta_{l,j}-p\right)M_{l,j}^{\star}\ind_{\left|M_{l,j}\right|\leq\tau}+\delta_{l,j}\varepsilon_{l,j}\ind_{\left|M_{l,j}\right|\leq\tau}+\left(\tau\delta_{l,j}-pM_{l,j}^{\star}\right)\ind_{\left|M_{l,j}\right|>\tau}\right)^{2}\right]\left\Vert \sum_{j=1}^{n}\bm{V}_{j,\cdot}^{0,\left(l\right)}\bm{V}_{j,\cdot}^{0,\left(l\right)\top}\right\Vert \\
 & \lesssim\left(\mathbb{E}\left[\left(\delta_{l,j}-p\right)^{2}\left(M_{l,j}^{\star}\right)^{2}\right]+\mathbb{E}\left[\delta_{l,j}\varepsilon_{l,j}^{2}\right]+\mathbb{E}\left[\left(\tau\delta_{l,j}-pM_{l,j}^{\star}\right)^{2}\ind_{\left|M_{l,j}\right|>\tau}\right]\right)\left\Vert \bm{V}^{0,\left(l\right)}\right\Vert _{\mathrm{F}}^{2}\\
 & \lesssim p\left(\left\Vert \bm{M}^{\star}\right\Vert _{\infty}^{2}+\sigma^{2}\right)\left\Vert \bm{V}^{0,\left(l\right)}\right\Vert _{\mathrm{F}}^{2}.
\end{align*}
Here the last line is an application of Markov inequality due to 
\[
\mathbb{P}\left(\left|M_{l,j}\right|>\tau\right)\leq\mathbb{P}\left(\left|\varepsilon_{l,j}\right|>\tau-\left|M_{l,j}^{\star}\right|\right)\le\mathbb{P}\left(\left|\varepsilon_{l,j}\right|>\tau/2\right)\leq\frac{\sigma^{2}}{\left(\tau/2\right)^{2}}.
\]
We are now ready to apply the matrix Bernstein inequality \citep[Theorem 6.1.1]{tropp2015introduction}:
\begin{align*}
&\left\Vert \left(\bm{M}^{0}-\bm{M}^{0,\left(l\right)}\right)\bm{V}^{0,\left(l\right)}\right\Vert _{\mathrm{F}} \\
&\quad \lesssim\frac{1}{p}\left(\sqrt{V\log n}+L\log n\right)\\
 &\quad \lesssim\sqrt{\frac{\left(\left\Vert \bm{M}^{\star}\right\Vert _{\infty}^{2}+\sigma^{2}\right)}{p}\log n}\left\Vert \bm{V}^{0,\left(l\right)}\right\Vert _{\mathrm{F}}+\frac{\left(\tau+\left\Vert \bm{M}^{\star}\right\Vert _{\infty}\right)}{p}\left\Vert \bm{V}^{0,\left(l\right)}\right\Vert _{2,\infty}\log n\\
 &\quad \lesssim\sqrt{\frac{n\left(\left\Vert \bm{M}^{\star}\right\Vert _{\infty}^{2}+\sigma^{2}\right)}{p}\log n}\left\Vert \bm{V}^{0,\left(l\right)}\right\Vert _{2,\infty}+\frac{\left(\tau+\left\Vert \bm{M}^{\star}\right\Vert _{\infty}\right)}{p}\left\Vert \bm{V}^{0,\left(l\right)}\right\Vert _{2,\infty}\log n.
\end{align*}
\end{proof}
\subsection{Proof of Lemma \ref{lem:aux}}

For \eqref{subeq:aux-1}, one has 
\begin{align*}
 & \left\Vert \bm{F}^{t\left(l\right)}\bm{R}^{t,\left(l\right)}-\bm{F}^{\star}\right\Vert _{2,\infty}\\
 & \leq\left\Vert \bm{F}^{t\left(l\right)}\bm{R}^{t,\left(l\right)}-\bm{F}^{t}\bm{H}^{t}\right\Vert _{2,\infty}+\left\Vert \bm{F}^{t}\bm{H}^{t}-\bm{F}^{\star}\right\Vert _{2,\infty}\\
 & \leq\left\Vert \bm{F}^{t\left(l\right)}\bm{R}^{t,\left(l\right)}-\bm{F}^{t}\bm{H}^{t}\right\Vert _{\mathrm{F}}+\left\Vert \bm{F}^{t}\bm{H}^{t}-\bm{F}^{\star}\right\Vert _{2,\infty}\\
 & \overset{\text{(i)}}{\lesssim}\kappa^{3/2}\left(\frac{\sigma}{\sigma_{\min}}\sqrt{\frac{n}{p}}+\frac{\left\Vert \bm{M}^{\star}\right\Vert _{\infty}}{\sigma_{\min}}\sqrt{\frac{n}{p}}\right)\log n\left\Vert \bm{F}^{\star}\right\Vert _{2,\infty},
\end{align*}
where (i) follows from Lemma \ref{lem:loo-1} and \ref{lem:loo-2}.

Turning attention to \eqref{subeq:aux-2}, we have 
\begin{align*}
\left\Vert \bm{F}^{t\left(l\right)}\bm{R}^{t,\left(l\right)}-\bm{F}^{\star}\right\Vert  & \leq\left\Vert \bm{F}^{t\left(l\right)}\bm{R}^{t,\left(l\right)}-\bm{F}^{t}\bm{H}^{t}\right\Vert _{\mathrm{F}}+\left\Vert \bm{F}^{t}\bm{H}^{t}-\bm{F}^{\star}\right\Vert \\
 & \lesssim\sqrt{r}\left(\frac{\sigma}{\sigma_{\min}}\sqrt{\frac{n}{p}}+\frac{\left\Vert \bm{M}^{\star}\right\Vert _{\infty}}{\sigma_{\min}}\sqrt{\frac{n}{p}}\right)\left\Vert \bm{F}^{\star}\right\Vert ,
\end{align*}
where the second line makes use of Lemma \ref{lem:loo-1} and \ref{lem:contraction}.

Moreover, 
\begin{align*}
 & \left\Vert \bm{F}^{t,\left(l\right)}\bm{H}^{t,\left(l\right)}-\bm{F}^{\star}\right\Vert \\
 & \leq\left\Vert \bm{F}^{t,\left(l\right)}\bm{H}^{t,\left(l\right)}-\bm{F}^{t}\bm{H}^{t}\right\Vert _{\mathrm{F}}+\left\Vert \bm{F}^{t}\bm{H}^{t}-\bm{F}^{\star}\right\Vert \\
 & \leq5\kappa\left\Vert \bm{F}^{t,\left(l\right)}\bm{R}^{t,\left(l\right)}-\bm{F}^{t}\bm{H}^{t}\right\Vert _{\mathrm{F}}+\left\Vert \bm{F}^{t}\bm{H}^{t}-\bm{F}^{\star}\right\Vert \\
 & \lesssim\left(\frac{\sigma}{\sigma_{\min}}\sqrt{\frac{n}{p}}+\frac{\left\Vert \bm{M}^{\star}\right\Vert _{\infty}}{\sigma_{\min}}\sqrt{\frac{n}{p}}\right)\sqrt{r}\left\Vert \bm{F}^{\star}\right\Vert ,
\end{align*}
which is due to \eqref{subeq:aux-4}, Lemma \ref{lem:loo-1} and \ref{lem:contraction}.

\eqref{subeq:aux-4} can be proved analogously as \citet[Lemma 18]{chen2020noisy}
and thus omitted here for brevity.

Finally, one has 
\begin{align*}
\left\Vert \bm{F}^{t}\right\Vert  & \leq\left\Vert \bm{F}^{t}\bm{H}^{t}-\bm{F}^{\star}\right\Vert +\left\Vert \bm{F}^{\star}\right\Vert \\
 & \overset{\text{(i)}}{\leq}C\sqrt{\kappa}\left(\frac{\sigma}{\sigma_{\min}}\sqrt{\frac{n}{p}}+\frac{\left\Vert \bm{M}^{\star}\right\Vert _{\infty}}{\sigma_{\min}}\sqrt{\frac{n}{p}}\right)\left\Vert \bm{F}^{\star}\right\Vert _{\mathrm{F}}+\left\Vert \bm{F}^{\star}\right\Vert \\
 & \overset{\text{(ii)}}{\leq}2\left\Vert \bm{F}^{\star}\right\Vert ,
\end{align*}
where (i) is due to \eqref{eq:operator-norm}, and (ii) holds as long
as $\frac{\sigma}{\sigma_{\min}}\sqrt{\frac{n}{p}}\ll1/\sqrt{\kappa r}$
and $np\gg\mu^{2}\kappa^{3}r^{3}$. In a similar way, we can prove
\eqref{subeq:aux-5} and \eqref{subeq:aux-6}.

\subsection{Proof of Lemma \ref{lem:1}}

Recalling the definition of $\bm{E}$ (cf.~\eqref{defn-E}), one
has 
\begin{align*}
\left\Vert \bm{E}\right\Vert  & \leq\left\Vert \bm{E}-\mathbb{E}\left[\bm{E}\right]\right\Vert +\left\Vert \mathbb{E}\left[\bm{E}\right]\right\Vert _{\mathrm{F}}\\
 & \overset{\text{(i)}}{\lesssim}\widetilde{\sigma}\sqrt{n}+n\frac{\sigma^{2}}{\tau}\\
 & \overset{\text{(ii)}}{\lesssim}\sigma\sqrt{\frac{n}{p}}+\sqrt{\frac{n}{p}}\left\Vert \bm{M}^{\star}\right\Vert _{\infty},
\end{align*}
where (i) makes use of standard matrix tail bounds \citep[Theorem 3.1.4]{chen2021spectral}
and (ii) follows from the definition of $\widetilde{\sigma}$ in \eqref{eq:defn-sigmatilde}.
In addition, we note that Weyl's inequality implies 
\[
\sigma_{r}\left(\bm{M}^{0}\right)\geq\sigma_{\min}-\left\Vert \bm{E}\right\Vert \geq\frac{1}{2}\sigma_{\min}.
\]
Applying Wedin's sin$\bm{\Theta}$ Theorem \citep[Theorem 2.3.1]{chen2021spectral}
then reveals 
\[
\max\left\{ \left\Vert \bm{U}\bm{Q}_{\bm{U}}-\bm{U}^{\star}\right\Vert ,\left\Vert \bm{V}\bm{Q}_{\bm{V}}-\bm{V}^{\star}\right\Vert \right\} \leq\frac{\sqrt{2}\left\Vert \bm{E}\right\Vert }{\sigma_{r}\left(\bm{M}^{0}\right)-\sigma_{r+1}\left(\bm{M}^{\star}\right)}\lesssim\frac{\sigma}{\sigma_{\min}}\sqrt{\frac{n}{p}}+\frac{\left\Vert \bm{M}^{\star}\right\Vert _{\infty}}{\sigma_{\min}}\sqrt{\frac{n}{p}}.
\]

Next, we turn to control $\Vert\bm{H}_{\bm{V}}-\bm{Q}_{\bm{V}}\Vert$.
The definition of $\bm{H}_{\bm{V}}$ (cf.~\eqref{defn-Huv}) enables
us to write the SVD of $\bm{H}_{\bm{V}}$ as 
\[
\bm{H}_{\bm{V}}=\bm{L}_{1}\cos\bm{\Theta}\bm{L}_{2}^{\top},
\]
where $\bm{\Theta}$ is a diagonal matrix consisting of the principal
angles between the subspaces $\bm{V}$ and $\bm{V}^{\star}$. Then
\eqref{defn-Quv} gives $\bm{Q}_{\bm{V}}=\bm{L}_{1}\bm{L}_{2}^{\top}$
and it follows that

\begin{equation}
\left\Vert \bm{H}_{\bm{V}}-\bm{Q}_{\bm{V}}\right\Vert =\left\Vert \bm{L}_{1}\left(\cos\bm{\Theta}-\bm{I}_{r}\right)\bm{L}_{2}^{\top}\right\Vert =\left\Vert 2\sin^{2}\left(\bm{\Theta}/2\right)\right\Vert \lesssim\left\Vert \sin\bm{\Theta}\right\Vert ^{2}\lesssim\frac{n}{\sigma_{\min}^{2}}\frac{\left\Vert \bm{M}^{\star}\right\Vert _{\infty}^{2}+\sigma^{2}}{p},\label{eq:lem1-HQ}
\end{equation}
where the last inequality invokes Wedin's sin$\bm{\Theta}$ Theorem
\citep[Theorem 2.3.1]{chen2021spectral} again. The results of $\bm{H}_{\bm{U}}$
can be obtained similarly.

The last inequality is a direct consequence of \eqref{subeq:lem1-3}
\[
\frac{1}{2}\leq\left\Vert \bm{Q}_{\bm{U}}\right\Vert -\left\Vert \bm{H}_{\bm{U}}-\bm{Q}_{\bm{U}}\right\Vert \leq\left\Vert \bm{H}_{\bm{U}}\right\Vert \leq\left\Vert \bm{H}_{\bm{U}}-\bm{Q}_{\bm{U}}\right\Vert +\left\Vert \bm{Q}_{\bm{U}}\right\Vert \leq2.
\]

\section{Technical lemmas}

\begin{lemma}\label{lem:bernoulli}Suppose $\{g_{i,j}\}_{i,j}$ are
i.i.d. Bernoulli random variables with parameter $p$. Then with probability
at least $1-O(n^{-10})$, one has 
\[
\left\Vert \left\{ g_{i,j}\right\} _{i,j}\right\Vert \lesssim np.
\]

\end{lemma}
\begin{proof}
Triangle inequality gives 
\begin{align*}
\left\Vert \left\{ g_{i,j}\right\} _{i,j}\right\Vert  & \leq\left\Vert \left\{ g_{i,j}\right\} _{i,j}-p\bm{1}\bm{1}^{\top}\right\Vert +\left\Vert p\bm{1}\bm{1}^{\top}\right\Vert \\
 & \lesssim\sqrt{np}+np\lesssim np.
\end{align*}
where the second line makes use of \citet[Lemma 3.2]{keshavan2010matrix}
and holds provided that $np\gg1$.
\end{proof}
\begin{lemma}Assume the matrix $\bm{E}\coloneqq\{E_{i,j}\}_{i,j}\in\mathbb{R}^{n\times n}$
consists of independent random variables obeying that for any $1\leq i,j\leq n$,
\[
\mathbb{E}\left[E_{i,j}\right]=0,\qquad\mathbb{E}\left[E_{i,j}^{2}\right]=\sigma_{i,j}^{2}\leq\sigma^{2},\qquad\left|E_{i,j}\right|\leq B.
\]

\begin{enumerate}
\item With probability exceeding $1-O(n^{-10})$, one has 
\begin{equation}
\left\Vert \mathcal{P}_{\Omega}\left(\bm{E}\right)\right\Vert \lesssim\sigma\sqrt{n}+B\sqrt{\log n}.\label{lem:noise}
\end{equation}
\item For any fixed matrix $\bm{A}$, one has 
\begin{equation}
\left\Vert \bm{E}\bm{A}\right\Vert _{2,\infty}\lesssim\sigma\left\Vert \bm{A}\right\Vert _{\mathrm{F}}\sqrt{\log n}+B\left\Vert \bm{A}\right\Vert _{2,\infty}\log n,\label{lem:bernstein}
\end{equation}
with probability over $1-O(n^{-100})$. 
\end{enumerate}
\end{lemma}
\begin{proof}
\citet[Equation 3.9]{chen2021spectral} gives \eqref{lem:noise}.
This is the same as \citet[Lemma 5]{yan2021inference}.
\end{proof}

\end{document}